\documentclass[reqno,11pt]{amsart}

\usepackage[top=2.0cm,bottom=2.0cm,left=3cm,right=3cm]{geometry}
\usepackage{amsthm,amsmath,amssymb,dsfont}
\usepackage{mathrsfs,amsfonts,functan,extarrows,mathtools}
\usepackage[colorlinks]{hyperref}

\usepackage{marginnote}
\usepackage{xcolor}
\usepackage{stmaryrd}
\usepackage{esint}
\usepackage{graphicx}
\usepackage{bm}

\newtheorem{theorem}{Theorem}[section]
\newtheorem{proposition}{Proposition}[section]

\newtheorem{lemma}{Lemma}[section]

\numberwithin{equation}{section}
\allowdisplaybreaks

\def\d{\mathrm{d}}

\def\R{\mathbb{R}}
\def\T{\mathbb{R}}

\def\eps{\varepsilon}
\def\div{\mathrm{div}}

\newcounter{wronumber}

\begin{document}
\title[From VMB to INSFM with Ohm's law]
			{The incompressible Navier-Stokes-Fourier-Maxwell system limits of the Vlasov-Maxwell-Boltzmann system for soft potentials: the noncutoff cases and  cutoff cases}

\author[Ning Jiang]{Ning Jiang}
\address[Ning Jiang]{\newline School of Mathematics and Statistics, Wuhan University, Wuhan, 430072, P. R. China}
\email{njiang@whu.edu.cn}

\author[Yuanjie Lei]{Yuanjie Lei${}^*$}
\address[Yuanjie Lei]
{\newline School of Mathematics and Statistics, Huazhong University of Sciences and Technology, Wuhan, 430074, P. R. China}
\email{leiyuanjie@hust.edu.cn}

\thanks{${}^*$ Corresponding author \quad \today}

\maketitle

\begin{abstract}
   We obtain the global-in-time and uniform in Knudsen number $\eps$ energy estimate for  the cutoff and non-cutoff scaled Vlasov-Maxwell-Boltzmann system for the soft potential. For the non-cutoff soft potential cases, our analysis relies heavily on additional dissipative mechanisms with respect to velocity, which are brought about by the strong angular singularity hypothesis, i.e. $\frac12\leq s<1$. In the case of cutoff cases, our proof relies on two new kinds of weight functions and complex construction of energy functions, and here we ask $\gamma\geq-1$. As a consequence, we justify the incompressible Navier-Stokes-Fourier-Maxwell equations with Ohm's law limit.\\



\end{abstract}



\tableofcontents


\section{Introduction.}\label{Sec:Introduction}

\subsection{The scaled VMB system}
The two-species Vlasov-Maxwell-Boltzmann (VMB) system describes the evolution of a gas of two species of oppositely charged particles (cations of charge $q^+ > 0$ and mass $m^+>0$, and anions of charge $-q^- <0$ and mass $m^->0$), subject to auto-induced electromagnetic forces.
Such a gas of charged particles equipped with a global neutrality condition is usually called a plasma. The unknowns $F^+(t,x,v) \geq 0 $ and $F^-(t,x,v) \geq 0 $ represent respectively the particle number densities of the positive charged ions (i.e., cations), and the negative charged ions (i.e., anions), which are at position $x\in \mathbb{R}^3$ with velocity $v\in \mathbb{R}^3$, at time $t\geq 0$. The VMB system reads as follows:
	\begin{align}\label{VMB-0}
	\left\{\begin{array}{rcl}
	  \partial_t F^+ + v\cdot \nabla_x F^+ + \tfrac{q^+}{m^+}(E + v \times B)\cdot\nabla_v F^+ & = &  Q(F^+, F^+) + Q(F^+, F^-)\,,
	  \\[5pt]
	  	\partial_t F^- + v\cdot \nabla_x F^- - \tfrac{q^-}{m^-}(E + v \times B)\cdot\nabla_v F^- & = & Q(F^-, F^-) + Q(F^-, F^+)\,,
	  \\[5pt] \displaystyle
	  	\mu_0\eps_0\partial_t E - \nabla_x \times B & = & -\mu_0\int_{\mathbb{R}^3}(q^+F^+- q^- F^-)v\,\mathrm{d}v\,,
	  \\[5pt]
	    \partial_t B + \nabla_{\!x} \times E & = & 0\,,
	  \\[5pt] \displaystyle
	  	\div_x E & = & \tfrac{1}{\eps_0}\int_{\mathbb{R}^3}(q^+F^+- q^- F^-)\,\mathrm{d}v\,,
	  \\[5pt]
	  	\div_x B & = & 0\,.
	\end{array}\right.
	\end{align}
Here the evolution of the densities $F^\pm$ are governed by the Vlasov-Boltzmann equations in the first two equations of \eqref{VMB-0}. This means that the variations of densities along the particle trajectories are subject to the influence of an auto-induced Lorentz force and inter-particle collisions in the gas. The electric field $E(t,x)$ and the magnetic field $B(t,x)$, which are generated by the motion of the particles in the plasma itself, are governed by the Maxwell equation. It consists of the Amp\`ere equation, Faraday's equation and Gauss' laws, representing in the third, fourth, and the last two equations, respectively.
The vacuum permeability and permittivity (or say, the magnetic and electric constants) are denoted, respectively, by the physical coefficients $\mu_0, \eps_0 >0$. Both species of particles are assumed here to have the same mass $m^\pm=m>0$ and charge $q^\pm = q >0$.

The Boltzmann collision operator, presented in the right-hand sides of the Vlasov-Boltzmann equations in \eqref{VMB-0}, is a quadratic form, acting on the velocity variables, associated to the bilinear operator $Q(F,G)(v)$. Before stating the results, we introduce the models, in particular the formats of the collision kernels. Since we treat both non-cutoff and cutoff kernels, for which the convenient representations are little different, we describe it separately in the following.

\subsubsection{Non-cutoff cases}

The Boltzmann collision operator $Q(F,G)(v)$ is given by
\begin{equation*}
	Q(F,G)(v)
	=\int_{\mathbb{R}^3\times \mathbb{S}^2}{\bf B}(v-u,\sigma)\{F(u')G(v')-F(u)G(v)\}
	\d\sigma \d u,
\end{equation*}
where in terms of velocities $v$ and $u$ before the collision, velocities  $v'$ and $u'$ after the collision are defined by
\begin{equation*}
	v'=\frac{v+u}{2}+\frac{|v-u|}{2}\sigma,\ \ \ \ u'=\frac{v+u}{2}-\frac{|v-u|}{2}\sigma,\quad \sigma\in \mathbb{S}^2.
\end{equation*}

The Boltzmann collision kernel ${\bf B}(v-u,\sigma)$ depends only on the relative velocity $|v-u|$ and on the
deviation angle $\theta$ given by $\cos\theta =\langle\sigma,\ (v-u)/{|v-u|}\rangle$, where $\langle\cdot, \cdot\rangle$ is the usual dot product in $\mathbb{R}^3$. Without loss of generality, we suppose that ${\bf B}(v-u,\sigma)$ is supported on $\cos\theta\geq 0$.

Throughout the paper, the collision kernel ${\bf B}(v-u,\sigma)$ is further supposed to satisfy the following assumptions:
\begin{itemize}
	\item[(A1).] ${\bf B}(v-u,\sigma)$ takes the product form in its argument as
	$$
	{\bf B}(v-u,\sigma)=\Phi(|v-u|){\bf b}(\cos\theta)
	$$
	with $\Phi$ and ${\bf b}$ being non-negative functions;
	
	\item[(A2).] The angular function $\sigma\rightarrow {\bf b}(\langle\sigma,(v-u)/|v-u|\rangle)$ is not integrable on ${\mathbb{S}}^2$, i.e.
	\[
	\int_{{\mathbb{S}}^2}{\bf b}(\cos\theta)\d\sigma=2\pi\int_0^{\pi/2}\sin\theta {\bf b}(\cos\theta)\d\theta=\infty.
	\]
	Moreover, there are two positive constants $c_b>0, 0<s<1$ such that
	\[
	\frac{c_b}{\theta^{1+2s}}\leq\sin\theta {\bf b}(\cos\theta)\leq \frac{1}{c_b\theta^{1+2s}};
	\]
	
	\item[(A3).] The kinetic function $z\rightarrow \Phi(|z|)$ satisfies
	\[
	\Phi(|z|)=C_\Phi|z|^\gamma
	\]
	for some positive constant $C_\Phi> 0.$ Here we should notice that the exponent $\gamma>-3$ is determined by the intermolecular interactive mechanism.
\end{itemize}

Usually, we call ${\bf B}(v-u,\sigma)$ as hard potentials collision kernels when $\gamma+2s\geq0$, and
soft potentials when $-3<\gamma<-2s$ with $0<s<1$. The current work is restricted to the case of
\begin{equation}\label{case}
	\max\left\{-3,-\tfrac32-2s\right\}<\gamma<-2s, \ \ \tfrac12\leq s<1\,.
\end{equation}
We call the case \eqref{case} with $\frac12\leq s<1$ as strong angular singularity as treated in \cite{DLYZ-KRM2013}, while for $0<s<\tfrac12$ weak angular singularity as in \cite{FLLZ-2018}.
\subsubsection{Cutoff cases}
For the cutoff collision kernels, $Q(F,G)(v)$ is defined as
defined as
\begin{equation*}
	\begin{split}
		Q(F,G)(v)
		=&\int_{\mathbb{R}^3\times \mathbb{S}^2}|u-v|^{\gamma}{\bf b}\left(\frac{\omega\cdot(v-u)}{|u-v|}\right)\left\{F(v')G(u')-F(v)G(u)\right\}
		\d\omega \d u\\
		\equiv&Q_{gain}(F,G)-Q_{loss}(F,G).
	\end{split}
\end{equation*}
Here $\omega\in\mathbb{S}^2$ and ${\bf b}$, the angular part of the collision kernel, satisfies the Grad angular cutoff assumption (cf. \cite{Grad-1963})
\begin{equation}\label{cutoff-assump}
	0\leq{\bf b}(\cos \vartheta)\leq C|\cos\vartheta|
\end{equation}
for some positive constant $C>0$. Here the deviation angle $\theta$ is given by $\cos\vartheta =\omega\cdot(v-u)/{|v-u|}$.
Moreover,
\begin{equation*}
	v'=v-[(v-u)\cdot\omega]\omega,\quad u'=u+[(v-u)\cdot\omega]\omega,
\end{equation*}
which denote velocities after a collision of particles having velocities $v, u$ before the collision and vice versa.

The exponent $\gamma\in(-3,1]$ in the kinetic part of the collision kernel is determined by the potential of intermolecular forces, which is classified into the soft potential case for $-3<\gamma<0$, the Maxwell molecular case for $\gamma=0$, and the hard potential case for $0<\gamma\leq 1$ which includes the hard sphere model with $\gamma=1$ and ${\bf b}(\cos\vartheta)=C|\cos\vartheta|$ for some positive constant $C>0$. Here we focus on the cutoff case $-1\leq \gamma<0$.

There have been extensive research on the well-posedness of the VMB. In late 80's, DiPerna and Lions developed the theory of global-in-time renormalized solutions with large initial data to the Boltzmann equation \cite{DL-Annals1989}, Vlasov-Maxwell equations \cite{DL-CPAM1989} and Vlasov-Poisson-Boltzmann equations \cite{Lions-Kyoto1994, Lions-Kyoto1994-2}. But for VMB there are severe  difficulties, among which the major one is that the {\em a priori} bounds coming from physical laws are not enough to prove the existence of global solutions, even in the renormalized sense. Recently, Ars\`enio and Saint-Raymond \cite{Arsenio-SRM-2016} eventually established global-in-time renormalized solutions with large initial data for VMB, both cut-off and non-cutoff collision kernels. We emphasize that for Boltzmann type equations, renormalized solutions are still the only global-in-time solutions without any smallness requirements on initial data. On the other line, in the context of classical solutions, through the so-called nonlinear energy method, Guo \cite{Guo-2003-Invent} constructed the classical solution of VMB near the global Maxwellian. Guo's work inspired sequences of follow-up research on VMB with more general collision kernels among which we only mention results for the most general collision kernels with or without angular cutoff assumptions, see \cite{DLYZ-KRM2013, DLYZ-CMP2017, FLLZ-2018}.

\subsection{Hydrodynamic limits of Boltzmann type equations}
One of the most important features of the Boltzmann equations (or more generally, kinetic equations) is their connections to the fluid equations. The so-called fluid regimes of the Boltzmann equations are those of asymptotic dynamics of the scaled Boltzmann equations when the Knudsen number $\eps$ is very small. Justifying these limiting processes rigorously has been an active research field from late 70's. Among many results obtained, the main contributions are the incompressible Navier-Stokes and Euler limits. There are two types of results in this field:

{\bf (Type-I)}, First obtaining the solutions of the scaled Boltzmann equation with {\em uniform} bounds in the Knudsen number $\eps$, then extracting a subsequence converging (at least weakly) to the solutions of the fluid equations as $\eps\rightarrow 0$.

{\bf (Type-II)}, First obtaining the solutions for the limiting fluid equations, then constructing a sequence of special solutions (near the Maxwellians) of the scaled Boltzmann equations for small Knudsen number $\eps$.

The key difference between the results of Type-I and  type-II are: in type-I, the solutions of the fluid equations are {\em not} known a priori, and are completely obtained from taking limits from the Boltzmann equations. In short, it is ``from kinetic to fluid"; In type-II, the solutions of the fluid equations are {\em known} first. In short, it is ``from fluid to kinetic". Usually, type-I results are harder to be achieved since it must obtain enough uniform (with respect to $\eps$) bounds for solutions of the scaled kinetic equations then compactness arguments give the convergence. This approach automatically provides the existence of both the original kinetic equations and the limiting macroscopic equations. On the other hand, type-II approach needs to employ the information of the limiting equations, then prove the solutions of the original kinetic equations with a {\em special} form (usually Hilbert expansion).

We remark that this classification of two approaches also appears in other asymptotical problems. For example, in their work on Kac's program \cite{Mischler-Mouhot}, Mischler and Mouhot called the type-I as ``bottom-up" and type-II as ``top-down". We quote their words here: {\em ...our answer is an ``inverse" answer inthe sense that our methodology is ``top-down" from the limit equation to the many-particle system rather than ``bottom-up" as was expected by Kac.}

The most successful achievement in type-I is the so-called BGL (named after Bardos, Golse and Levermore) program.  From late 80's, Bardos, Golse and Levermore initialized the program to justify Leray's solutions of the incompressible Navier-Stokes equations from DiPerna-Lions' renormalized solutions \cite{BGL-1991-JSP, BGL-CPAM1993}. They proved the first convergence result with five additional technical assumptions. After ten years efforts by Bardos, Golse, Levermore, Lions, Masmoudi and Saint-Raymond, see for example \cite{BGL3, LM3, LM4, GL}, the first complete convergence result without any additional compactness assumption was proved by Golse and Saint-Raymond in \cite{G-SRM-Invent2004} for cutoff Maxwell collision kernel, and in \cite{G-SRM2009} for hard cutoff potentials. Later on, it was extended by Levermore and Masmoudi \cite{LM} to include soft potentials. Recently Ars\'enio got the similar results for non-cutoff case \cite{Arsenio, Arsenio-Masmoudi} based on the existence result of \cite{AV-CPAM2002}. Furthermore, by Jiang, Levermore, Masmoudi and Saint-Raymond, these results were extended to bounded domain where the Boltzmann equation was endowed with the Maxwell reflection boundary condition \cite{Masmoudi-SRM-CPAM2003, JLM-CPDE2010, JM-CPAM2017}, based on the solutions obtained by Mischler \cite{mischler2010asens}.

Another direction in type-I is in the context of classical solutions. The first work in this type is Bardos-Ukai \cite{b-u}. They started from the scaled Boltzmann equation for cut-off hard potentials, and proved the global existence of classical solutions $g_\eps$ uniformly in $0< \eps <1$. The key feature of Bardos-Ukai's work is that they only need the smallness of the initial data, and did not assume the smallness of the Knudsen number $\eps$ in uniform estimate. After having the uniform in $\eps$ solutions $g_\eps$, taking limits can provide a classical solution of the incompressible Navier-Stokes equations with small initial data. Bardos-Ukai's approach heavily depends on the sharp estimates from the spectral analysis on the linearized Boltzmann operator $\mathcal{L}$, and the semigroup method (the semigroup generated by the scaled linear operator $\eps^{-2}\mathcal{L}+\eps^{-1}v\cdot\nabla_{\!x}$). It seems that it is hardly extended to cutoff soft potential, and even harder for the non-cutoff cases, since it is well-known that the operator $\mathcal{L}$ has continuous spectrum in those cases. On the torus, semigroup approach was used by Briant \cite{Briant-JDE-2015} and Briant, Merino-Aceituno and Mouhot \cite{BMM-arXiv-2014} to prove incompressible Navier-Stokes limit by employing the functional analysis breakthrough of Gualdani-Mischler-Mouhot \cite{GMM}. Again, their results are for cut-off kernels with hard potentials. Recently, there is type-I convergence result on the incompressible Navier-Stokes limit of the Boltzmann equation by the first author and his collaborators. In \cite{JXZ-Indiana2018}, the uniform in $\eps$ global existence of the Boltzmann equation with or without cutoff assumption was obtained and the global energy estimates were established.

Most of the type-II results are based on the Hilbert expansion  and obtained in the context of classical solutions. It was started from Nishida and Caflisch's work on the compressible Euler limit \cite{Nishida, Caflisch, KMN}. Their approach was revisitied by Guo, Jang and Jiang, combining with nonlinear energy method to apply to the acoustic limit \cite{GJJ-KRM2009, GJJ-CPAM2010, JJ-DCDS2009}. After then this process was used for the incompressible limits, for examples, \cite{DEL-89} and \cite{Guo-2006-CPAM}. In \cite{DEL-89}, De Masi-Esposito-Lebowitz considered  Navier-Stokes limit in dimension 2. More recently, using the nonlinear energy method, in \cite{Guo-2006-CPAM} Guo justified the Navier-Stokes limit (and beyond, i.e. higher order terms in Hilbert expansion). This was extended in \cite{JX-SIMA2015} to more general initial data which allow the fast acoustic waves.  These results basically say that, given the initial data which is needed in the classical solutions of the Navier-Stokes equations, it can be constructed the solutions of the Boltzmann equation of the form $F_\eps = M +  \sqrt{M}(\eps g_1+ \eps^2 g_2 + \cdots + \eps^n g_n+ \eps^k g^R_\eps)$, where $g_1, g_2, \cdots $ can be determined by the Hilbert expansion, and $g^R_\eps$ is the error term. In particular, the first order fluctuation $g_1 = \rho_1 + \mathrm{u}_1\!\cdot\! v + \theta_1(\frac{|v|^2}{2}-\frac{3}{2})$, where $(\rho_1, \mathrm{u}_1, \theta_1)$ is the solutions to the incompressible Navier-Stokes equations.

Regarding to the Vlasov-Poisson-Boltzmann (VPB) system and VMB, the corresponding fluid limits are more fruitful since the effects of electric and magnetic fields are considered. Analytically, the limits from scaled VPB are similar to the Boltzmann equations because VPB couples with an extra Poisson equation, which has very good enough regularity. This usually does not bring much difficulties. For the limits of VPB, see recent results \cite{GJL-2019, JZ-2020}.

However, for the VMB, the situation is quite different. The corresponding hydrodynamic limits are much harder, even at the formal level, since it is coupled with Maxwell equations which are essentially hyperbolic. In a recent remarkable breakthrough \cite{Arsenio-SRM-2016}, Ars\'enio and Saint-Raymond not only proved the existence of renormalized solutions of VMB,  more importantly, also justified various limits (depending on the scalings) towards incompressible viscous electro-magneto-hydrodynamics. Among these limits, the most singular one is to the two-fluid incompressible Navier-Stokes-Fourier-Maxwell (in brief, NSFM) system with Ohm's law.

The proofs in \cite{Arsenio-SRM-2016} for justifying the weak limit from a sequence of solutions of VMB  to a dissipative solution of incompressible NSFM  are extremely hard. Part of the reasons are, besides many difficulties of the existence of renormalized solutions of VMB itself, the current understanding for the incompressible NSFM with Ohm's law is far from complete. From the point view of mathematical analysis, NSFM have the behavior which is more similar to the much less understood incompressible Euler equations than to the Navier-Stokes equations. That is the reason in \cite{Arsenio-SRM-2016}, they consider the so-called dissipative solutions of NSFM rather than the usual weak solutions. The dissipative solutions were introduced by Lions for 3-dimensional incompressible Euler equations (see section 4.4 of \cite{Lions-1996}).

The studies of incompressible NSFM just started in recent years (for the introduction of physical background, see \cite{Biskamp, Davidson}).  For weak solutions, the existence of global-in-time Leray type weak solutions are completely open, even in 2-dimension.  The first breakthrough comes from Masmoudi \cite{Masmoudi-JMPA2010}, who proved in 2-dimensional case the existence and uniqueness of global strong solutions of incompressible NSFM (in fact, the system he considered in \cite{Masmoudi-JMPA2010} is little different with the NSFM in this paper, but the analytic analysis are basically the same) for the initial data $(v^{in}, E^{in}, B^{in})\in L^2(\mathbb{R}^2)\times (H^s(\mathbb{R}^2))^2$ with $s>0$. It is notable that in \cite{Masmoudi-JMPA2010}, the divergence-free condition of the magnetic field $B$ or the decay property of the linear part coming from Maxwell's equations is {\em not} used.  Ibrahim and Keraani \cite{Ibrahim-Keraani-2011-SIMA} considered the data $(v^{in}, E^{in}, B^{in}) \in \dot B^{1/2}_{2,1}(\mathbb{R}^3)\times (\dot H^{1/2}(\mathbb{R}^3))^2$ for 3-dimensional, and $(v_0, E_0, B_0)\in \dot B^0_{2,1}(\mathbb{R}^2)\times (L^2_{log}(\mathbb{R}^2))^2$ for 2-dimensional case. Later on, German, Ibrahim and Masmoudi \cite{GIM2014} refines the previous results by running a fixed-point argument to obtain mild solutions, but taking the initial velocity field in the natural Navier-Stokes space $H^{1/2}$. In their results the regularity of the initial velocity and electromagnetic fields is lowered. Furthermore, they employed an $L^2L^\infty$-estimate on the velocity field, which significantly simplifies the fixed-point arguments used in \cite{Ibrahim-Keraani-2011-SIMA}. For some other asymptotic problems related, say, the derivation of the MHD from the Navier-Stokes-Maxwell system in the context of weak solutions, see Ars\'enio-Ibrahim-Masmoudi \cite{AIM-ARMA-2015}. Recently, in \cite{JL-CMS-2018} the authors of the current paper proved the global classical solutions of the incompressible NSFM with small intial data, by using the decay properties of both the electric field and  the wave equation with linear damping of the divergence free magnetic field. This key idea was already used in \cite{GIM2014}.

Regarding to the hydrodynamic limits of VMB in the context of classical solutions, the only previous result belongs to Jang \cite{Jang-ARMA2008}. In fact, in \cite{Jang-ARMA2008}, it was taken a very special scaling under which the magnetic effect appeared only at a higher order. As a consequence, it vanished in the limit as the Knudsen number $\eps\rightarrow 0$. So in the limiting equations derived in \cite{Jang-ARMA2008}, there was no equations for the magnetic field at all. We emphasize that in \cite{Jang-ARMA2008}, the author took the Hilbert expansion approach, and the classical solutions to the VMB were based on those of the limiting equations. So the convergence results in \cite{Jang-ARMA2008} belong to the type-II results, as we named in the last subsection.

The main purpose of this paper is to obtain the type-I results for the hydrodynamic limits to Cauchy problem of the Vlasov-Maxwell-Boltzmann system \eqref{VMB-F}-\eqref{IC-VMB-F} for the soft potentials, including noncutoff cases and cutoff cases, when the Knudsen number $\eps$ tends to zero. The key issue is that we do {\em not} employ the Hilbert expansion method, which as mentioned above, has two disadvantages: first, it only gives a special type of the solution of the VMB. In other words, the solution with the expansion form of some finite terms. Second, the solution to the limiting equations must be known {\em before} the existence of VMB. The approach employed in this paper is to obtain a family of the global in time solutions $F_\eps^\pm (t,x,v)$ to the scaled VMB with energy estimate uniform in $0< \eps < 1$. Based on this uniform energy estimate, the moments of the fluctuations of $F_\eps^\pm (t,x,v)$ around the global Maxwellian converge to the solutions to the incompressible Navier-Stokes-Fourier-Maxwell (NSFM) equations. This approach automatically provides a classical solution to the NSFM equations. The first named author of this paper and Luo did this for noncutoff VMB with hard sphere collision kernel in \cite{Jiang-Luo-2022-Ann.PDE}. This paper treats the technically much harder general soft potentials cases, both noncutoff and cutoff collision kernels.

\subsubsection{Incompressible NSFM limit of VMB}
To obtain the incompressible NSFM equations, formally, we follow the scales settled in \cite{Arsenio-SRM-2016}. More specifically, we consider the following scaled two-species VMB system:
\begin{equation}\label{VMB-F}
	\left\{
	\begin{array}{l}
		\partial_t F_\eps^\pm + \tfrac{1}{\eps} v \cdot \nabla_x F_\eps^\pm \pm \tfrac{1}{\eps} ( \eps E_\eps + v \times B_\eps ) \cdot \nabla_v F_\eps^\pm = \tfrac{1}{\eps^2} Q (F_\eps^\pm, F_\eps^\pm) + \tfrac{1}{\eps^2} Q (F_\eps^\pm , F_\eps^\mp) \,,\\[2mm]
		\partial_t E_\eps - \nabla_x \times B_\eps = - \tfrac{1}{\eps^2} \int_{\R^3} (F_\eps^+ - F_\eps^-) v \d v \,,\\[2mm]
		\partial_t B_\eps + \nabla_x \times E_\eps = 0 \,,\\[2mm]
		\div_x E_\eps = \tfrac{1}{\eps} \int_{\R^3} ( F_\eps^+ - F_\eps^- ) \d v \,,\\[2mm]
		\div_x B_\eps = 0
	\end{array}
	\right.
\end{equation}
with initial data
\begin{equation}\label{IC-VMB-F}
	\begin{aligned}
		F_\eps^\pm (0, x, v) = F_\eps^{\pm, in} (x,v) \in \R \,, \quad E_\eps (0,x) = E_\eps^{in} (x) \in \R^3 \,, \quad B_\eps (0, x) = B_\eps^{in} (x) \in \R^3 \,.
	\end{aligned}
\end{equation}

It is well-known that the global equilibrium for the two-species VMB is $[{M(v)},{M}(v)]$, where the normalized global {\em Maxwellian} is
\begin{equation}\label{GM}
 {M}(v) = \tfrac{1}{(2 \pi)^\frac{3}{2}} \exp(- \tfrac{|v|^2}{2}) \,.
\end{equation}
Set
$$F_\eps^\pm (t,x,v) = {M}(v)  + \eps \sqrt{{M(v)}} f_\eps^\pm (t,x,v),$$
this leads to the perturbed two-species VMB
\begin{equation}\label{VMB-F-perturbative}
  \left\{
    \begin{array}{l}
      \partial_t f_\eps + \tfrac{1}{\eps} \big[ v \cdot  \nabla_x f_\eps + q_0 ( \eps E_\eps + v \times B_\eps ) \cdot \nabla_v f_\eps \big] + \tfrac{1}{\eps^2} \mathscr{L} f_\eps - \tfrac{1}{\eps}  (E_\eps \cdot v) \sqrt{{M}}  q_1\\[2mm]
       \qquad \qquad = \tfrac{1}{2} q_0 (E_\eps \cdot v) f_\eps + \tfrac{1}{\eps} \mathscr{T} (f_\eps, f_\eps) \,,\\[2mm]
      \partial_t E_\eps - \nabla_x \times B_\eps = - \tfrac{1}{\eps} \int_{\R^3} f_\eps \cdot {q_1} v \sqrt{{M}} \d v \,,\\[2mm]
      \partial_t B_\eps + \nabla_x \times E_\eps = 0 \,,\\[2mm]
      \div_x E_\eps = \int_{\R^3} f_\eps \cdot {q_1} \sqrt{{M}} \d v \,, \ \div_x B_\eps = 0 \,,
    \end{array}
  \right.
\end{equation}
{where $f_\eps = [ f_\eps^+ , f_\eps^- ]$ represents the vector in $\R^2$ with the components $f_\eps^\pm$, the $2 \times 2$ diagonal matrix $q_0 = \textit{diag} (1, -1)$, the vector $q_1 = [1, -1]$, the linearized collision operator $\mathscr{L}f_\varepsilon$ and the nonlinear collision term $\mathscr{T} (f_\varepsilon,f_\varepsilon)$ are respectively defined by
\[\mathscr{L}f_\varepsilon=[\mathscr{L}_+f_\varepsilon,\mathscr{L}_-f_\varepsilon],\quad\quad\quad\mathscr{T}(f_\varepsilon,g)=[\mathscr{T}_+(f_\varepsilon,g),\mathscr{T}_-(f_\varepsilon,g)]\]
with
\begin{equation*}
\begin{aligned}
\mathscr{L}_\pm f_\varepsilon =& -{M}^{-1/2}
\left\{{Q\left( {M},{M}^{1/2}(f_\varepsilon^\pm+f_\varepsilon^\mp)\right)+ 2Q\left( {M}^{1/2}f_\varepsilon^\pm, {M}\right)}\right\},\\[2mm]
\mathscr{T}_{\pm}(f_\varepsilon,g) =&{M}^{-1/2}Q\left({M}^{1/2}f_\varepsilon^{\pm},{M}^{1/2}g^\pm\right)
+{M}^{-1/2}Q\left({M}^{1/2}f_\varepsilon^{\pm},{M}^{1/2}g^\mp\right).
\end{aligned}
\end{equation*}
For the linearized Boltzmann collision operator $\mathscr{L}$, it is well known, cf. \cite{Guo-2003-Invent, DLYZ-CMP2017}, that it is non-negative and the null space $\mathcal{N}$ of $\mathscr{L}$ is given by
\begin{equation*}
{\mathcal{ N}}={\textrm{span}}\left\{[1,0]{M}^{1/2} , [0,1]{M}^{1/2}, [v_i,v_i]{{M}}^{1/2} (1\leq i\leq3),[|v|^2,|v|^2]{M}^{1/2}\right\}.
\end{equation*}
If we define ${\bf P}$ as the orthogonal projection from $L^2({\mathbb{R}}^3_ v)\times L^2({\mathbb{R}}^3_ v)$ to $\mathcal{N}$, then for any given function $f(t, x, v )\in L^2({\mathbb{R}}^3_ v)$, one has
\begin{equation*}
  {\bf P}f_\varepsilon =\left\{{\rho_\varepsilon^+(t, x)[1,0]+\rho_\varepsilon^-(t, x)[0,1]+\sum_{i=1}^{3}u_\varepsilon^{i}(t, x) [1,1]v_i+\theta_\varepsilon(t, x)[1,1](| v|^2-3)}\right\}{M}^{1/2}
\end{equation*}
with
\begin{equation*}
  \rho_\varepsilon^\pm=\int_{{\mathbb{R}}^3}{M}^{1/2}f_\varepsilon^\pm \d v,\quad
  u_{\varepsilon,i}=\frac12\int_{{\mathbb{R}}^3} v  _i {M}^{1/2}(f_\varepsilon^++f_\varepsilon^-)\d v,\quad
  \theta_\varepsilon=\frac{1}{12}\int_{{\mathbb{R}}^3}(| v|^2-3){M}^{1/2}(f_\varepsilon^++f_\varepsilon^-) \d v.
\end{equation*}
Therefore, we have the following macro-micro decomposition with respect to the given global Maxwellian ${M}$ which was introduced in \cite{Guo-IUMJ-2004}
\begin{equation}\label{macro-micro}
 f_\varepsilon(t,x, v)={\bf P}f_\varepsilon(t,x, v)+\{{\bf I}-{\bf P}\}f_\varepsilon(t, x, v)
\end{equation}
where ${\bf I}$ denotes the identity operator and ${\bf P}f_\varepsilon$ and $\{{\bf I}-{\bf P}\}f_\varepsilon$ are called the macroscopic and the microscopic component of $f_\varepsilon(t,x,v)$, respectively.}

{

Using the moment method, Ars\'enio and Saint-Raymond in \cite{Arsenio-SRM-2016} proved the following limit:
  \begin{equation}\label{NSFM-Limit}
    \begin{aligned}
      f_\eps \rightarrow ( \rho + \tfrac{1}{2} n ) \tfrac{{q_1} + {q_2}}{2} \sqrt{{M}} + ( \rho - \tfrac{1}{2} n ) \tfrac{{q_2} - {q_1}}{2} \sqrt{{M}} + u \cdot v {q_2} \sqrt{{M}} + \theta ( \tfrac{|v|^2}{2} - \tfrac{3}{2} )q_2 \sqrt{{M}}\,,
    \end{aligned}
  \end{equation}
where $(\rho, n, u, \theta, E, B)$ satisfies the following two-fluid incompressible NSFM with Ohm's law:
\begin{equation}\label{INSFM-Ohm}
  \left\{
    \begin{array}{l}
      \partial_t u + u \cdot \nabla_x u - {\mu} \Delta_x u + \nabla_x p = \tfrac{1}{2} ( n E + j \times B ) \,, \qquad \div_x \, u = 0 \,, \\ [2mm]
      \partial_t \theta + u \cdot \nabla_x \theta - \kappa \Delta_x \theta = 0 \,, \qquad\qquad\qquad\qquad\qquad\quad\ \, \rho + \theta = 0 \,, \\ [2mm]
      \partial_t E - \nabla_x \times B = - j \,, \qquad\qquad\qquad\qquad\qquad\qquad\ \ \ \, \div_x \, E = n \,, \\ [2mm]
      \partial_t B + \nabla_x \times E = 0 \,, \qquad\qquad\qquad\qquad\qquad\qquad\qquad \div_x \, B = 0 \,, \\ [2mm]
      \qquad \qquad j - nu = \sigma \big( - \tfrac{1}{2} \nabla_x n + E + u \times B \big) \,, \quad\quad\quad\quad\quad w = \tfrac{3}{2} n \theta \,,
    \end{array}
  \right.
\end{equation}
where the vectors $q_1 = [1, -1]$, $q_2 = [1, 1]$, for a detailed definition of the viscosity $\mu$, heat conductivity $\kappa$, and electrical conductivity $\sigma$, please refer to \cite{Arsenio-SRM-2016} for their derivation.
We will not give detailed formal derivation here. In our proof of Theorem \ref{Main-Thm-1}, we indeed provide how the NSFM can be derived from VMB with soft potential, including the noncutoff cases and cutoff cases.

\subsection{Notations}
\begin{itemize}
	\item For convention, we index the usual $L^p$ space by the name of the concerned variable. So we have, for $p \in [1, + \infty]$,
	\begin{equation*}
		L^p_{[0,T]} = L^p ([0,T]) \,, \ L^p_x = L^p (\T^3) \,, \ L^p_v = L^p (\R^3) \,, \ L^p_{x,v} = L^p (\mathbb{R}^3 \times \mathbb{R}^3) \,.
	\end{equation*}
	{For $p = 2$, we use the notations $ ( \cdot \,, \cdot )_{L^2_x} $, $ \langle \cdot \,, \cdot \rangle_{L^2_v} $ and $ ( \cdot \,, \cdot )_{L^2_{x,v}} $ to represent the inner product on the Hilbert spaces $L^2_x$, $L^2_v$ and $L^2_{x,v}$, respectively;}
\item $
  \langle \cdot \rangle = \sqrt{1 + |\cdot|^2} \,;
$
\item
For any multi-indexes $\alpha = (\alpha_1, \alpha_2, \alpha_3)$ and $ \beta = ( \beta_1, \beta_2, \beta_3 )$ in $\mathbb{N}^3$ we denote the $(\alpha,\beta)^{th}$ partial derivative by
\begin{equation*}
  \partial^\alpha_\beta  = \partial^\alpha_x \partial^\beta_v = \partial^{\alpha_1}_{x_1} \partial^{\alpha_2}_{x_2} \partial^{\alpha_3}_{x_3} \partial^{\beta_1}_{v_1} \partial^{\beta_2}_{v_2} \partial^{\beta_3}_{v_3} \,;
\end{equation*}
\item As in \cite{AMUXY-JFA-2012},
\begin{eqnarray*}
	|g|_{L^2_D}^2=|g|_{D}^2&\equiv&\int_{\mathbb{R}^6\times \mathbb{S}^2}{\bf B}(v-u,\sigma)\mu(u)\left(g'-g\right)^2\d u\d v\d\sigma\\
	&&+\int_{\mathbb{R}^6\times \mathbb{S}^2}g(u)^2\left(\mu(u')^{\frac12}-\mu(u)^{\frac12}\right)^2\d u\d v\d\sigma;
\end{eqnarray*}
and $\|g\|_{D}=\|g\|_{L^2_xL^2_D}=\||g|_{L^2_D}\|_{L^2_x}$.

\item
For $l\in \mathbb{R}$ , $\langle v\rangle=\sqrt{1+| v|^2}$, $L_l^2$  denotes the weighted space with norm
$
|g|_{L^2_{l}}^2\equiv\int_{\mathbb{R}^3_v}|g(v)|^2\langle v\rangle^{2l}dv.
$
The weighted frictional Sobolev norm $|g(v)|_{H^s_l}^2=|\langle v\rangle^lf(v)|_{H^s}^2$ is given by
\begin{equation*}
	|g(v)|_{H^s_l}^2=|f|^2_{L^2_l}+\int_{\mathbb{R}^3}\d v\int_{\mathbb{R}^3}\d v'
	\frac{[\langle v\rangle^lg(v)-\langle v'\rangle^lg(v')]^2}{|v-v'|^{3+2s}}\chi_{|v-v'|\leq1},
\end{equation*}
where $\chi_{\Omega}$ is the standard indicator function of the set $\Omega$.
Moreover, in $\mathbb{R}^3_x\times\mathbb{R}^3_v$, $\|\cdot\|_{H^s_\gamma}=\||\cdot|_{H^s_\gamma}\|_{L^2_x}$ is used;
\item
$
|g|_\nu\equiv|g\langle v\rangle ^{\frac\gamma2}|_{L^2_v},\ 	\|g\|_\nu\equiv\|g\langle v\rangle ^{\frac\gamma2}\|_{L^2_xL^2_v}
$ for $-3<\gamma\leq 1$;
\item We use $\Lambda^{-\varrho}g(t,x,v)$ to denote
\[\Lambda^{-\varrho}g(t,x,v)=\int_{\mathbb{R}^3_\xi}|\xi|^{-\varrho}\widehat{g}(t,\xi,v)e^{2\pi ix\cdot\xi}\d\xi,\]
where $\widehat{g}(t,\xi,v)=\int_{\mathbb{R}^3_x}g(t,x,v)e^{-2\pi ix\cdot\xi}\d x$.
\end{itemize}
\subsection{The structure}
In the next section, we will deduce the main results for both non-cutoff soft potentials and cutoff soft potentials. In Section 3 and Section 4, we mainly deal with the uniform bound estimate of solutions for non-cutoff soft potentials and cutoff soft potentials respectively, which is independent of time $t$ and $\varepsilon$. In Section 5, based on
the uniform global in time energy bound, we take the limit to derive the incompressible
NSFM system with Ohm's law. Some basic properties of the linear collision operator and
bilinear symmetric operator will be given  in Appendix.

\section{Main results}
\subsection{Non-cutoff cases}
To state for brevity, we introduce the following energy and dissipation rate functional respectively: for some large $N \in \mathbb{N}$,
\begin{eqnarray}
    \mathcal{E}_N(t) &= & \|f_\eps \|^2_{H^N_{x}} + \| E_\eps \|^2_{H^N_x} + \| B_\eps \|^2_{H^N_x}, \label{E-N}\\
    \mathcal{D}_N(t) &= & \tfrac{1}{\eps^2} \| \{{\bf I-P}\}f_{\varepsilon}\|^2_{H^N_{x}L^2_D} + \| \nabla_x {\bf P} f_\eps \|^2_{H^{N-1}_x L^2_v} + \| E_\eps \|^2_{H^{N-1}_x} + \| \nabla_x B_\eps \|^2_{H^{N-2}_x} \label{D-N}
\end{eqnarray}

Due to the weaker dissipation of the Boltzmann operator in the soft potential case rather than the hard cases, in order to deal with the external force term brought by the Lorenian electric-magnetic force, especially including the growth of velocity $v$, we need to introduce the time-velocity weight
\begin{equation}\label{weight}
  w_\ell(\alpha,\beta)=e^{\frac {q\langle v\rangle}{(1+t)^\vartheta}}\langle v\rangle^{4(\ell-|\alpha|-|\beta|)},\ q<<1,\ \ell\geq N.
\end{equation}

The energy and dissipation rate functional with respect to $w_\ell(\alpha,\beta)$ are introduced respectively by
\begin{eqnarray}
\mathcal{E}_{N,\ell}(t)
&=&\sum_{|\alpha|+|\beta|\leq N-1}\left\|w_\ell(\alpha,\beta)\partial^\alpha_\beta \{{\bf I-P}\}f_{\varepsilon}\right\|^2+{(1+t)^{-\frac{1+\epsilon_0}{2}}}\varepsilon^2\sum_{|\alpha|=N}\left\|w_\ell(\alpha,0)\partial^\alpha f_\eps\right\|^2\nonumber\\
&&+{(1+t)^{-\frac{1+\epsilon_0}{2}}}\sum_{|\alpha|+|\beta|=N,\atop|\beta|\geq1}\left\|w_\ell(\alpha,\beta)\partial^\alpha_\beta \{{\bf I-P}\}f_{\varepsilon}\right\|^2,\label{E-N-w}\\
\mathcal{D}_{N,\ell}(t)
&=&
\tfrac{1}{\eps^2} \sum_{|\alpha|+|\beta|\leq N-1}\|w_\ell(\alpha,\beta)\partial^\alpha \{{\bf I-P}\}f_{\varepsilon}\|^2_{D}\nonumber\\
&&+ \frac{q\vartheta}{(1+t)^{1+\vartheta}}\sum_{|\alpha|+|\beta|\leq N-1}\|w_\ell(\alpha,\beta)\partial^\alpha_\beta \{{\bf I-P}\}f_{\varepsilon}\langle v\rangle^{\frac12}\|^2\nonumber\\
&&+{(1+t)^{-\frac{1+\epsilon_0}{2}}}\left\{\sum_{|\alpha|=N}\|w_\ell(\alpha,0)\partial^\alpha f_\eps\|^2_{D}+ \frac{q\vartheta}{(1+t)^{1+\vartheta}}\varepsilon^2\sum_{|\alpha|=N}\|w_\ell(\alpha,0)\partial^\alpha f_\eps\langle v\rangle^{\frac12}\|^2\right\}\nonumber\\
&&+{(1+t)^{-\frac{1+\epsilon_0}{2}}}\tfrac{1}{\eps^2} \sum_{|\alpha|+|\beta|= N,\atop |\beta|\geq 1}\|w_\ell(\alpha,\beta)\partial^\alpha \{{\bf I-P}\}f_{\varepsilon}\|^2_{D}\nonumber\\
&&+ {(1+t)^{-\frac{1+\epsilon_0}{2}}}\frac{q\vartheta}{(1+t)^{1+\vartheta}}\sum_{|\alpha|+|\beta|= N,\atop |\beta|\geq 1}\|w_\ell(\alpha,\beta)\partial^\alpha_\beta \{{\bf I-P}\}f_{\varepsilon}\langle v\rangle^{\frac12}\|^2.\label{D-N-w}
\end{eqnarray}
To obtain the desired temporal time decays, we introduce the energy and dissipation rate functional with the lowest $k\mbox{-}$order space-derivative
\begin{eqnarray}
\mathcal{E}_{k\rightarrow N_0}(t)&=&\sum_{|\alpha|=k}^{N_0}\left\|\partial^\alpha[f_\varepsilon,E_\varepsilon,B_\varepsilon]\right\|^2,
\label{E-k}\\
\mathcal{D}_{k\rightarrow N_0}(t)&=&
\left\|\nabla^{k}[E_\varepsilon,\rho_\varepsilon^+-\rho_\varepsilon^-]\right\|^2+\sum_{k+1\leq|\alpha|\leq N_0-1}\left\|
\partial^\alpha[{\bf P}f_\varepsilon,E_\varepsilon,B_\varepsilon]\right\|^2\nonumber\\
&&+\sum_{|\alpha|=N_0}\left\|\partial^\alpha {\bf P}f_\varepsilon\right\|^2+\tfrac{1}{\eps^2}\sum_{k\leq |\alpha| \leq N_0}\left\|\partial^\alpha{\bf\{I-P\}}f_\varepsilon\right\|^2_{D},\label{D-k}
\end{eqnarray}
respectively.

 The first part of the theorems is the global existence of the scaled two-species VMB system \eqref{VMB-F-perturbative} with uniform energy estimate with respect to the Knudsen number $0 < \eps \leq 1$, both for non-cutoff and cutoff collision kernels.

\begin{theorem}\label{Main-Thm-1}
Assume that
\begin{itemize}
	\item $
		\max\left\{-3,-\frac32-2s\right\}<\gamma<-2s, \ \ \frac12\leq s<1$;
\item 	$0<q\ll1$, $0<\varepsilon<1$;
\item $\frac12<\varrho<\frac32$,
$0<\vartheta\leq \frac\varrho2-\frac14$ ;
\item $0<\epsilon_0\leq 2(1+\varrho)$;
\item $\bar{l}$ is a properly large positive constant;
\item  $N_0\geq 4$, $N=2N_0$ and $l\geq \bar{l}+N+\frac12$;
\item  the initial data
\begin{eqnarray}\label{Def-Y_0}
	Y_{f_\varepsilon,E_\varepsilon,B_\varepsilon}(0)&\equiv&\sum_{|\alpha|\leq N}\left\|e^{ {q\langle v\rangle}}\langle v\rangle^{4(l-|\alpha|-|\beta|)} \partial^\alpha f_{\varepsilon,0}\right\|\nonumber\\
	&&+\|[E_{\varepsilon,0},B_{\varepsilon,0}]\|_{H^N_x}+\|\Lambda^{-\varrho}[f_{\varepsilon,0},E_{\varepsilon,0},B_{\varepsilon,0}]\|
\end{eqnarray}
is less than a sufficiently small positive constant, which is independent of $\varepsilon$.
\end{itemize}
{Then the Cauchy problem to \eqref{VMB-F-perturbative} admits admits the unique global-in-time solutions  $[f_\varepsilon,E_\varepsilon, B_\varepsilon]\in H^N_xL^2_v\times H^N_x\times H^N_x$, }we can also deduce that there exist energy functionals $\mathcal{E}_{N}(t)$, $\mathcal{E}_{N,l}(t)$ and the corresponding energy dissipation functionals $\mathcal{D}_{N}(t)$, $\mathcal{D}_{N,l}(t)$ which satisfy \eqref{E-N}, \eqref{D-N}, \eqref{E-N-w} and \eqref{D-N-w} respectively such that
\begin{eqnarray}\label{Main-Thm-1-1}
&&\frac{\d}{\d t}\left\{\mathcal{E}_{N,l}(t)+\mathcal{E}_{N}(t)\right\}+\mathcal{D}_{N}(t)+\mathcal{D}_{N-1,l}(t)
\lesssim0
\end{eqnarray}
{holds for all $0\leq t\leq T$.}

Meanwhile, we also get the large time behavior in the following result:
\begin{eqnarray}
  \mathcal{E}_{k \rightarrow N_0}(t)\lesssim Y^2_{f_\varepsilon,E_\varepsilon,B_\varepsilon}(0)(1+t)^{-k-\varrho}, \ k=0,1,\cdots, N_0-2,
\end{eqnarray}
where $\mathcal{E}_{k \rightarrow N_0}(t)$ is defined in \eqref{E-k}.
\end{theorem}

\subsection{Cutoff cases}
We introduce the time-velocity weight:
\[\overline {w}_{l}(\alpha,\beta)=\overline {w}_{l}(\alpha,\beta)(t,v)=\langle v\rangle^{l-|\alpha|-2|\beta|}e^{\frac{q\langle v\rangle^2}{(1+t)^{\vartheta}}}\]
and
\begin{eqnarray}
	\overline{\mathcal{E}}_{N-1,l_1}(t)&\equiv&\sum_{|\alpha|+|\beta|\leq N-1}	\left\|\overline {w}_{l_1}(\alpha,\beta)\partial^\alpha_\beta\{{\bf I-P}\}f_{\varepsilon}\right\|^2,\\
	\overline{\mathcal{D}}_{N-1,l_1}(t)&\equiv&\sum_{|\alpha|+|\beta|\leq N-1}	\left\|\overline {w}_{l_1}(\alpha,\beta)\partial^\alpha_\beta\{{\bf I-P}\}f_{\varepsilon}\right\|^2_\nu\nonumber\\
	&&+\sum_{|\alpha|+|\beta|\leq N-1}	\frac{q\vartheta}{(1+t)^{1+\vartheta}}\left\|\overline {w}_{l_1}(\alpha,\beta)\partial^\alpha_\beta\{{\bf I-P}\}f_{\varepsilon}\langle v\rangle\right\|^2.
\end{eqnarray}
Define another weight
\[\widetilde{w}_{\ell}(\alpha,\beta)=\widetilde{w}_{\ell}(\alpha,\beta)(t,v)=\langle v\rangle^{{\ell}-|\alpha|-\frac12|\beta|}e^{\frac{q\langle v\rangle^2}{(1+t)^{\vartheta}}}.\]
When $n\leq N$,
\begin{eqnarray}
	\mathcal{E}_\ell^{n,j}(t)&\equiv&
	\chi_{n=N}(1+t)^{-\sigma_{N,0}}\sum_{|\alpha|=N}\|\widetilde{w}_\ell(\alpha,0)\partial^\alpha f_{\varepsilon}\|^2\nonumber\\
	&&+\sum_{|\alpha|+|\beta|= n,\atop |\beta|\leq j}
	(1+t)^{-\sigma_{n,|\beta|}}
	\left\|\widetilde{w}_\ell(\alpha,\beta)\partial^\alpha_\beta \{{\bf I-P}\}f_{\varepsilon}\right\|^2\nonumber\\
	&&+\sum_{|\alpha|+|\beta|= n,\atop|\beta|\leq j}
	(1+t)^{-\sigma_{n,|\beta|}}
	\left\|\widetilde{w}_\ell(\alpha,\beta)\partial^\alpha_\beta \{{\bf I-P}\}f_{\varepsilon}\right\|^2,\\
	\mathcal{D}_\ell^{n,j}(t)&\equiv&
	\sum_{|\alpha|+|\beta|= n,\atop|\beta|\leq j}
	(1+t)^{-\sigma_{n,|\beta|}}
	\left\|\widetilde{w}_\ell(\alpha,\beta)\partial^\alpha_\beta \{{\bf I-P}\}f_{\varepsilon}\right\|^2_\nu\nonumber\\
	&&+\sum_{|\alpha|+|\beta|= n,\atop|\beta|\leq j}
	(1+t)^{-\sigma_{n,|\beta|}}\frac{q\vartheta{\varepsilon^2}}{(1+t)^{1+\vartheta}}\|\widetilde{w}_\ell(\alpha,\beta)\partial^\alpha_\beta \{{\bf I-P}\}f_{\varepsilon}\langle v\rangle\|^2.
\end{eqnarray}
Furthermore, for brevity,
\begin{eqnarray}
	\mathbb{E}^{(n)}_\ell(t)\equiv	\sum_{j\leq n}\mathcal{E}_\ell^{n,j}(t),\ \ \mathbb{D}^{(n)}_\ell(t)\equiv	\sum_{j\leq n}\mathcal{D}_\ell^{n,j}(t).
\end{eqnarray}
To obtain the desired temporal time decays, we introduce the energy and dissipation rate functional with the lowest $k\mbox{-}$order space-derivative
\begin{eqnarray}
	\mathcal{E}_{k\rightarrow N_0}(t)&=&\sum_{|\alpha|=k}^{N_0}\left\|\partial^\alpha[f_\varepsilon,E_\varepsilon,B_\varepsilon]\right\|^2,
	\label{E-k-cut}\\
	\mathcal{D}_{k\rightarrow N_0}(t)&=&
	\left\|\nabla^{k}[E_\varepsilon,\rho_\varepsilon^+-\rho_\varepsilon^-]\right\|^2+\sum_{k+1\leq|\alpha|\leq N_0-1}\left\|
	\partial^\alpha[{\bf P}f_\varepsilon,E_\varepsilon,B_\varepsilon]\right\|^2\nonumber\\
	&&+\sum_{|\alpha|=N_0}\left\|\partial^\alpha {\bf P}f_\varepsilon\right\|^2+\tfrac{1}{\eps^2}\sum_{k\leq |\alpha| \leq N_0}\left\|\partial^\alpha{\bf\{I-P\}}f_\varepsilon\right\|^2_{\nu},\label{D-k-cut}
\end{eqnarray}
respectively. Meanwhile, we also define
\begin{eqnarray}
	\mathcal{E}_{1\rightarrow N_0-1,\ell}(t)
	&=&\sum_{|\alpha|+|\beta|\leq N_0-1,\atop|\alpha|\geq 1}\left\|\overline{w}_\ell(\alpha,\beta)\partial^\alpha_\beta \{{\bf I-P}\}f_{\varepsilon}\right\|^2,\label{E-N-0-1-w}\\
\mathcal{D}_{1\rightarrow N_0-1,\ell}(t)
&=&\sum_{|\alpha|+|\beta|\leq N_0-1,\atop|\alpha|\geq 1}\left\|\overline{w}_\ell(\alpha,\beta)\partial^\alpha_\beta \{{\bf I-P}\}f_{\varepsilon}\right\|^2_\nu.\label{D-N-0-1-w}
\end{eqnarray}

\begin{theorem}\label{Main-Thm-2}
Assume that
\begin{itemize}
	\item $-1\leq\gamma<0$;
	\item $N_0\geq 5$, $N=2N_0$;
	\item 	$0<q\ll1$, $0<\varepsilon<1$;
\item $\frac12\leq \varrho<\frac32$, $0<\vartheta\leq \frac23\rho$;
\item $0<\epsilon_0\leq 2(1+\varrho)$;
	\item $\sigma_{N,0}=\frac{1+\epsilon_0}2$, $\sigma_{n,0}=0$ for $n\leq N-1$, $\sigma_{n,j+1}-\sigma_{n,j}=\frac{1+\epsilon_0}2$;
	\item $\bar{l}$ is a properly large positive constant;
	\item $l_1\geq N+\bar{l}$, $\tilde{\ell}\geq\frac32\sigma_{N-1,N-1}$, $\ell_1\geq l_1+\tilde{\ell}+\frac12$, $\overline{\ell}_0\geq \ell_1+\frac32N$, $l_0\geq \overline{\ell}_0+\frac52$, $\ell_0\geq l_0+\tilde{\ell}+\frac12$, $l^{H}\geq \ell_0+N$,
\end{itemize}
if the initial data
\begin{eqnarray}\label{Def-Y_0}
	Y_{f_\varepsilon,E_\varepsilon,B_\varepsilon}(0)&\equiv&\sum_{|\alpha|+|\beta|\leq N}\left\|e^{ {q\langle v\rangle^2}}\langle v\rangle^{l^H} \partial^\alpha_\beta f_{\varepsilon,0}\right\|\nonumber\\
	&&+\|[E_{\varepsilon,0},B_{\varepsilon,0}]\|_{H^N_x}+\|\Lambda^{-\varrho}[f_{\varepsilon,0},E_{\varepsilon,0},B_{\varepsilon,0}]\|
\end{eqnarray}
is less than a sufficiently small positive constant, which is independent of $\varepsilon$,
{then the Cauchy problem to \eqref{VMB-F-perturbative} admits the unique global-in-time solutions $[f_\varepsilon,E_\varepsilon, B_\varepsilon]\in H^N_xL^2_v\times H^N_x\times H^N_x$, } and
we can deduce that
\begin{eqnarray}\label{Main-Thm-2-1}
	&&\frac{\d}{\d t}\left\{\sum_{n\leq N_0}\mathbb{E}_{\ell_0}^{(n)}(t)+\sum_{N_0+1\leq n\leq N-1}\mathbb{E}_{\ell_1}^{(n)}(t)+\varepsilon^2\mathbb{E}_{\ell_1}^{(N)}(t)\right\}\nonumber\\
	&&+\frac{\d}{\d t}\left\{\overline{\mathcal{E}}_{N-1,l_1}(t)+\overline{\mathcal{E}}_{N_0-1,l_0}(t)+\mathcal{E}_{N}(t)+\|\Lambda^{-\varrho}[f_\varepsilon,E_\varepsilon,B_\varepsilon]\|^2\right\}+\sum_{n\leq N_0}\mathbb{D}_{\ell_0}^{(n)}(t)\nonumber\\
	&&+\sum_{N_0+1\leq n\leq N-1}\mathbb{D}_{\ell_1}^{(n)}(t)+\varepsilon^2\mathbb{D}_{\ell_1}^{(N)}(t)+\overline{\mathcal{D}}_{N-1,l_1}(t)+\overline{\mathcal{D}}_{N_0-1,l_0}(t)+\mathcal{D}_{N}(t)
	\lesssim0
\end{eqnarray}
{ holds for all $0\leq t\leq T$.}

	Meanwhile, we also get the large time behavior in the following result:
	\begin{eqnarray}
		\mathcal{E}_{k\rightarrow N_0}(t)\lesssim Y^2_{f_\varepsilon,E_\varepsilon,B_\varepsilon}(0)(1+t)^{-k-\varrho}, \ k=0,1,\cdots, N_0-2,
	\end{eqnarray}
	where $\mathcal{E}_{k\rightarrow N_0}(t)$ is defined in \eqref{E-k}.
\end{theorem}
\subsection{The limits}
The second is on the two-fluid incompressible NSFM limit with Ohm's law as $\eps \rightarrow 0$, taken from the solutions $(f_\eps , E_\eps , B_\eps)$ of system \eqref{VMB-F-perturbative} which are constructed in the first theorem.
\begin{theorem}\label{Main-Thm-3}

Take the assumption as in Theorem \ref{Main-Thm-1}. Assume that the initial data $( f_{\eps,0} , E_{\eps,0}, B_{\eps,0} )$ satisfy
	\begin{enumerate}
		\item $f_{\eps,0} \in H^N_xL^2_{v}$, $E_{\eps,0}$, $B_{\eps,0}\in H^N_x$;
		\item there exist scalar functions $\rho{(0,x)}$, $ \theta{(0,x)}$,  $n{(0,x)}\in H^N_x$ and vector-valued functions $u{(0,x)} $,  $E{(0,x)}$, $B{(0,x)} \in H^N_x$ such that
		\begin{equation}
		  \begin{array}{l}
			f_{\eps,0} \rightarrow f{(0,x,v)}\quad \textrm{strongly in} \ H^N_{x}L^2_v \,, \\
			E_{\eps,0} \rightarrow E{(0,x)} \quad \textrm{strongly in} \ H^N_x \,, \\
			B_{\eps,0} \rightarrow B{(0,x)} \quad \textrm{strongly in} \ H^N_x
		  \end{array}
		\end{equation}
		as $\eps \rightarrow 0$, where $f{(0,x,v)}$ is of the form
		\begin{equation}
		  \begin{aligned}
		    f(0,x,v)= & ( \rho(0,x) + \tfrac{1}{2} n(0,x) ) \tfrac{\mathsf{q}_1 + \mathsf{q}_2}{2} \sqrt{M} + ( \rho(0,x) - \tfrac{1}{2} n(0,x) ) \tfrac{\mathsf{q}_2 - \mathsf{q}_1}{2} \sqrt{M} \\
		    & + u(0,x) \cdot v \mathsf{q}_2 \sqrt{M} + \theta(0,x) ( \tfrac{|v|^2}{2} - \tfrac{3}{2} ) \mathsf{q}_2 \sqrt{M} \,.
		  \end{aligned}
		\end{equation}
	\end{enumerate}
  Let $(f_\eps, E_\eps, B_\eps)$ be the family of solutions to the scaled two-species VML \eqref{VMB-F-perturbative} constructed in Theorem \ref{Main-Thm-1}. Then, as $\eps \rightarrow 0$,
  \begin{equation}
    \begin{aligned}
      f_\eps \rightarrow ( \rho + \tfrac{1}{2} n ) \tfrac{\mathsf{q}_1 + \mathsf{q}_2}{2} \sqrt{M} + ( \rho - \tfrac{1}{2} n ) \tfrac{\mathsf{q}_2 - \mathsf{q}_1}{2} \sqrt{M} + u \cdot v \mathsf{q}_2 \sqrt{M} + \theta ( \tfrac{|v|^2}{2} - \tfrac{3}{2} ) \sqrt{M}
    \end{aligned}
  \end{equation}
  weakly-$\star$ in $t \geq 0$, strongly in $H^{N-1}_{x}L^2_v$ and weakly in $H^N_{x}L^2_v$, and
  \begin{equation}
    \begin{aligned}
      E_\eps \rightarrow E \quad \textrm{and } \quad B_\eps \rightarrow B
    \end{aligned}
  \end{equation}
  strongly in $C( \R^+; H^{N-1}_x )$, weakly-$\star$ in $t \geq 0$ and weakly in $H^N_x$. Here
  $$( u, \theta , n, E, B ) \in C(\R^+; H^{N-1}_x) \cap L^\infty (\R^+ ; H^N_x)$$
  is the solution to the incompressible NSFM \eqref{INSFM-Ohm} with Ohm's law, which has the initial data
  \begin{equation}
    \begin{aligned}
      u|_{t=0} = \mathcal{P} u(0,x) \,, \ \theta |_{t=0} = \tfrac{3}{5} \theta(0,x) - \tfrac{2}{5} \rho(0,x) \,, \ E|_{t=0} = E(0,x)\,, \ B|_{t=0} = B(0,x) \,,
    \end{aligned}
  \end{equation}
  where $\mathcal{P}$ is the Leray projection. Moreover, the convergence of the moments holds:
  \begin{equation}
    \begin{aligned}
      & \mathcal{P} \langle f_\eps , \tfrac{1}{2} \mathsf{q}_2 v \sqrt{M} \rangle_{L^2_v} \rightarrow u \,, \\
      & \langle f_\eps , \tfrac{1}{2} \mathsf{q}_2 ( \tfrac{|v|^2}{5} - 1 ) \sqrt{M} \rangle_{L^2_v} \rightarrow \theta \,,
    \end{aligned}
  \end{equation}
  strongly in $C(\R^+ ; H^{N-1}_x)$, weakly-$\star$ in $t \geq 0$ and weakly in $H^N_x$ as $\eps \rightarrow 0$.
\end{theorem}

\subsection{The idea and the outline of the proof}
In order to better illustrate the difficulties encountered in the soft potential considered in this paper, especially when compared with the hard sphere model.  To this end, we first state the main strategies adopted in \cite{Jiang-Luo-2022-Ann.PDE} to deal with the hard sphere model.
\subsubsection{The strategy in \cite{Jiang-Luo-2022-Ann.PDE} for the hard sphere model}

\begin{itemize}
\item First, under the hard sphere model, the following dissipative mechanism can be derived from the compulsion of linear collision operators:
\[\|\cdot \|^2_\nu\equiv\|\cdot \langle v\rangle^{\frac12}\|^2,\]
from this, we can see that the first power of the velocity in the dissipative mechanism is sufficient to control a single increase in the velocity due to the Lorentz force, namely the following terms, $E_\varepsilon\cdot v f_\varepsilon$, $E_\varepsilon\cdot \nabla_v f_\varepsilon$ and $\frac1\varepsilon v\times B_\varepsilon\cdot\nabla_v f_\varepsilon$.
\item Secondly, in order to control the singularity of $\varepsilon$ in the magnetic field term, especially the singularity of macroquantities in its energy estimation, they note that the linear operators corresponding to the two particles cancel out, making the correlation estimation of the macroquantities 0. For example, the following equation is true, i.e.
\[\left(\frac1\varepsilon v\times \partial^{\alpha_1}_x B_\varepsilon\cdot\partial_x^{\alpha-\alpha_1}{\bf P} f_\varepsilon, \partial^\alpha_x {\bf P}f_\varepsilon\right)=0.\]
\item  Finally, in order to obtain the dissipative mechanism of the electromagnetic field, they used Ohm's law to obtain the equation containing Damping term of the electric field quantity. This discovery played a very important role in their final estimation uniformly in both time $t$  and $\varepsilon$.
\end{itemize}
\subsubsection {The difficulty for soft potentials}
However, for the soft potential case we consider in this paper, we will face more difficulties. We must control not only the growth of velocity, but also the growth of time, and more importantly, the growth of singularity. The specific difficulties are stated as follows:
\begin{itemize}
\item [i)]First, the dissipation mechanism of the linear collision operator is much weaker than that of the hard sphere model, namely, the dissipation mechanism is as follows:
\[{\|\cdot\|_\nu\equiv\|\cdot\langle v\rangle^{\frac\gamma2}\|, \ \ \ \gamma<0,}\]
we find that this dissipative mechanism cannot control the growth of velocity in the nonlinear term containing Lorentz forces.
\item [ii)]Secondly, refer to the practices about the dynamic model with external force rather than the hard sphere model, such as Guo in \cite{Guo-2012-JAMS} proposed algebraic weight to deal with the Coulomb potential of VPL,  and  Duan-Liu-Yang-Zhao in \cite{DLYZ-KRM2013},  Duan-Lei-Yang-Zhao in \cite{DLYZ-CMP2017}  proposed to use the extra dissipation of exponential weight to deal with the soft potential of VMB and so on. Whether algebraic weight or exponential weight function, when we use the weight function, the weighted energy estimate of a linear operator $\mathscr{L}$ of the highest order weighted function will produce the singular term of the highest order macroscopic quantity, as follows:
\begin{eqnarray}\label{hard-1}
 \frac1{\varepsilon^2}\left(\mathscr{L}\partial^\alpha f_\varepsilon,w^2\partial^\alpha f_\varepsilon\right)
  \gtrsim\cdots-\frac1{\varepsilon^2}\left\|\partial^\alpha {\bf P}f_{\varepsilon}\right\|^2,
\end{eqnarray}
which leads to a new singular term $\frac1{\varepsilon^2}\left\|\partial^\alpha {\bf P}f_{\varepsilon}\right\|^2$. This new singularity term can not be controlled any more!

\item [iii)]{Thirdly,  in order to control the increase of velocity in the nonlinear term caused by Lorentz force, the only feasible method so far is to use the strategy of exponential weight functions $w_\ell$ i.e. $\langle v\rangle^\ell e^{\frac{q\langle v\rangle}{(1+t)^{\vartheta}}}$ or $\langle v\rangle^\ell e^{\frac{q\langle v\rangle^2}{(1+t)^{\vartheta}}}$. We use the extra dissipative term brought by the exponential weight function to deal with the increase of velocity in the nonlinear term. The electromagnetic field is required to have a certain time decay rates. For example, for the soft potential case under the Grad’s angular cutoff hypothesis, how to deal with the related energy estimation of nonlinear terms containing magnetic fields is shown as follows:
\begin{equation}\label{hard-6}
\frac1\varepsilon\|\partial^{e_i}_x B_\varepsilon\|_{L^\infty_x}\|\langle v\rangle^{\frac12} w_\ell \partial_{\beta+e_i}^{\alpha-e_i}f_\varepsilon\|\|\langle v\rangle^{\frac12} w_\ell \partial_{\beta}^{\alpha}f_\varepsilon\|.
\end{equation}
We see that even if the magnetic field term has a certain time decay i.e. $\frac{1}{(1+t)^{1+\vartheta}}$, this term cannot be controlled by the additional dissipative term
\begin{equation}\label{extra-dissipation}
\frac{q\vartheta}{(1+t)^{1+\vartheta}}\|\langle v\rangle w_\ell \partial_{\beta}^{\alpha}f_\varepsilon\|^2
\end{equation}
 due to the existence of this singular term $\frac1\varepsilon$ in \eqref{hard-6}.
 Therefore, how to control the weighted energy estimation $\eqref{hard-6}$ is the most challenging problem in our proof.}
\item [iv)]{Forthly, for the transport term $\frac1\varepsilon v\cdot \nabla_x f_\varepsilon$, under the hard sphere model or in the case of hard potential, the energy estimation of this term can be controlled by the dissipative mechanism i.e. $\frac1{\varepsilon}\|\cdot\|_\nu$ caused by the corresponding linear operator $\frac1{\varepsilon^2}\mathcal{L}$, regardless of the existence of the singular term with or without weight estimation. For the soft potential case, however, the situation is quite different. As estimated below:
\begin{equation}\label{hard-7}
\frac1{\varepsilon}(\partial^\alpha_\beta(v\cdot\nabla_x f_\varepsilon), w_\ell\partial^\alpha_\beta f_\varepsilon)\leq \frac1{\varepsilon}\|w_\ell \partial^{\alpha+e_i}_{\beta-e_i} f_\varepsilon\|\|w_\ell \partial^\alpha_\beta f_\varepsilon\|,
\end{equation}
where we assume that the weight function $w_\ell$ has no contribution with respect to velocity increasing or decreasing, we can see that there is a loss of velocity due to the dissipative mechanism of linear operators $\mathcal{L}$. If we expect to use the extra dissipation mechanism \eqref{extra-dissipation} to control \eqref{hard-7}, we find that since the transport term $\frac1\varepsilon v\cdot \nabla_x f_\varepsilon$ is linear, unlike \eqref{hard-6}, it lacks some kind of time decay, and there are also singular term $\frac1\varepsilon$. So, in a sense, controlling \eqref{hard-7} is harder than controlling \eqref{hard-6}.}
\item [v)] {Finally, from the above discussion, a basic premise is that the electromagnetic field must have a certain time decay rates. Due to the existence of the singularity $\frac1\varepsilon$ in nonlinearity terms, using the linear analysis combined with the principle of Duhamel's principle, we should be able to get some time decay rates of electromagnetic field, but this time decay rates is $O(\varepsilon^{-1})$. Now let's look at \eqref{hard-6}, which has a higher degree of singularity $\frac1\varepsilon$, which makes it harder to estimate. Therefore, whether we can get better time decay rates, which is independent of $\varepsilon$?}
\end{itemize}
\subsubsection{The uniform estimates for the non-cutoff cases}
To overcome the above difficulty induced by the singularity terms, the main strategies
and novelties can be summarized in the following parts.
      \begin{itemize}
        \item [1)]For the $L^2\mbox{-}$energy estimate of spatial derivatives, assuming that the electromagnetic field has some time decay, we can obtain a direct estimation in the following :\begin{eqnarray}
        	&&\frac{\d}{\d t}\mathcal{E}_{N}(t)+\mathcal{D}_{N}(t)\nonumber\\
        	&\lesssim &{M_1(1+t)^{-\varrho-\frac32}}\left\{\left\|\langle v\rangle^{\frac 74}\nabla^N_x\{{\bf I-P}\}f_{\varepsilon}\right\|^2+\sum_{|\alpha'|= N-1}\|\langle v\rangle^{\frac74}\partial^{\alpha'}\nabla_v \{{\bf I-P}\}f_{\varepsilon}\|^2\right\}+\cdots.\nonumber
        \end{eqnarray}

         \item [2)]To control the growth of the velocity on the right-hand side of the above inequality, we need to obtain the energy estimate of the highest-order weighted spatial derivative. As \eqref{hard-1}, we multiply $\varepsilon^2$ to the energy estimates with weight on the highest-order spatial derivatives, such that $\frac1{\varepsilon^2}\left\|\nabla^N_x {\bf P}f_{\varepsilon}\right\|^2$ can be controlled by the macroscopic dissipation terms.
\item[3)] To overcome singular factors $\frac1\varepsilon$ and the velocity growth for the transport term $\frac1\varepsilon v\cdot\nabla_x f_\varepsilon$ and the nonlinear term containing magnetic field $\frac1{\varepsilon}v\times B_\varepsilon\cdot\nabla_v f_\varepsilon$, we can design a time-velocity weight function as follows:
\[w_\ell(\alpha,\beta)=e^{\frac{q\langle v\rangle}{(1+t)^\vartheta}}\langle v\rangle^{4(\ell-|\alpha|-|\beta|)}.\]
We make full use of the dissipative mechanism and weight function brought by the strong angular singularity i.e. $\frac12\leq s<1$, and have the following related estimates such as
 \begin{eqnarray}
	&&\frac1\varepsilon\left(\partial^\alpha_\beta\left[ v\cdot \nabla_x\{{\bf I-P}\}f_{\varepsilon}\right],w^2_l(\alpha,\beta)\partial^\alpha_\beta\{{\bf I-P}\}f_{\varepsilon}\right)\nonumber\\
	&=&-\frac1\varepsilon\int_{\mathbb{R}^3_x\times\mathbb{R}^3_\xi}i\xi_j\mathcal{F}\left[\partial^{\alpha+e_i}_{\beta-e_i}\{{\bf I-P}\}f_{\varepsilon} w_l(\alpha+e_i,\beta-e_i)\langle v\rangle^{-\frac32}\right]\nonumber\\
	&&\quad\quad\quad\quad\quad\times\overline{\mathcal{F}\left[\langle v\rangle^{\frac32}w_l(\alpha,\beta)\partial^\alpha_{\beta-e_j}\{{\bf I-P}\}f_{\varepsilon}\right]}d\xi dx\nonumber\\
	&\lesssim&\left\|w_l(\alpha+e_i,\beta-e_i)\partial^{\alpha+e_i}_{\beta-e_i}\{{\bf I-P}\}f_{\varepsilon}\right\|_D^2+\frac\eta{\varepsilon^2}\left\|w_l(\alpha,\beta-e_j)\partial^\alpha_{\beta-e_j}\{{\bf I-P}\}f_{\varepsilon}\right\|_D^2\nonumber	
\end{eqnarray}
and
\begin{eqnarray}
	&&\frac1\varepsilon\left(\partial^\alpha_\beta\left[ v\times B_\varepsilon\cdot\nabla_{ v}\{{\bf I-P}\}f_{\varepsilon}\right],w^2_l(\alpha,\beta)\partial^\alpha_\beta\{{\bf I-P}\}f_{\varepsilon}\right)\nonumber\\
	&=&\frac1\varepsilon\sum_{|\alpha_1|=1,\beta_1=0}\left(\partial^{\alpha_1}_{\beta_1}\left[ v\times B_\varepsilon\right]\cdot\partial^{\alpha-\alpha_1}_{\beta-\beta_1}\left[\nabla_{ v}\{{\bf I-P}\}f_{\varepsilon}\right],w^2_l(\alpha,\beta)\partial^\alpha_\beta\{{\bf I-P}\}f_{\varepsilon}\right)+\cdots\nonumber\\
	&\lesssim&\frac1\varepsilon\|\partial^{e_i}B_\varepsilon\|_{L^\infty_x}\left\|w_l(\alpha-e_i,\beta)\partial^{\alpha-e_i}_{\beta}\{{\bf I-P}\}f_{\varepsilon}\right\|_D^2+\frac\eta\varepsilon\left\|w_l(\alpha,\beta)\partial^\alpha_{\beta}\{{\bf I-P}\}f_{\varepsilon}\right\|_D^2+\cdots\nonumber
\end{eqnarray}
where we used the fact that
\[v_j \langle v\rangle^{\frac32}w_l(\alpha,\beta)\leq v_j \langle v\rangle^{-\frac52}w_l(\alpha-e_i,\beta)\leq\langle v\rangle^{-\frac32}w_l(\alpha-e_i,\beta).\]
In this way, combined with other related estimates, we can obtain the estimates we want,
\begin{eqnarray}
	&&\frac{\d}{\d t}\left\{\sum_{|\alpha|=N}{\varepsilon^2}\left\|w_l(\alpha,0)\partial^\alpha f_{\varepsilon}\right\|^2
	+\sum_{|\alpha|+|\beta|= N,\atop|\beta|\geq 1}\left\|w_l(\alpha,\beta)\partial^\alpha_\beta \{{\bf I-P}\}f_{\varepsilon}\right\|^2\right\}+\cdots\nonumber\\
	&\lesssim&{\varepsilon^2}\left\|\nabla^N_x E_\varepsilon\right\|\left\|\nabla^N_xf_{\varepsilon}\right\|_\nu+\cdots.
\end{eqnarray}
In order to overcome the regularization loss of the electromagnetic field, we multiply the above inequality by a time factor $(1+t)^{-\frac{1+\epsilon_0}{2}}$ so that it can be controlled in the following
\[(1+t)^{-\frac{1+\epsilon_0}{2}}{\varepsilon^2}\left\|\nabla^N_x E_\varepsilon\right\|\left\|\nabla^N_xf_{\varepsilon}\right\|_\nu\lesssim(1+t)^{-{(1+\epsilon_0)}}\mathcal{E}_N(t)+\eta\mathcal{D}_N(t).\]
\item[4)]Except for  the energy estimation of the highest spatial derivative, we must use the microscopic projection equation \eqref{I-P-cut} to estimate them. Similar to the highest order estimate, we can obtain the following estimate:
	\begin{eqnarray}
	&&\frac{\d}{\d t}\sum_{|\alpha|+|\beta|\leq N-1}\left\|w_l(\alpha,\beta)\partial^\alpha_\beta \{{\bf I-P}\}f_{\varepsilon}\right\|^2+\cdots\lesssim\cdots+\mathcal{D}_{N}(t).
\end{eqnarray}

\item [5)]In fact, our weighted energy estimate above heavily relies on the time decay of the low-order derivatives of the electromagnetic field. Therefore, we need to obtain time decay results that are independent of $\varepsilon$. By using interpolation methods for derivatives and velocity, carefully treating the linear and nonlinear terms containing the singularity $\frac1\varepsilon$, assuming the boundedness of the norm of the negative-order Sobolev space i.e. $\|\Lambda^{-\varrho}[f_\varepsilon,E_\varepsilon,B_\varepsilon]\|$, we can obtain the following estimate:
\begin{equation}
	\mathcal{E}_{k\rightarrow N_0}(t)\lesssim\cdots(1+t)^{-(k+\varrho)},\quad 0\leq t\leq T.
\end{equation}
To ensure the validity of the decay estimate above, we use standard estimates to obtain the boundedness of the norm of the negative-order Sobolev space i.e. $\|\Lambda^{-\varrho}[f_\varepsilon,E_\varepsilon,B_\varepsilon]\|$.
Finally, by combining the above strategies, we construct the a priori estimates and close it.
      \end{itemize}
\subsubsection{The uniform estimates for the angular cutoff cases}
\begin{itemize}
	\item [i)]Unlike the non-angular truncated case with strong angular singularity, the extra difficulty in the angular truncated case lies in the fact that the dissipation of the linear collision operator does not have a dissipation mechanism for velocity derivatives, which makes it very difficult to obtain uniformity estimates for the solutions. Specifically, for the case of soft potential with angular truncation, we need to consider the singularity $\frac1\varepsilon$, time growth, and velocity growth in the relevant energy estimates.
	\item [ii)] Similar to the non-angular truncated case, we can obtain the following estimate:
		\begin{eqnarray}
		&&\frac{\d}{\d t}\mathcal{E}_{N}(t)+\mathcal{D}_{N}(t)\nonumber\\
		&\lesssim &\left\|E_\varepsilon\right\|_{L^\infty_x}^2\left\|\langle v\rangle^{\frac 32}\nabla^N_xf_{\varepsilon}\right\|^2+\left\|\nabla_xB_\varepsilon\right\|_{L^\infty_x}^2\left\|\langle v\rangle^{\frac 32}\nabla^{N-1}_x\nabla_v\{{\bf I-P}\}f_{\varepsilon}\right\|^2+\cdots
	\end{eqnarray}
Therefore, we need a weighted energy estimate for the spatial-velocity highest-order derivatives. Due to the lack of dissipation mechanisms related to velocity derivatives, we need to design the following weight function related to time and velocity:
\[\widetilde{w}_{\ell}(\alpha,\beta)(t,v)=\langle v\rangle^{{\ell}-|\alpha|-\frac12|\beta|}e^{\frac{q\langle v\rangle^2}{(1+t)^{\vartheta}}}\]
The greatest advantage of designing an algebraic weight function here is that it can balance singularity and velocity growth such as for $|\alpha_1|=1$
\begin{eqnarray}\label{hard-1-cut}
	&&{\frac1\varepsilon}\left|\left([ v\times \partial^{\alpha_1}B_\varepsilon\cdot\nabla_v \partial^{\alpha-\alpha_1}{\bf \{I-P\}}f_\varepsilon],\widetilde{w}^2_{\ell_1}(\alpha,0)\partial^\alpha {\bf \{I-P\}}f_{\varepsilon}\right)\right|\nonumber\\
	&\lesssim&{\frac1\varepsilon}\|\partial^{\alpha_1}B_\varepsilon\|_{L^\infty_x}\|\langle v\rangle^{\frac14} \widetilde{w}_{\ell_1}(\alpha-\alpha_1,e_i)\partial^{\alpha-\alpha_1}_{e_i}{\bf \{I-P\}}f_{\varepsilon}\|\|\langle v\rangle^{\frac14} \widetilde{w}_{\ell_1}(\alpha,0 )\partial^{\alpha}{\bf \{I-P\}}f_{\varepsilon}\|\nonumber\\
		&\lesssim&\cdots+{\frac\eta\varepsilon}(1+t)^{-\frac{1+\vartheta}2}\|\langle v\rangle^{\frac14} w_\ell(\alpha,0 )\partial^{\alpha}{\bf \{I-P\}}f_{\varepsilon}\|^2,
\end{eqnarray}
and
	\begin{eqnarray}\label{hard-2-cut}
	&&\frac1\varepsilon\left(\partial^\alpha_\beta\left[ v\cdot \nabla_x\{{\bf I-P}\}f_{\varepsilon}\right],\widetilde{w}^2_{\ell_1}(\alpha,\beta)\partial^\alpha_\beta\{{\bf I-P}\}f_{\varepsilon}\right)\nonumber\\
	&=&-\frac1\varepsilon\int_{\mathbb{R}^3_x\times\mathbb{R}^3_v}\langle v\rangle^{\frac12}\partial^{\alpha+e_i}_{\beta-e_i}\{{\bf I-P}\}f_{\varepsilon} \widetilde{w}_{\ell_1}(\alpha+e_i,\beta-e_i)\widetilde{w}_{\ell_1}(\alpha,\beta)\partial^\alpha_\beta\{{\bf I-P}\}f_{\varepsilon}dvdx\nonumber\\
	&\lesssim&\frac1\varepsilon(1+t)^{\frac{1+\vartheta}2}\left\|\langle v\rangle^{\frac14}\widetilde{w}_{\ell_1}(\alpha+e_i,\beta-e_i)\partial^{\alpha+e_i}_{\beta-e_i}\{{\bf I-P}\}f_{\varepsilon}\right\|^2\nonumber\\
	&&+\frac\eta\varepsilon(1+t)^{-\frac{1+\vartheta}2}\left\|\langle v\rangle^{\frac14}\widetilde{w}_{\ell_1}(\alpha,\beta)\partial^{\alpha}_{\beta}\{{\bf I-P}\}f_{\varepsilon}\right\|^2	
\end{eqnarray}
where we used the fact
\begin{eqnarray}
\langle v\rangle \widetilde{w}_{\ell_1}(\alpha,0 )=\widetilde{w}_{\ell_1}(\alpha-\alpha_1,e_i)\langle v\rangle^{\frac12},
\widetilde{w}_{\ell_1}(\alpha,\beta)=\langle v\rangle^{\frac12}\widetilde{w}_{\ell_1}(\alpha+e_i,\beta-e_i).
\end{eqnarray}
The last terms on the right-hand side of both \eqref{hard-1-cut} and \eqref{hard-2-cut} can be bounded by the dissipation terms like
\begin{eqnarray}
	&&(1+t)^{-(1+\vartheta)}\left\|\langle v\rangle\widetilde{w}_{\ell_1}(\alpha,\beta)\partial^{\alpha}_{\beta}\{{\bf I-P}\}f_{\varepsilon}\right\|^2+\frac1{\varepsilon^2}\left\|\widetilde{w}_{\ell_1}(\alpha,\beta)\partial^{\alpha}_{\beta}\{{\bf I-P}\}f_{\varepsilon}\right\|^2_\nu,
\end{eqnarray}
provided that $\gamma\geq -1$.
To control the first term on the right-hand side of \eqref{hard-2-cut}, we need to design different time increments for the various spatial velocity derivatives such that
	\begin{eqnarray}\label{hard-3-cut}
	&&{\varepsilon^2}\frac{d}{dt}\mathbb{E}^{(N)}_{\ell_1}(t)+{\varepsilon^2}\mathbb{D}^{(N)}_{\ell_1}(t)
	\lesssim\eta(1+t)^{-2\sigma_{N,0}}\left\|\nabla^N_x E_\varepsilon\right\|^2+\mathcal{E}_N(t)\mathcal{E}_{1\rightarrow N_0-1,\overline{\ell}_0}(t)+\cdots,
\end{eqnarray}
Here we take
\[	\sigma_{N,0}=\frac{1+\epsilon_0}2,\ 	\sigma_{N,|\beta|}-\sigma_{N,|\beta|-1}=\frac{1+\vartheta}2, |\beta|\geq1.	\]
Following a similar approach, we have the following estimates established:
\begin{eqnarray}\label{hard-4-cut}
	\frac{\d}{\d t}\sum_{N_0+1\leq n\leq N-1}\mathbb{E}_{\ell_1}^{(n)}(t)+\cdots
	&\lesssim&\mathcal{D}_{N}(t)+\mathcal{E}_N(t)\mathcal{E}_{1\rightarrow N_0-1,\overline{\ell}_0}(t)+\cdots,\\
		\frac{\d}{\d t}\sum_{n\leq N_0}\mathbb{E}_{\ell_0}^{(n)}(t)+\cdots
	&\lesssim&\mathcal{D}_{N_0+1}(t)+\cdots.\nonumber
\end{eqnarray}
\item [iii)] The above weighted energy estimate heavily relies on the time decay of the low-order derivatives of the solutions, such as
\[	\mathcal{E}_{k\rightarrow N_0}(t)\lesssim(1+t)^{-(k+\varrho)},\  \mathcal{E}_{1\rightarrow N_0-1,\overline{\ell}_0}(t)\lesssim(1+t)^{-(1+\varrho)}\]
These time decay estimates can be obtained using interpolation techniques, which is similar to the non-cutoff case. For the sake of brevity, we will not elaborate on this further here. To ensure the validity of these decay estimates, we need boundedness of certain weighted energy norm estimates. However, the weighted energy estimates in the aforementioned inequalities i.e. \eqref{hard-3-cut} and\eqref{hard-4-cut} are with respect to time growth with $(1+t)^{\sigma_{n,j}}$.
\item[iv)] To do so, we introduce another weight function represented as follows:
\[\overline {w}_{l}(\alpha,\beta)(t,v)=\langle v\rangle^{l-|\alpha|-2|\beta|}e^{\frac{q\langle v\rangle^2}{(1+t)^{\vartheta}}}.\] The advantage of this weight function is that the treatment of the transport term does not involve any growth in both velocity and time
such as
\begin{eqnarray}
	&&\frac1\varepsilon\left(\partial^\alpha_\beta\left[ v\cdot \nabla_x\{{\bf I-P}\}f_{\varepsilon}\right],\overline{w}^2_{l_1-|\beta|}\partial^\alpha_\beta\{{\bf I-P}\}f_{\varepsilon}\right)\nonumber\\
	&=&-\frac1\varepsilon\int_{\mathbb{R}^3_x\times\mathbb{R}^3_v}\langle v\rangle^{-1}\partial^{\alpha+e_i}_{\beta-e_i}\{{\bf I-P}\}f_{\varepsilon} \overline{w}_{l_1}(\alpha+e_i,\beta-e_i)\overline {w}_{l_1}(\alpha,\beta)\partial^\alpha_\beta\{{\bf I-P}\}f_{\varepsilon}dvdx\nonumber\\
	&\lesssim&\frac1\varepsilon\left\|\langle v\rangle^{\frac\gamma2}\overline{w}_{l_1}(\alpha+e_i,\beta-e_i)\partial^{\alpha+e_i}_{\beta-e_i}\{{\bf I-P}\}f_{\varepsilon}\right\|^2+\frac\eta\varepsilon\left\|\langle v\rangle^{\frac\gamma2}\overline{w}_{l_1}(\alpha,\beta)\partial^{\alpha}_{\beta}\{{\bf I-P}\}f_{\varepsilon}\right\|^2,	
\end{eqnarray}
where we used the fact
\begin{equation}\label{wight-bound-1-trans.}
	\overline {w}_{l_1}(\alpha,\beta)=\langle v\rangle^{-1}\overline{w}_{l_1}(\alpha+e_i,\beta-e_i).\end{equation}
However, the disadvantage of this weight function is that when performing corresponding weighted energy estimates for nonlinear terms containing magnetic fields, the weight function will unavoidably introduce a linear growth in velocity such as
\begin{eqnarray}
	&&\frac1\varepsilon\left(\partial^\alpha_\beta\left[ v\times B_\varepsilon\cdot\nabla_{ v}\{{\bf I-P}\}f_{\varepsilon}\right],\overline{w}^2_{l_1}(\alpha,\beta)\partial^\alpha_\beta\{{\bf I-P}\}f_{\varepsilon}\right)\nonumber\\
	&=&\frac1\varepsilon\sum_{|\alpha_1|=1,\beta_1=0}\left(\partial^{\alpha_1}_{\beta_1}\left[ v\times B_\varepsilon\right]\cdot\partial^{\alpha-\alpha_1}_{\beta-\beta_1}\left[\nabla_{ v}\{{\bf I-P}\}f_{\varepsilon}\right],\overline{w}^2_{l_1}(\alpha,\beta)\partial^\alpha_\beta\{{\bf I-P}\}f_{\varepsilon}\right)+\cdots\nonumber
\end{eqnarray}
In contrast to the previous equation \eqref{wight-bound-1-trans.}, the weight function here results in linearly increasing velocity, i.e.
\[\overline{w}_{l_1}(\alpha,\beta)= \langle v\rangle\overline{w}_{l_1}(\alpha-\alpha_1,\beta+e_i),\ |\alpha_1|=1,\beta_1=0.\]
To account for this, we apply Sobolev inequalities, interpolation method and Young inequalites to get
\begin{eqnarray}\label{hard-5-cut}
	&\lesssim&\sum_{|\alpha_1|=1}{\frac1\varepsilon}\|\partial^{\alpha_1}B_\varepsilon\|_{L^\infty_x}\|\langle v\rangle^{\frac52} \overline{w}_{l_1}(\alpha-\alpha_1,\beta+e_i)\partial^{\alpha-\alpha_1}_{\beta+e_i}{\bf \{I-P\}}f_{\varepsilon}\|\nonumber\\
	&&\times\|\langle v\rangle^{-\frac{1}2} \overline {w}_{l_1}(\alpha,\beta)\partial^{\alpha}_\beta{\bf \{I-P\}}f_{\varepsilon}\|+\cdots\nonumber\\
	&\lesssim&\sum_{|\alpha_1|=1}
	\|\partial^{\alpha_1}B_\varepsilon\|^{\frac2\theta}_{L^\infty_x}\|\langle v\rangle^{\tilde{\ell}} \overline{w}_{l_1}(\alpha-\alpha_1,\beta+e_i)\partial^{\alpha-\alpha_1}_{\beta +e_i}{\bf \{I-P\}}f_{\varepsilon}\|^2+\cdots
\end{eqnarray}
where we take $\theta$ as
$\theta=\frac3{\tilde{\ell}+\frac12}$.
Here we note that assuming the electromagnetic field has some decay, as the weight index $\widetilde{\ell}$ increases, the overall decay will be fast enough so that the following inequality holds
\[\sum_{|\alpha_1|=1}
\|\partial^{\alpha_1}B_\varepsilon\|^{\frac2{\theta}}_{L^\infty_x}\lesssim (1+t)^{-\sigma_{N,N}}.\]
At this point, we choose the weight index $\ell_1$ in both \eqref{hard-3-cut} and \eqref{hard-4-cut} such that the following inequality holds.
\[\langle v\rangle^{\tilde{\ell}} \overline{w}_{l_1}(\alpha-\alpha_1,\beta+e_i)\leq\widetilde{w}_{\ell_1}(\alpha-\alpha_1,\beta+e_i)\langle v\rangle^{-\frac12}.\]
 Then, the first term on the right-hand side of \eqref{hard-5-cut} can be controlled
by the correponding dissipation terms on the left-hand side of both \eqref{hard-3-cut} and \eqref{hard-4-cut}.
Thus, one can deduce that
	\begin{eqnarray}
	&&\frac{d}{dt}\overline{\mathcal{E}}_{N-1,l_1}(t)+\overline{\mathcal{D}}_{N-1,l_1}(t)
	\lesssim\mathcal{D}_N(t)+\eta\sum_{N_0+1\leq n\leq N-1}\mathbb{D}^{(n)}_{\ell_1}(t)+\cdots.
\end{eqnarray}
\item [v)]Building upon the aforementioned technique and combining it with other relevant estimations, we can construct the a priori estimates \eqref{The-a-priori-estimtates-cut} and close it to obtain uniform bound estimates, which is independent of time $t$ and $\varepsilon$.
\end{itemize}

\subsubsection{The limits}
Based on the uniform in $\varepsilon$ estimates, we employ the moment method to rigorously justify the
hydrodynamic limit from the perturbed VMB to the two fluid incompressible NSFM
system with Ohm's law as \cite{Jiang-Luo-2022-Ann.PDE}.

\section{The non-cutoff VMB system}
\subsection{Lyapunov inequality for the energy functional $\mathcal{E}_{N}(t)$}

Without generality, we take $N=2N_0$ with $N_0\geq 3$ for brevity, since we do not attempt to obtain the optimal regularity index.
\begin{lemma}\label{lemma3.7}
For $|\alpha|\leq N$, it holds that
\begin{eqnarray}\label{lemma3.7-1}
&&\frac{\d}{\d t}\left\|\partial^\alpha [f_\varepsilon,E_\varepsilon,B_\varepsilon]\right\|^2+\frac1{\varepsilon^2}\left\|\partial^\alpha\{{\bf I-P}\}f_{\varepsilon}\right\|_D^2\nonumber\\
&\lesssim &\|E_\varepsilon\|^2_{L^\infty_x} \left\|\langle v\rangle^{\frac 74}\nabla^N_x\{{\bf I-P}\}f_{\varepsilon}\right\|^2+\|\nabla_xB_\varepsilon\|^2_{L^\infty_x}\sum_{|\alpha'|= N-1}\|\langle v\rangle^{\frac74}\partial^{\alpha'}\nabla_v \{{\bf I-P}\}f_{\varepsilon}\|^2\nonumber\\
&&+\mathcal{E}_N(t)\left\{\left\|\langle v\rangle^{\frac 74}\{{\bf I-P}\}f_{\varepsilon}\right\|_{H^{N-1}_xL^2_v}^2+\sum_{|\alpha'|\leq N-2}\|\langle v\rangle^{\frac74}\partial^{\alpha'}\nabla_v \{{\bf I-P}\}f_{\varepsilon}\|^2\right\}\nonumber\\
&&+\mathcal{E}_N(t)\mathcal{D}_{N}(t)+\eta\mathcal{D}_{N}(t)\nonumber
\end{eqnarray}
for all $0\leq t\leq T$.
\end{lemma}
\begin{proof}
First of all, it is straightforward to establish the energy identities
\begin{eqnarray}\label{spatial-without}
&&\frac12\frac{\d}{\d t}\left(\left\|\partial^\alpha f_{\varepsilon}\right\|^2+ \left\|\partial^\alpha[E_\varepsilon,B_\varepsilon])\right\|^2\right)+\frac1{\varepsilon^2}\left(\mathscr{L}\partial^\alpha f_\varepsilon,\partial^\alpha f_\varepsilon\right)\nonumber\\
&=&\left(\partial^\alpha\left(\frac{q_0}{2}E_\varepsilon\cdot vf_\varepsilon\right),\partial^\alpha f_\varepsilon\right)-\left(\partial^\alpha\left(q_0 E_\varepsilon\cdot\nabla_vf_\varepsilon\right),\partial^\alpha f_\varepsilon\right)\\
&&-\frac1\varepsilon\left(\partial^\alpha\left(q_0(v\times B_\varepsilon)\cdot\nabla_vf_\varepsilon\right),\partial^\alpha f_\varepsilon\right)+\frac1\varepsilon\left(\partial^\alpha\mathscr{T}(f_\varepsilon,f_\varepsilon),\partial^\alpha f_\varepsilon\right).\nonumber
\end{eqnarray}
The coercivity property of the linear operator $\mathscr{L}$ i.e. \eqref{L-cut-1} tells us that
\[
\frac1{\varepsilon^2}\left(\mathscr{L}\partial^\alpha f_\varepsilon,\partial^\alpha f_\varepsilon\right)\gtrsim\frac1{\varepsilon^2}\|\partial^\alpha{\bf\{I-P\}}f_{\varepsilon}\|_D^2.
\]
For the four terms on the right-hand side of \eqref{spatial-without}, we estimate them one by one in the following.

{\bf Case 1: $\alpha=0$.}
By applying macro-micro decomposition, $L^2\mbox{-}L^3\mbox{-}L^6$ or $L^2\mbox{-}L^\infty\mbox{-}L^2$ Sobolev inequalities and Cauchy inequalities, one has
\begin{eqnarray}
 &&\left(\frac{q_0}{2}E_\varepsilon\cdot vf_\varepsilon, f_\varepsilon\right)\nonumber\\
 &=& \left(\frac{q_0}{2}E_\varepsilon\cdot v{\bf P}f_\varepsilon, {\bf P}f_\varepsilon\right)+\left(\frac{q_0}{2}E_\varepsilon\cdot v{\bf P}f_\varepsilon, {\bf \{I-P\}}f_{\varepsilon}\right)\nonumber\\
 &&+\left(\frac{q_0}{2}E_\varepsilon\cdot v{\bf \{I-P\}}f_\varepsilon, {\bf P}f_\varepsilon\right)+\left(\frac{q_0}{2}E_\varepsilon\cdot v{\bf \{I-P\}}f_\varepsilon, {\bf \{I-P\}}f_{\varepsilon}\right)\\
 &\lesssim&\|f_{\varepsilon}\|_{H^1_xL^2_v}\|E_\varepsilon\|\|\nabla_xf_{\varepsilon}\|_{D}+\|E_\varepsilon\|_{L^\infty_x}^2\|{\bf \{I-P\}}f_{\varepsilon}\langle v\rangle^{\frac34}\|^2+\eta\|{\bf \{I-P\}}f_{\varepsilon}\|^2_{D}\nonumber\\
&\lesssim&\|f_{\varepsilon}\|^2_{H^1_xL^2_v}\|E_\varepsilon\|^2+\eta\|\nabla_xf_{\varepsilon}\|^2_{D}+\|E_\varepsilon\|_{L^\infty_x}^2\|{\bf \{I-P\}}f_{\varepsilon}\langle v\rangle^{\frac74}\|^2+\eta\|{\bf \{I-P\}}f_{\varepsilon}\|^2_{D}\nonumber\\
&\lesssim&\mathcal{E}_2(t)\|E_\varepsilon\|^2+\eta\|\nabla_xf_{\varepsilon}\|^2_{D}+\|E_\varepsilon\|_{L^\infty_x}^2\|{\bf \{I-P\}}f_{\varepsilon}\langle v\rangle^{\frac74}\|^2+\eta\|{\bf \{I-P\}}f_{\varepsilon}\|^2_{D}\nonumber\\
&\lesssim&\mathcal{E}_2(t)\|E_\varepsilon\|^2+\|E_\varepsilon\|_{L^\infty_x}^2\|{\bf \{I-P\}}f_{\varepsilon}\langle v\rangle^{\frac74}\|^2+\eta\mathcal{D}_{2}(t).\nonumber
\end{eqnarray}
Integrating in part with respect to $v$ yields that
\begin{eqnarray}
  -\left(\left(q_0 E_\varepsilon\cdot\nabla_vf_\varepsilon\right), f_\varepsilon\right)-\frac1\varepsilon\left(\left(q_0(v\times B_\varepsilon)\cdot\nabla_vf_\varepsilon\right), f_\varepsilon\right)=0.
\end{eqnarray}
By using Lemma \ref{Gamma-noncut}, we can deduce that the last term can be dominated by
$$
 \mathcal{E}_2(t)\mathcal{D}_{2}(t)+\frac{\eta}{{\varepsilon^2}}\|{\{\bf I- P\}}f_{\varepsilon}\|^2_{D}.
$$
By collecting the above related estimates, we arrive at
\begin{eqnarray}\label{0-order}
&&\frac12\frac{\d}{\d t}\left(\left\|f_{\varepsilon}\right\|^2+ \left\|[E_\varepsilon,B_\varepsilon])\right\|^2\right)+\frac1{\varepsilon^2}\|{\bf \{I-P\}}f_{\varepsilon}\|^2_\nu\nonumber\\
&\lesssim&\mathcal{E}_2(t)\mathcal{D}_{2}(t)+\mathcal{E}_2(t)\|E_\varepsilon\|^2+\|E_\varepsilon\|_{L^\infty_x}^2\|{\bf \{I-P\}}f_{\varepsilon}\langle v\rangle^{\frac74}\|^2+\eta\mathcal{D}_2(t).
\end{eqnarray}

{\bf Case 2: $1\leq |\alpha|\leq N$.} By using Sobolev inequalities,Cauchy inequality and macro-micro decomposition, one can deduce that the first term can be bounded by
\begin{eqnarray}\label{alpha-without-w}
&&\left(\partial^\alpha\left(\frac{q_0}{2}E_\varepsilon\cdot vf_\varepsilon\right),\partial^\alpha f_\varepsilon\right)\nonumber\\
&\lesssim&\|E_\varepsilon\|_{L^\infty_x} \left\|\langle v\rangle^{\frac 74}\partial^{\alpha}f_{\varepsilon}\right\|
\left\|\langle v\rangle^{-\frac {3}4}\partial^\alpha f_{\varepsilon}\right\|\nonumber\\
&&+\chi_{|\alpha|\geq3}\sum_{1\leq|\alpha_1|\leq |\alpha|-2}\left\|\partial^{\alpha_1}E_\varepsilon\right\|_{L^\infty_x}
\left\|\langle v\rangle^{\frac 74}\partial^{\alpha-\alpha_1}f_{\varepsilon}\right\|
\left\|\langle v\rangle^{-\frac {3}4}\partial^\alpha f_{\varepsilon}\right\|\nonumber\\
&&+\chi_{|\alpha|\geq2}\sum_{|\alpha_1|=|\alpha|-1}\int_{\mathbb{R}^3_x\times\mathbb{R}^3_v}|\partial^{\alpha_1}E_\varepsilon|
|\langle v\rangle^{\frac74}\partial^{\alpha-\alpha_1}f_\varepsilon|
|\langle v\rangle^{-\frac {3}4}\partial^\alpha f_\varepsilon|dvdx\nonumber
\\
&&+\chi_{|\alpha|\geq1}\int_{\mathbb{R}^3_x\times\mathbb{R}^3_v}|\partial^{\alpha}E_\varepsilon|
|\langle v\rangle^{\frac74}f_\varepsilon|
|\langle v\rangle^{-\frac {3}4}\partial^\alpha f_\varepsilon|dvdx\nonumber\\
&\lesssim&\|E_\varepsilon\|^2_{L^\infty_x} \left\|\langle v\rangle^{\frac 74}\partial^{\alpha}f_{\varepsilon}\right\|^2+\chi_{|\alpha|\geq3}\sum_{1\leq|\alpha_1|\leq |\alpha|-2}\left\|\partial^{\alpha_1}E_\varepsilon\right\|^2_{L^\infty_x}
\left\|\langle v\rangle^{\frac 74}\partial^{\alpha-\alpha_1}f_{\varepsilon}\right\|^2
+\eta\left\|\langle v\rangle^{-\frac {3}4}\partial^\alpha f_{\varepsilon}\right\|^2\nonumber\\
&&+\chi_{|\alpha|\geq2}\sum_{|\alpha_1|=|\alpha|-1}\int_{\mathbb{R}^3_x\times\mathbb{R}^3_v}|\partial^{\alpha_1}E_\varepsilon|
|\langle v\rangle^{\frac 32}\partial^{\alpha-\alpha_1}{\bf P}f_\varepsilon|
|\langle v\rangle^{-\frac {1}2}\partial^\alpha f_\varepsilon|dvdx\nonumber\\[2mm]
&&+\chi_{|\alpha|\geq1}\int_{\mathbb{R}^3_x\times\mathbb{R}^3_v}|\partial^{\alpha}E_\varepsilon|
|\langle v\rangle^{\frac74}{\bf P}f_\varepsilon|
|\langle v\rangle^{-\frac {3}4}\partial^\alpha f_\varepsilon|dvdx\nonumber\\
&&+\chi_{|\alpha|\geq2}\sum_{|\alpha_1|=|\alpha|-1}\int_{\mathbb{R}^3_x\times\mathbb{R}^3_v}|\partial^{\alpha_1}E_\varepsilon|
|\langle v\rangle^{\frac 74}\partial^{\alpha-\alpha_1}{\bf\{I-P\}}f_\varepsilon|
|\langle v\rangle^{-\frac {3}4}\partial^\alpha f_\varepsilon|dvdx\nonumber\\[2mm]
&&+\chi_{|\alpha|\geq1}\int_{\mathbb{R}^3_x\times\mathbb{R}^3_v}|\partial^{\alpha}E_\varepsilon|
|\langle v\rangle^{\frac74}{\bf \{I-P\}}f_\varepsilon|
|\langle v\rangle^{-\frac {3}4}\partial^\alpha f_\varepsilon|dvdx\nonumber\\
&\lesssim&\|E_\varepsilon\|^2_{L^\infty_x} \left\|\langle v\rangle^{\frac 74}\nabla^N_x\{{\bf I-P}\}f_{\varepsilon}\right\|^2+\mathcal{E}_N(t)\left\|\langle v\rangle^{\frac 74}\{{\bf I-P}\}f_{\varepsilon}\right\|_{H^{N-1}_xL^2_v}^2\nonumber\\
&&+\mathcal{E}_N(t)\mathcal{D}_{N}(t)+\eta\left\|\partial^\alpha f_{\varepsilon}\right\|_{D}^2.
\end{eqnarray}

By a similar way, note that $\left(E_\varepsilon\cdot \partial^\alpha\nabla_vf,\partial^\alpha f_\varepsilon\right)=0,$ we can deduce that
\begin{eqnarray}
  &&\left(\partial^\alpha\left(q_0 E_\varepsilon\cdot\nabla_vf_\varepsilon\right),\partial^\alpha f_\varepsilon\right)\nonumber\\
  &\lesssim&\mathcal{E}_N(t)\mathcal{D}_{N}(t)+\mathcal{E}_N(t)\left\|\nabla_v\{{\bf I-P}\}f_{\varepsilon}\langle v\rangle^{\frac 74}\right\|_{H^{N-1}_xL^2_v}^2+\eta\left\|\partial^\alpha f_{\varepsilon}\right\|_{D}^2.
\end{eqnarray}

As for the third term on the right-hand side of \eqref{spatial-without}, by using macro-micro decomposition, we have
\begin{eqnarray}
&&\frac1\varepsilon\left(\partial^\alpha\left(q_0(v\times B_\varepsilon)\cdot\nabla_vf_\varepsilon\right),\partial^\alpha f_\varepsilon\right)\nonumber\\
&=&\frac1\varepsilon\left(\partial^\alpha\left(q_0(v\times B_\varepsilon)\cdot\nabla_v{\bf P}f_\varepsilon\right),\partial^\alpha {\bf P}f_\varepsilon\right)\nonumber\\
&&+\frac1\varepsilon\left(\partial^\alpha\left(q_0(v\times B_\varepsilon)\cdot\nabla_v{\bf P}f_\varepsilon\right),\partial^\alpha {\{\bf I- P\}}f_\varepsilon\right)\nonumber\\
&&+\frac1\varepsilon\left(\partial^\alpha\left(q_0(v\times B_\varepsilon)\cdot\nabla_v{\{\bf I- P\}}f_\varepsilon\right),\partial^\alpha {\bf P}f_\varepsilon\right)\nonumber\\
&&+\frac1\varepsilon\left(\partial^\alpha\left(q_0(v\times B_\varepsilon)\cdot\nabla_v{\{\bf I- P\}}f_\varepsilon\right),\partial^\alpha {\{\bf I- P\}}f_\varepsilon\right)\nonumber\\
&=&\frac1\varepsilon\left(\partial^\alpha\left(q_0(v\times B_\varepsilon)\cdot\nabla_v{\bf P}f_\varepsilon\right),\partial^\alpha {\{\bf I- P\}}f_\varepsilon\right)\nonumber\\
&&+\frac1\varepsilon\left(\partial^\alpha\left(q_0(v\times B_\varepsilon)\cdot\nabla_v{\{\bf I- P\}}f_\varepsilon\right),\partial^\alpha {\bf P}f_\varepsilon\right)\nonumber\\
&&+\frac1\varepsilon\left(\partial^\alpha\left(q_0(v\times B_\varepsilon)\cdot\nabla_v{\{\bf I- P\}}f_\varepsilon\right),\partial^\alpha {\{\bf I- P\}}f_\varepsilon\right).\label{B-without-w}
\end{eqnarray}
where we use the fact that
$$\frac1\varepsilon\left(\partial^\alpha\left(q_0(v\times B_\varepsilon)\cdot\nabla_v{\bf P}f_\varepsilon\right),\partial^\alpha {\bf P}f_\varepsilon\right)=0,$$
due to the kernel structure of ${\bf P}$ and the integral of oddness function with respect to velocity $v$ over $\mathbb{R}^3_v$.

Using various Sobolev inequalities and Cauchy inequality, we can get that the first two terms on the right-hand side of \eqref{B-without-w} can be controlled by
\begin{eqnarray}
&&\frac1\varepsilon\left(\partial^\alpha\left(q_0(v\times B_\varepsilon)\cdot\nabla_v{\bf P}f_\varepsilon\right),\partial^\alpha {\{\bf I- P\}}f_\varepsilon\right)\nonumber\\
&&+\frac1\varepsilon\left(\partial^\alpha\left(q_0(v\times B_\varepsilon)\cdot\nabla_v{\{\bf I- P\}}f_\varepsilon\right),\partial^\alpha {\bf P}f_\varepsilon\right)\nonumber\\
&\lesssim&\|B_\varepsilon\|^2_{H^N_x}(t)\mathcal{D}_{N}(t)+\frac{\eta}{{\varepsilon^2}}\sum_{|\alpha'|\leq \alpha}\|\partial^{\alpha'}{\{\bf I- P\}}f_{\varepsilon}\|^2_D,
\end{eqnarray}
as for the last term on the right-hand side of \eqref{B-without-w}, by using a similar way as the estimates on $\left(\partial^\alpha\left(q_0 E_\varepsilon\cdot\nabla_vf_\varepsilon\right),\partial^\alpha f_\varepsilon\right)$, we can deduce that
\begin{eqnarray*}
&&\frac1\varepsilon\left(\partial^\alpha\left(q_0(v\times B_\varepsilon)\cdot\nabla_v{\{\bf I- P\}}f_\varepsilon\right),\partial^\alpha {\{\bf I- P\}}f_\varepsilon\right)\\
&=&\frac1\varepsilon\left(q_0(v\times B_\varepsilon)\cdot\nabla_v\partial^\alpha{\{\bf I- P\}}f_\varepsilon,\partial^\alpha {\{\bf I- P\}}f_\varepsilon\right)\\
&&+\sum_{1\leq|\alpha_1|\leq |\alpha|}\frac1\varepsilon\left(q_0(v\times \partial ^{\alpha_1}B_\varepsilon)\cdot\nabla_v\partial^{\alpha-\alpha_1}{\{\bf I- P\}}f_\varepsilon,\partial^\alpha {\{\bf I- P\}}f_\varepsilon\right)\\
&=&\sum_{1\leq|\alpha_1|\leq |\alpha|}\frac1\varepsilon\left(q_0(v\times \partial ^{\alpha_1}B_\varepsilon)\cdot\nabla_v\partial^{\alpha-\alpha_1}{\{\bf I- P\}}f_\varepsilon\langle v\rangle^{\frac34},\partial^\alpha {\{\bf I- P\}}f_\varepsilon\langle v\rangle^{-\frac34}\right)\\
&\lesssim&\|\nabla_xB_\varepsilon\|_{L^\infty_x}\sum_{|\alpha'|= N-1}\|\partial^{\alpha'}\nabla_v \{{\bf I-P}\}f_{\varepsilon}\langle v\rangle^{\frac74}\|^2\nonumber\\
&&+\|B_\varepsilon\|_{H^N_x}^2\sum_{|\alpha'|\leq N-2}\|\partial^{\alpha'}\nabla_v \{{\bf I-P}\}f_{\varepsilon}\langle v\rangle^{\frac74}\|^2
+\frac{\eta}{{\varepsilon^2}}\|\partial^\alpha{\{\bf I- P\}}f_{\varepsilon}\|^2_D.
\end{eqnarray*}

Now we arrive at
\begin{eqnarray}
  &&\frac1\varepsilon\left(\partial^\alpha\left(q_0(v\times B_\varepsilon)\cdot\nabla_vf_\varepsilon\right),\partial^\alpha f_\varepsilon\right)\\
&\lesssim&\mathcal{E}_N(t)\mathcal{D}_{N}(t)+\|\nabla_xB_\varepsilon\|_{L^\infty_x}\sum_{|\alpha'|= N-1}\|\partial^{\alpha'}\nabla_v \{{\bf I-P}\}f_{\varepsilon}\langle v\rangle^{\frac74}\|^2\nonumber\\
&&+\|B_\varepsilon\|_{H^N_x}^2\sum_{|\alpha'|\leq N-2}\|\partial^{\alpha'}\nabla_v \{{\bf I-P}\}f_{\varepsilon}\langle v\rangle^{\frac74}\|^2
+\eta\mathcal{D}_{N}(t).\nonumber
\end{eqnarray}

By using Lemma \ref{Gamma-noncut} and Cauchy inequality, one can get that the last term on the right-hand side of \eqref{B-without-w} can be dominated by
$$
\lesssim \mathcal{E}_{N}(t)\mathcal{D}_N(t)+\frac{\eta}{{\varepsilon^2}}\sum_{\alpha'\leq \alpha}\|\partial^{\alpha'}{\{\bf I- P\}}f_{\varepsilon}\|^2_D.
$$
Now we arrive at by collecting the above estimates
\begin{eqnarray}\label{E1}
&&\frac{\d}{\d t}\sum_{1\leq|\alpha|\leq N}\left(\left\|\partial^\alpha  f_{\varepsilon}\right\|^2+\left\|\partial^\alpha[E_\varepsilon,B_\varepsilon]\right\|^2\right)
+\sum_{1\leq|\alpha|\leq N}\frac1{\varepsilon^2}\left\|\partial^\alpha\{{\bf I-P}\}f_{\varepsilon}\right\|^2_{D}\nonumber\\
&\lesssim&\| E_\varepsilon\|^{2}_{L^{\infty}_x}\left\|\langle v\rangle^{\frac 74}\nabla^N_xf_{\varepsilon}\right\|^2+\mathcal{E}_N(t)\left\|\langle v\rangle^{\frac 74}\{{\bf I-P}\}f_{\varepsilon}\right\|_{H^{N-1}_xL^2_v}^2\nonumber\\
&&+\mathcal{E}_N(t)\mathcal{D}_{N}(t)+\|\nabla_xB_\varepsilon\|_{L^\infty_x}\sum_{|\alpha'|= N-1}\|\partial^{\alpha'}\nabla_v \{{\bf I-P}\}f_{\varepsilon}\langle v\rangle^{\frac74}\|^2\nonumber\\
&&+\mathcal{E}_N(t)\sum_{|\alpha'|\leq N-2}\|\partial^{\alpha'}\nabla_v \{{\bf I-P}\}f_{\varepsilon}\langle v\rangle^{\frac74}\|^2
+\eta\mathcal{D}_{N}(t).\nonumber
\end{eqnarray}

By plugging \eqref{0-order} and \eqref{E1} into \eqref{spatial-without}, one has \eqref{lemma3.7-1}.
Thus the proof of Lemma \ref{lemma3.7} is complete.
\end{proof}

\begin{lemma}\label{mac-dissipation}

	There exists an interactive energy functional $\mathcal{E}^{int}_N(t)$ satisfying
	$$
	\mathcal{E}^{int}_N(t)\lesssim\sum_{|\alpha|\leq N}\left\|\partial^\alpha[f_\varepsilon,E_\varepsilon,B_\varepsilon]\right\|^2$$
	such that
	\begin{eqnarray}\label{mac-dissipation-1}
		&&\frac{\d}{\d t}\mathcal{E}^{int}_N(t)+\left\|\nabla_x[\rho_\varepsilon^\pm,u_\varepsilon,\theta_\varepsilon]\right\|^2_{H^{N-1}_x}+\|\rho_\varepsilon^+-\rho_\varepsilon^-\|^2+\|E_\varepsilon\|^2_{H^{N-1}_x}+\|\nabla_xB_\varepsilon\|^2_{H^{N-2}_x}\nonumber\\
		&\lesssim&\frac1{\varepsilon^2}\sum_{|\alpha|\leq N}\left\|\partial^\alpha\{{\bf I-P}\}f_{\varepsilon}\right\|_{D}^2+\mathcal{E}_N(t)\mathcal{D}_N(t).
	\end{eqnarray}
\end{lemma}
\begin{proof}
This lemma can be proved by macroscopically projecting the original equation. For the sake of brevity in exposition, the details of the proof are provided in Appendix \ref{Macro}.
\end{proof}
{\bf Assumption 1:}
\begin{eqnarray}\label{Assump-1}
\sup_{0<t\leq T}\left\{(1+t)^{\varrho+\frac32}\left\|E_\varepsilon\right\|^2_{L^\infty_x}+(1+t)^{\varrho+\frac52}\left\|\nabla_xB_\varepsilon\right\|_{L^\infty_x}^2+\mathcal{E}_{N}(t)\right\}\leq M_1
\end{eqnarray}
where $M_1$ is a sufficiently small positive constant.
\begin{proposition}
Under {\bf Assumption 1},
there exist an energy functional $\mathcal{E}_{N}(t)$ and the corresponding energy dissipation functional $\mathcal{D}_{N}(t)$ which satisfy (\ref{E-N}), (\ref{D-N}) respectively such that
\begin{eqnarray}\label{prof-1}
&&\frac{\d}{\d t}\mathcal{E}_{N}(t)+\mathcal{D}_{N}(t)\nonumber\\
&\lesssim &M_1(1+t)^{-\varrho-\frac32}\left\{\left\|\langle v\rangle^{\frac 74}\nabla^N_x\{{\bf I-P}\}f_{\varepsilon}\right\|^2+\sum_{|\alpha'|= N-1}\|\langle v\rangle^{\frac74}\partial^{\alpha'}\nabla_v \{{\bf I-P}\}f_{\varepsilon}\|^2\right\}\nonumber\\
&&+\mathcal{E}_N(t)\sum_{|\alpha'|+|\beta'|\leq N-1}\|\langle v\rangle^{\frac74}\partial^{\alpha'}_{\beta'} \{{\bf I-P}\}f_{\varepsilon}\|^2
\end{eqnarray}
{holds for all $0\leq t\leq T$.}
\end{proposition}
\begin{proof}
  A proper linear combination of \eqref{lemma3.7-1} and \eqref{mac-dissipation-1} gives \eqref{prof-1}.
\end{proof}
\subsection{The top-order energy estimates with weight}
To control the first term on the right-hand side of \eqref{prof-1}, one need the energy estimate with the weight. However, To overcome this difficulty term, we have the following the result:

\begin{lemma}\label{noncut-N-wight} It holds that
\begin{eqnarray}\label{N-wight-1}
&&\frac{\d}{\d t}\sum_{|\alpha|=N}{\varepsilon^2}\left\|w_l(\alpha,0)\partial^\alpha f_{\varepsilon}\right\|^2
+\frac{\vartheta q{\varepsilon^2}}{(1+t)^{1+\vartheta}}\sum_{|\alpha|=N}\left\|\langle v\rangle^{\frac12} w_l(\alpha,0)\partial^\alpha f_{\varepsilon}\right\|^2+\sum_{|\alpha|=N}\left\|w_l(\alpha,0)\partial^\alpha f_{\varepsilon}\right\|^2_{D}\nonumber\\
&\lesssim&\mathcal{D}_N(t)+
{\varepsilon^2}\left\|\partial^\alpha E_\varepsilon\right\|\left\|{M}^\delta\partial^\alpha f_{\varepsilon}\right\|+\mathcal{E}_N(t)\mathcal{D}_{N,l}(t)
\end{eqnarray}
{for all $0\leq t\leq T$.}
\end{lemma}
\begin{proof}
For this purpose, the standard energy estimate on $\partial^\alpha f$ with $|\alpha|= N$ weighted by the time-velocity dependent function $w_{l}(\alpha,0)(t,v)$ gives

\begin{equation}\label{N-f-w}
\begin{aligned}
&\frac{\d}{\d t}\left\|w_{l}(\alpha,0)\partial^\alpha f_{\varepsilon}\right\|^2
+\frac{\vartheta q}{(1+t)^{1+\vartheta}}\left\|\langle v\rangle^{\frac12} w_{l}(\alpha,0)\partial^\alpha f_{\varepsilon}\right\|^2+\frac1{\varepsilon^2}\|w_{l}(\alpha,0)\partial^\alpha f_{\varepsilon}\|_{D}^2\\[2mm]
\lesssim&\frac1{\varepsilon^2}\left\|\partial^\alpha {\bf P}f_{\varepsilon}\right\|^2+\frac1{\varepsilon^2}\left\|\partial^\alpha \{{\bf I-P}\}f_{\varepsilon}\right\|_{D}^2+
\left\|\partial^\alpha E_\varepsilon\right\|\left\|{M}^\delta\partial^\alpha f_{\varepsilon}\right\|\\
&+\left|\left(\partial^\alpha (E_\varepsilon\cdot vf_\varepsilon),w^2_{l}(\alpha,0)\partial^\alpha f_\varepsilon\right)\right|+\left|\left(\partial^\alpha (E_\varepsilon\cdot \nabla_vf_\varepsilon),w^2_{l}(\alpha,0)\partial^\alpha f_\varepsilon\right)\right|\\[2mm]
&+\frac1{\varepsilon}\left|\left(\partial^\alpha[ v\times B_\varepsilon\cdot\nabla_v f_\varepsilon],w^2_{l}(\alpha,0)\partial^\alpha f_\varepsilon\right)\right|+\frac1{\varepsilon}\left|\left(\partial^\alpha \mathscr{T}(f_\varepsilon,f_\varepsilon),w^2_{l}(\alpha,0)\partial^\alpha f_\varepsilon\right)\right|
\end{aligned}
\end{equation}
where the coercivity estimates on $\mathscr{L}$ i.e. \eqref{L-noncut-2} tell us that
\begin{eqnarray*}
	&&\frac1{\varepsilon^2}\left(\mathscr{L}\partial^\alpha f_\varepsilon,w^2_{l}(\alpha,0)\partial^\alpha f_\varepsilon\right)\\
	&\gtrsim&\frac1{\varepsilon^2}\|w_{l}(\alpha,0)\partial^\alpha f_{\varepsilon}\|_{D}^2-\frac1{\varepsilon^2}\|\partial^\alpha f_{\varepsilon}\|_{D}^2\\
	&\gtrsim&\frac1{\varepsilon^2}\|w_{l}(\alpha,0)\partial^\alpha f_{\varepsilon}\|_{D}^2-\frac1{\varepsilon^2}\left\|\partial^\alpha {\bf P}f_{\varepsilon}\right\|^2-\frac1{\varepsilon^2}\left\|\partial^\alpha \{{\bf I-P}\}f_{\varepsilon}\right\|_{D}^2.
\end{eqnarray*}

To control the singularity $\frac1{\epsilon^2}$ in $\frac1{\varepsilon^2}\left\|\partial^\alpha {\bf P}f_{\varepsilon}\right\|^2$ on the right-hand side of \eqref{N-f-w}, we have to multiply ${\varepsilon^2}$ to \eqref{N-f-w} such that
\begin{equation}\label{alpha-f-1}
\begin{aligned}
&\frac{\d}{\d t}{\varepsilon^2}\left\|w_l(\alpha,0)\partial^\alpha f_{\varepsilon}\right\|^2
+\frac{\vartheta q{\varepsilon^2}}{(1+t)^{1+\vartheta}}\left\|\langle v\rangle^\frac12 w_l(\alpha,0)\partial^\alpha f_{\varepsilon}\right\|^2+\left\|w_l(\alpha,0)\partial^\alpha f_{\varepsilon}\right\|^2_{D}\\[2mm]
\lesssim&\left\|\partial^\alpha {\bf P}f_{\varepsilon}\right\|^2+\left\|\partial^\alpha \{{\bf I-P}\}f_{\varepsilon}\right\|_{D}^2+
{\varepsilon^2}\left\|\partial^\alpha E_\varepsilon\right\|\left\|{M}^\delta\partial^\alpha f_{\varepsilon}\right\|\\
&+{\varepsilon^2}\left|\left(\partial^\alpha (E_\varepsilon\cdot vf_\varepsilon),w^2_l(\alpha,0)\partial^\alpha f_\varepsilon\right)\right|+{\varepsilon^2}\left|\left(\partial^\alpha (E_\varepsilon\cdot \nabla_vf_\varepsilon),w^2_l(\alpha,0)\partial^\alpha f_\varepsilon\right)\right|\\
&+{\varepsilon}\left|\left(\partial^\alpha[ v\times B_\varepsilon\cdot\nabla_v f_\varepsilon],w^2_l(\alpha,0)\partial^\alpha f_\varepsilon\right)\right|+{\varepsilon}\left|\left(\partial^\alpha \mathscr{T}(f_\varepsilon,f_\varepsilon),w^2_l(\alpha,0)\partial^\alpha f_\varepsilon\right)\right|.
\end{aligned}
\end{equation}
By using Sobolev inequalities, one has
\begin{eqnarray}
&&{\varepsilon^2}\left|\left(\partial^\alpha (E_\varepsilon\cdot vf_\varepsilon),w^2_l(\alpha,0)\partial^\alpha f_\varepsilon\right)\right|\\
&\lesssim&{\varepsilon^2}\|E_\varepsilon\|_{L^\infty_x}\left\|\langle v\rangle^{1/2}w_l(\alpha,0)\partial^{\alpha}f_{\varepsilon}\right\|
\left\|\langle v\rangle^{1/2}w_l(\alpha,0)\partial^\alpha f_{\varepsilon}\right\|\nonumber\\
&&+\sum_{\substack{1\leq|\alpha_1|\leq N-2}}{\varepsilon^2}\left\|\partial^{\alpha_1}E_\varepsilon\right\|_{L^\infty_x}
\left\|\langle v\rangle^{\frac74}w_l(\alpha,0)\partial^{\alpha-\alpha_1}f_{\varepsilon}\right\|
\left\|\langle v\rangle^{-\frac34}w_l(\alpha,0)\partial^\alpha f_{\varepsilon}\right\|\nonumber\\
&&+\sum_{\substack{|\alpha_1|= N-1}}{\varepsilon^2}\left\|\partial^{\alpha_1}E_\varepsilon\right\|_{L^3_x}
\left\|\langle v\rangle^{\frac74}w_l(\alpha,0)\partial^{\alpha-\alpha_1}f_{\varepsilon}\right\|_{L^6_x}
\left\|\langle v\rangle^{-\frac34}w_l(\alpha,0)\partial^\alpha f_{\varepsilon}\right\|\nonumber\\
&&+{\varepsilon^2}\left\|\partial^{\alpha}E_\varepsilon\right\|\left\|\langle v\rangle^{\frac74}w_l(\alpha,0)f_{\varepsilon}\right\|_{L^\infty_x}
\left\|\langle v\rangle^{-\frac34}w_l(\alpha,0)\partial^\alpha f_{\varepsilon}\right\|\nonumber\\
&\lesssim&{\varepsilon^2}\|E_\varepsilon\|_{L^\infty_x}\left\|\langle v\rangle^{1/2}w_l(\alpha,0)\partial^{\alpha}f_{\varepsilon}\right\|^2\nonumber\\
&&+\sum_{\substack{1\leq|\alpha_1|\leq N-2}}{\varepsilon^2}\left\|\partial^{\alpha_1}E_\varepsilon\right\|_{L^\infty_x}^2
\left\|\langle v\rangle^{\frac74}w_l(\alpha,0)\partial^{\alpha-\alpha_1}f_{\varepsilon}\right\|^2\nonumber\\
&&+\sum_{\substack{|\alpha_1|= N-1}}{\varepsilon^2}\left\|\partial^{\alpha_1}E_\varepsilon\right\|_{L^3_x}^2
\left\|\langle v\rangle^{\frac74}w_l(\alpha,0)\partial^{\alpha-\alpha_1}f_{\varepsilon}\right\|^2_{L^6_x}\nonumber\\
&&+{\varepsilon^2}\left\|\partial^{\alpha}E_\varepsilon\right\|^2\left\|\langle v\rangle^{\frac74}w_l(\alpha,0)f_{\varepsilon}\right\|^2_{L^\infty_x}
+\eta\left\|\langle v\rangle^{-\frac34}w_l(\alpha,0)\partial^\alpha f_{\varepsilon}\right\|^2\nonumber\\
&\lesssim&{\varepsilon^2}\|E_\varepsilon\|_{L^\infty_x}\left\|\langle v\rangle^{1/2}w_l(\alpha,0)\partial^{\alpha}f_{\varepsilon}\right\|^2+\mathcal{E}_N(t)\mathcal{D}_{N,l}(t)
+\eta\left\|\langle v\rangle^{-\frac34}w_l(\alpha,0)\partial^\alpha f_{\varepsilon}\right\|^2\nonumber
\end{eqnarray}
where we used the fact that
\[\langle v\rangle^{\frac74}w_l(\alpha,0)\lesssim w_l(\alpha-\alpha_1)\langle v\rangle^{\frac74-4|\alpha_1|}\lesssim w_l(\alpha-\alpha_1)\langle v\rangle^{-\frac34}.\]

As for ${\varepsilon}\left|\left(\partial^\alpha[ v\times B_\varepsilon\cdot\nabla_v f_\varepsilon],w^2_l(\alpha,0)\partial^\alpha f_\varepsilon\right)\right|$, one need to estimate it very carefully because of lack of $\varepsilon-$order,
\begin{eqnarray*}
&&{\varepsilon}\left|\left(\partial^\alpha[ v\times B_\varepsilon\cdot\nabla_v f_\varepsilon],w^2_l(\alpha,0)\partial^\alpha f_\varepsilon\right)\right|\\
&\leq&{\varepsilon}\left|\left(\partial^\alpha[ v\times B_\varepsilon\cdot\nabla_v {\bf P}f_\varepsilon],w^2_l(\alpha,0)\partial^\alpha {\bf P}f_\varepsilon\right)\right|\\
&&+{\varepsilon}\left|\left(\partial^\alpha[ v\times B_\varepsilon\cdot\nabla_v {\bf P}f_\varepsilon],w^2_l(\alpha,0)\partial^\alpha {\bf \{I-P\}}f_{\varepsilon}\right)\right|\\
&&+{\varepsilon}\left|\left(\partial^\alpha[ v\times B_\varepsilon\cdot\nabla_v {\bf \{I-P\}}f_\varepsilon],w^2_l(\alpha,0)\partial^\alpha {\bf P}f_\varepsilon\right)\right|\\
&&+{\varepsilon}\left|\left(\partial^\alpha[ v\times B_\varepsilon\cdot\nabla_v {\bf \{I-P\}}f_\varepsilon],w^2_l(\alpha,0)\partial^\alpha {\bf \{I-P\}}f_{\varepsilon}\right)\right|\\
&\leq&{\varepsilon}\left|\left(\partial^\alpha[ v\times B_\varepsilon\cdot\nabla_v {\bf P}f_\varepsilon],w^2_l(\alpha,0)\partial^\alpha {\bf \{I-P\}}f_{\varepsilon}\right)\right|\\
&&+{\varepsilon}\left|\left(\partial^\alpha[ v\times B_\varepsilon\cdot\nabla_v {\bf \{I-P\}}f_\varepsilon],w^2_l(\alpha,0)\partial^\alpha {\bf P}f_\varepsilon\right)\right|\\
&&+{\varepsilon}\left|\left(\partial^\alpha[ v\times B_\varepsilon\cdot\nabla_v {\bf \{I-P\}}f_\varepsilon],w^2_l(\alpha,0)\partial^\alpha {\bf \{I-P\}}f_{\varepsilon}\right)\right|\\
&\lesssim&\mathcal{E}_N(t)\mathcal{D}_N(t)+\eta\|\partial^\alpha f_{\varepsilon}\|_{D}^2+{\varepsilon}\left|\left(\partial^\alpha[ v\times B_\varepsilon\cdot\nabla_v {\bf \{I-P\}}f_\varepsilon],w^2_l(\alpha,0)\partial^\alpha {\bf \{I-P\}}f_{\varepsilon}\right)\right|
\end{eqnarray*}
where the last term can be bounded by
\begin{eqnarray*}
 &&{\varepsilon}\left|\left(\partial^\alpha[ v\times B_\varepsilon\cdot\nabla_v {\bf \{I-P\}}f_\varepsilon],w^2_l(\alpha,0)\partial^\alpha {\bf \{I-P\}}f_{\varepsilon}\right)\right|\\
 &=&{\varepsilon}\left|\left([ v\times B_\varepsilon\cdot\nabla_v \partial^\alpha{\bf \{I-P\}}f_\varepsilon],w^2_l(\alpha,0)\partial^\alpha {\bf \{I-P\}}f_{\varepsilon}\right)\right|\\
 &&+\sum_{1\leq|\alpha_1|\leq N}{\varepsilon}\left|\left([ v\times \partial^{\alpha_1}B_\varepsilon\cdot\nabla_v \partial^{\alpha-\alpha_1}{\bf \{I-P\}}f_\varepsilon],w^2_l(\alpha,0)\partial^\alpha {\bf \{I-P\}}f_{\varepsilon}\right)\right|\\
 &=&\sum_{1\leq|\alpha_1|\leq N}{\varepsilon}\left|\left([ v\times \partial^{\alpha_1}B_\varepsilon\cdot\nabla_v \partial^{\alpha-\alpha_1}{\bf \{I-P\}}f_\varepsilon],w^2_l(\alpha,0)\partial^\alpha {\bf \{I-P\}}f_{\varepsilon}\right)\right|\\
 &\lesssim&\sum_{1\leq|\alpha_1|\leq N}{\varepsilon}\int_{\mathbb{R}^3_\xi\times\mathbb{R}^3_v}|\xi||\partial^{\alpha_1}B_\varepsilon|
 \left|\mathcal{F}_v\left[w_l(\alpha-\alpha_1) \partial^{\alpha-\alpha_1}{\bf \{I-P\}}f_{\varepsilon}\langle v\rangle^{\frac52-4|\alpha_1|}\right]\right|\\
 &&\times \left|\mathcal{F}_v\left[w_l(\alpha,0)\partial^\alpha {\bf \{I-P\}}f_{\varepsilon}\langle v\rangle^{-\frac32}\right]\right|d\xi dv\\
  &\lesssim&\sum_{1\leq|\alpha_1|\leq N}{\varepsilon}\int_{\mathbb{R}^3_\xi\times\mathbb{R}^3_v}|\xi||\partial^{\alpha_1}B_\varepsilon|
 \left|\mathcal{F}_v\left[w_l(\alpha-\alpha_1) \partial^{\alpha-\alpha_1}{\bf \{I-P\}}f_{\varepsilon}\langle v\rangle^{-\frac32}\right]\right|^2d\xi dv\\
 &&+\eta{\varepsilon}\int_{\mathbb{R}^3_\xi\times\mathbb{R}^3_v} |\xi|\left|\mathcal{F}_v\left[w_l(\alpha,0)\partial^\alpha {\bf \{I-P\}}f_{\varepsilon}\langle v\rangle^{-\frac32}\right]\right|^2d\xi dv\\
 &\lesssim&\mathcal{E}_N(t)\mathcal{D}_{N,l}(t)+\eta\left\|w_l(\alpha,0)\partial^\alpha {\bf \{I-P\}}f_{\varepsilon}\right\|_D^2.
\end{eqnarray*}

Similarly, one also has
\begin{eqnarray}
	&&{\varepsilon^2}\left|\left(\partial^\alpha (E_\varepsilon\cdot \nabla_vf_\varepsilon),w^2_l(\alpha,0)\partial^\alpha f_\varepsilon\right)\right|\\
	&\lesssim&{\varepsilon^2}\|E_\varepsilon\|_{L^\infty_x}\left\|\langle v\rangle^{1/2}w_l(\alpha,0)\partial^{\alpha}f_{\varepsilon}\right\|^2+\mathcal{E}_N(t)\mathcal{D}_{N,l}(t)+\eta\left\|w_l(\alpha,0)\partial^\alpha {\bf \{I-P\}}f_{\varepsilon}\right\|_D^2\nonumber\\		&\lesssim&{\varepsilon^2}M_1^{\frac12}(1+t)^{-\frac34-\frac\varrho2}\left\|\langle v\rangle^{1/2}w_l(\alpha,0)\partial^{\alpha}f_{\varepsilon}\right\|^2+\mathcal{E}_N(t)\mathcal{D}_{N,l}(t)+\eta\left\|w_l(\alpha,0)\partial^\alpha {\bf \{I-P\}}f_{\varepsilon}\right\|_D^2\nonumber\\
	&\lesssim&{\varepsilon^2}M_1^{\frac12}(1+t)^{-1-\vartheta}\left\|\langle v\rangle^{1/2}w_l(\alpha,0)\partial^{\alpha}f_{\varepsilon}\right\|^2+\mathcal{E}_N(t)\mathcal{D}_{N,l}(t)+\eta\left\|w_l(\alpha,0)\partial^\alpha {\bf \{I-P\}}f_{\varepsilon}\right\|_D^2\nonumber
\end{eqnarray}
where we use {\bf Assumption 1} and take
\[0<\vartheta\leq\frac\varrho2-\frac14.\]
The last term on the right-hand side of \eqref{alpha-f-1} can be bounded by
\begin{equation}
\begin{aligned}
&{\varepsilon}\left|\left(\partial^\alpha \mathscr{T}(f_\varepsilon,f_\varepsilon),w^2_l(\alpha,0)\partial^\alpha f_\varepsilon\right)\right|\\
\lesssim& \mathcal{E}_N(t)\mathcal{D}_{N,l}(t)+\eta\left\|w_l(\alpha,0)\partial^\alpha f_{\varepsilon}\right\|_{D}^2.
\end{aligned}
\end{equation}
By collecting the related inequalities into \eqref{alpha-f-1}, one has \eqref{N-wight-1},
which complete the proof of Lemma \ref{noncut-N-wight}.
\end{proof}
In addition to the above highest-order energy estimates with the wight $w(\alpha,0)$ for $|\alpha|=N$, to avoid the macroscopic part with the singularity factor $\frac 1{\epsilon^2}$, for the other cases with weight, we applying the micro projection equality.

\begin{lemma}\label{N-micro-weight}
It holds that
  \begin{eqnarray}\label{N-micro-weight-1}
&&\frac{\d}{\d t}\sum_{|\alpha|+|\beta|= N,\atop|\alpha|\leq N-1}\left\|w_l(\alpha,\beta)\partial^\alpha_\beta \{{\bf I-P}\}f_{\varepsilon}\right\|^2+\frac1{\varepsilon^2}\sum_{|\alpha|+|\beta|= N,\atop|\alpha|\leq N-1}\left\|w_l(\alpha,\beta)\partial^\alpha_\beta \{{\bf I-P}\}f_{\varepsilon}\right\|^2_D\nonumber\\
&&+\frac{q\vartheta}{(1+t)^{1+\vartheta}}\sum_{|\alpha|+|\beta|= N,\atop|\alpha|\leq N-1}\|\langle v\rangle^{\frac12}w_l(\alpha,\beta)\partial^\alpha_\beta \{{\bf I-P}\}f_{\varepsilon}\|^2\nonumber\\
&\lesssim&\sum_{|\alpha|=N-1,|\beta|=1}\left\|w_l(\alpha+e_i,\beta-e_i)\partial^{\alpha+e_i}_{\beta-e_i}\{{\bf I-P}\}f_{\varepsilon}\right\|_D^2+\mathcal{E}_N(t)\mathcal{D}_{N,l}(t)+\mathcal{D}_{N}(t).
\end{eqnarray}
\end{lemma}
\begin{proof}
To this end, by applying the microscopic projection $\{{\bf I-P}\}$ to the first equation of (\ref{VMB-F-perturbative}), we can get that
\begin{eqnarray}\label{I-P-cut}
	&&\partial_t\{{\bf I-P}\}f_{\varepsilon}+\frac1\varepsilon v\cdot \nabla_x\{{\bf I-P}\}f_{\varepsilon}-\frac1 \varepsilon E_\varepsilon\cdot v{M}^{1/2}q_1+\frac1 {\varepsilon^2} \mathscr{L}f_\varepsilon\nonumber\\
	&=&\frac1\varepsilon\{{\bf I-P}\}\left[ q_0v\times B_\varepsilon\cdot\nabla_{ v}f_\varepsilon\right]+\frac1 \varepsilon{\bf P}(v\cdot\nabla_x f_\varepsilon)-\frac1 \varepsilon v\cdot\nabla_x{\bf P}f_\varepsilon\nonumber\\
	&&+\{{\bf I-P}\}\left[ -\frac12q_0 v\cdot E_\varepsilon f_\varepsilon+q_0E_\varepsilon\cdot\nabla_{ v}f_\varepsilon\right]+\frac1\varepsilon\mathscr{T}(f_\varepsilon,f_\varepsilon).
\end{eqnarray}

Applying $\partial^\alpha_\beta$ into \eqref{I-P-cut}, multiplying $w^2_{l}(\alpha,\beta)\partial^\alpha\{{\bf I-P}\}f_{\varepsilon}$ and integrating the result identity over $\mathbb{R}^3_x\times\mathbb{R}^3_v$, then we have
\begin{eqnarray}\label{I-P-w}
&&\frac12\frac{\d}{\d t}\|w_l(\alpha,\beta)\partial^\alpha_\beta\{{\bf I-P}\}f_{\varepsilon}\|^2+\frac{q\vartheta}{(1+t)^{1+\vartheta}}\left\|\langle v\rangle^{\frac12} w_{l}(\alpha,\beta)\partial^\alpha_\beta\{{\bf I-P}\}f_{\varepsilon}\right\|^2\nonumber\\
&&+\frac1 {\varepsilon^2} \left(\partial^\alpha_\beta\mathscr{L}f_\varepsilon,w^2_l(\alpha,\beta)\partial^\alpha_\beta\{{\bf I-P}\}f_{\varepsilon}\right)\nonumber\\
&=&-\frac1\varepsilon\left(\partial^\alpha_\beta[ v\cdot \nabla_x\{{\bf I-P}\}f_{\varepsilon}],w^2_l(\alpha,\beta)\partial^\alpha_\beta\{{\bf I-P}\}f_{\varepsilon}\right)\nonumber\\
&&+\frac1 \varepsilon \left(\partial^\alpha_\beta \left[E_\varepsilon\cdot v{M}^{1/2}q_1\right],w^2_l(\alpha,\beta)\partial^\alpha_\beta\{{\bf I-P}\}f_{\varepsilon}\right)\nonumber\\
&&+\frac1 \varepsilon\left(\partial^\alpha_\beta\left[{\bf P}(v\cdot\nabla_x f_\varepsilon)-\frac1 \varepsilon v\cdot\nabla_x{\bf P}f_\varepsilon\right],w^2_l(\alpha,\beta)\partial^\alpha_\beta\{{\bf I-P}\}f_{\varepsilon}\right)\nonumber\\
&&+\frac1\varepsilon\left(\partial^\alpha_\beta\left[\{{\bf I-P}\}\left[q_0 v\times B_\varepsilon\cdot\nabla_{ v}f_\varepsilon\right]\right],w^2_l(\alpha,\beta)\partial^\alpha_\beta\{{\bf I-P}\}f_{\varepsilon}\right)\nonumber\\
&&+\left(\partial^\alpha_\beta\{{\bf I-P}\}\left[ -\frac12q_0 v\cdot E_\varepsilon f_\varepsilon+q_0E_\varepsilon\cdot\nabla_{ v}f_\varepsilon\right],w^2_l(\alpha,\beta)\partial^\alpha_\beta\{{\bf I-P}\}f_{\varepsilon}\right)\nonumber\\
&&+\frac1\varepsilon\left(\partial^\alpha_\beta \mathscr{T}(f_\varepsilon,f_\varepsilon), w^2_l(\alpha,\beta)\partial^\alpha_\beta\{{\bf I-P}\}f_{\varepsilon}\right).
\end{eqnarray}
The coercivity estimates on the linear operator $\mathscr{L}$ i.e. \eqref{L-cut-3} yields that
 \begin{eqnarray}
   &&\frac1 {\varepsilon^2} \left(\partial^\alpha_\beta\mathscr{L}f_\varepsilon,w^2_l(\alpha,\beta)\partial^\alpha_\beta\{{\bf I-P}\}f_{\varepsilon}\right)\nonumber\\
   &\gtrsim&\frac1 {\varepsilon^2} \|w_l(\alpha,\beta)\partial^\alpha_\beta\{{\bf I-P}\}f_{\varepsilon}\|_{D}^2-\frac1 {\varepsilon^2}\|\partial^\alpha\{{\bf I-P}\}f_{\varepsilon}\|_{D}^2.
 \end{eqnarray}
As for the transport term on the right-hand side of \eqref{I-P-w}, one has
 \begin{eqnarray}
 	&&\frac1\varepsilon\left(\partial^\alpha_\beta\left[ v\cdot \nabla_x\{{\bf I-P}\}f_{\varepsilon}\right],w^2_l(\alpha,\beta)\partial^\alpha_\beta\{{\bf I-P}\}f_{\varepsilon}\right)\nonumber\\
 	&=&-\frac1\varepsilon\int_{\mathbb{R}^3_x\times\mathbb{R}^3_v}\partial^{\alpha+e_i}_{\beta-e_i}\{{\bf I-P}\}f_{\varepsilon} w_l(\alpha+e_i,\beta-e_i)w_l(\alpha,\beta)\partial^\alpha_\beta\{{\bf I-P}\}f_{\varepsilon}dvdx\nonumber\\
	&=&-\frac1\varepsilon\int_{\mathbb{R}^3_x\times\mathbb{R}^3_v}\partial^{\alpha+e_i}_{\beta-e_i}\{{\bf I-P}\}f_{\varepsilon} w_l(\alpha+e_i,\beta-e_i)\langle v\rangle^{-\frac32}\langle v\rangle^{\frac32}w_l(\alpha,\beta)\partial_{e_j}\partial^\alpha_{\beta-e_j}\{{\bf I-P}\}f_{\varepsilon}dvdx\nonumber\\
		&=&-\frac1\varepsilon\int_{\mathbb{R}^3_x\times\mathbb{R}^3_\xi}i\xi_j\mathcal{F}_v\left[\partial^{\alpha+e_i}_{\beta-e_i}\{{\bf I-P}\}f_{\varepsilon} w_l(\alpha+e_i,\beta-e_i)\langle v\rangle^{-\frac32}\right]\nonumber\\
		&&\quad\quad\quad\quad\quad\times\overline{\mathcal{F}_v\left[\langle v\rangle^{\frac32}w_l(\alpha,\beta)\partial^\alpha_{\beta-e_j}\{{\bf I-P}\}f_{\varepsilon}\right]}d\xi dx\nonumber\\
				&\lesssim&\int_{\mathbb{R}^3_x\times\mathbb{R}^3_\xi}|\xi|\left|\mathcal{F}_v\left[\partial^{\alpha+e_i}_{\beta-e_i}\{{\bf I-P}\}f_{\varepsilon} w_l(\alpha+e_i,\beta-e_i)\langle v\rangle^{-\frac32}\right]\right|^2d\xi dx\nonumber\\
		&&+\frac\eta{\varepsilon^2}\int_{\mathbb{R}^3_x\times\mathbb{R}^3_\xi}|\xi|\left|\overline{\mathcal{F}_v\left[\langle v\rangle^{\frac32}w_l(\alpha,\beta)\partial^\alpha_{\beta-e_j}\{{\bf I-P}\}f_{\varepsilon}\right]}\right|^2d\xi dx\nonumber\\
		&\lesssim&\int_{\mathbb{R}^3_x\times\mathbb{R}^3_\xi}|\xi|\left|\mathcal{F}_v\left[\langle v\rangle^{-\frac32} w_l(\alpha+e_i,\beta-e_i)\partial^{\alpha+e_i}_{\beta-e_i}\{{\bf I-P}\}f_{\varepsilon}\right]\right|^2d\xi dx\nonumber\\
		&&+\frac\eta{\varepsilon^2}\int_{\mathbb{R}^3_x\times\mathbb{R}^3_\xi}|\xi|\left|\overline{\mathcal{F}_v\left[\langle v\rangle^{-\frac32}w_l(\alpha,\beta-e_j)\partial^\alpha_{\beta-e_j}\{{\bf I-P}\}f_{\varepsilon}\right]}\right|^2d\xi dx\nonumber\\
		&\lesssim&\left\|w_l(\alpha+e_i,\beta-e_i)\partial^{\alpha+e_i}_{\beta-e_i}\{{\bf I-P}\}f_{\varepsilon}\right\|_D^2+\frac\eta{\varepsilon^2}\left\|w_l(\alpha,\beta-e_j)\partial^\alpha_{\beta-e_j}\{{\bf I-P}\}f_{\varepsilon}\right\|_D^2\nonumber\\
		&\lesssim&\left\|w_l(\alpha+e_i,\beta-e_i)\partial^{\alpha+e_i}_{\beta-e_i}\{{\bf I-P}\}f_{\varepsilon}\right\|_D^2+\eta\mathcal{D}_{N,l}(t).	
 \end{eqnarray}
Here we used the fact
\[\langle v\rangle^{\frac32}w_l(\alpha,\beta)=\langle v\rangle^{\frac32}w_l(\alpha,\beta-e_j)\langle v\rangle^{-4}\leq \langle v\rangle^{-\frac32}w_l(\alpha,\beta-e_j).\]
For the second and third term on the right-hand side of \eqref{I-P-w}, it is straightford to compute that
\begin{eqnarray}
  &&\frac1 \varepsilon \left(\partial^\alpha_\beta \left[E_\varepsilon\cdot v{M}^{1/2}q_1\right],w^2_l(\alpha,\beta)\partial^\alpha_\beta\{{\bf I-P}\}f_{\varepsilon}\right)\nonumber\\
  &\lesssim&\|\partial^\alpha E_\varepsilon\|^2+\frac\eta {\varepsilon^2}\|\partial^\alpha\{{\bf I-P}\}f_{\varepsilon}\|_{D}^2\lesssim\mathcal{D}_N(t).
\end{eqnarray}
and
\begin{eqnarray}
  &&\frac1 \varepsilon\left(\partial^\alpha_\beta\left[{\bf P}(v\cdot\nabla_x f_\varepsilon)-\frac1 \varepsilon v\cdot\nabla_x{\bf P}f_\varepsilon\right],w^2_l(\alpha,\beta)\partial^\alpha_\beta\{{\bf I-P}\}f_{\varepsilon}\right)\nonumber\\
  &\lesssim&\|\nabla^{|\alpha|+1}f_{\varepsilon}\|_{D}^2+\frac\eta {\varepsilon^2}\|\partial^\alpha\{{\bf I-P}\}f_{\varepsilon}\|_{D}^2\lesssim\mathcal{D}_N(t).
\end{eqnarray}
As for the fourth term on the right-hand side of \eqref{I-P-w}, one has
\begin{eqnarray}\label{mix-B-w}
	&&\frac1\varepsilon\left(\partial^\alpha_\beta\left[\{{\bf I-P}\}\left[ v\times B_\varepsilon\cdot\nabla_{ v}f_\varepsilon\right]\right],w^2_l(\alpha,\beta)\partial^\alpha_\beta\{{\bf I-P}\}f_{\varepsilon}\right)\nonumber\\
	&=&	\frac1\varepsilon\left(\partial^\alpha_\beta\left[ v\times B_\varepsilon\cdot\nabla_{ v}\{{\bf I-P}\}f_{\varepsilon}\right],w^2_l(\alpha,\beta)\partial^\alpha_\beta\{{\bf I-P}\}f_{\varepsilon}\right)\nonumber\\
	&&+\frac1\varepsilon\left(\partial^\alpha_\beta\left[ v\times B_\varepsilon\cdot\nabla_{ v}{\bf P}f_\varepsilon\right],w^2_l(\alpha,\beta)\partial^\alpha_\beta\{{\bf I-P}\}f_{\varepsilon}\right)\nonumber\\
	&&+\frac1\varepsilon\left(\partial^\alpha_\beta\left[{\bf P}\left[ v\times B_\varepsilon\cdot\nabla_{ v}f_\varepsilon\right]\right],w^2_l(\alpha,\beta)\partial^\alpha_\beta\{{\bf I-P}\}f_{\varepsilon}\right).
\end{eqnarray}
Applying macroscopic part and the Cauchy inequality, one can deduce that
\begin{eqnarray}
&&\frac1\varepsilon\left(\partial^\alpha_\beta\left[ v\times B_\varepsilon\cdot\nabla_{ v}{\bf P}f_\varepsilon\right],w^2_l(\alpha,\beta)\partial^\alpha_\beta\{{\bf I-P}\}f_{\varepsilon}\right)\nonumber\\
&&+\frac1\varepsilon\left(\partial^\alpha_\beta\left[{\bf P}\left[ v\times B_\varepsilon\cdot\nabla_{ v}f_\varepsilon\right]\right],w^2_l(\alpha,\beta)\partial^\alpha_\beta\{{\bf I-P}\}f_{\varepsilon}\right)\nonumber\\	 &\lesssim&\mathcal{E}_N(t)\mathcal{D}_N(t)+\frac{\eta}{\varepsilon^2}\|\partial^\alpha\{{\bf I-P}\}f_{\varepsilon}\|^2_D.
	\end{eqnarray}
For the first term on the right-hand side of \eqref{mix-B-w},
\begin{eqnarray}
	&&\frac1\varepsilon\left(\partial^\alpha_\beta\left[ v\times B_\varepsilon\cdot\nabla_{ v}\{{\bf I-P}\}f_{\varepsilon}\right],w^2_l(\alpha,\beta)\partial^\alpha_\beta\{{\bf I-P}\}f_{\varepsilon}\right)\nonumber\\
&=&\frac1\varepsilon\sum_{\alpha_1\leq\alpha,\beta_1\leq\beta}\left(\partial^{\alpha_1}_{\beta_1}\left[ v\times B_\varepsilon\right]\cdot\partial^{\alpha-\alpha_1}_{\beta-\beta_1}\left[\nabla_{ v}\{{\bf I-P}\}f_{\varepsilon}\right],w^2_l(\alpha,\beta)\partial^\alpha_\beta\{{\bf I-P}\}f_{\varepsilon}\right),\nonumber
\end{eqnarray}
when $|\alpha_1|=1,\beta_1=0$,
we apply the Parseval identity to get
\begin{eqnarray}
&\lesssim&\frac1\varepsilon\int_{\mathbb{R}^3_\xi\times\mathbb{R}^3_\xi}|\xi||\partial^{e_i}B_\varepsilon|\left|\mathcal{F}_v\left[v_j \langle v\rangle^{\frac32}w_l(\alpha,\beta)\partial^{\alpha-e_i}_{\beta}\{{\bf I-P}\}f_{\varepsilon}\right]\right|\nonumber\\
&&\times\left|\langle v\rangle^{-\frac32}w_l(\alpha,\beta)\mathcal{F}_v\left[\partial^{\alpha}_{\beta}\{{\bf I-P}\}f_{\varepsilon}\right]\right|d\xi dv\nonumber\\
&\lesssim&\frac1\varepsilon\int_{\mathbb{R}^3_\xi\times\mathbb{R}^3_\xi}|\xi||\partial^{e_i}B_\varepsilon|\left|\mathcal{F}_v\left[v_j \langle v\rangle^{\frac32}w_l(\alpha,\beta)\partial^{\alpha-e_i}_{\beta}\{{\bf I-P}\}f_{\varepsilon}\right]\right|^2d\xi dv\nonumber\\
&&+\frac1\varepsilon\int_{\mathbb{R}^3_\xi\times\mathbb{R}^3_\xi}|\xi|\left|\langle v\rangle^{-\frac32}w_l(\alpha,\beta)\mathcal{F}_v\left[\partial^{\alpha}_{\beta}\{{\bf I-P}\}f_{\varepsilon}\right]\right|^2d\xi dv\nonumber\\
&\lesssim&\frac1\varepsilon\|\partial^{e_i}B_\varepsilon\|_{L^\infty_x}\int_{\mathbb{R}^3_\xi\times\mathbb{R}^3_\xi}|\xi|\left|\mathcal{F}_v\left[ \langle v\rangle^{-\frac32}w_l(\alpha-e_i,\beta)\partial^{\alpha-e_i}_{\beta}\{{\bf I-P}\}f_{\varepsilon}\right]\right|^2d\xi dv\nonumber\\
&&+\frac1\varepsilon\int_{\mathbb{R}^3_\xi\times\mathbb{R}^3_\xi}|\xi|\left|\langle v\rangle^{-\frac32}w_l(\alpha,\beta)\mathcal{F}_v\left[\partial^{\alpha}_{\beta}\{{\bf I-P}\}f_{\varepsilon}\right]\right|^2d\xi dv\nonumber\\
&\lesssim&\mathcal{E}_N(t)\mathcal{D}_{N,l}(t)+\frac{\eta}{\varepsilon^2}\|\partial^\alpha\{{\bf I-P}\}f_{\varepsilon}\|^2_D
\end{eqnarray}
where we used the fact that
\[v_j \langle v\rangle^{\frac32}w_l(\alpha,\beta)\leq v_j \langle v\rangle^{-\frac52}w_l(\alpha-e_i,\beta)\leq\langle v\rangle^{-\frac32}w_l(\alpha-e_i,\beta).\]
The other cases have similar upper bound by utilizing a similar way. Thus, one has
\begin{eqnarray}
	&&\frac1\varepsilon\left(\partial^\alpha_\beta\left[ v\times B_\varepsilon\cdot\nabla_{ v}\{{\bf I-P}\}f_{\varepsilon}\right],w^2_l(\alpha,\beta)\partial^\alpha_\beta\{{\bf I-P}\}f_{\varepsilon}\right)\nonumber\\
	&\lesssim&\mathcal{E}_N(t)\mathcal{D}_{N,l}(t)+\frac{\eta}{\varepsilon^2}\|\partial^\alpha\{{\bf I-P}\}f_{\varepsilon}\|^2_D.
	\end{eqnarray}
By using the similar argument as the estimate on \eqref{mix-B-w},
one has
\begin{eqnarray}
 &&\left(\partial^\alpha_\beta\{{\bf I-P}\}\left[ -\frac12q_0 v\cdot E_\varepsilon f_\varepsilon+q_0E_\varepsilon\cdot\nabla_{ v}f_\varepsilon\right],w^2_l(\alpha,\beta)\partial^\alpha_\beta\{{\bf I-P}\}f_{\varepsilon}\right)\nonumber\\
&\lesssim&\mathcal{E}_N(t)\mathcal{D}_{N,l}(t)+\frac{\eta}{\varepsilon^2}\|\partial^\alpha\{{\bf I-P}\}f_{\varepsilon}\|^2_D.
\end{eqnarray}
By using \eqref{Gamma-noncut-1}-\eqref{Gamma-noncut-2},
one can deduce that
\begin{eqnarray}
	&&\frac1\varepsilon\left(\partial^\alpha_\beta \mathscr{T}(f_\varepsilon,f_\varepsilon), w^2_l(\alpha,\beta)\partial^\alpha_\beta\{{\bf I-P}\}f_{\varepsilon}\right)\nonumber\\
	&\lesssim&\mathcal{E}_N(t)\mathcal{D}_{N,l}(t)++\frac{\eta}{\varepsilon^2}\|\partial^\alpha\{{\bf I-P}\}f_{\varepsilon}\|^2_D.
\end{eqnarray}
 Now by collecting the above related estimates into \eqref{I-P-w}, we arrive at
 \begin{equation}\label{3.68}
\begin{aligned}
&\frac{\d}{\d t}\left\|w_{l}(\alpha,\beta)\partial^\alpha_\beta\{{\bf I-P}\}f_{\varepsilon}\right\|^2+\frac1{\varepsilon^2}\left\|w_{l}(\alpha,\beta)\partial^\alpha_\beta\{{\bf I-P}\}f_{\varepsilon}\right\|^2_{D}\\[2mm]
&+\frac{q\vartheta}{(1+t)^{1+\vartheta}}\left\|\langle v\rangle^\frac12 w_{l}(\alpha,\beta)\partial^\alpha_\beta\{{\bf I-P}\}f_{\varepsilon}\right\|^2\\[2mm]
\lesssim&\left\|w_l(\alpha+e_i,\beta-e_i)\partial^{\alpha+e_i}_{\beta-e_i}\{{\bf I-P}\}f_{\varepsilon}\right\|_D^2+\mathcal{E}_N(t)\mathcal{D}_{N,l}(t)+\mathcal{D}_{N}(t)+\eta\mathcal{D}_{N,l}(t).
\end{aligned}
\end{equation}
The other cases can be dominated by the same upper bound as \eqref{3.68}.
A proper linear combination (\ref{3.68}) with $|\alpha|+|\beta|\leq N, |\alpha|\leq N-1$ implies \eqref{N-micro-weight-1}.
\end{proof}
A proper linear combination of  Lemma \ref{noncut-N-wight} and Lemma \ref{N-micro-weight} gives that
\begin{proposition}\label{mix-w-N}
It holds that
\begin{eqnarray}\label{mix-w-N-estimates}
&&\frac{\d}{\d t}\left\{\sum_{|\alpha|=N}{\varepsilon^2}\left\|w_l(\alpha,0)\partial^\alpha f_{\varepsilon}\right\|^2
+\sum_{|\alpha|+|\beta|= N,\atop|\beta|\geq 1}\left\|w_l(\alpha,\beta)\partial^\alpha_\beta \{{\bf I-P}\}f_{\varepsilon}\right\|^2\right\}\nonumber\\
&&+\frac{\vartheta q{\varepsilon^2}}{(1+t)^{1+\vartheta}}\sum_{|\alpha|=N}\left\|\langle v\rangle^{\frac12} w_l(\alpha,0)\partial^\alpha f_{\varepsilon}\right\|^2+\sum_{|\alpha|=N}\left\|w_l(\alpha,0)\partial^\alpha f_{\varepsilon}\right\|^2_{D}\nonumber\\
&&+\frac1{\varepsilon^2}\sum_{|\alpha|+|\beta|= N,\atop|\beta|\geq 1}\left\|w_l(\alpha,\beta)\partial^\alpha_\beta \{{\bf I-P}\}f_{\varepsilon}\right\|^2_D\nonumber\\
&&+\frac{q\vartheta}{(1+t)^{1+\vartheta}}\sum_{|\alpha|+|\beta|= N,\atop|\beta|\geq 1}\|\langle v\rangle^{\frac12}w_l(\alpha,\beta)\partial^\alpha_\beta \{{\bf I-P}\}f_{\varepsilon}\|^2\nonumber\\
&\lesssim&{\varepsilon^2}\left\|\partial^\alpha E_\varepsilon\right\|\left\|{M}^\delta\partial^\alpha f_{\varepsilon}\right\|+\mathcal{E}_N(t)\mathcal{D}_{N,l}(t)+\mathcal{D}_{N}(t).
\end{eqnarray}
\end{proposition}
\subsection{The low-order energy estimates with weight}
By utilizing a slightly different technique as Proposition \ref{mix-w-N}, we also have
\begin{proposition}\label{mix-w-N-1}
	It holds that
	\begin{eqnarray}\label{mix-w-N-1-1}
		&&\frac{\d}{\d t}\sum_{|\alpha|+|\beta|\leq N-1}\left\|w_l(\alpha,\beta)\partial^\alpha_\beta \{{\bf I-P}\}f_{\varepsilon}\right\|^2+\frac1{\varepsilon^2}\sum_{|\alpha|+|\beta|\leq N-1}\left\|w_l(\alpha,\beta)\partial^\alpha_\beta \{{\bf I-P}\}f_{\varepsilon}\right\|^2_D\nonumber\\
		&&+\frac{q\vartheta}{(1+t)^{1+\vartheta}}\sum_{|\alpha|+|\beta|\leq N-1}\|\langle v\rangle^{\frac12}w_l(\alpha,\beta)\partial^\alpha_\beta \{{\bf I-P}\}f_{\varepsilon}\|^2\nonumber\\
		&\lesssim&\mathcal{E}_N(t)\mathcal{D}_{N,l}(t)+\mathcal{D}_{N}(t).
	\end{eqnarray}

\end{proposition}
\subsection{Lyapunov inequality for the energy functionals}
\begin{proposition}\label{prop1}
Under {\bf Assumption 1},
take $l\geq N+\frac12$, we can deduce that
\begin{eqnarray}\label{prop1-1}
&&\frac{\d}{\d t}\left\{\mathcal{E}_{N,l}(t)+\mathcal{E}_{N}(t)\right\}+\mathcal{D}_{N}(t)+\mathcal{D}_{N,l}(t)
\lesssim0
\end{eqnarray}
{ holds for all $0\leq t\leq T$.}
\end{proposition}
\begin{proof}
	Recalling the definition of $\mathcal{D}_{N,l}(t)$, \eqref{prof-1} tells us that
	  \begin{eqnarray}\label{end-1}
		&&\frac{\d}{\d t}\mathcal{E}_{N}(t)+\mathcal{D}_{N}(t)\nonumber\\
		&\lesssim &M_1(1+t)^{-(1+\epsilon_0)}(1+t)^{\frac{1+\epsilon_0}2}(1+t)^{-\frac{1+\epsilon_0}2}\nonumber\\
		&&\times\left\{\left\|\langle v\rangle^{\frac 74}\nabla^N_x\{{\bf I-P}\}f_{\varepsilon}\right\|^2+\left\|\langle v\rangle^{\frac 74}\nabla^{N-1}_x\nabla_v\{{\bf I-P}\}f_{\varepsilon}\right\|^2\right\}\nonumber\\
		&&+\mathcal{E}_N(t)\sum_{|\alpha'|+|\beta'|\leq N-1}\|\langle v\rangle^{\frac74}\partial^{\alpha'}_{\beta'} \{{\bf I-P}\}f_{\varepsilon}\|^2\nonumber\\
		&\lesssim&M_1\mathcal{D}_{N,l}(t)+\mathcal{E}_N(t)\mathcal{D}_{N,l}(t).
	\end{eqnarray}
where we ask that $l\geq N+\frac12$.

  By multiplying $(1+t)^{-\epsilon_0}$ into \eqref{end-1}, one has
  \begin{eqnarray}\label{end-2}
&&\frac{\d}{\d t}\left\{(1+t)^{-\epsilon_0}\mathcal{E}_{N}(t)\right\}+\epsilon_0(1+t)^{-1-\epsilon_0}\mathcal{E}_{N}(t)+(1+t)^{-\epsilon_0}\mathcal{D}_{N}(t)\nonumber\\
&\lesssim&M_1\mathcal{D}_{N,l}(t)+(1+t)^{-\epsilon_0}\mathcal{E}_N(t)\mathcal{D}_{N,l}(t).
\end{eqnarray}

By multiplying $(1+t)^{-\frac{1+\epsilon_0}2}$ into \eqref{mix-w-N-estimates}, one has
\begin{eqnarray}\label{end-3}
&&\frac{\d}{\d t}\left\{(1+t)^{-\frac{1+\epsilon_0}2}\left\{\sum_{|\alpha|=N}{\varepsilon^2}\left\|w_l(\alpha,0)\partial^\alpha f_{\varepsilon}\right\|^2
+\sum_{|\alpha|+|\beta|= N,\atop|\alpha|\leq N-1}\left\|w_\ell(\alpha,\beta)\partial^\alpha_\beta \{{\bf I-P}\}f_{\varepsilon}\right\|^2\right\}\right\}\nonumber\\
&&+(1+t)^{-\frac{1+\epsilon_0}2}\frac{\vartheta q{\varepsilon^2}}{(1+t)^{1+\vartheta}}\sum_{|\alpha|=N}\left\|\langle v\rangle^{\frac12} w_l(\alpha,0)\partial^\alpha f_{\varepsilon}\right\|^2+(1+t)^{-\frac{1+\epsilon_0}2}\sum_{|\alpha|=N}\left\|w_l(\alpha,0)\partial^\alpha f_{\varepsilon}\right\|^2_{D}\nonumber\\
&&+(1+t)^{-\frac{1+\epsilon_0}2}\frac1{\varepsilon^2}\sum_{|\alpha|+|\beta|= N}\left\|w_\ell(\alpha,\beta)\partial^\alpha_\beta \{{\bf I-P}\}f_{\varepsilon}\right\|^2_D\nonumber\\
&&+(1+t)^{-\frac{1+\epsilon_0}2}\frac{q\vartheta}{(1+t)^{1+\vartheta}}\sum_{|\alpha|+|\beta|= N}\|w_\ell(\alpha,\beta)\partial^\alpha_\beta \{{\bf I-P}\}f_{\varepsilon}\langle v\rangle^{\frac12}\|^2\nonumber\\
&\lesssim&(1+t)^{-\frac{1+\epsilon_0}2}\mathcal{D}_{N}(t)+(1+t)^{-\frac{1+\epsilon_0}2}\mathcal{E}_N(t)\mathcal{D}_{N,l}(t)\nonumber\\
&&+(1+t)^{-\frac{1+\epsilon_0}2}
\sum_{|\alpha|=N}{\varepsilon^2}\left\|\partial^\alpha E_\varepsilon\right\|\left\|{M}^\delta\partial^\alpha f_{\varepsilon}\right\|\nonumber\\
&\lesssim&\mathcal{D}_{N}(t)+\mathcal{E}_{N}(t)\mathcal{D}_{N,l}(t)+\eta(1+t)^{-1-\epsilon_0}\|\nabla^N_xE_\varepsilon\|^2.
\end{eqnarray}

Thus \eqref{prop1-1} follows from \eqref{end-1},  \eqref{end-2}, \eqref{end-3} and \eqref{mix-w-N-1-1}, which complete the proof of this proposition.
\end{proof}
\subsection{The temporal time decay estimate on $\mathcal{E}_{k\rightarrow N_0}(t)$}
To ensure {\bf Assumption} 1 in \eqref{Assump-1}, this subsection is devoted into the temporal time decay rates for $[f_\varepsilon,E_\varepsilon,B_\varepsilon]$ to
the Cauchy problem \eqref{VMB-F-perturbative}.

{\bf Assumption 2:}
\[\sup_{0<t\leq T}\left\{\|\Lambda^{-\varrho}[f_\varepsilon,E_\varepsilon,B_\varepsilon]\|^2+\mathcal{E}_{N,l}(t)\right\}\leq M_2,\]
where $M_2$ is a sufficiently small positive constant.

\begin{lemma}\label{k-sum} Under {\bf Assumption 1} and {\bf Assumption 2},  there exists a suitably large constant $\bar{l}$, and take $l\geq\bar{l}$, $\widetilde{k}=\min\{k+1, N_0-1\}$ let $N_0\geq 4$, $N=2N_0$, one has the following estimates:
	\begin{itemize}
		\item[(i).] For $k=0,1,\cdots,N_0-1$, it holds that
		\begin{equation}\label{k-sum-1}
			\begin{aligned}
				&\frac{\d}{\d t}\left(\left\|\nabla^kf_{\varepsilon}\right\|^2+\left\|\nabla^k[E_\varepsilon,B_\varepsilon]\right\|^2\right)+\frac1{\varepsilon^2}\left\|\nabla^k\{{\bf I-P}\}f_{\varepsilon}\right\|^2_{D}\\[2mm]
				\lesssim&\max\{M_1,M_2\}\left(\left\|\nabla^{\widetilde{k}}[E_\varepsilon,B_\varepsilon]\right\|^2+\left\|\nabla^{\widetilde{k}}{\bf P}f_{\varepsilon}\right\|^2\right)+\eta\left\|\nabla^{\widetilde{k}}f_{\varepsilon}\right\|_{D}^2.
			\end{aligned}
		\end{equation}
		\item[(iii).] For $k=N_0$, it follows that
		\begin{equation}\label{k-sum-2}
			\begin{aligned}
				&\frac{\d}{\d t}\left(\left\|\nabla^{N_0}f_{\varepsilon}\right\|^2+\left\|\nabla^{N_0}[E_\varepsilon,B_\varepsilon]\right\|^2\right)
				+\frac1{\varepsilon^2}\left\|\nabla^{N_0}\{{\bf I-P}\}f_{\varepsilon}\right\|^2_{D}\\[2mm]
				\lesssim&\max\{M_1,M_2\}\left(\left\|\nabla^{N_0-1}[E_\varepsilon,B_\varepsilon]\right\|^2
				+\left\|\nabla^{N_0}f_{\varepsilon}\right\|_{D}^2\right)+\eta\left\|\nabla^{N_0}f_{\varepsilon}\right\|_{D}^2.
			\end{aligned}
		\end{equation}
		\item[(iii).] For $k=0,1,2\cdots,N_0-1$, there exist interactive energy functionals $G^k_f(t)$ satisfying
		\[
		G^k_f(t)\lesssim \left\|\nabla^k[f_\varepsilon,E_\varepsilon,B_\varepsilon]\right\|^2+\left\|\nabla^{k+1}[f_\varepsilon,E_\varepsilon,B_\varepsilon]\right\|^2+\left\|\nabla^{k+2}E_\varepsilon\right\|^2
		\]
		such that
		\begin{eqnarray}\label{k-sum-3}
			&&\frac{\d}{\d t}G^k_f(t)+\left\|\nabla^k[E_\varepsilon,\rho_{\varepsilon}^+-\rho^-_{_\varepsilon}]\right\|_{H^1}^2+\left\|\nabla^{k+1}[{\bf P}f_\varepsilon,B_\varepsilon])\right\|^2\nonumber\\
			&\lesssim&\max\{M_1,M_2\}\left(\left\|\nabla^{\widetilde{k}}[E_\varepsilon,B_\varepsilon]\right\|^2+\left\|\nabla^{\widetilde{k}}f_{\varepsilon}\right\|^2_{D}\right)
			+\frac1{{\varepsilon^2}}\left\|\nabla^k\{{\bf I-P}\}f_{\varepsilon}\right\|^2_{H^2_xL^2_{D}}.
		\end{eqnarray}
	\end{itemize}
\end{lemma}
\begin{proof}
	We can use a similar approach to Lemma \ref{k-sum-cut} to prove this lemma. For brevity, we omit its detailed proof.
\end{proof}
\begin{proposition}\label{noncut-decay}
Under {\bf Assumption 1} and {\bf Assumption 2},
there exist an energy functional $\mathcal{E}_{k\rightarrow N_0}(t)$ and the corresponding energy dissipation rate functional $\mathcal{D}_{k\rightarrow N_0}(t)$ satisfying \eqref{E-k} and \eqref{D-k} respectively such that
\begin{equation}\label{noncut-decay-1}
\frac{\d}{\d t}\mathcal{E}_{k\rightarrow N_0}(t)+\mathcal{D}_{k\rightarrow N_0}(t)\leq 0
\end{equation}
{holds for $k=0,1,2,\cdots, N_0-2$ and all $0\leq t\leq T.$}

Furthermore, we can get that
\begin{equation}\label{noncut-decay-2}
\mathcal{E}_{k\rightarrow N_0}(t)\lesssim\max\{M_1,M_2\}(1+t)^{-(k+\varrho)},\quad 0\leq t\leq T.
\end{equation}
\end{proposition}
\begin{proof}
\eqref{noncut-decay-1} follows from \eqref{k-sum-1}, \eqref{k-sum-2} and \eqref{k-sum-3}. To deduce \eqref{noncut-decay-2},
\begin{equation*}
\begin{aligned}
\left\|\nabla^k[{\bf P}f_\varepsilon,E_\varepsilon,B_\varepsilon]\right\|\leq \left\|\nabla^{k+1}[{\bf P}f_\varepsilon,E_\varepsilon,B_\varepsilon]\right\|^{\frac{k+\varrho}{k+1+\varrho}}
\left\|\Lambda^{-\varrho}[{\bf P}f_\varepsilon,E_\varepsilon,B_\varepsilon]\right\|^{\frac{1}{k+1+\varrho}}.
\end{aligned}
\end{equation*}
The above inequality together with the facts that
\begin{eqnarray}
\left\|\nabla^m{\bf\{I-P\}}f_{\varepsilon}\right\|&\leq& \left\|\langle v\rangle^{-\frac12}\nabla^m{\bf\{I-P\}}f_{\varepsilon}\right\|^{\frac{k+\varrho}{k+1+\varrho}}
\left\|\langle v\rangle^{-\frac{(\gamma+2){k+\frac12}}{2}}\nabla^m{\bf\{I-P\}}f_{\varepsilon}\right\|^{\frac{1}{k+1+\varrho}},\nonumber\\
\left\|\nabla^{N_0}[E_\varepsilon,B_\varepsilon]\right\|&\lesssim&\left\|\nabla^{N_0-1}[E_\varepsilon,B_\varepsilon]\right\|^{\frac{k+\varrho}{k+1+\varrho}}
\left\|\nabla^{N_0+k+\frac12}[E_\varepsilon,B_\varepsilon]\right\|^{\frac{1}{k+1+\varrho}}
\end{eqnarray}
imply
\begin{equation*}
\begin{aligned}
\mathcal{E}_{k\rightarrow N_0}(t)\leq \left(\mathcal{D}_{k\rightarrow N_0}(t)\right)^{\frac{k+\varrho}{k+1+\varrho}}\left\{\max\{M_1,M_2\}\right\}^{\frac{1}{k+1+\varrho}}.
\end{aligned}
\end{equation*}
Hence, we deduce that
\begin{equation*}
\frac{\d}{\d t}\mathcal{E}_{k\rightarrow N_0}(t)+\left\{\max\{M_1,M_2\}\right\}^{-\frac{1}{k+\varrho}}
\left\{\mathcal{E}_{k\rightarrow N_0}(t)\right\}^{1+\frac{1}{k+\varrho}}\leq 0
\end{equation*}
and we can get by solving the above inequality directly that
\begin{equation*}
\mathcal{E}_{k\rightarrow N_0}(t)\lesssim\max\{M_1,M_2\}(1+t)^{-{k+\varrho}}.
\end{equation*}
This completes the proof of Lemma \ref{noncut-decay}.
\end{proof}
\subsection{The estimates on the negative Sobolev space}
To ensure {\bf Assumption} 2, this subsection is devoted into
bound on $\left\|\Lambda^{-\varrho}[f_\varepsilon,E_\varepsilon,B_\varepsilon](t)\right\|$.

The first one is on the estimate on $\|[f_\varepsilon,E_\varepsilon,B_\varepsilon](t)\|_{\dot{H}^{-\varrho}}$.
\begin{lemma}\label{Lemma4.1} It holds that
	\begin{eqnarray}\label{Lemma4.1-1}
		&&\frac{\d}{\d t}\left(\left\|\Lambda^{-\varrho}f_{\varepsilon}\right\|^2+\left\|\Lambda^{-\varrho}[E_\varepsilon,B_\varepsilon]\right\|^2\right)+\frac1{\varepsilon^2}\left\|\Lambda^{-\varrho}\{{\bf I-P}\}f_{\varepsilon}\right\|_{D}^2\nonumber\\
		 &\lesssim&\left\|\Lambda^{-\varrho}f_{\varepsilon}\right\|\left(\left\|\Lambda^{\frac34-\frac\varrho2}E_\varepsilon\right\|^2+\left\|\Lambda^{\frac34-\frac\varrho2}f_{\varepsilon}\right\|^2_{D}+\left\|\Lambda^{\frac32-\varrho}B_\varepsilon\right\|^2
		+\frac1{\varepsilon^2}\left\|\Lambda^{1-\varrho}{\bf\{I-P\}}f_{\varepsilon}\right\|^2_{D}\right)\nonumber\\
			&&+\left\|f_{\varepsilon}\langle v\rangle^{\frac52}\right\|_{L^2_xH^1_v}^2\left(\left\|\Lambda^{\frac32-\varrho}E_\varepsilon\right\|^2+\left\|\Lambda^{\frac32-\varrho}B_\varepsilon\right\|^2
		+\frac1{\varepsilon^2}\left\|\Lambda^{1-\varrho}{\bf\{I-P\}}f_{\varepsilon}\right\|^2_{D}\right)
		\nonumber\\
		&&+\left\|\langle v\rangle^{\frac32}f_\varepsilon\right\|^2_{L^2_xH^3_{s+\frac\gamma2}}\left\|\langle v\rangle^{\frac32}\Lambda^{\frac32-\varrho}f_\varepsilon\right\|_{H^1_{s+\frac\gamma2}}^2.
		\end{eqnarray}
\end{lemma}
\begin{proof}
	By taking Fourier transform of the first equation of \eqref{VMB-F-perturbative} with respect to $x$,  multiplying the resulting identity by $|\xi|^{-2\varrho}\bar{\hat{f}}_\pm$ with $\bar{\hat{f}}_\pm$ being the complex conjugate of $\hat{f}_\pm$, and integrating the final result with respect to $\xi$ and $v$ over $\mathbb{R}^3_\xi\times\mathbb{R}^3_v$ that
	we can get by using the coercivity property $\mathscr{L}$
	\begin{eqnarray}\label{4.6}
		&&\frac{\d}{\d t}\left(\left\|\Lambda^{-\varrho}f_{\varepsilon}\right\|^2+\left\|\Lambda^{-\varrho}[E_\varepsilon,B_\varepsilon]\right\|^2\right)+\frac1{\varepsilon^2}\left\|\Lambda^{-\varrho}\{{\bf I-P}\}f_{\varepsilon}\right\|_D^2\nonumber\\
		&\lesssim&\left|\left(v  \cdot\mathcal{F}[q_0 E_\varepsilon f_\varepsilon]\mid|\xi|^{-2\varrho}\hat{f}_\varepsilon\right)\right|+\left|\left(\mathcal{F}[q_0 E_\varepsilon\cdot\nabla_{ v  }f_\varepsilon]\mid|\xi|^{-2\varrho}\hat{f}_\varepsilon\right)\right|\nonumber\\
		&&+\frac1{\varepsilon}\left|\left(\mathcal{F}[q_0 v\times B_\varepsilon\cdot\nabla_{ v  }f_\varepsilon]\mid|\xi|^{-2\varrho}\hat{f}_\varepsilon\right)\right|+\frac1{\varepsilon}\left|\left(\mathcal{F}[{ \mathscr{T}}(f_\varepsilon,f_\varepsilon)]\mid|\xi|^{-2\varrho}\hat{f}_\varepsilon\right)\right|.
	\end{eqnarray}
	One has by macro-micro decomposition
	\begin{eqnarray*}
		&&\left(\mathcal{F}[E_\varepsilon\cdot v f_\varepsilon]\mid|\xi|^{-2\varrho}\hat{f}_\varepsilon\right)\nonumber\\
		&=&\left(\mathcal{F}[E_\varepsilon\cdot v{\bf P}f_\varepsilon]\mid|\xi|^{-2\varrho}\mathcal{F}[{\bf P}f_\varepsilon]\right)
		+\left(\mathcal{F}[E_\varepsilon\cdot v{\bf P}f_\varepsilon]\mid|\xi|^{-2\varrho}\mathcal{F}[{\bf \{I-P\}}f_\varepsilon]\right)\nonumber\\
		&&+\left(\mathcal{F}[E_\varepsilon\cdot v{\bf \{I-P\}}f_\varepsilon]\mid|\xi|^{-2\varrho}\mathcal{F}[{\bf P}f_\varepsilon]\right)+\left(\mathcal{F}[E_\varepsilon\cdot v{\bf \{I-P\}}f_\varepsilon]\mid|\xi|^{-2\varrho}\mathcal{F}[{\bf \{I-P\}}f_\varepsilon]\right).
	\end{eqnarray*}	
The first three terms on the right-hand side of the above equation can be bounded by
	\begin{eqnarray}\label{E-s}
		&\lesssim&\left\|\Lambda^{-\varrho}\left(E_\varepsilon\cdot {M}^{\delta}f_\varepsilon\right)\right\|\left\|\Lambda^{-\varrho}\left({M}^{\delta}f_\varepsilon\right)\right\|
		\lesssim\left\|E_\varepsilon\cdot {M}^{\delta}f_{\varepsilon}\right\|_{L_x^{\frac6{3+2\varrho}}}\left\|\Lambda^{-\varrho}\left({M}^{\delta}f_\varepsilon\right)\right\|
		\nonumber\\
		&\lesssim&\|E_\varepsilon\|_{L_x^{\frac{12}{3+2\varrho}}}\left\|{M}^{\delta}f_{\varepsilon}\right\|_{L_x^{\frac{12}{3+2\varrho}}}
		 \left\|\Lambda^{-\varrho}\left({M}^{\delta}f_\varepsilon\right)\right\|\lesssim\left\|\Lambda^{\frac34-\frac\varrho2}E_\varepsilon\right\|\left\|\Lambda^{\frac34-\frac\varrho2}\left({M}^{\delta}f_\varepsilon\right)\right\|\left\|\Lambda^{-\varrho}\left({M}^{\delta}f_\varepsilon\right)\right\|
		\nonumber\\
		&\lesssim&\left\|\Lambda^{-\varrho}f_{\varepsilon}\right\|\left(\left\|\Lambda^{\frac34-\frac\varrho2}E_\varepsilon\right\|^2+\left\|\Lambda^{\frac34-\frac\varrho2}f_{\varepsilon}\right\|^2_{D}\right).
	\end{eqnarray}
	As for the last term, one has
	\begin{eqnarray}\label{E-mic-s}
		&&\left(\mathcal{F}[E_\varepsilon\cdot v{\bf \{I-P\}}f_\varepsilon]\mid|\xi|^{-2\varrho}\mathcal{F}[{\bf \{I-P\}}f_\varepsilon]\right)\nonumber\\
		&\lesssim&\left\|\Lambda^{-\varrho}\left(E_\varepsilon\cdot v {\bf \{I-P\}}f_\varepsilon \langle v\rangle^{\frac32}\right)\right\| \left\|\Lambda^{-\varrho}\left({\bf\{I-P\}}f_\varepsilon\right)\right\|_D\nonumber\\
		&\lesssim&\left\|E_\varepsilon \cdot  v{\bf \{I-P\}}f_\varepsilon \langle v\rangle^{\frac32}\right\|_{L_x^{\frac6{3+2\varrho}}}\left\|\Lambda^{-\varrho}{\bf\{I-P\}}f_{\varepsilon}\right\|_D\nonumber\\
		&\lesssim&\left(\|E_\varepsilon\|_{L_x^{\frac3\varrho}}
		\left\|{\bf \{I-P\}}f_\varepsilon \langle v\rangle^{\frac52}\right\|\right)^2+\eta\left\|\Lambda^{-\varrho}{\bf\{I-P\}}f_{\varepsilon}\right\|^2_D\nonumber\\
		&\lesssim&\left\|{\bf \{I-P\}}f_\varepsilon \langle v\rangle^{\frac52}\right\|^2\left\|\Lambda^{\frac32-\varrho}E_\varepsilon\right\|^2
		+\eta\left\|\Lambda^{-\varrho}{\bf\{I-P\}}f_{\varepsilon}\right\|^2_D.
	\end{eqnarray}
	Consequently, one has
	\begin{eqnarray}
		&&\left|\left(\mathcal{F}[E_\varepsilon\cdot v f_\varepsilon]\mid|\xi|^{-2\varrho}\hat{f}_\varepsilon\right)\right|\nonumber\\
		&\lesssim&\left\|\Lambda^{-\varrho}f_{\varepsilon}\right\|\left(\left\|\Lambda^{\frac34-\frac\varrho2}E_\varepsilon\right\|^2+\left\|\Lambda^{\frac34-\frac\varrho2}f_{\varepsilon}\right\|^2_{D}\right)\nonumber\\
		&&+\left\|{\bf \{I-P\}}f_\varepsilon \langle v\rangle^{\frac52}\right\|^2\left\|\Lambda^{\frac32-\varrho}E_\varepsilon\right\|^2
		+\eta\left\|\Lambda^{-\varrho}{\bf\{I-P\}}f_{\varepsilon}\right\|^2_D.
	\end{eqnarray}
Similarly, one has
	\begin{eqnarray}
	&&\left|\left(\mathcal{F}[E_\varepsilon \cdot \nabla_vf_\varepsilon]\mid|\xi|^{-2\varrho}\hat{f}_\varepsilon\right)\right|\nonumber\\
		&\lesssim&\left\|\Lambda^{-\varrho}f_{\varepsilon}\right\|\left(\left\|\Lambda^{\frac34-\frac\varrho2}E_\varepsilon\right\|^2+\left\|\Lambda^{\frac34-\frac\varrho2}f_{\varepsilon}\right\|^2_{D}\right)\nonumber\\
		&&+\left\|\nabla_v{\bf \{I-P\}}f_\varepsilon \langle v\rangle^{\frac32}\right\|^2\left\|\Lambda^{\frac32-\varrho}E_\varepsilon\right\|^2
		+\eta\left\|\Lambda^{-\varrho}{\bf\{I-P\}}f_{\varepsilon}\right\|^2_D.
	\end{eqnarray}
	For the third term on the right-hand side of \eqref{4.6}, we have by repeating the argument used in deducing the estimate on the first two terms that
	\begin{eqnarray}
		&&\frac1{\varepsilon}\left|\left(\mathcal{F}[{q_0} v\times B_\varepsilon\cdot\nabla_{ v  }f_\varepsilon]\mid|\xi|^{-2\varrho}\hat{f}_\varepsilon\right)\right|\nonumber\\
		&\lesssim&\frac1{\varepsilon}\left|\left(\mathcal{F}[{q_0} v\times B_\varepsilon\cdot\nabla_{ v  }{\bf P}f_\varepsilon]\mid|\xi|^{-2\varrho}{\bf P}\hat{f}_\varepsilon\right)\right|\nonumber\\
		&&+\frac1{\varepsilon}\left|\left(\mathcal{F}[{q_0} v\times B_\varepsilon\cdot\nabla_{ v  }{\bf P}f_\varepsilon]\mid|\xi|^{-2\varrho}{\bf \{I-P\}}\hat{f}_\varepsilon\right)\right|\nonumber\\
		&&+\frac1{\varepsilon}\left|\left(\mathcal{F}[{q_0} v\times B_\varepsilon\cdot\nabla_{ v  }{\bf \{I-P\}}f_\varepsilon]\mid|\xi|^{-2\varrho}{\bf P}\hat{f}_\varepsilon\right)\right|\nonumber\\
		&&+\frac1{\varepsilon}\left|\left(\mathcal{F}[{q_0} v\times B_\varepsilon\cdot\nabla_{ v  }{\bf \{I-P\}}f_\varepsilon]\mid|\xi|^{-2\varrho}{\bf \{I-P\}}\hat{f}_\varepsilon\right)\right|\nonumber\\
		&=&\frac1{\varepsilon}\left|\left(\mathcal{F}[{q_0} v\times B_\varepsilon\cdot\nabla_{ v  }{\bf P}f_\varepsilon]\mid|\xi|^{-2\varrho}{\bf \{I-P\}}\hat{f}_\varepsilon\right)\right|\nonumber\\
		&&+\frac1{\varepsilon}\left|\left(\mathcal{F}[{q_0} v\times B_\varepsilon\cdot\nabla_{ v  }{\bf \{I-P\}}f_\varepsilon]\mid|\xi|^{-2\varrho}{\bf P}\hat{f}_\varepsilon\right)\right|\nonumber\\
		&&+\frac1{\varepsilon}\left|\left(\mathcal{F}[{q_0} v\times B_\varepsilon\cdot\nabla_{ v  }{\bf \{I-P\}}f_\varepsilon]\mid|\xi|^{-2\varrho}{\bf \{I-P\}}\hat{f}_\varepsilon\right)\right|
	\end{eqnarray}
	where we used the fact that
	\[\frac1{\varepsilon}\left|\left(\mathcal{F}[{q_0} v\times B_\varepsilon\cdot\nabla_{ v  }{\bf P}f_\varepsilon]\mid|\xi|^{-2\varrho}{\bf P}\hat{f}_\varepsilon\right)\right|=0.\]
	Similar with the argument used in \eqref{E-s}, one has
	\begin{eqnarray}
		&&\frac1{\varepsilon}\left|\left(\mathcal{F}[{q_0} v\times B_\varepsilon\cdot\nabla_{ v  }{\bf P}f_\varepsilon]\mid|\xi|^{-2\varrho}{\bf \{I-P\}}\hat{f}_\varepsilon\right)\right|\nonumber\\
		&&+\frac1{\varepsilon}\left|\left(\mathcal{F}[{q_0} v\times B_\varepsilon\cdot\nabla_{ v  }{\bf \{I-P\}}f_\varepsilon]\mid|\xi|^{-2\varrho}{\bf P}\hat{f}_\varepsilon\right)\right|\nonumber\\
		&\lesssim&\left(\left\|\Lambda^{-\varrho}f_{\varepsilon}\right\|^2+\left\|f_{\varepsilon}\right\|^2\right)\left(\left\|\Lambda^{\frac32-\varrho}B_\varepsilon\right\|^2
		+\frac1{\varepsilon^2}\left\|\Lambda^{1-\varrho}{\bf\{I-P\}}f_{\varepsilon}\right\|^2_{D}\right)
		\nonumber\\
		&&+\frac\eta{\varepsilon^2}\left\|\Lambda^{-\varrho}{\bf\{I-P\}}f_{\varepsilon}\right\|^2_{D}.
	\end{eqnarray}
By using the similar way as \eqref{E-mic-s}, one has
\begin{eqnarray}
	&& \frac1{\varepsilon}\left|\left(\mathcal{F}[{q_0} v\times B_\varepsilon\cdot\nabla_{ v  }{\bf \{I-P\}}f_\varepsilon]\mid|\xi|^{-2\varrho}{\bf \{I-P\}}\hat{f}_\varepsilon\right)\right|\nonumber\\
	&\lesssim&\left\|\nabla_v{\bf \{I-P\}}f \langle v\rangle^{\frac52}\right\|^2\left\|\Lambda^{\frac32-\varrho}B_\varepsilon\right\|^2
		+ \frac\eta{\varepsilon^2}\left\|\Lambda^{-\varrho}{\bf\{I-P\}}f_{\varepsilon}\right\|^2_{D},
	\end{eqnarray}
	consequently, one  has
	\begin{eqnarray}\label{B-s-neg}
		&&\frac1{\varepsilon}\left|\left(\mathcal{F}[{q_0} v\times B_\varepsilon\cdot\nabla_{ v  }f_\varepsilon]\mid|\xi|^{-2\varrho}\hat{f}_\varepsilon\right)\right|\nonumber\\
		&\lesssim&\left\|\Lambda^{-\varrho}f_{\varepsilon}\right\|\left(\left\|\Lambda^{1-\varrho}B_\varepsilon\right\|^2+\frac1{\varepsilon^2}\left\|\Lambda^{1-\varrho}{\bf\{I-P\}}f_{\varepsilon}\right\|^2_{D}\right)\nonumber\\
		&&+\left\|f\langle v\rangle^{\frac52}\right\|^2\left\|\Lambda^{1-\varrho}B_\varepsilon\right\|^2
		+\frac\eta{\varepsilon^2}\left\|\Lambda^{-\varrho}{\bf\{I-P\}}f_{\varepsilon}\right\|^2_{D}\nonumber\\
		&\lesssim&\left\{\left\|\Lambda^{-\varrho}f_{\varepsilon}\right\|+\left\|f\langle v\rangle^{\frac52}\right\|^2\right\}\left(\left\|\Lambda^{1-\varrho}B_\varepsilon\right\|^2+\frac1{\varepsilon^2}\left\|\Lambda^{1-\varrho}{\bf\{I-P\}}f_{\varepsilon}\right\|^2_{D}\right)\nonumber\\
		&&+\left\|\nabla_v{\bf \{I-P\}}f \langle v\rangle^{\frac52}\right\|^2\left\|\Lambda^{\frac32-\varrho}B_\varepsilon\right\|^2
		+ \frac\eta{\varepsilon^2}\left\|\Lambda^{-\varrho}{\bf\{I-P\}}f_{\varepsilon}\right\|^2_{D}.
	\end{eqnarray}
	As to the last term on the right-hand side of \eqref{4.6},  one has
	\begin{eqnarray*}
		&&\frac1{\varepsilon}\left(\mathcal{F}[{\mathscr{T}}(f_\varepsilon,f_\varepsilon)]\mid|\xi|^{-2\varrho}\mathcal{F}\left[{{\bf\{I-P\}}f_{\varepsilon}}\right]\right)\\
		&\lesssim&\left\|\Lambda^{-\varrho}\left(\langle v\rangle^{\frac32}{\mathscr{T}}(f_\varepsilon,f_\varepsilon)\right)\right\|^2+\frac\eta{\varepsilon^2}\left\|\Lambda^{-\varrho}{\bf\{I-P\}}f_{\varepsilon}\right\|^2_{D}\\
	&\lesssim&\left\|\langle v\rangle^{\frac32}{\mathscr{T}}(f_\varepsilon,f_\varepsilon)\right\|_{L_x^{\frac{6}{3+2\varrho}}L^2_v}^2+\frac\eta{\varepsilon^2}\left\|\Lambda^{-\varrho}{\bf\{I-P\}}f_{\varepsilon}\right\|^2_{D}\\
&\lesssim&\left\|\left|\langle v\rangle^{\frac32}f_\varepsilon\right|_{H^3_{s+\frac\gamma2}}\left|\langle v\rangle^{\frac32}f_\varepsilon\right|_{H^1_{s+\frac\gamma2}}\right\|^2_{L_x^{\frac{6}{3+2\varrho}}}+\frac\eta{\varepsilon^2}\left\|\Lambda^{-\varrho}{\bf\{I-P\}}f_{\varepsilon}\right\|^2_{D}\\
&\lesssim&\left\|\langle v\rangle^{\frac32}f_\varepsilon\right\|^2_{L^2_xH^3_{s+\frac\gamma2}}\left\|\langle v\rangle^{\frac32}\Lambda^{\frac32-\varrho}f_\varepsilon\right\|_{H^1_{s+\frac\gamma2}}^2+\frac\eta{\varepsilon^2}\left\|\Lambda^{-\varrho}{\bf\{I-P\}}f_{\varepsilon}\right\|^2_{D}.
	\end{eqnarray*}

	Substituting the estimates  into (\ref{4.6}) yields \eqref{Lemma4.1-1}, which complete the proof of Lemma $\ref{Lemma4.1}$.
\end{proof}
\begin{lemma}\label{Lemma4.2}
	There exists an interactive functional $G^{-\varrho}_{f_\varepsilon}(t)$ satisfying
	\begin{equation*}
		G^{-\varrho}_{f_\varepsilon}(t)\lesssim \left\|\Lambda^{1-\varrho}[f_\varepsilon,E_\varepsilon,B_\varepsilon]\right\|^2+\left\|\Lambda^{-\varrho}[f_\varepsilon,E_\varepsilon,B_\varepsilon]\right\|^2+\|\Lambda^{\frac32}E_\varepsilon\|^2
	\end{equation*}
	such that
	\begin{eqnarray}\label{Lemma4.2-1}
		&&\frac{\d}{\d t}G^{-\varrho}_{f_\varepsilon}(t)+\left\|\Lambda^{1-\varrho}{\bf P}f_{\varepsilon}\right\|^2+\left\|\Lambda^{1-\varrho}[E_\varepsilon,B_\varepsilon]\right\|^2+\left\|\Lambda^{-\varrho}E_\varepsilon\right\|_{H^1_x}^2+\left\|\Lambda^{-\varrho}(\rho_\varepsilon^+-\rho_\varepsilon^-)\right\|^2_{H^1_x}\nonumber\\
		&\lesssim&\frac1{\varepsilon^2}\left\|\Lambda^{-\varrho}\{{\bf I-P}\}f_{\varepsilon}\right\|^2_{H^2_xL^2_{D}}+\mathcal{E}_{2}(t)\mathcal{D}_{2}(t)
	\end{eqnarray}
	holds for any $0\leq t\leq T$.
\end{lemma}
\begin{proof}
	This lemma can be proved by a similar way as Lemma \ref{mac-dissipation}, we omit its proof for brevity.
\end{proof}
Based on the above two lemmas, one has
\begin{proposition}\label{prop-neg}
	Under Under {\bf Assumption 1} and {\bf Assumption 2}, one has
	\begin{eqnarray}\label{negative-estimate}
		&&\frac{\d}{\d t}\left(\left\|\Lambda^{-\varrho}f_{\varepsilon}\right\|^2+\left\|\Lambda^{-\varrho}[E_\varepsilon,B_\varepsilon]\right\|^2+\kappa_1G^{-\varrho}_{f_\varepsilon}(t)\right)+\frac1{\varepsilon^2}\left\|\Lambda^{-\varrho}\{{\bf I-P}\}f_{\varepsilon}\right\|_{D}^2\nonumber\\
		&&+\kappa_1\left\|\Lambda^{1-\varrho}{\bf P}f_{\varepsilon}\right\|^2+\kappa_1\left\|\Lambda^{1-\varrho}[E_\varepsilon,B_\varepsilon]\right\|^2+\kappa_1\left\|\Lambda^{-\varrho}E_\varepsilon\right\|_{H^1_x}^2+\kappa_1\left\|\Lambda^{-\varrho}(\rho_\varepsilon^+-\rho_\varepsilon^-)\right\|^2_{H^1_x}\nonumber\\
		&\lesssim&\left(M_2^{\frac12}+M_2\right)\left(\left\|\nabla_x[{\bf P}f_\varepsilon,E_\varepsilon,B_\varepsilon]\right\|^2
		+\frac1{\varepsilon^2}\left\|\nabla_x{\bf\{I-P\}}f_{\varepsilon}\right\|^2_{D}+\left\|\langle v\rangle^{\frac32}\nabla_xf_\varepsilon\right\|_{L^2_xH^1_{s+\frac\gamma2}}^2\right)\nonumber\\
		&&+\frac1{\varepsilon^2}\left\|\nabla^2_x\{{\bf I-P}\}f_{\varepsilon}\right\|^2_{{D}}+\mathcal{E}_{2}(t)\mathcal{D}_{2}(t).
	\end{eqnarray}
\end{proposition}
\begin{proof}
	By the interpolation with respect to the spatial derivatives, one has
	\begin{eqnarray}
		&&\frac1{\varepsilon^2}\left\|\Lambda^{1-\varrho}{\bf\{I-P\}}f_{\varepsilon}\right\|^2_{H^1_xL^2_{D}}\nonumber\\
		&\lesssim&\frac1{\varepsilon^2}\left\|\nabla^2_x{\bf\{I-P\}}f_{\varepsilon}\right\|^2_{D}+
		\frac\eta{\varepsilon^2}\left\|\Lambda^{-\varrho}{\bf\{I-P\}}f_{\varepsilon}\right\|^2_{D},
	\end{eqnarray}
	and for $\frac12<\varrho<\frac32$, then $\frac34-\frac\varrho2> 1-\varrho$,
	 \[\left\|\Lambda^{\frac34-\frac\varrho2}E_\varepsilon\right\|^2+\left\|\Lambda^{\frac34-\frac\varrho2}B_\varepsilon\right\|^2\lesssim\eta\left\|\Lambda^{1-\varrho}[E_\varepsilon,B_\varepsilon]\right\|^2+C_\eta\left\|\nabla_x[E_\varepsilon,B_\varepsilon]\right\|^2,\]
Similarly,
\[\left\|\Lambda^{\frac32-\varrho}[E_\varepsilon,B_\varepsilon]\right\|^2\lesssim\eta\left\|\Lambda^{1-\varrho}[E_\varepsilon,B_\varepsilon]\right\|^2+C_\eta\left\|\nabla_x[E_\varepsilon,B_\varepsilon]\right\|^2.\]	
By using the above estimates, one can deduce  \eqref{negative-estimate} from a proper linear combination of \eqref{Lemma4.1-1} and \eqref{Lemma4.2-1}.
\end{proof}

\subsection{The a priori estimates}
Based on Proposition \ref{prop1} and \ref{prop-neg}, we are ready to construct the a priori estimates
\begin{equation}\label{Def-a-priori}
  \sup_{0<t\leq T}\left\{\left\|\Lambda^{-\varrho}[f_\varepsilon,E_\varepsilon,B_\varepsilon]\right\|^2+\mathcal{E}_{N,l}(t)+\mathcal{E}_{N}(t)\right\}\leq \overline{M}
\end{equation}
where $ \overline{M}\leq \min\{M_1,M_2\}$ is a sufficiently small positive constant.

One has
\begin{eqnarray}
\frac{\d}{\d t}\left\{\left\|\Lambda^{-\varrho}[f_\varepsilon,E_\varepsilon,B_\varepsilon]\right\|^2+\mathcal{E}_{N,l}(t)+\mathcal{E}_{N}(t)\right\}
+\mathcal{D}_{N}(t)+\mathcal{D}_{N,l}(t)\lesssim 0,
\end{eqnarray}
which gives that
\begin{eqnarray}
  &&\left\|\Lambda^{-\varrho}[f_\varepsilon,E_\varepsilon,B_\varepsilon](t)\right\|^2+\mathcal{E}_{N,l}(t)+\mathcal{E}_{N}(t)\nonumber\\
  &\lesssim&\left\|\Lambda^{-\varrho}[f_\varepsilon,E_\varepsilon,B_\varepsilon](0)\right\|^2+\mathcal{E}_{N,l}(0)+\mathcal{E}_{N}(0)
  \lesssim Y^2_{f_\varepsilon,E_\varepsilon,B_\varepsilon}(0)
\end{eqnarray}
where $Y_{f_\varepsilon,E_\varepsilon,B_\varepsilon}(0)$ is defined in \eqref{Def-Y_0},
thus we close the a priori estimates if the initial data $Y_{f_\varepsilon,E_\varepsilon,B_\varepsilon}(0)$ is taken as small sufficiently.

Furthermore, we can get from Proposition \ref{noncut-decay}
\begin{equation}
\mathcal{E}_{k\rightarrow N_0}(t)\lesssim Y^2_{f_\varepsilon,E_\varepsilon,B_\varepsilon}(0)(1+t)^{-(k+\varrho)},\quad 0\leq t\leq T.
\end{equation}

\section{The cutoff VMB system}

\subsection{Lyapunov inequality for the energy functional $\mathcal{E}_{N}(t)$}

\begin{proposition}
There exist an energy functional $\mathcal{E}_{N}(t)$ and the corresponding energy dissipation functional $\mathcal{D}_{N}(t)$ which satisfy (\ref{E-N}), (\ref{D-N}) respectively such that
	\begin{eqnarray}\label{prof-1-cut}
		&&\frac{\d}{\d t}\mathcal{E}_{N}(t)+\mathcal{D}_{N}(t)\nonumber\\
		&\lesssim &\left\|E_\varepsilon\right\|_{L^\infty_x}^2\left\|\langle v\rangle^{\frac 32}\nabla^N_xf_{\varepsilon}\right\|^2+\left\|\nabla_xB_\varepsilon\right\|_{L^\infty_x}^2\left\|\langle v\rangle^{\frac 32}\nabla^{N-1}_x\nabla_v\{{\bf I-P}\}f_{\varepsilon}\right\|^2\nonumber\\
		&&+\mathcal{E}_N(t)\sum_{|\alpha'|+|\beta'|\leq N-1}\left\|\langle v\rangle^{\frac32}\partial^{\alpha'}_{\beta'} \{{\bf I-P}\}f_{\varepsilon}\right\|^2+\mathcal{E}_N(t)\mathcal{D}_N(t).
	\end{eqnarray}
\end{proposition}
\begin{proof}
This proposition can be proved by a similar way as to deduce \eqref{prof-1}, we omit its proof for brevity.	
	\end{proof}

\subsection{The bound for energy functionals $\overline{\mathcal{E}}_{N-1,l_1}(t)$ and $\overline{\mathcal{E}}_{N_0-1,l_0}(t)$}
{\bf Assumption I:}
\begin{eqnarray}\label{Assump-1-cut}
	\sup_{0<t\leq T}\left\{\sum_{k\leq N_0-2}(1+t)^{k+\varrho}\mathcal{E}_{k\rightarrow N_0}(t)+(1+t)^{1+\epsilon_1}\overline{\mathcal{E}}_{N_0-1,l_0}(t)+\mathcal{E}_{N}(t)\right\}\leq \mathcal{M}_1
\end{eqnarray}
where $\mathcal{M}_1$ is a sufficiently small positive constant.

\begin{proposition}\label{bound-E}
Under {\bf Assumption 1},
assume $\tilde{\ell}\geq\frac32\sigma_{N-1,N-1}$, $\theta=\frac3{\tilde{\ell}+\frac12}$ and $\ell_1\geq l_1+\tilde{\ell}+\frac12$, then one has
	\begin{eqnarray}\label{bound-E-1}
		&&\frac{\d}{\d t}\overline{\mathcal{E}}_{N-1,l_1}(t)+\overline{\mathcal{D}}_{N-1,l_1}(t)\nonumber\\
			&\lesssim&\mathcal{D}_N(t)+\overline{\mathcal{E}}_{N-1,l_1}(t)\overline{\mathcal{D}}_{N-1,l_1}(t)+{\mathcal{M}_1}^{\frac2{\theta}}\sum_{N_0+1\leq n\leq N-1}\mathbb{D}^{(n)}_{\ell_1}(t).
	\end{eqnarray}
\end{proposition}
\begin{proof}	
	Applying $\partial^\alpha_\beta$ into \eqref{I-P-cut}, integrating the result identity over $\mathbb{R}^3_x\times\mathbb{R}^3_v$ by multiplying $\overline{w}^2_{l_1}(\alpha,\beta)\partial^\alpha_\beta\{{\bf I-P}\}f_{\varepsilon}$, then we have
	\begin{eqnarray}\label{bound-I-P}
		&&\frac12\frac{\d}{\d t}\|\overline {w}_{l_1}(\alpha,\beta)\partial^\alpha_\beta\{{\bf I-P}\}f_{\varepsilon}\|^2+\frac{q\vartheta}{(1+t)^{1+\vartheta}}\left\|\langle v\rangle\overline {w}_{l_1}(\alpha,\beta)\partial^\alpha_\beta\{{\bf I-P}\}f_{\varepsilon}\right\|^2\nonumber\\
		&&+\frac1 {\varepsilon^2} \left(\partial^\alpha_\beta\mathscr{L}f_\varepsilon,\overline {w}^2_{l_1}(\alpha,\beta)\partial^\alpha_\beta\{{\bf I-P}\}f_{\varepsilon}\right)\nonumber\\
		&=&-\frac1\varepsilon\left(\partial^\alpha_\beta[ v\cdot \nabla_x\{{\bf I-P}\}f_{\varepsilon}],\overline {w}^2_{l_1}(\alpha,\beta)\partial^\alpha_\beta\{{\bf I-P}\}f_{\varepsilon}\right)\nonumber\\
		&&+\frac1 \varepsilon \left(\partial^\alpha_\beta \left[E_\varepsilon\cdot v{M}^{1/2}q_1\right],\overline {w}^2_{l_1}(\alpha,\beta)\partial^\alpha_\beta\{{\bf I-P}\}f_{\varepsilon}\right)\nonumber\\
		&&+\frac1 \varepsilon\left(\partial^\alpha_\beta\left[{\bf P}(v\cdot\nabla_x f_\varepsilon)-\frac1 \varepsilon v\cdot\nabla_x{\bf P}f_\varepsilon\right],\overline {w}^2_{l_1}(\alpha,\beta)\partial^\alpha_\beta\{{\bf I-P}\}f_{\varepsilon}\right)\nonumber\\
		&&+\frac1\varepsilon\left(\partial^\alpha_\beta\left[\{{\bf I-P}\}\left[q_0 v\times B_\varepsilon\cdot\nabla_{ v}f_\varepsilon\right]\right],\overline {w}^2_{l_1}(\alpha,\beta)\partial^\alpha_\beta\{{\bf I-P}\}f_{\varepsilon}\right)\nonumber\\
		&&+\left(\partial^\alpha_\beta\{{\bf I-P}\}\left[ -\frac12q_0 v\cdot E_\varepsilon f_\varepsilon+q_0E_\varepsilon\cdot\nabla_{ v}f_\varepsilon\right],\overline {w}^2_{l_1}(\alpha,\beta)\partial^\alpha_\beta\{{\bf I-P}\}f_{\varepsilon}\right)\nonumber\\
			&&+\frac1\varepsilon\left(\partial^\alpha_\beta\mathscr{T}(f_\varepsilon,f_\varepsilon),\overline {w}^2_{l_1}(\alpha,\beta)\partial^\alpha_\beta\{{\bf I-P}\}f_{\varepsilon}\right).
	\end{eqnarray}
	The coercivity estimates on the linear operator $\mathscr{L}$ yields that
	\begin{eqnarray}
		&&\frac1 {\varepsilon^2} \left(\partial^\alpha_\beta\mathscr{L}f_\varepsilon,\overline {w}^2_{l_1}(\alpha,\beta)\partial^\alpha_\beta\{{\bf I-P}\}f_{\varepsilon}\right)\nonumber\\
		&\gtrsim&\frac1 {\varepsilon^2} \|\overline {w}_{l_1}(\alpha,\beta)\partial^\alpha_\beta\{{\bf I-P}\}f_{\varepsilon}\|_\nu^2-\frac1 {\varepsilon^2}\|\partial^\alpha\{{\bf I-P}\}f_{\varepsilon}\|_\nu^2.
	\end{eqnarray}
	As for the transport term on the right-hand side of \eqref{bound-I-P}, one has
	\begin{eqnarray}
		&&\frac1\varepsilon\left(\partial^\alpha_\beta\left[ v\cdot \nabla_x\{{\bf I-P}\}f_{\varepsilon}\right],\overline{w}^2_{l_1-|\beta|}\partial^\alpha_\beta\{{\bf I-P}\}f_{\varepsilon}\right)\nonumber\\
		&=&-\frac1\varepsilon\int_{\mathbb{R}^3_x\times\mathbb{R}^3_v}\langle v\rangle^{-1}\partial^{\alpha+e_i}_{\beta-e_i}\{{\bf I-P}\}f_{\varepsilon} \overline{w}_{l_1}(\alpha+e_i,\beta-e_i)\overline {w}_{l_1}(\alpha,\beta)\partial^\alpha_\beta\{{\bf I-P}\}f_{\varepsilon}dvdx\nonumber\\
		&\lesssim&\frac1\varepsilon\left\|\langle v\rangle^{\frac\gamma2}\overline{w}_{l_1}(\alpha+e_i,\beta-e_i)\partial^{\alpha+e_i}_{\beta-e_i}\{{\bf I-P}\}f_{\varepsilon}\right\|^2+\frac\eta\varepsilon\left\|\langle v\rangle^{\frac\gamma2}\overline{w}_{l_1}(\alpha,\beta)\partial^{\alpha}_{\beta}\{{\bf I-P}\}f_{\varepsilon}\right\|^2,	
	\end{eqnarray}
	where we used the fact
	\[\overline {w}_{l_1}(\alpha,\beta)=\langle v\rangle^{-1}\overline{w}_{l_1}(\alpha+e_i,\beta-e_i).\]	
Now we deal with the nonlinear term induced by magnetic field $B(t,x)$,
	\begin{eqnarray}
		&&\frac1\varepsilon\left(\partial^\alpha_\beta\left[ v\times B_\varepsilon\cdot\nabla_{ v}\{{\bf I-P}\}f_{\varepsilon}\right],\overline{w}^2_{l_1}(\alpha,\beta)\partial^\alpha_\beta\{{\bf I-P}\}f_{\varepsilon}\right)\nonumber\\
		&=&\frac1\varepsilon\sum_{\alpha_1\leq\alpha,\beta_1\leq\beta}\left(\partial^{\alpha_1}_{\beta_1}\left[ v\times B_\varepsilon\right]\cdot\partial^{\alpha-\alpha_1}_{\beta-\beta_1}\left[\nabla_{ v}\{{\bf I-P}\}f_{\varepsilon}\right],\overline{w}^2_{l_1}(\alpha,\beta)\partial^\alpha_\beta\{{\bf I-P}\}f_{\varepsilon}\right),\nonumber
	\end{eqnarray}
	when $|\alpha_1|=1,\beta_1=0$,
	we apply Sobolev inequalities, interpolation method and Young inequalites to get
	\begin{eqnarray}
		&\lesssim&\sum_{|\alpha_1|=1}{\frac1\varepsilon}\|\partial^{\alpha_1}B_\varepsilon\|_{L^\infty_x}\|\langle v\rangle^{\frac52} \overline{w}_{l_1}(\alpha-\alpha_1,\beta+e_i)\partial^{\alpha-\alpha_1}_{\beta+e_i}{\bf \{I-P\}}f_{\varepsilon}\|\nonumber\\
		&&\times\|\langle v\rangle^{-\frac{1}2} \overline {w}_{l_1}(\alpha,\beta)\partial^{\alpha}_\beta{\bf \{I-P\}}f_{\varepsilon}\|\nonumber\\
		&\lesssim&\sum_{|\alpha_1|=1}
		\|\partial^{\alpha_1}B_\varepsilon\|^2_{L^\infty_x}\|\langle v\rangle^{\frac{5}2} \overline{w}_{l_1}(\alpha-\alpha_1,\beta+e_i)\partial^{\alpha-\alpha_1}_{\beta +e_i}{\bf \{I-P\}}f_{\varepsilon}\|^2\nonumber\\
		&&+{\frac\eta{\varepsilon^2}}\| \overline {w}_{l_1}(\alpha,\beta)\partial^{\alpha}_\beta{\bf \{I-P\}}f_{\varepsilon}\|_\nu^2\nonumber\\
		&\lesssim&\sum_{|\alpha_1|=1}
		\left\{	\left\{\|\partial^{\alpha_1}B_\varepsilon\|^2_{L^\infty_x}\right\}^{\frac1\theta}\|\langle v\rangle^{\tilde{\ell}} \overline{w}_{l_1}(\alpha-\alpha_1,\beta+e_i)\partial^{\alpha-\alpha_1}_{\beta +e_i}{\bf \{I-P\}}f_{\varepsilon}\|^2\right\}^{\theta}\nonumber\\
		&&\times\left\{\|\langle v\rangle^{-\frac{1}2} \overline{w}_{l_1}(\alpha-\alpha_1,\beta+e_i)\partial^{\alpha-\alpha_1}_{\beta +e_i}{\bf \{I-P\}}f_{\varepsilon}\|^2\right\}^{1-\theta}\nonumber\\
		&&+{\frac\eta{\varepsilon^2}}\| \overline {w}_{l_1}(\alpha,\beta)\partial^{\alpha}_\beta{\bf \{I-P\}}f_{\varepsilon}\|_\nu^2\nonumber\\
		&\lesssim&\sum_{|\alpha_1|=1}
	\|\partial^{\alpha_1}B_\varepsilon\|^{\frac2{\theta}}_{L^\infty_x}\|\langle v\rangle^{\tilde{\ell}} \overline{w}_{l_1}(\alpha-\alpha_1,\beta+e_i)\partial^{\alpha-\alpha_1}_{\beta +e_i}{\bf \{I-P\}}f_{\varepsilon}\|^2\nonumber\\
	&&+\eta\sum_{|\alpha_1|=1}\| \overline{w}_{l_1}(\alpha-\alpha_1,\beta+e_i)\partial^{\alpha-\alpha_1}_{\beta +e_i}{\bf \{I-P\}}f_{\varepsilon}\|_\nu^2+{\frac\eta{\varepsilon^2}}\| \overline {w}_{l_1}(\alpha,\beta)\partial^{\alpha}_\beta{\bf \{I-P\}}f_{\varepsilon}\|_\nu^2,	
	\end{eqnarray}
	where we take $\theta$ as
	\[\theta=\frac3{\tilde{\ell}+\frac12}\]
	and use the fact that
	\[\langle v\rangle^{\frac32}\overline{w}_{l_1}(\alpha,\beta)\leq \langle v\rangle^{\frac52}\overline{w}_{l_1}(\alpha-e_i,\beta+e_i),\ |\alpha_1|=1,\beta_1=0.\]
	When $|\alpha_1|=2,\beta_1=0$ or $|\alpha_1|=3,\beta_1=0$, one has the similar bound
\begin{eqnarray}
		&\lesssim&\sum_{|\alpha_1|=2, 3}
	\|\partial^{\alpha_1}B_\varepsilon\|^{\frac2\theta}_{L^\infty_x}\|\langle v\rangle^{\tilde{\ell}} \overline{w}_{l_1}(\alpha-\alpha_1,\beta+e_i)\partial^{\alpha-\alpha_1}_{\beta +e_i}{\bf \{I-P\}}f_{\varepsilon}\|^2\nonumber\\
	&&+\eta\sum_{|\alpha_1|=2,3}\| \overline{w}_{l_1}(\alpha-\alpha_1,\beta+e_i)\partial^{\alpha-\alpha_1}_{\beta +e_i}{\bf \{I-P\}}f_{\varepsilon}\|_\nu^2+{\frac\eta{\varepsilon^2}}\| \overline {w}_{l_1}(\alpha,\beta)\partial^{\alpha}_\beta{\bf \{I-P\}}f_{\varepsilon}\|_\nu^2.
\end{eqnarray}
While for $4\leq |\alpha_1|\leq N-2$, since
\[\langle v\rangle^{\frac32}\overline{w}_{l_1}(\alpha,\beta)\leq \langle v\rangle^{-\frac12}\overline{w}_{l_1}(\alpha-\alpha_1,\beta+e_i),\]
one deduces that
\begin{eqnarray}
		&\lesssim&\sum_{4\leq |\alpha_1|\leq N-2}{\frac1\varepsilon}\|\partial^{\alpha_1}B_\varepsilon\|_{L^\infty_x}\|\langle v\rangle^{-\frac12} \overline{w}_{l_1}(\alpha-\alpha_1,\beta+e_i)\partial^{\alpha-\alpha_1}_{\beta+e_i}{\bf \{I-P\}}f_{\varepsilon}\|\nonumber\\
	&&\times\|\langle v\rangle^{-\frac{1}2} \overline {w}_{l_1}(\alpha,\beta)\partial^{\alpha}_\beta{\bf \{I-P\}}f_{\varepsilon}\|\nonumber\\
	&\lesssim&\sum_{4\leq |\alpha_1|\leq N-2}
	\|\partial^{\alpha_1}B_\varepsilon\|^2_{L^\infty_x}\| \overline{w}_{l_1}(\alpha-\alpha_1,\beta+e_i)\partial^{\alpha-\alpha_1}_{\beta +e_i}{\bf \{I-P\}}f_{\varepsilon}\|_\nu^2\nonumber\\
	&&+{\frac\eta{\varepsilon^2}}\| \overline {w}_{l_1}(\alpha,\beta)\partial^{\alpha}_\beta{\bf \{I-P\}}f_{\varepsilon}\|_\nu^2\nonumber\\
	&\lesssim&\mathcal{E}_N(t)\overline{\mathcal{D}}_{N-1,l_1}(t)+{\frac\eta{\varepsilon^2}}\| \overline {w}_{l_1}(\alpha,\beta)\partial^{\alpha}_\beta{\bf \{I-P\}}f_{\varepsilon}\|_\nu^2.\nonumber
	\end{eqnarray}
If $|\alpha_1|=N-1$, one also has that the term can be bounded by
\[\mathcal{E}_N(t)\overline{\mathcal{D}}_{N-1,l_1}(t)+{\frac\eta{\varepsilon^2}}\| \overline {w}_{l_1}(\alpha,\beta)\partial^{\alpha}_\beta{\bf \{I-P\}}f_{\varepsilon}\|_\nu^2.\]
Thus, one has
	\begin{eqnarray}
		&&\frac1\varepsilon\left(\partial^\alpha_\beta\left[ v\times B_\varepsilon\cdot\nabla_{ v}\{{\bf I-P}\}f_{\varepsilon}\right],\overline{w}^2_{l_1}(\alpha,\beta)\partial^\alpha_\beta\{{\bf I-P}\}f_{\varepsilon}\right)\nonumber\\
		&\lesssim&\sum_{|\alpha_1|=1,2, 3}
	\|\partial^{\alpha_1}B_\varepsilon\|^{\frac2\theta_{\varepsilon}}_{L^\infty_x}\|\langle v\rangle^{\tilde{\ell}} \overline{w}_{l_1}(\alpha-\alpha_1,\beta+e_i)\partial^{\alpha-\alpha_1}_{\beta +e_i}{\bf \{I-P\}}f_{\varepsilon}\|^2\nonumber\\
	&&+\mathcal{E}_N(t)\overline{\mathcal{D}}_{N-1,{l_1}}(t)+\eta\overline{\mathcal{D}}_{N-1,{l_1}}(t).
	\end{eqnarray}
	By using the similar argument, one also has
		\begin{eqnarray}
		&&\left(\partial^\alpha_\beta\left[ E_\varepsilon\cdot\nabla_{ v}\{{\bf I-P}\}f_{\varepsilon}\right],\overline{w}^2_{l_1}(\alpha,\beta)\partial^\alpha_\beta\{{\bf I-P}\}f_{\varepsilon}\right)\nonumber\\
		&\lesssim&\sum_{|\alpha_1|\leq 3}
		\|\partial^{\alpha_1}E_\varepsilon\|^{\frac2\theta}_{L^\infty_x}\|\langle v\rangle^{\tilde{\ell}} \overline{w}_{l_1}(\alpha-\alpha_1,\beta+e_i)\partial^{\alpha-\alpha_1}_{\beta +e_i}{\bf \{I-P\}}f_{\varepsilon}\|^2\nonumber\\
		&&+\mathcal{E}_N(t)\overline{\mathcal{D}}_{N-1,{l_1}}(t)+\eta\overline{\mathcal{D}}_{N-1,{l_1}}(t)
	\end{eqnarray}
and
		\begin{eqnarray}
	&&\left(\partial^\alpha_\beta\left[ E_\varepsilon\cdot v\{{\bf I-P}\}f_{\varepsilon}\right],\overline{w}^2_{l_1}(\alpha,\beta)\partial^\alpha_\beta\{{\bf I-P}\}f_{\varepsilon}\right)\nonumber\\
	&\lesssim&\sum_{|\alpha_1|\leq 3}
	\|\partial^{\alpha_1}E_\varepsilon\|^{\frac2\theta}_{L^\infty_x}\|\langle v\rangle^{\tilde{\ell}} \overline{w}_{l_1}(\alpha-\alpha_1,\beta)\partial^{\alpha-\alpha_1}_{\beta}{\bf \{I-P\}}f_{\varepsilon}\|^2\nonumber\\
	&&+\mathcal{E}_N(t)\overline{\mathcal{D}}_{N-1,{l_1}}(t)+\eta\overline{\mathcal{D}}_{N-1,{l_1}}(t).
\end{eqnarray}
For the last term, one has
	\begin{eqnarray}
		\frac1\varepsilon\left(\partial^\alpha_\beta\mathscr{T}(f_\varepsilon,f_\varepsilon),\overline {w}^2_{l_1}(\alpha,\beta)\partial^\alpha_\beta\{{\bf I-P}\}f_{\varepsilon}\right)
		\lesssim\overline{\mathcal{E}}_{N-1,{l_1}}(t)\overline{\mathcal{D}}_{N-1,{l_1}}(t)+\eta\overline{\mathcal{D}}_{N-1,{l_1}}(t).
	\end{eqnarray}
	Now by collecting the above related estimates into \eqref{bound-I-P}, we arrive at
	\begin{eqnarray}\label{bound-sum}
			&&\frac{\d}{\d t}\left\|\overline{w}_{l_1}(\alpha,\beta)\partial^\alpha_\beta\{{\bf I-P}\}f_{\varepsilon}\right\|^2+\frac1{\varepsilon^2}\left\|\overline{w}_{l_1}(\alpha,\beta)\partial^\alpha_\beta\{{\bf I-P}\}f_{\varepsilon}\right\|^2_{\nu}\nonumber\\
			&&+\frac{q\vartheta}{(1+t)^{1+\vartheta}}\left\|\langle v\rangle \overline{w}_{l_1}(\alpha,\beta)\partial^\alpha_\beta\{{\bf I-P}\}f_{\varepsilon}\right\|^2\nonumber\\
			&\lesssim&\left\|\overline{w}_{l_1}(\alpha+e_i,\beta-e_i)\partial^{\alpha+e_i}_{\beta-e_i}\{{\bf I-P}\}f_{\varepsilon}\right\|_\nu^2+\overline{\mathcal{E}}_{N-1,{l_1}}(t)\overline{\mathcal{D}}_{N-1,{l_1}}(t)+\eta\overline{\mathcal{D}}_{N-1,{l_1}}(t)\nonumber\\
			&&+\sum_{|\alpha_1|\leq 3}
			\|\partial^{\alpha_1}E_\varepsilon\|^{\frac2\theta}_{L^\infty_x}\|\langle v\rangle^{\tilde{\ell}} \overline{w}_{l_1}(\alpha-\alpha_1,\beta+e_i)\partial^{\alpha-\alpha_1}_{\beta +e_i}{\bf \{I-P\}}f_{\varepsilon}\|^2\nonumber\\
			&&+\sum_{|\alpha_1|\leq 3}
			\|\partial^{\alpha_1}E_\varepsilon\|^{\frac2\theta}_{L^\infty_x}\|\langle v\rangle^{\tilde{\ell}} \overline{w}_{l_1}(\alpha-\alpha_1,\beta)\partial^{\alpha-\alpha_1}_{\beta}{\bf \{I-P\}}f_{\varepsilon}\|^2\nonumber\\
			&&+\sum_{|\alpha_1|\leq 3}
			\|\partial^{\alpha_1}B_\varepsilon\|^{\frac2\theta}_{L^\infty_x}\|\langle v\rangle^{\tilde{\ell}} \overline{w}_{l_1}(\alpha-\alpha_1,\beta+e_i)\partial^{\alpha-\alpha_1}_{\beta +e_i}{\bf \{I-P\}}f_{\varepsilon}\|^2\nonumber\\[2mm]
			&&+\mathcal{D}_N(t)+\mathcal{E}_N(t)\mathcal{D}_N(t)\nonumber\\
			&\lesssim&\left\|\overline{w}_{l_1}(\alpha+e_i,\beta-e_i)\partial^{\alpha+e_i}_{\beta-e_i}\{{\bf I-P}\}f_{\varepsilon}\right\|_\nu^2+\eta\overline{\mathcal{D}}_{N-1,{l_1}}(t)\nonumber\\
			&&+\overline{\mathcal{E}}_{N-1,l_1}(t)\overline{\mathcal{D}}_{N-1,l_1}(t)+{\mathcal{M}_1}^{\frac2\theta}\sum_{N_0+1\leq n\leq N-1}\mathbb{D}^{(n)}_{\ell_1}(t)+\mathcal{D}_N(t)+\mathcal{E}_N(t)\mathcal{D}_N(t).
	\end{eqnarray}
Where we used the fact that, if $\ell_1\geq l_1+\tilde{\ell}+\frac12$, then
\[\langle v\rangle^{\tilde{\ell}} \overline{w}_{l_1}(\alpha-\alpha_1,\beta+e_i)=\langle v\rangle^{\tilde{\ell}+{l_1}-|\alpha|-2|\beta|}\leq\langle v\rangle^{\ell_1-|\alpha|-\frac12|\beta|}\langle v\rangle^{-\frac12}=\widetilde{w}_{\ell_1}(\alpha,\beta)\langle v\rangle^{-\frac12}.\]
	A proper linear combination (\ref{bound-sum}) with $|\alpha|+|\beta|\leq N, |\alpha|\leq N-1$ implies \eqref{bound-E-1}.
\end{proof}
For later use, we also need the following result:
\begin{proposition}\label{bound-E}
	Under {\bf Assumption I},
	assume $\tilde{\ell}\geq\frac32\sigma_{N_0-1,N_0-1}$, $\theta=\frac3{\tilde{\ell}+\frac12}$ and $l_0\geq l+\tilde{\ell}+\frac12$, then one has
	\begin{eqnarray}\label{bound-E-1}
		&&\frac{\d}{\d t}\overline{\mathcal{E}}_{N_0-1,l_0}(t)+\overline{\mathcal{D}}_{N-1,l_0}(t)\nonumber\\
		&\lesssim&\mathcal{D}_{N_0}(t)+\overline{\mathcal{E}}_{N_0-1,l_0}(t)\overline{\mathcal{D}}_{N_0-1,l_0}(t)+{\mathcal{M}_1}^{\frac2{\theta}}\sum_{n\leq N_0-1}\mathbb{D}^{(n)}_{\ell_0}(t).
	\end{eqnarray}
\end{proposition}
\begin{proof}
The proof is similar to Proposition \ref{bound-E}, we omit its proof for brevity.
\end{proof}

\subsection{Lyapunov inequality for the energy functionals $\mathbb{E}^{(n)}_\ell(t)$}

\begin{lemma}\label{lemma3.8-cut}
	Take $\overline{\ell}_0\geq \ell_1+\frac32N$, for $|\alpha|=N$, it holds that
	\begin{eqnarray}\label{lemma3.8-1-cut}
		&&\frac{\d}{\d t}\left\|\widetilde{w}_{\ell_1}(\alpha,0)\partial^\alpha f_{\varepsilon}\right\|^2
		+\frac1{\varepsilon^2}\left\|\widetilde{w}_{\ell_1}(\alpha,0)\partial^\alpha f_{\varepsilon}\right\|^2_{\nu}+\frac{\vartheta q}{(1+t)^{1+\vartheta}}\left\|\langle v\rangle \widetilde{w}_{\ell_1}(\alpha,0)\partial^\alpha f_{\varepsilon}\right\|^2\nonumber\\
		&\lesssim&\frac1{\varepsilon^2}\|\nabla^N_x{\bf P}f_{\varepsilon}\|^2+\mathcal{D}_N(t)+
		\left\|\partial^\alpha E_\varepsilon\right\|\left\|{M}^\delta\partial^\alpha f_{\varepsilon}\right\|\nonumber\\
		&&+\sum_{1\leq|\alpha_1|\leq 2}(1+t)^{1+\vartheta}\|\partial^{\alpha_1}[E_\varepsilon,B_\varepsilon]\|^2_{L^\infty_x}\sum_{N-1\leq n\leq N}(1+t)^{\sigma_{n,1}}\mathcal{D}^{n,1}_{\ell_1}(t)\nonumber\\
		&&+\sum_{3\leq|\alpha_1|\leq N_0+1}\|\partial^{\alpha_1}[E_\varepsilon,B_\varepsilon]\|^2\sum_{N_0+1\leq n\leq N-1}(1+t)^{\sigma_{n,1}}\mathcal{D}^{n,1}_{\ell_1}(t)\nonumber\\
		&&+\mathcal{E}_N(t)\mathcal{E}_{1\rightarrow N_0-1,\overline{\ell}_0}(t)+\mathcal{E}_{N-1,l_1}(t)\sum_{n\leq N}(1+t)^{\sigma_{n,0}}\mathcal{D}^{n,0}_{\ell_1}(t)
	\end{eqnarray}
	{for all $0\leq t\leq T$.}
\end{lemma}
\begin{proof}
	The proofs of this lemma is similar to Lemma \ref{noncut-N-wight}. For the sake of brevity, let's just point out the differences from Lemma \ref{noncut-N-wight}.
	
	Recalling the definition of the weight $\widetilde{w}_{\ell_1}(\alpha,\beta)(t,v)$, one has
	\[\langle v\rangle \widetilde{w}_{\ell_1}(\alpha,0)=\langle v\rangle \widetilde{w}_{\ell_1}(\alpha-e_i,e_i)\langle v\rangle^{-\frac12}\leq \langle v\rangle^{\frac12}\widetilde{w}_{\ell_1}(\alpha-e_i,e_i),\]	
	it follows that
	\begin{eqnarray}\label{N-weight-B-cut}
		&&{\frac1\varepsilon}\left|\left(\partial^\alpha[ v\times B_\varepsilon\cdot\nabla_v {\bf \{I-P\}}f_\varepsilon],\widetilde{w}^2_{\ell_1}(\alpha,0)\partial^\alpha {\bf \{I-P\}}f_{\varepsilon}\right)\right|\nonumber\\
		&=&{\frac1\varepsilon}\left|\left([ v\times B_\varepsilon\cdot\nabla_v \partial^\alpha{\bf \{I-P\}}f_\varepsilon],\widetilde{w}^2_{\ell_1}(\alpha,0)\partial^\alpha {\bf \{I-P\}}f_{\varepsilon}\right)\right|\nonumber\\
		&&+\sum_{1\leq|\alpha_1|\leq N}{\frac1\varepsilon}\left|\left([ v\times \partial^{\alpha_1}B_\varepsilon\cdot\nabla_v \partial^{\alpha-\alpha_1}{\bf \{I-P\}}f_\varepsilon],\widetilde{w}^2_{\ell_1}(\alpha,0)\partial^\alpha {\bf \{I-P\}}f_{\varepsilon}\right)\right|\nonumber\\
		&=&\sum_{|\alpha_1|=1}{\frac1\varepsilon}\left|\left([ v\times \partial^{\alpha_1}B_\varepsilon\cdot\nabla_v \partial^{\alpha-\alpha_1}{\bf \{I-P\}}f_\varepsilon],\widetilde{w}^2_{\ell_1}(\alpha,0)\partial^\alpha {\bf \{I-P\}}f_{\varepsilon}\right)\right|\nonumber\\
		&&+\sum_{|\alpha_1|=2}{\frac1\varepsilon}\left|\left([ v\times \partial^{\alpha_1}B_\varepsilon\cdot\nabla_v \partial^{\alpha-\alpha_1}{\bf \{I-P\}}f_\varepsilon],\widetilde{w}^2_{\ell_1}(\alpha,0)\partial^\alpha {\bf \{I-P\}}f_{\varepsilon}\right)\right|\nonumber\\
		&&+\sum_{3\leq|\alpha_1|\leq N_0-2}{\frac1\varepsilon}\left|\left([ v\times \partial^{\alpha_1}B_\varepsilon\cdot\nabla_v \partial^{\alpha-\alpha_1}{\bf \{I-P\}}f_\varepsilon],\widetilde{w}^2_{\ell_1}(\alpha,0)\partial^\alpha {\bf \{I-P\}}f_{\varepsilon}\right)\right|\nonumber\\
		&&+\sum_{|\alpha_1|= N_0-1}{\frac1\varepsilon}\left|\left([ v\times \partial^{\alpha_1}B_\varepsilon\cdot\nabla_v \partial^{\alpha-\alpha_1}{\bf \{I-P\}}f_\varepsilon],\widetilde{w}^2_{\ell_1}(\alpha,0)\partial^\alpha {\bf \{I-P\}}f_{\varepsilon}\right)\right|\nonumber\\
			&&+\sum_{|\alpha_1|= N_0}{\frac1\varepsilon}\left|\left([ v\times \partial^{\alpha_1}B_\varepsilon\cdot\nabla_v \partial^{\alpha-\alpha_1}{\bf \{I-P\}}f_\varepsilon],\widetilde{w}^2_{\ell_1}(\alpha,0)\partial^\alpha {\bf \{I-P\}}f_{\varepsilon}\right)\right|\nonumber\\
				&&+\sum_{|\alpha_1|= N_0+1}{\frac1\varepsilon}\left|\left([ v\times \partial^{\alpha_1}B_\varepsilon\cdot\nabla_v \partial^{\alpha-\alpha_1}{\bf \{I-P\}}f_\varepsilon],\widetilde{w}^2_{\ell_1}(\alpha,0)\partial^\alpha {\bf \{I-P\}}f_{\varepsilon}\right)\right|\nonumber\\
					&&+\sum_{N_0+2\leq|\alpha_1|\leq N}{\frac1\varepsilon}\left|\left([ v\times \partial^{\alpha_1}B_\varepsilon\cdot\nabla_v \partial^{\alpha-\alpha_1}{\bf \{I-P\}}f_\varepsilon],\widetilde{w}^2_{\ell_1}(\alpha,0)\partial^\alpha {\bf \{I-P\}}f_{\varepsilon}\right)\right|\nonumber\\
		&\lesssim&\sum_{|\alpha_1|=1}{\frac1\varepsilon}\|\partial^{\alpha_1}B_\varepsilon\|_{L^\infty_x}\|\langle v\rangle^{\frac14} \widetilde{w}_{\ell_1}(\alpha-\alpha_1,e_i)\partial^{\alpha-\alpha_1}_{e_i}{\bf \{I-P\}}f_{\varepsilon}\|\|\langle v\rangle^{\frac14} \widetilde{w}_{\ell_1}(\alpha,0 )\partial^{\alpha}{\bf \{I-P\}}f_{\varepsilon}\|\nonumber\\
		&&+\sum_{|\alpha_1|=2}{\frac1\varepsilon}\|\partial^{\alpha_1}B_\varepsilon\|_{L^\infty_x}\|\langle v\rangle^{-\frac14} \widetilde{w}_{\ell_1}(\alpha-\alpha_1,e_i)\partial^{\alpha-\alpha_1}_{e_i}{\bf \{I-P\}}f_{\varepsilon}\|\|\langle v\rangle^{-\frac14} \widetilde{w}_{\ell_1}(\alpha,0 )\partial^{\alpha}{\bf \{I-P\}}f_{\varepsilon}\|\nonumber\\
		&&+\sum_{3\leq|\alpha_1|\leq N_0-2}{\frac1\varepsilon}\|\partial^{\alpha_1}B_\varepsilon\|_{L^\infty_x}\|\langle v\rangle^{2-|\alpha_1|} \widetilde{w}_{\ell_1}(\alpha-\alpha_1,e_i)\partial^{\alpha-\alpha_1}_{e_i}{\bf \{I-P\}}f_{\varepsilon}\|\|\langle v\rangle^{-\frac12} \widetilde{w}_{\ell_1}(\alpha,0 )\partial^{\alpha}{\bf \{I-P\}}f_{\varepsilon}\|\nonumber\\
			&&+\sum_{|\alpha_1|= N_0-1}{\frac1\varepsilon}\|\partial^{\alpha_1}B_\varepsilon\|_{L^6_x}\|\langle v\rangle^{2-|\alpha_1|} \widetilde{w}_{\ell_1}(\alpha-\alpha_1,e_i)\partial^{\alpha-\alpha_1}_{e_i}{\bf \{I-P\}}f_{\varepsilon}\|_{L^3_x}\|\langle v\rangle^{-\frac12} \widetilde{w}_{\ell_1}(\alpha,0 )\partial^{\alpha}{\bf \{I-P\}}f_{\varepsilon}\|\nonumber\\
				&&+\sum_{|\alpha_1|= N_0}{\frac1\varepsilon}\|\partial^{\alpha_1}B_\varepsilon\|\|\langle v\rangle^{2-|\alpha_1|} \widetilde{w}_{\ell_1}(\alpha-\alpha_1,e_i)\partial^{\alpha-\alpha_1}_{e_i}{\bf \{I-P\}}f_{\varepsilon}\|_{L^\infty_x}\|\langle v\rangle^{-\frac12} \widetilde{w}_{\ell_1}(\alpha,0 )\partial^{\alpha}{\bf \{I-P\}}f_{\varepsilon}\|\nonumber\\
					&&+\sum_{|\alpha_1|= N_0+1}{\frac1\varepsilon}\|\partial^{\alpha_1}B_\varepsilon\|\|\langle v\rangle^{2-|\alpha_1|} \widetilde{w}_{\ell_1}(\alpha-\alpha_1,e_i)\partial^{\alpha-\alpha_1}_{e_i}{\bf \{I-P\}}f_{\varepsilon}\|_{L^\infty_x}\|\langle v\rangle^{-\frac12} \widetilde{w}_{\ell_1}(\alpha,0 )\partial^{\alpha}{\bf \{I-P\}}f_{\varepsilon}\|\nonumber\\
		&&+\sum_{N_0+2\leq|\alpha_1|\leq N}{\frac1\varepsilon}\left|\left([ v\times \partial^{\alpha_1}B_\varepsilon\cdot\nabla_v \partial^{\alpha-\alpha_1}{\bf \{I-P\}}f_\varepsilon],\widetilde{w}^2_{\ell_1}(\alpha,0)\partial^\alpha {\bf \{I-P\}}f_{\varepsilon}\right)\right|\nonumber\\
		&\lesssim&\sum_{1\leq|\alpha_1|\leq 2}{\frac1\varepsilon}(1+t)^{\frac{1+\vartheta}2}\|\partial^{\alpha_1}B_\varepsilon\|^2_{L^\infty_x}\|\langle v\rangle^{\frac14} \widetilde{w}_{\ell_1}(\alpha-\alpha_1,e_i)\partial^{\alpha-\alpha_1}_{e_i}{\bf \{I-P\}}f_{\varepsilon}\|^2\nonumber\\
			&&+\sum_{3\leq|\alpha_1|\leq N_0-2}\|\partial^{\alpha_1}B_\varepsilon\|^2_{L^\infty_x}\|\langle v\rangle^{2-|\alpha_1|} \widetilde{w}_{\ell_1}(\alpha-\alpha_1,e_i)\partial^{\alpha-\alpha_1}_{e_i}{\bf \{I-P\}}f_{\varepsilon}\|^2\nonumber\\
		&&+\sum_{|\alpha_1|= N_0-1}\|\partial^{\alpha_1}B_\varepsilon\|^3_{L^6_x}\|\langle v\rangle^{2-|\alpha_1|} \widetilde{w}_{\ell_1}(\alpha-\alpha_1,e_i)\partial^{\alpha-\alpha_1}_{e_i}{\bf \{I-P\}}f_{\varepsilon}\|_{L^3_x}^2\nonumber\\
		&&+\sum_{|\alpha_1|= N_0}\|\partial^{\alpha_1}B_\varepsilon\|^2\|\langle v\rangle^{2-|\alpha_1|} \widetilde{w}_{\ell_1}(\alpha-\alpha_1,e_i)\partial^{\alpha-\alpha_1}_{e_i}{\bf \{I-P\}}f_{\varepsilon}\|_{L^\infty_x}^2\nonumber\\
		&&+\sum_{|\alpha_1|= N_0+1}\|\partial^{\alpha_1}B_\varepsilon\|^2\|\langle v\rangle^{2-|\alpha_1|} \widetilde{w}_{\ell_1}(\alpha-\alpha_1,e_i)\partial^{\alpha-\alpha_1}_{e_i}{\bf \{I-P\}}f_{\varepsilon}\|_{L^\infty_x}^2\nonumber\\
		&&+\sum_{N_0+2\leq|\alpha_1|\leq N}{\frac1\varepsilon}\left|\left([ v\times \partial^{\alpha_1}B_\varepsilon\cdot\nabla_v \partial^{\alpha-\alpha_1}{\bf \{I-P\}}f_\varepsilon],\widetilde{w}^2_{\ell_1}(\alpha,0)\partial^\alpha {\bf \{I-P\}}f_{\varepsilon}\right)\right|\nonumber\\
		&&+{\frac\eta\varepsilon}(1+t)^{-\frac{1+\vartheta}2}\|\langle v\rangle^{\frac14} w_\ell(\alpha,0 )\partial^{\alpha}{\bf \{I-P\}}f_{\varepsilon}\|^2+\frac\eta{\varepsilon^2}\left\|\widetilde{w}_{\ell_1}(\alpha,0)\partial^\alpha {\bf \{I-P\}}f_{\varepsilon}\right\|_\nu^2\nonumber\\
			&\lesssim&\sum_{1\leq|\alpha_1|\leq 2}{\frac1\varepsilon}(1+t)^{\frac{1+\vartheta}2}\|\partial^{\alpha_1}B_\varepsilon\|^2_{L^\infty_x}\|\langle v\rangle^{\frac14} \widetilde{w}_{\ell_1}(\alpha-\alpha_1,e_i)\partial^{\alpha-\alpha_1}_{e_i}{\bf \{I-P\}}f_{\varepsilon}\|^2\nonumber\\
		&&+\sum_{3\leq|\alpha_1|\leq N_0+1}\|\partial^{\alpha_1}B_\varepsilon\|^2\sum_{N_0+1\leq n\leq N-1}(1+t)^{\sigma_{n,1}}\mathcal{D}^{n,1}_{\ell_1}(t)\nonumber\\
		&&+\sum_{N_0+2\leq|\alpha_1|\leq N}{\frac1\varepsilon}\left|\left([ v\times \partial^{\alpha_1}B_\varepsilon\cdot\nabla_v \partial^{\alpha-\alpha_1}{\bf \{I-P\}}f_\varepsilon],\widetilde{w}^2_{\ell_1}(\alpha,0)\partial^\alpha {\bf \{I-P\}}f_{\varepsilon}\right)\right|\nonumber\\
		&&+{\frac\eta\varepsilon}(1+t)^{-\frac{1+\vartheta}2}\|\langle v\rangle^{\frac14} \widetilde{w}_{\ell_1}(\alpha,0 )\partial^{\alpha}{\bf \{I-P\}}f_{\varepsilon}\|^2+\eta\left\|\widetilde{w}_{\ell_1}(\alpha,0)\partial^\alpha {\bf \{I-P\}}f_{\varepsilon}\right\|_\nu^2\nonumber\\
			&\lesssim&\sum_{1\leq|\alpha_1|\leq 2}(1+t)^{1+\vartheta}\|\partial^{\alpha_1}B_\varepsilon\|^2_{L^\infty_x}\sum_{N-1\leq n\leq N}(1+t)^{\sigma_{n,1}}\mathcal{D}^{n,1}_{\ell_1}(t)\nonumber\\\nonumber\\
		&&+\sum_{3\leq|\alpha_1|\leq N_0+1}\|\partial^{\alpha_1}B_\varepsilon\|^2\sum_{N_0+1\leq n\leq N-1}(1+t)^{\sigma_{n,1}}\mathcal{D}^{n,1}_{\ell_1}(t)+\mathcal{E}_N(t)\mathcal{E}_{1\rightarrow N_0-1,\overline{\ell}_0}(t)\nonumber\\
		&&+{\frac\eta\varepsilon}(1+t)^{-\frac{1+\vartheta}2}\|\langle v\rangle^{\frac14} \widetilde{w}_{\ell_1}(\alpha,0 )\partial^{\alpha}{\bf \{I-P\}}f_{\varepsilon}\|^2+\frac\eta{\varepsilon^2}\left\|\widetilde{w}_{\ell_1}(\alpha,0)\partial^\alpha {\bf \{I-P\}}f_{\varepsilon}\right\|_\nu^2
	\end{eqnarray}
	where the third term on the right-hand side of the above inequalities can dominated by
	\[	\frac1{\varepsilon^2}\left\|\widetilde{w}_{\ell_1}(\alpha,0)\partial^\alpha f_{\varepsilon}\right\|^2_{\nu}+\frac{\vartheta q}{(1+t)^{1+\vartheta}}\left\|\langle v\rangle \widetilde{w}_{\ell_1}(\alpha,0)\partial^\alpha f_{\varepsilon}\right\|^2.\]
By combining the other similar estimates as Lemma \ref{noncut-N-wight}, we complete the proof of Lemma \ref{lemma3.8-cut}.
\end{proof}
\begin{lemma}\label{N-micro-w-cut}
Assume $\overline{\ell}_0\geq \ell_1+\frac32N$, for ${|\alpha|+|\beta|= N,|\beta|\geq 1}$,	one also has
	\begin{eqnarray}\label{N-micro-w-cut-1}
		&&\frac{\d}{\d t}\left\|\widetilde{w}_{\ell_1}(\alpha,\beta)\partial^\alpha_\beta \{{\bf I-P}\}f_{\varepsilon}\right\|^2+\frac1{\varepsilon^2}\left\|\widetilde{w}_{\ell_1}(\alpha,\beta)\partial^\alpha_\beta \{{\bf I-P}\}f_{\varepsilon}\right\|^2_\nu\nonumber\\
		&&+\frac{q\vartheta}{(1+t)^{1+\vartheta}}\|\langle v\rangle\widetilde{w}_{\ell_1}(\alpha,\beta)\partial^\alpha_\beta \{{\bf I-P}\}f_{\varepsilon}\|^2\nonumber\\
		&\lesssim&\frac1\varepsilon(1+t)^{\frac{1+\epsilon_0}2}\left\|\langle v\rangle^{\frac14}\widetilde{w}_{\ell_1}(\alpha+e_i,\beta-e_i)\partial^{\alpha+e_i}_{\beta-e_i}\{{\bf I-P}\}f_{\varepsilon}\right\|^2\nonumber\\
		&&+\sum_{1\leq|\alpha_1|\leq 2}(1+t)^{1+\vartheta}\|\partial^{\alpha_1}B_\varepsilon\|^2_{L^\infty_x}(1+t)^{\sigma_{N+1-|\alpha_1|,|\beta|+1}}\mathcal{D}^{N+1-|\alpha_1|,|\beta|+1}_{\ell_1}(t)\nonumber\\\nonumber\\
	&&+\sum_{3\leq|\alpha_1|\leq N_0+1}\|\partial^{\alpha_1}B_\varepsilon\|^2\sum_{N_0+1\leq n\leq N-1}(1+t)^{\sigma_{n,|\beta|+1}}\mathcal{D}^{n,|\beta|+1}_{\ell_1}(t)\nonumber\\
	&&+\mathcal{D}_{N}(t)+\mathcal{E}_N(t)\mathcal{E}_{1\rightarrow N_0-1,\overline{\ell}_0}(t)+\mathcal{E}_{N-1,l_1}(t)\sum_{n\leq N-1}(1+t)^{\sigma_{n,|\beta|}}\mathcal{D}^{n,|\beta|}_{\ell_1}(t).
	\end{eqnarray}
\end{lemma}
\begin{proof}
	To this end, applying $\partial^\alpha_\beta$ into \eqref{I-P-cut}, multiplying $\widetilde{w}^2_{\ell_1}(\alpha,\beta)\partial^\alpha_\beta\{{\bf I-P}\}f_{\varepsilon}$ and integrating the result identity over $\mathbb{R}^3_x\times\mathbb{R}^3_v$, then we have
	\begin{eqnarray}\label{I-P-w-cut}
		&&\frac12\frac{\d}{\d t}\|\widetilde{w}_{\ell_1}(\alpha,\beta)\partial^\alpha_\beta\{{\bf I-P}\}f_{\varepsilon}\|^2+\frac{q\vartheta}{(1+t)^{1+\vartheta}}\left\|\langle v\rangle \widetilde{w}_{\ell_1}(\alpha,\beta)\partial^\alpha_\beta\{{\bf I-P}\}f_{\varepsilon}\right\|^2\nonumber\\
		&&+\frac1 {\varepsilon^2} \left(\partial^\alpha_\beta\mathscr{L}f_\varepsilon,\widetilde{w}^2_{\ell_1}(\alpha,\beta)\partial^\alpha_\beta\{{\bf I-P}\}f_{\varepsilon}\right)\nonumber\\
		&=&-\frac1\varepsilon\left(\partial^\alpha_\beta[ v\cdot \nabla_x\{{\bf I-P}\}f_{\varepsilon}],\widetilde{w}^2_{\ell_1}(\alpha,\beta)\partial^\alpha_\beta\{{\bf I-P}\}f_{\varepsilon}\right)\nonumber\\
		&&+\frac1 \varepsilon \left(\partial^\alpha_\beta \left[E_\varepsilon\cdot v{M}^{1/2}q_1\right],\widetilde{w}^2_{\ell_1}(\alpha,\beta)\partial^\alpha_\beta\{{\bf I-P}\}f_{\varepsilon}\right)\nonumber\\
		&&+\frac1 \varepsilon\left(\partial^\alpha_\beta\left[{\bf P}(v\cdot\nabla_x f_\varepsilon)-\frac1 \varepsilon v\cdot\nabla_x{\bf P}f_\varepsilon\right],\widetilde{w}^2_{\ell_1}(\alpha,\beta)\partial^\alpha_\beta\{{\bf I-P}\}f_{\varepsilon}\right)\nonumber\\
		&&+\frac1\varepsilon\left(\partial^\alpha_\beta\left[\{{\bf I-P}\}\left[ v\times B_\varepsilon\cdot\nabla_{ v}f_\varepsilon\right]\right],\widetilde{w}^2_{\ell_1}(\alpha,\beta)\partial^\alpha_\beta\{{\bf I-P}\}f_{\varepsilon}\right)\nonumber\\
		&&+\frac1\varepsilon\left(\partial^\alpha_\beta\left[\{{\bf I-P}\}\left[q_0 v\times B_\varepsilon\cdot\nabla_{ v}f_\varepsilon\right]\right],\widetilde{w}^2_{\ell_1}\partial^\alpha_\beta\{{\bf I-P}\}f_{\varepsilon}\right)\nonumber\\
	&&+\left(\partial^\alpha_\beta\{{\bf I-P}\}\left[ -\frac12q_0 v\cdot E_\varepsilon f_\varepsilon+q_0E_\varepsilon\cdot\nabla_{ v}f_\varepsilon\right],\widetilde{w}^2_{\ell_1}\partial^\alpha_\beta\{{\bf I-P}\}f_{\varepsilon}\right)\nonumber\\
	&&+\frac1\varepsilon\left(\partial^\alpha_\beta\mathscr{T}(f_\varepsilon,f_\varepsilon),\widetilde{w}^2_{\ell_1}\partial^\alpha_\beta\{{\bf I-P}\}f_{\varepsilon}\right).
	\end{eqnarray}
	The coercivity estimates on the linear operator $\mathscr{L}$ yields that
	\begin{eqnarray}
		&&\frac1 {\varepsilon^2} \left(\partial^\alpha_\beta\mathscr{L}f_\varepsilon,\widetilde{w}^2_{\ell_1}(\alpha,\beta)\partial^\alpha_\beta\{{\bf I-P}\}f_{\varepsilon}\right)\nonumber\\
		&\gtrsim&\frac1 {\varepsilon^2} \|\widetilde{w}_{\ell_1}(\alpha,\beta)\partial^\alpha_\beta\{{\bf I-P}\}f_{\varepsilon}\|_{\nu}^2-\frac1 {\varepsilon^2}\|\partial^\alpha\{{\bf I-P}\}f_{\varepsilon}\|_{\nu}^2.
	\end{eqnarray}
	As for the transport term on the right-hand side of \eqref{I-P-w-cut}, one has
	\begin{eqnarray}
		&&\frac1\varepsilon\left(\partial^\alpha_\beta\left[ v\cdot \nabla_x\{{\bf I-P}\}f_{\varepsilon}\right],\widetilde{w}^2_{\ell_1}(\alpha,\beta)\partial^\alpha_\beta\{{\bf I-P}\}f_{\varepsilon}\right)\nonumber\\
		&=&-\frac1\varepsilon\int_{\mathbb{R}^3_x\times\mathbb{R}^3_v}\langle v\rangle^{\frac12}\partial^{\alpha+e_i}_{\beta-e_i}\{{\bf I-P}\}f_{\varepsilon} \widetilde{w}_{\ell_1}(\alpha+e_i,\beta-e_i)w_l(\alpha,\beta)\partial^\alpha_\beta\{{\bf I-P}\}f_{\varepsilon}dvdx\nonumber\\
		&\lesssim&\frac1\varepsilon(1+t)^{\frac{1+\epsilon_0}2}\left\|\langle v\rangle^{\frac14}\widetilde{w}_{\ell_1}(\alpha+e_i,\beta-e_i)\partial^{\alpha+e_i}_{\beta-e_i}\{{\bf I-P}\}f_{\varepsilon}\right\|^2\nonumber\\
		&&+\frac\eta\varepsilon(1+t)^{-\frac{1+\epsilon_0}2}\left\|\langle v\rangle^{\frac14}\widetilde{w}_{\ell_1}(\alpha,\beta)\partial^{\alpha}_{\beta}\{{\bf I-P}\}f_{\varepsilon}\right\|^2,	
	\end{eqnarray}
where we used the fact
	\[\widetilde{w}_{\ell_1}(\alpha,\beta)=\langle v\rangle^{\frac12}\widetilde{w}_{\ell_1}(\alpha+e_i,\beta-e_i).\]
For the nonlinear term induced by magnetic field $B(t,x)$, simiar to \eqref{N-weight-B-cut}, one has
	\begin{eqnarray}\label{B-nonlinear}
		&&\frac1\varepsilon\left(\partial^\alpha_\beta\left[ v\times B_\varepsilon\cdot\nabla_{ v}\{{\bf I-P}\}f_{\varepsilon}\right],\widetilde{w}^2_{\ell_1}(\alpha,\beta)\partial^\alpha_\beta\{{\bf I-P}\}f_{\varepsilon}\right)\nonumber\\
		&=&\frac1\varepsilon\sum_{\alpha_1\leq\alpha,\beta_1\leq\beta}\left(\partial^{\alpha_1}_{\beta_1}\left[ v\times B_\varepsilon\right]\cdot\partial^{\alpha-\alpha_1}_{\beta-\beta_1}\left[\nabla_{ v}\{{\bf I-P}\}f_{\varepsilon}\right],\widetilde{w}^2_{\ell_1}(\alpha,\beta)\partial^\alpha_\beta\{{\bf I-P}\}f_{\varepsilon}\right)\nonumber\\
			&\lesssim&\sum_{1\leq|\alpha_1|\leq 2}(1+t)^{1+\vartheta}\|\partial^{\alpha_1}B_\varepsilon\|^2_{L^\infty_x}(1+t)^{\sigma_{N+1-|\alpha_1|,|\beta|+1}}\mathcal{D}^{N+1-|\alpha_1|,|\beta|+1}_{\ell_1}(t)\nonumber\\\nonumber\\
		&&+\sum_{3\leq|\alpha_1|\leq N_0+1}\|\partial^{\alpha_1}B_\varepsilon\|^2\sum_{N_0+1\leq n\leq N-1}(1+t)^{\sigma_{n,|\beta|+1}}\mathcal{D}^{n,|\beta|+1}_{\ell_1}(t)+\mathcal{E}_N(t)\mathcal{E}_{1\rightarrow N_0-1,\overline{\ell}_0}(t)\nonumber\\
		&&+{\frac\eta\varepsilon}(1+t)^{-\frac{1+\vartheta}2}\|\langle v\rangle^{\frac14} \widetilde{w}_{\ell_1}(\alpha,\beta )\partial^{\alpha}_\beta{\bf \{I-P\}}f_{\varepsilon}\|^2+\frac\eta{\varepsilon^2}\left\|\widetilde{w}_{\ell_1}(\alpha,\beta)\partial^\alpha_\beta {\bf \{I-P\}}f_{\varepsilon}\right\|_\nu^2.
	\end{eqnarray}

Similarly, one also has
\begin{eqnarray}
	&&\left(\partial^\alpha_\beta\left[q_0 E\cdot\nabla_{ v}\{{\bf I-P}\}f_{\varepsilon}-\frac12q_0E_\varepsilon\cdot v\{{\bf I-P}\}f_{\varepsilon}\right],\widetilde{w}_{\ell_1}(\alpha,\beta )\partial^\alpha_\beta\{{\bf I-P}\}f_{\varepsilon}\right)\nonumber\\
	&\lesssim&\sum_{1\leq|\alpha_1|\leq 2}(1+t)^{1+\vartheta}\|\partial^{\alpha_1}E_\varepsilon\|^2_{L^\infty_x}(1+t)^{\sigma_{N+1-|\alpha_1|,|\beta|+1}}\mathcal{D}^{N+1-|\alpha_1|,|\beta|+1}_{\ell_1}(t)\nonumber\\\nonumber\\
&&+\sum_{3\leq|\alpha_1|\leq N_0+1}\|\partial^{\alpha_1}E_\varepsilon\|^2\sum_{N_0+1\leq n\leq N-1}(1+t)^{\sigma_{n,|\beta|+1}}\mathcal{D}^{n,|\beta|+1}_{\ell_1}(t)+\mathcal{E}_N(t)\mathcal{E}_{1\rightarrow N_0-1,\overline{\ell}_0}(t)\nonumber\\
&&+{\frac\eta\varepsilon}(1+t)^{-\frac{1+\vartheta}2}\|\langle v\rangle^{\frac14} \widetilde{w}_{\ell_1}(\alpha,\beta )\partial^{\alpha}_\beta{\bf \{I-P\}}f_{\varepsilon}\|^2+\frac\eta{\varepsilon^2}\left\|\widetilde{w}_{\ell_1}(\alpha,\beta)\partial^\alpha_\beta {\bf \{I-P\}}f_{\varepsilon}\right\|_\nu^2\nonumber\\
	&&+\|E_\varepsilon\|_{L^\infty_x}\|\langle v\rangle ^{\frac12}\widetilde{w}_{\ell_1}(\alpha,\beta)\partial^{\alpha}_{\beta}{\bf \{I-P\}}f_{\varepsilon}\|^2.
\end{eqnarray}
The last term can be dominated by
\begin{eqnarray}
	&&\frac1\varepsilon\left(\partial^\alpha_\beta\mathscr{T}(f_\varepsilon,f_\varepsilon),\widetilde{w}_{\ell_1}(\alpha,\beta )\partial^\alpha_\beta f_\varepsilon\right)\nonumber\\
	&\lesssim&\mathcal{E}_{N-1,l_1}(t)\sum_{n\leq N-1}(1+t)^{\sigma_{n,|\beta|}}\mathcal{D}^{n,|\beta|}_{\ell_1}(t)+\frac\eta{\varepsilon^2}\left\|\widetilde{w}_{\ell_1}(\alpha,\beta)\partial^\alpha_\beta {\bf \{I-P\}}f_{\varepsilon}\right\|_\nu^2.
\end{eqnarray}
\eqref{N-micro-w-cut-1} follows by collecting the above estimates into \eqref{I-P-w-cut}, thus we complete the proof of this lemma.
\end{proof}
A proper linear combination of  Lemma \ref{lemma3.8-cut}.
 and Lemma \ref{N-micro-w-cut},  one has
\begin{proposition}\label{Highest}
	Under {\bf Assumption 1} and take
	\begin{equation}\label{sigma-N-Assump}
	\sigma_{N,0}=\frac{1+\epsilon_0}2,\ 	\sigma_{N,|\beta|}-\sigma_{N,|\beta|-1}=\frac{1+\vartheta}2, |\beta|\geq1,	
\end{equation}
where $0<\epsilon_0\ll 1$,
	 one can deduce that
	\begin{eqnarray}\label{large-N}
		&&{\varepsilon^2}\frac{\d}{\d t}\mathbb{E}^{(N)}_{\ell_1}(t)+{\varepsilon^2}\mathbb{D}^{(N)}_{\ell_1}(t)\nonumber\\
		&\lesssim&\eta{\varepsilon^2}(1+t)^{-2\sigma_{N,0}}\left\|\nabla^N_x E_\varepsilon\right\|^2+\mathcal{D}_{N}(t)\nonumber\\
	&&+\mathcal{M}_1\left\{(1+t)^{-1^+}\mathcal{E}_N(t)+	\sum_{N_0+1\leq n\leq N-1}\mathbb{D}^{(n)}_{\ell_1}(t)+\varepsilon^2\mathbb{D}^{(N)}_{\ell_1}(t)\right\}.
	\end{eqnarray}
\end{proposition}
\begin{proof}
	To control the singularty term
	\[\frac1{\varepsilon^2}\|\nabla^N_x{\bf P}f_{\varepsilon}\|^2\]
	on the righ-hand side of \eqref{lemma3.8-1-cut}, we firstly multiply $\varepsilon^2$ into \eqref{lemma3.8-1-cut}, then one has
		\begin{eqnarray}\label{N-cut-0}
		&&\varepsilon^2\frac{\d}{\d t}\left\|w_\ell(\alpha,0)\partial^\alpha f_{\varepsilon}\right\|^2
		+\left\|w_\ell(\alpha,0)\partial^\alpha f_{\varepsilon}\right\|^2_{\nu}+\frac{\vartheta q\varepsilon^2}{(1+t)^{1+\vartheta}}\left\|\langle v\rangle w_\ell(\alpha,0)\partial^\alpha f_{\varepsilon}\right\|^2\nonumber\\
		&\lesssim&\|\nabla^N_x{\bf P}f_{\varepsilon}\|^2+\mathcal{D}_N(t)+
	\left\|\partial^\alpha E_\varepsilon\right\|\left\|{M}^\delta\partial^\alpha f_{\varepsilon}\right\|\nonumber\\
	&&+{\varepsilon^2}\sum_{1\leq|\alpha_1|\leq 2}(1+t)^{1+\vartheta}\|\partial^{\alpha_1}B_\varepsilon\|^2_{L^\infty_x}(1+t)^{\sigma_{N+1-|\alpha_1|,1}}\mathcal{D}^{N+1-|\alpha_1|,1}_{\ell_1}(t)\nonumber\\
	&&+{\varepsilon^2}\sum_{3\leq|\alpha_1|\leq N_0+1}\|\partial^{\alpha_1}B_\varepsilon\|^2\sum_{N_0+1\leq n\leq N-1}(1+t)^{\sigma_{n,1}}\mathcal{D}^{n,1}_{\ell_1}(t)+\mathcal{E}_N(t)\mathcal{E}^{1\rightarrow N_0-1}_{\overline{\ell}_0}(t)\nonumber\\
	&\lesssim&\mathcal{D}_N(t)+
	\left\|\partial^\alpha E_\varepsilon\right\|\left\|{M}^\delta\partial^\alpha f_{\varepsilon}\right\|\nonumber\\
	&&+{\varepsilon^2}\sum_{1\leq|\alpha_1|\leq 2}(1+t)^{1+\vartheta}\|\partial^{\alpha_1}B_\varepsilon\|^2_{L^\infty_x}(1+t)^{\sigma_{N+1-|\alpha_1|,1}}\mathcal{D}^{N+1-|\alpha_1|,1}_{\ell_1}(t)\nonumber\\
	&&+{\varepsilon^2}\sum_{3\leq|\alpha_1|\leq N_0+1}\|\partial^{\alpha_1}B_\varepsilon\|^2\sum_{N_0+1\leq n\leq N-1}(1+t)^{\sigma_{n,1}}\mathcal{D}^{n,1}_{\ell_1}(t)+\mathcal{E}_N(t)\mathcal{E}_{1\rightarrow N_0-1,\overline{\ell}_0}(t).
	\end{eqnarray}

	When $|\alpha|=N-1,\ |\beta|=1$, we hope that the term
	\[\frac1\varepsilon(1+t)^{\frac{1+\epsilon_0}2}\left\|\langle v\rangle^{\frac14}\widetilde{w}_{\ell_1}(\alpha+e_i,\beta-e_i)\partial^{\alpha+e_i}_{\beta-e_i}\{{\bf I-P}\}f_{\varepsilon}\right\|^2\]
	on the right-hand side of \eqref{N-micro-w-cut-1} can be dominated by the corresponding  dissipation terms on the left-hand side of \eqref{N-cut-0}, to do so, we multiply $\varepsilon^2$ into \eqref{N-micro-w-cut-1} to get
	\begin{eqnarray}\label{N-cut-1}
		&&{\varepsilon^2}\frac{\d}{\d t}\left\|\widetilde{w}_{\ell_1}(\alpha,\beta)\partial^\alpha_\beta \{{\bf I-P}\}f_{\varepsilon}\right\|^2+\left\|\widetilde{w}_{\ell_1}(\alpha,\beta)\partial^\alpha_\beta \{{\bf I-P}\}f_{\varepsilon}\right\|^2_\nu\nonumber\\
		&&+\frac{q\vartheta{\varepsilon^2}}{(1+t)^{1+\vartheta}}\|\widetilde{w}_{\ell_1}(\alpha,\beta)\partial^\alpha_\beta \{{\bf I-P}\}f_{\varepsilon}\langle v\rangle_\varepsilon\|^2\nonumber\\
		&\lesssim&\varepsilon(1+t)^{\frac{1+\epsilon_0}2}\left\|\langle v\rangle^{\frac14}\widetilde{w}_{\ell_1}(\alpha+e_i,\beta-e_i)\partial^{\alpha+e_i}_{\beta-e_i}\{{\bf I-P}\}f_{\varepsilon}\right\|^2\nonumber\\
	&&+\varepsilon^2\sum_{1\leq|\alpha_1|\leq 2}(1+t)^{1+\vartheta}\|\partial^{\alpha_1}B_\varepsilon\|^2_{L^\infty_x}(1+t)^{\sigma_{N+1-|\alpha_1|,|\beta|+1}}\mathcal{D}^{N+1-|\alpha_1|,|\beta|+1}_{\ell_1}(t)\nonumber\\
	&&+\varepsilon^2\sum_{3\leq|\alpha_1|\leq N_0+1}\|\partial^{\alpha_1}B_\varepsilon\|^2\sum_{N_0+1\leq n\leq N-1}(1+t)^{\sigma_{n,|\beta|+1}}\mathcal{D}^{n,|\beta|+1}_{\ell_1}(t)\nonumber\\
	&&+\mathcal{D}_{N}(t)+\mathcal{E}_N(t)\mathcal{E}_{1\rightarrow N_0-1,\overline{\ell}_0}(t).
	\end{eqnarray}	
	Furthermore, to control the time increasing rates $(1+t)^{\frac{1+\vartheta}2}$ for the first term on the right-hand side of \eqref{N-cut-1}, we multiply $(1+t)^{-\sigma_{|\alpha|+|\beta|,|\beta|}}$ into  \eqref{N-cut-1} and \eqref{N-cut-0}, and take a proper linear combination of the result equalities, then we have
	\begin{eqnarray}
		&&{\varepsilon^2}\frac{\d}{\d t}\left\{\sum_{|\alpha|=N}(1+t)^{-\sigma_{N,0}}\left\|\widetilde{w}_{\ell_1}(\alpha,0)\partial^\alpha f_{\varepsilon}\right\|^2
		+\sum_{|\alpha|+|\beta|= N,\atop|\beta|\geq 1}(1+t)^{-\sigma_{N,|\beta|}}\left\|\widetilde{w}_{\ell_1}(\alpha,\beta)\partial^\alpha_\beta \{{\bf I-P}\}f_{\varepsilon}\right\|^2\right\}\nonumber\\
		&&+\sum_{|\alpha|=N}(1+t)^{-\sigma_{N,0}}\left\{\left\|\widetilde{w}_{\ell_1}(\alpha,0)\partial^\alpha f_{\varepsilon}\right\|^2_{\nu}+\frac{\vartheta q{\varepsilon^2}}{(1+t)^{1+\vartheta}}\sum_{|\alpha|=N}\left\|\langle v\rangle \widetilde{w}_{\ell_1}(\alpha,0)\partial^\alpha f_{\varepsilon}\right\|^2\right\}\nonumber\\
		&&+\sum_{|\alpha|+|\beta|= N,\atop|\beta|\geq1}(1+t)^{-\sigma_{N,|\beta|}}\left\{\left\|\widetilde{w}_{\ell_1}(\alpha,\beta)\partial^\alpha_\beta \{{\bf I-P}\}f_{\varepsilon}\right\|^2_\nu+\frac{q\vartheta{\varepsilon^2}}{(1+t)^{1+\vartheta}}\|\widetilde{w}_{\ell_1}(\alpha,\beta)\partial^\alpha_\beta \{{\bf I-P}\}f_{\varepsilon}\langle v\rangle_\varepsilon\|^2\right\}\nonumber\\
		&\lesssim&{\varepsilon^2}(1+t)^{-\sigma_{N,0}}\left\|\nabla^N_x E_\varepsilon\right\|\left\|{M}^\delta\nabla^N_x f_{\varepsilon}\right\|+\mathcal{D}_N(t)\nonumber\\	 &&+\mathcal{M}_1\left\{(1+t)^{-(1+\epsilon_0)}\mathcal{E}_N(t)+	 \sum_{N_0+1\leq n\leq N-1}\mathbb{D}^{(n)}_{\ell_1}(t)+\varepsilon^2\mathbb{D}^{(N)}_{\ell_1}(t)\right\}\nonumber\\
		&\lesssim&\eta{\varepsilon^2}(1+t)^{-2\sigma_{N,0}}\left\|\nabla^N_x E_\varepsilon\right\|^2+\mathcal{D}_N(t)\nonumber\\	
		&&+\mathcal{M}_1\left\{(1+t)^{-(1+\epsilon_0)}\mathcal{E}_N(t)+	\sum_{N_0+1\leq n\leq N-1}\mathbb{D}^{(n)}_{\ell_1}(t)+\varepsilon^2\mathbb{D}^{(N)}_{\ell_1}(t)\right\}.
	\end{eqnarray}
	where we use \eqref{sigma-N-Assump}.
\end{proof}

\begin{proposition}\label{sub-highest}
For $N_0+1\leq n\leq N-1$, it holds that
	\begin{eqnarray}\label{large-N-1}
		&&\frac{\d}{\d t}\sum_{N_0+1\leq n\leq N-1}\mathbb{E}_{\ell_1}^{(n)}(t)+\sum_{N_0+1\leq n\leq N-1}\mathbb{D}_{\ell_1}^{(n)}(t)\nonumber\\
		&\lesssim&\mathcal{D}_{N}(t)+\mathcal{M}_1\left\{(1+t)^{-(1+\epsilon_0)}\mathcal{E}_N(t)+	\sum_{N_0+1\leq n\leq N_0}\mathbb{D}^{(n)}_{\ell_1}(t)\right\}.
	\end{eqnarray}
	
\end{proposition}
\begin{proof}
	This propositon can be proved by a similar way as Proposition \ref {Highest}, we omit its proof for brevity.
\end{proof}
When $n\leq N_0$, one also has
\begin{proposition}\label{mix-N-0-w}
For any $\ell_0\geq N_0$, it holds that
	\begin{eqnarray}\label{small-N_0}
		&&\frac{\d}{\d t}\sum_{n\leq N_0}\mathbb{E}_{\ell_0}^{(n)}(t)+\sum_{n\leq N_0}\mathbb{D}_{\ell_0}^{(n)}(t)
	\lesssim\mathcal{D}_{N_0+1}(t)+\mathcal{M}_1\sum_{ n\leq N_0}\mathbb{D}^{(n)}_{\ell_0}(t).
	\end{eqnarray}
	
\end{proposition}
\subsection{The temporal time decay rates for
	$\mathcal{E}_{k\rightarrow N_0}(t)$ and $\mathcal{E}_{1\rightarrow N_0-1,\overline{\ell}_0}(t)$}

To this end, we need
{\bf Assumption II:}
\[\sup_{0\leq t\leq T}\left\{\overline{\mathcal{E}}_{N_0-1,\bar{\ell}_0+\frac52}(t)+\overline{\mathcal{E}}_{N-1,l_1}(t)+\|\Lambda^{-\varrho}[f_\varepsilon,E_\varepsilon,B_\varepsilon]\|^2\right\}\leq \mathcal{M}_2,\]
where $\mathcal{M}_2$ is a sufficiently small positive constant.
\begin{lemma}\label{k-sum-cut} Under {\bf Assumption I} and {\bf Assumption II},  there exists suitably large constants $\bar{l}$, and take $l\geq\bar{l}$, $\widetilde{k}=\min\{k+1, N_0-1\}$ let $N_0\geq 4$, $N=2N_0$, one has the following estimates:
	\begin{itemize}
		\item[(i).] For $k=0,1,\cdots,N_0-1$, it holds that
		\begin{equation}\label{k-sum-cut-1}
			\begin{aligned}
				&\frac{\d}{\d t}\left(\left\|\nabla^kf_{\varepsilon}\right\|^2+\left\|\nabla^k[E_\varepsilon,B_\varepsilon]\right\|^2\right)+\frac1{\varepsilon^2}\left\|\nabla^k\{{\bf I-P}\}f_{\varepsilon}\right\|^2_{\nu}\\[2mm]
				\lesssim&\max\{\mathcal{M}_1,\mathcal{M}_2\}\left(\left\|\nabla^{\widetilde{k}}[E_\varepsilon,B_\varepsilon]\right\|^2+\left\|\nabla^{\widetilde{k}}{\bf P}f_{\varepsilon}\right\|^2\right)+\eta\left\|\nabla^{\widetilde{k}}f_{\varepsilon}\right\|_{\nu}^2.
			\end{aligned}
		\end{equation}
		\item[(iii).] For $k=N_0$, it follows that
		\begin{equation}\label{k-sum-cut-2}
			\begin{aligned}
				&\frac{\d}{\d t}\left(\left\|\nabla^{N_0}f_{\varepsilon}\right\|^2+\left\|\nabla^{N_0}[E_\varepsilon,B_\varepsilon]\right\|^2\right)
				+\frac1{\varepsilon^2}\left\|\nabla^{N_0}\{{\bf I-P}\}f_{\varepsilon}\right\|^2_{\nu}\\[2mm]
				\lesssim&\max\{\mathcal{M}_1,\mathcal{M}_2\}\left(\left\|\nabla^{N_0-1}[E_\varepsilon,B_\varepsilon]\right\|^2
				+\left\|\nabla^{N_0}f_{\varepsilon}\right\|_{\nu}^2\right)+\eta\left\|\nabla^{N_0}f_{\varepsilon}\right\|_{\nu}^2.
			\end{aligned}
		\end{equation}
		\item[(iii).] For $k=0,1,2\cdots,N_0-1$, there exist interactive energy functionals $G^k_f(t)$ satisfying
		\[
		G^k_{f_\varepsilon}(t)\lesssim \left\|\nabla^k[f_\varepsilon,E_\varepsilon,B_\varepsilon]\right\|^2+\left\|\nabla^{k+1}[f_\varepsilon,E_\varepsilon,B_\varepsilon]\right\|^2+\left\|\nabla^{k+2}E_\varepsilon\right\|^2
		\]
		such that
		\begin{eqnarray}\label{k-sum-cut-3}
			&&\frac{\d}{\d t}G^k_{f_\varepsilon}(t)+\left\|\nabla^k[E_\varepsilon,\rho_{\varepsilon}^+-\rho^-_{\varepsilon}]\right\|_{H^1_x}^2+\left\|\nabla^{k+1}[{\bf P}f_\varepsilon,B_\varepsilon])\right\|^2\nonumber\\
			&\lesssim&\max\{\mathcal{M}_1,\mathcal{M}_2\}\left(\left\|\nabla^{\widetilde{k}}[E_\varepsilon,B_\varepsilon]\right\|^2+\left\|\nabla^{\widetilde{k}}f_{\varepsilon}\right\|^2_{\nu}\right)
			+\frac1{{\varepsilon^2}}\left\|\nabla^k\{{\bf I-P}\}f_{\varepsilon}\right\|^2_{\nu}.
		\end{eqnarray}
	\end{itemize}
\end{lemma}
\begin{proof}
	We can use the method of proving Lemma 3.4-3.5 in \cite{Lei-Zhao-JFA-2014} to prove this lemma. For the sake of brevity, we only provide some distinct and crucial estimates here, and similar calculations can be carried out for other related estimates.
	
	For the case $1\leq j\leq k-1\leq N_0-2$, we can control \[\int_{{\mathbb{R}}^3_x\times{\mathbb{R}}^3_v}\frac1\varepsilon\nabla^j_x[v\times B_\varepsilon]\cdot\nabla_v\nabla^{k-j}_x{\bf\{I- P\}}f_\varepsilon\nabla^k_x{{\bf\{I- P\}}f_{\varepsilon}}dxdv,\]
	by using interpolation with respect to both spatial-velocity derivatives and velocity. Take $\varrho=\frac12$ for brevity,
	when $j=1$, the estimate is given as follows:
	\begin{eqnarray*}
		&&	\int_{{\mathbb{R}}^3_x\times{\mathbb{R}}^3_v}\frac1\varepsilon\nabla_x[v\times B_\varepsilon]\cdot\nabla_v\nabla^{k-1}_x{\bf\{I- P\}}f_\varepsilon\nabla^k_x{{\bf\{I- P\}}f_{\varepsilon}}dxdv\\
		&\lesssim&\frac1\varepsilon\left\|\nabla_xB_\varepsilon\right\|_{L^\infty_x}\left\|\nabla_v\nabla^{k-1}_x{\bf\{I- P\}}f_{\varepsilon}\right\|\left\|\nabla^k_x{\bf\{I- P\}}f_\varepsilon\langle v\rangle\right\| \\
		&\lesssim&\frac1\varepsilon\left\|\Lambda^{-{\frac12}}B_\varepsilon\right\|^{\frac{k-\frac32}{k+\frac23}}\left\|\nabla^{k+1}B_\varepsilon\right\|^{\frac{3}{k+\frac32}}
		\left\|\nabla^{m_{1}+1}_v\nabla^{k-1}_x{\bf\{I- P\}}f_{\varepsilon}\right\|^{\frac1{m_{1}+1}}\\{}
		&&\times\left\|\Lambda^{-\frac12}{\bf\{I-P\}}f_{\varepsilon}\right\|^{\frac{m_1}{m_1+1}\times\frac1{k+\frac12}}\left\|\nabla^{k}{\bf\{I- P\}}f_{\varepsilon}\right\|^{\frac{m_1}{m_1+1}\times\frac{k-\frac12}{k+\frac12}}\left\|\nabla^k{\bf\{I- P\}}f_\varepsilon\langle v\rangle\right\| \\{}
		&\lesssim&\frac1\varepsilon\left\|\Lambda^{-{\frac12}}B_\varepsilon\right\|^{\frac{k-\frac32}{k+\frac23}}
		\left\|\nabla^{m_{1}+1}_v\nabla^{k-1}_x{\bf\{I- P\}}f_{\varepsilon}\right\|^{\frac1{m_{1}+1}}\left\|\Lambda^{-\frac12}{\bf\{I-P\}}f_{\varepsilon}\right\|^{\frac{m_1}{m_1+1}\times\frac1{k+\frac12}}\\{}
		&&\times\left\|\nabla^{k+1}B_\varepsilon\right\|^{\frac{3}{k+\frac32}}\left\|\nabla^{k}{\bf\{I- P\}}f_{\varepsilon}\right\|^{\frac{m_1}{m_1+1}\times\frac{k-\frac12}{k+\frac12}}\left\|\nabla^k{\bf\{I- P\}}f_\varepsilon\langle v\rangle\right\| \\{}
		&\lesssim&\frac1\varepsilon\left\|\Lambda^{-{\frac12}}B_\varepsilon\right\|^{\frac{k-\frac32}{k+\frac23}}
		\left\|\nabla^{m_{1}+1}_v\nabla^{k-1}_x{\bf\{I- P\}}f_{\varepsilon}\right\|^{\frac1{m_{1}+1}}\left\|\Lambda^{-\frac12}{\bf\{I-P\}}f_{\varepsilon}\right\|^{\frac{m_1}{m_1+1}\times\frac1{k+\frac12}}\\{}
		&&\times\left\|\nabla^{k+1}B_\varepsilon\right\|^{\frac{3}{k+\frac32}}\left\|\langle v\rangle^{-\frac12}\nabla^{k}{\bf\{I- P\}}f_{\varepsilon}\right\|^{\frac{m_1}{m_1+1}\times\frac{k-\frac12}{k+\frac12}\times\frac{\bar{l}_1}{\bar{l}_1+\frac12}}\left\|\langle v\rangle^{\bar{l}_1}\nabla^{k}{\bf\{I- P\}}f_{\varepsilon}\right\|^{\frac{m_1}{m_1+1}\times\frac{k-\frac12}{k+\frac12}\times\frac{1}{\bar{l}_1+\frac12}}\nonumber\\
		&&\times\left\|\langle v\rangle^{-\frac12}\nabla^k{\bf\{I- P\}}f_{\varepsilon}\right\|^{\frac{\bar{l}_2-1}{\bar{l}_2+\frac12}}\left\|\langle v\rangle^{\bar{l}_2}\nabla^k{\bf\{I- P\}}f_{\varepsilon}\right\|^{\frac{\frac32}{\bar{l}_2+\frac12}} \\{}
		&=&\frac1\varepsilon\left\|\Lambda^{-{\frac12}}B_\varepsilon\right\|^{\frac{k-\frac32}{k+\frac23}}
		\left\|\nabla^{m_{1}+1}_v\nabla^{k-1}_x{\bf\{I- P\}}f_{\varepsilon}\right\|^{\frac1{m_{1}+1}}\left\|\Lambda^{-\frac12}{\bf\{I-P\}}f_{\varepsilon}\right\|^{\frac{m_1}{m_1+1}\times\frac1{k+\frac12}}\\{}
		&&\times\left\|\langle v\rangle^{\bar{l}_2}\nabla^k{\bf\{I- P\}}f_{\varepsilon}\right\|^{\frac{\frac32}{\bar{l}_2+\frac12}} \left\|\langle v\rangle^{\bar{l}_1}\nabla^{k}{\bf\{I- P\}}f_{\varepsilon}\right\|^{\frac{m_1}{m_1+1}\times\frac{k-\frac12}{k+\frac12}\times\frac{1}{\bar{l}_1+\frac12}}\\{}
		&&\times\left\|\nabla^{k+1}B_\varepsilon\right\|^{\frac{3}{k+\frac32}}\left\|\langle v\rangle^{-\frac12}\nabla^{k}{\bf\{I- P\}}f_{\varepsilon}\right\|^{\frac{m_1}{m_1+1}\times\frac{k-\frac12}{k+\frac12}\times\frac{\bar{l}_1}{\bar{l}_1+\frac12}}\left\|\langle v\rangle^{-\frac12}\nabla^k{\bf\{I- P\}}f_{\varepsilon}\right\|^{\frac{\bar{l}_2-1}{\bar{l}_2+\frac12}}\nonumber\\
		&\lesssim&\max\left\{\left\|\langle v\rangle^{\max{\{\bar{l}_1,\bar{l}_2\}}}\nabla^k{\bf\{I- P\}}f_{\varepsilon}\right\|,\left\|\nabla^{m_{1}+1}_v\nabla^{k-1}_x{\bf\{I- P\}}f_{\varepsilon}\right\|^{\frac1{m_{1}+1}},\left\|\Lambda^{-{\frac12}}[f,B_\varepsilon]\right\|^2\right\}\nonumber\\
		&&\times\left(\left\|\nabla^{k+1}B_\varepsilon\right\|^2+\left\|\nabla^k\{{\bf I-P}\}f_{\varepsilon}\right\|^2_\nu\right)+\frac\eta{\varepsilon^2}\left\|\nabla^k\{{\bf I-P}\}f_{\varepsilon}\right\|^2_\nu\nonumber\\
		&\lesssim&\max\{\mathcal{M}_1,\mathcal{M}_2\}\left(\left\|\nabla^{k+1}B_\varepsilon\right\|^2+\left\|\nabla^k\{{\bf I-P}\}f_{\varepsilon}\right\|^2_\nu\right)+\frac\eta{\varepsilon^2}\left\|\nabla^k\{{\bf I-P}\}f_{\varepsilon}\right\|^2_\nu.
	\end{eqnarray*}
	In fact, we can take proper $m_1$, which satisfies
	\[m_1>\frac k2-1-\frac3{8k},\ m_1+k\leq 2N_0-1\]
	such that
	\[\ \frac{3}{k+\frac32}+\frac{m_1}{m_1+1}\times\frac{k-\frac12}{k+\frac12}>1,\]
	so obviously there exists suitably large constants $\bar{l}_1$ and $\bar{l}_2$ such that the following equation holds
	\[\frac{3}{k+\frac32}+\frac{m_1}{m_1+1}\times\frac{k-\frac12}{k+\frac12}\times\frac{\bar{l}_1}{\bar{l}_1+\frac12}+\frac{\bar{l}_2-1}{\bar{l}_2+\frac12}=2.\]
	Here we ask $\bar{l}\geq \max{\{\bar{l}_1,\bar{l}_2\}}$.
	Based on the aforementioned estimated ideas, we can prove \eqref{k-sum-cut-1}, \eqref{k-sum-cut-2} and \eqref{k-sum-cut-3}, thus the proof of this lemma is complete.
\end{proof}
\begin{proposition}\label{cut-decay}
	Under {\bf Assumption I} and {\bf Assumption II}, assume $N_0\geq 5$, $N=2N_0$,
	there exist an energy functional $\mathcal{E}_{k\rightarrow N_0}(t)$ and the corresponding energy dissipation rate functional $\mathcal{D}_{k\rightarrow N_0}(t)$ satisfying \eqref{E-k} and \eqref{D-k} respectively such that
	\begin{equation}\label{cut-decay-1}
		\frac{\d}{\d t}\mathcal{E}_{k\rightarrow N_0}(t)+\mathcal{D}_{k\rightarrow N_0}(t)\leq 0
	\end{equation}
	{holds for $k=0,1,2,\cdots, N_0-2$ and all $0\leq t\leq T.$}
	
	Furthermore, we can get that
	\begin{equation}\label{cut-decay-2}
		\mathcal{E}_{k\rightarrow N_0}(t)\lesssim\min\{\mathcal{M}_1,\mathcal{M}_2\}(1+t)^{-(k+\varrho)}.
	\end{equation}
\end{proposition}
\begin{proof}
	From \eqref{k-sum-cut-1}, \eqref{k-sum-cut-1} and \eqref{k-sum-cut-3}, one has \eqref{cut-decay-1}.
	The decay estimtates \eqref{cut-decay-2} can be proved by a similar way as Proposition \ref{noncut-decay}, we omit its proof for brevity.
	\end{proof}
	
Besides, we need to deduce the temporal time decay rates for $\mathcal{E}_{1\rightarrow N_0-1,\overline{\ell}_0}(t)$.
\begin{proposition}\label{cut-w-decay}
Under {\bf Assumption I} and {\bf Assumption II},
	it holds that
	\begin{eqnarray}\label{cut-w-decay-1}
		\frac{\d}{\d t}\left\{\mathcal{E}_{1\rightarrow N_0-1,\overline{\ell}_0}(t)+\mathcal{E}_{1\rightarrow N_0}(t)\right\}+\mathcal{D}_{1\rightarrow N_0-1,\overline{\ell}_0}(t)+\mathcal{D}_{1\rightarrow N_0}(t)\lesssim 0.
	\end{eqnarray}
Furthermore, one also deduces that
	\begin{eqnarray}\label{cut-w-decay-2}
&&\mathcal{E}_{1\rightarrow N_0-1,\overline{\ell}_0}(t)+\mathcal{E}_{1\rightarrow N_0}(t)\lesssim\min\{\mathcal{M}_1,\mathcal{M}_2\} (1+t)^{-1-\varrho}.
	\end{eqnarray}
\end{proposition}
\begin{proof}	
		Recalling the definition of $\mathcal{E}_{1\rightarrow N_0-1,\overline{\ell}_0}(t)$ in \eqref{E-N-0-1-w},
	applying $\partial^\alpha_\beta$ into \eqref{I-P-cut}, integrating the result identity over $\mathbb{R}^3_x\times\mathbb{R}^3_v$ by multiplying $\overline{w}^2_{\overline{\ell}_0}(\alpha,\beta)\partial^\alpha_\beta\{{\bf I-P}\}f_{\varepsilon}$ with $|\alpha|\geq 1$, then we have
	\begin{eqnarray}\label{bound-I-P-low-w}
		&&\frac12\frac{\d}{\d t}\|\overline {w}_{\overline{\ell}_0}(\alpha,\beta)\partial^\alpha_\beta\{{\bf I-P}\}f_{\varepsilon}\|^2+\frac{q\vartheta}{(1+t)^{1+\vartheta}}\left\|\langle v\rangle\overline {w}_{\overline{\ell}_0}(\alpha,\beta)\partial^\alpha_\beta\{{\bf I-P}\}f_{\varepsilon}\right\|^2\nonumber\\
		&&+\frac1 {\varepsilon^2} \left(\partial^\alpha_\beta\mathscr{L}f_\varepsilon,\overline {w}^2_{\overline{\ell}_0}(\alpha,\beta)\partial^\alpha_\beta\{{\bf I-P}\}f_{\varepsilon}\right)\nonumber\\
		&=&-\frac1\varepsilon\left(\partial^\alpha_\beta[ v\cdot \nabla_x\{{\bf I-P}\}f_{\varepsilon}],\overline {w}^2_{\overline{\ell}_0}(\alpha,\beta)\partial^\alpha_\beta\{{\bf I-P}\}f_{\varepsilon}\right)\nonumber\\
		&&+\frac1 \varepsilon \left(\partial^\alpha_\beta \left[E_\varepsilon\cdot v{M}^{1/2}q_1\right],\overline {w}^2_{\overline{\ell}_0}(\alpha,\beta)\partial^\alpha_\beta\{{\bf I-P}\}f_{\varepsilon}\right)\nonumber\\
		&&+\frac1 \varepsilon\left(\partial^\alpha_\beta\left[{\bf P}(v\cdot\nabla_x f_\varepsilon)-\frac1 \varepsilon v\cdot\nabla_x{\bf P}f_\varepsilon\right],\overline {w}^2_{\overline{\ell}_0}(\alpha,\beta)\partial^\alpha_\beta\{{\bf I-P}\}f_{\varepsilon}\right)\nonumber\\
		&&+\frac1\varepsilon\left(\partial^\alpha_\beta\left[\{{\bf I-P}\}\left[ q_0v\times B_\varepsilon\cdot\nabla_{ v}f_\varepsilon\right]\right],\overline {w}^2_{\overline{\ell}_0}(\alpha,\beta)\partial^\alpha_\beta\{{\bf I-P}\}f_{\varepsilon}\right)\nonumber\\
		&&+\left(\partial^\alpha_\beta\{{\bf I-P}\}\left[\tfrac{1}{2} q_0 (E_\varepsilon \cdot v) f_\varepsilon  -q_0E_\varepsilon\cdot\nabla_vf_\varepsilon \right],\overline {w}^2_{\overline{\ell}_0}(\alpha,\beta)\partial^\alpha_\beta\{{\bf I-P}\}f_{\varepsilon}\right)\nonumber\\
		&&+\tfrac{1}{\eps} \left(\partial^\alpha_\beta\mathscr{T} (f_\varepsilon, f_\varepsilon),\overline {w}^2_{\overline{\ell}_0}(\alpha,\beta)\partial^\alpha_\beta\{{\bf I-P}\}f_{\varepsilon}\right).
	\end{eqnarray}
	The coercivity estimates on the linear operator $\mathscr{L}$ yields that
	\begin{eqnarray}
		&&\frac1 {\varepsilon^2} \left(\partial^\alpha_\beta\mathscr{L}f_\varepsilon,\overline {w}^2_{\overline{\ell}_0}(\alpha,\beta)\partial^\alpha_\beta\{{\bf I-P}\}f_{\varepsilon}\right)\nonumber\\
		&\gtrsim&\frac1 {\varepsilon^2} \|\overline {w}_{\overline{\ell}_0}(\alpha,\beta)\partial^\alpha_\beta\{{\bf I-P}\}f_{\varepsilon}\|_\nu^2-\frac1 {\varepsilon^2}\|\partial^\alpha\{{\bf I-P}\}f_{\varepsilon}\|_\nu^2.
	\end{eqnarray}
	As for the transport term on the right-hand side of \eqref{bound-I-P-low-w}, since
		\[\overline {w}_{\overline{\ell}_0}(\alpha,\beta)=\langle v\rangle^{-1}\overline{w}_{\overline{\ell}_0}(\alpha+e_i,\beta-e_i),\]
		one has
	\begin{eqnarray}
		&&\frac1\varepsilon\left(\partial^\alpha_\beta\left[ v\cdot \nabla_x\{{\bf I-P}\}f_{\varepsilon}\right],\overline{w}^2_{\overline{\ell}_0}(\alpha,\beta)\partial^\alpha_\beta\{{\bf I-P}\}f_{\varepsilon}\right)\nonumber\\
		&=&-\frac1\varepsilon\int_{\mathbb{R}^3_x\times\mathbb{R}^3_v}\langle v\rangle^{-1}\partial^{\alpha+e_i}_{\beta-e_i}\{{\bf I-P}\}f_{\varepsilon} \overline{w}_{\overline{\ell}_0}(\alpha+e_i,\beta-e_i)\overline {w}_{\overline{\ell}_0}(\alpha,\beta)\partial^\alpha_\beta\{{\bf I-P}\}f_{\varepsilon}dvdx\nonumber\\
		&\lesssim&\frac1\varepsilon\left\|\langle v\rangle^{\frac\gamma2}\overline{w}_{\overline{\ell}_0}(\alpha+e_i,\beta-e_i)\partial^{\alpha+e_i}_{\beta-e_i}\{{\bf I-P}\}f_{\varepsilon}\right\|^2+\frac\eta\varepsilon\left\|\langle v\rangle^{\frac\gamma2}\overline{w}_{\overline{\ell}_0}(\alpha,\beta)\partial^{\alpha}_{\beta}\{{\bf I-P}\}f_{\varepsilon}\right\|^2\nonumber\\
		&\lesssim&\frac1\varepsilon\left\|\overline{w}_{\overline{\ell}_0}(\alpha+e_i,\beta-e_i)\partial^{\alpha+e_i}_{\beta-e_i}\{{\bf I-P}\}f_{\varepsilon}\right\|_\nu^2+\frac\eta\varepsilon\left\|\overline{w}_{\overline{\ell}_0}(\alpha,\beta)\partial^{\alpha}_{\beta}\{{\bf I-P}\}f_{\varepsilon}\right\|^2_\nu.	
	\end{eqnarray}
The second and third terms on the right-hand side of \eqref{bound-I-P-low-w} can be dominated by
	\begin{eqnarray}
		\|\partial^\alpha E_\varepsilon\|^2+\|\partial^\alpha \{{\bf I-P}\}f_{\varepsilon}\|_\nu^2+\|\nabla^{|\alpha|+1}_x {\bf P}f_{\varepsilon}\|^2\lesssim \mathcal{D}_{|\alpha|\rightarrow N_0}(t).
	\end{eqnarray}
	For the nonlinear term induced by magnetic field $B(t,x)$,
	\begin{eqnarray}
		&&\frac1\varepsilon\left(\partial^\alpha_\beta\left[ v\times B_\varepsilon\cdot\nabla_{ v}\{{\bf I-P}\}f_{\varepsilon}\right],\overline{w}^2_{\overline{\ell}_0}(\alpha,\beta)\partial^\alpha_\beta\{{\bf I-P}\}f_{\varepsilon}\right)\nonumber\\
		&=&\frac1\varepsilon\sum_{\alpha_1\leq\alpha,\beta_1\leq\beta}\left(\partial^{\alpha_1}_{\beta_1}\left[ v\times B_\varepsilon\right]\cdot\partial^{\alpha-\alpha_1}_{\beta-\beta_1}\left[\nabla_{ v}\{{\bf I-P}\}f_{\varepsilon}\right],\overline{w}^2_{\overline{\ell}_0}(\alpha,\beta)\partial^\alpha_\beta\{{\bf I-P}\}f_{\varepsilon}\right),\nonumber
	\end{eqnarray}
	when $|\alpha_1|=1,\beta_1=0$,
	we apply Sobolev inequalities, interpolation method and Young inequalites to get
	\begin{eqnarray}\label{B-w-decay}
		&\lesssim&\sum_{|\alpha_1|=1}{\frac1\varepsilon}\|\partial^{\alpha_1}B_\varepsilon\|_{L^\infty_x}\|\langle v\rangle^{\frac52} \overline{w}_{\overline{\ell}_0}(\alpha-\alpha_1,\beta+e_i)\partial^{\alpha-\alpha_1}_{\beta+e_i}{\bf \{I-P\}}f_{\varepsilon}\|\nonumber\\
		&&\times\|\langle v\rangle^{-\frac{1}2} \overline {w}_{\overline{\ell}_0}(\alpha,\beta)\partial^{\alpha}_\beta{\bf \{I-P\}}f_{\varepsilon}\|\nonumber\\
		&\lesssim&\sum_{|\alpha_1|=1}
		\|\partial^{\alpha_1}B_\varepsilon\|^2_{L^\infty_x}\|\langle v\rangle^{\frac{5}2} \overline{w}_{\overline{\ell}_0}(\alpha-\alpha_1,\beta+e_i)\partial^{\alpha-\alpha_1}_{\beta +e_i}{\bf \{I-P\}}f_{\varepsilon}\|^2\nonumber\\
		&&+{\frac\eta{\varepsilon^2}}\| \overline {w}_{\overline{\ell}_0}(\alpha,\beta)\partial^{\alpha}_\beta{\bf \{I-P\}}f_{\varepsilon}\|_\nu^2\nonumber\\
			&\lesssim&\mathcal{D}_{1\rightarrow N_0}(t)\overline{\mathcal{E}}_{N_0-1,\overline{\ell}_0+\frac52}(t)+{\frac\eta{\varepsilon^2}}\| \overline {w}_{\overline{\ell}_0}(\alpha,\beta)\partial^{\alpha}_\beta{\bf \{I-P\}}f_{\varepsilon}\|_\nu^2.
	\end{eqnarray}
The other cases have similar bound as \eqref{B-w-decay}. Consequently, one has
	\begin{eqnarray}
		&&\frac1\varepsilon\left(\partial^\alpha_\beta\left[ v\times B_\varepsilon\cdot\nabla_{ v}\{{\bf I-P}\}f_{\varepsilon}\right],\overline{w}^2_{l_1}(\alpha,\beta)\partial^\alpha_\beta\{{\bf I-P}\}f_{\varepsilon}\right)\nonumber\\
		&\lesssim&\mathcal{D}_{1\rightarrow N_0}(t)\overline{\mathcal{E}}_{N_0-1,\overline{\ell}_0+\frac52}(t)+{\frac\eta{\varepsilon^2}}\| \overline {w}_{\overline{\ell}_0}(\alpha,\beta)\partial^{\alpha}_\beta{\bf \{I-P\}}f_{\varepsilon}\|_\nu^2.
	\end{eqnarray}
	By using the similar argument, one also has
	\begin{eqnarray}
		&&\left(\partial^\alpha_\beta\left[ E_\varepsilon\cdot\nabla_{ v}\{{\bf I-P}\}f_{\varepsilon}\right],\overline{w}^2_{\overline{\ell}_0}(\alpha,\beta)\partial^\alpha_\beta\{{\bf I-P}\}f_{\varepsilon}\right)\nonumber\\
		&\lesssim&\mathcal{D}_{1\rightarrow N_0}(t)\overline{\mathcal{E}}_{N_0-1,\overline{\ell}_0+\frac32}(t)+{\frac\eta{\varepsilon^2}}\| \overline {w}_{\overline{\ell}_0}(\alpha,\beta)\partial^{\alpha}_\beta{\bf \{I-P\}}f_{\varepsilon}\|_\nu^2
	\end{eqnarray}
	and
	\begin{eqnarray}
		&&\left(\partial^\alpha_\beta\left[ E_\varepsilon\cdot v\{{\bf I-P}\}f_{\varepsilon}\right],\overline{w}^2_{\overline{\ell}_0}(\alpha,\beta)\partial^\alpha_\beta\{{\bf I-P}\}f_{\varepsilon}\right)\nonumber\\
		&\lesssim&\mathcal{D}_{1\rightarrow N_0}(t)\overline{\mathcal{E}}_{N_0-1,\overline{\ell}_0+\frac32}(t)+{\frac\eta{\varepsilon^2}}\| \overline {w}_{\overline{\ell}_0}(\alpha,\beta)\partial^{\alpha}_\beta{\bf \{I-P\}}f_{\varepsilon}\|_\nu^2,
	\end{eqnarray}
	one has
	\begin{eqnarray}
	&&\tfrac{1}{\eps} \left(\partial^\alpha_\beta\mathscr{T} (f_\varepsilon, f_\varepsilon),\overline {w}^2_{\overline{\ell}_0}(\alpha,\beta)\partial^\alpha_\beta\{{\bf I-P}\}f_{\varepsilon}\right)\nonumber\\
		&\lesssim&\overline{\mathcal{E}}_{N_0-1,\overline{\ell}_0}(t)\mathcal{D}_{1\rightarrow N_0-1,\overline{\ell}_0}(t)+{\frac\eta{\varepsilon^2}}\| \overline {w}_{\overline{\ell}_0}(\alpha,\beta)\partial^{\alpha}_\beta{\bf \{I-P\}}f_{\varepsilon}\|_\nu^2.
	\end{eqnarray}
	Now by collecting the above related estimates into \eqref{bound-I-P-low-w}, we arrive at
	\begin{eqnarray}\label{bound-sum-decay}
		&&\frac{\d}{\d t}\left\|\overline{w}_{\overline{\ell}_0}(\alpha,\beta)\partial^\alpha_\beta\{{\bf I-P}\}f_{\varepsilon}\right\|^2+\frac1{\varepsilon^2}\left\|\overline{w}_{\overline{\ell}_0}(\alpha,\beta)\partial^\alpha_\beta\{{\bf I-P}\}f_{\varepsilon}\right\|^2_{\nu}\nonumber\\
		&&+\frac{q\vartheta}{(1+t)^{1+\vartheta}}\left\|\langle v\rangle \overline{w}_{\overline{\ell}_0}(\alpha,\beta)\partial^\alpha_\beta\{{\bf I-P}\}f_{\varepsilon}\right\|^2\nonumber\\
		&\lesssim&\left\|\overline{w}_{\overline{\ell}_0}(\alpha+e_i,\beta-e_i)\partial^{\alpha+e_i}_{\beta-e_i}\{{\bf I-P}\}f_{\varepsilon}\right\|_\nu^2+\mathcal{D}_{1\rightarrow N_0}(t)\overline{\mathcal{E}}_{N_0-1,\overline{\ell}_0+\frac52}(t)+\mathcal{D}_{|\alpha|\rightarrow N_0}(t),
	\end{eqnarray}
Taking summation for \eqref{bound-sum-decay} with $|\alpha|+|\beta|\leq N_0-1, |\alpha|\geq 1$, one has
\begin{eqnarray}
\frac{\d}{\d t}\mathcal{E}_{1\rightarrow N_0-1,\overline{\ell}_0}(t)+\mathcal{D}_{1\rightarrow N_0-1,\overline{\ell}_0}(t)\lesssim \mathcal{D}_{1\rightarrow N_0}(t)
	\end{eqnarray}
which combining with \eqref{cut-decay-1} yields that
\begin{eqnarray}
	\frac{\d}{\d t}\left\{\mathcal{E}_{1\rightarrow N_0-1,\overline{\ell}_0}(t)+\mathcal{E}_{1\rightarrow N_0}(t)\right\}+\mathcal{D}_{1\rightarrow N_0-1,\overline{\ell}_0}(t)+\mathcal{D}_{1\rightarrow N_0}(t)\lesssim 0.
\end{eqnarray}
\eqref{cut-w-decay-2} follows by a similar way as \eqref{cut-decay-2}, thus we complete the proof of this Propostion \ref{cut-w-decay}.
\end{proof}
\subsection{The bound for $\|\Lambda^{-\varrho}[f_\varepsilon,E_\varepsilon,B_\varepsilon](t)\|^2$}
To ensure {\bf Assumption 2}, one also need the following result:
\begin{proposition}\label{prop-neg-cut}
	Under Under {\bf Assumption 1} and {\bf Assumption 2}, one has
	\begin{eqnarray}\label{prop-neg-cut-1}
		&&\frac{\d}{\d t}\left(\left\|\Lambda^{-\varrho}f_{\varepsilon}\right\|^2+\left\|\Lambda^{-\varrho}[E_\varepsilon,B_\varepsilon]\right\|^2+\kappa_1G^{-\varrho}_{f_\varepsilon}(t)\right)+\frac1{\varepsilon^2}\left\|\Lambda^{-\varrho}\{{\bf I-P}\}f_{\varepsilon}\right\|_{\nu}^2\nonumber\\
		&&+\kappa_1\left\|\Lambda^{1-\varrho}{\bf P}f_{\varepsilon}\right\|^2+\kappa_1\left\|\Lambda^{1-\varrho}[E_\varepsilon,B_\varepsilon]\right\|^2+\kappa_1\left\|\Lambda^{-\varrho}E_\varepsilon\right\|_{H^1_x}^2+\kappa_1\left\|\Lambda^{-\varrho}(\rho_\varepsilon^+-\rho_\varepsilon^-)\right\|^2_{H^1}\nonumber\\
		&\lesssim&\mathcal{M}_2\left\|\nabla_xf_{\varepsilon}\right\|^2_{\nu}.
	\end{eqnarray}
\end{proposition}
\begin{proof}
This proposition can be proved by a similar way as Proposition \ref{prop-neg}, we omit its proof for brevity.
\end{proof}
\subsection{The a priori estimates}
Now we are ready to construct the a priori estimates:
\begin{eqnarray}\label{The-a-priori-estimtates-cut}
X(T)&\equiv&\sup_{0\leq t\leq T}\left\{\sum_{n\leq N_0}\mathbb{E}_{\ell_0}^{(n)}(t)+\sum_{N_0+1\leq n\leq N-1}\mathbb{E}_{\ell_1}^{(n)}(t)+\varepsilon^2\mathbb{E}_{\ell_1}^{(N)}(t)\right\}\nonumber\\
&&+\sup_{0\leq t\leq T}\left\{\overline{\mathcal{E}}_{N-1,l_1}(t)+\overline{\mathcal{E}}_{N_0-1,l_0}(t)+\mathcal{E}_{N}(t)+\|\Lambda^{-\varrho}[f_\varepsilon,E_\varepsilon,B_\varepsilon]\|^2\right\}\leq \mathcal{M}.
\end{eqnarray}
\begin{proposition}\label{prop1-cut}
Assume that
\begin{itemize}
	\item $N_0\geq 5$, $N=2N_0$;
	\item $-1\leq\gamma<0$, $\frac12\leq \varrho<\frac32$, $0<\vartheta\leq \frac23\rho$;
	\item $0<\epsilon_0\leq 2(1+\varrho)$;
	\item $\sigma_{N,0}=\frac{1+\epsilon_0}2$, $\sigma_{n,0}=0$ for $n\leq N-1$, $\sigma_{n,j+1}-\sigma_{n,j}=\frac{1+\epsilon_0}2$;
	\item $\bar{l}$ is a sufficiently large positive constant;
	\item $l_1\geq N+\bar{l}$, $\tilde{\ell}\geq\frac32\sigma_{N-1,N-1}$, $\ell_1\geq l_1+\tilde{\ell}+\frac12$, $\overline{\ell}_0\geq \ell_1+\frac32N$, $l_0\geq \overline{\ell}_0+\frac52$, $\ell_0\geq l_0+\tilde{\ell}+\frac12$,
\end{itemize}
under the a priori estimates \eqref{The-a-priori-estimtates-cut}, we can deduce that
	\begin{eqnarray}\label{cut-end}
		&&\frac{\d}{\d t}\left\{\sum_{n\leq N_0}\mathbb{E}_{\ell_0}^{(n)}(t)+\sum_{N_0+1\leq n\leq N-1}\mathbb{E}_{\ell_1}^{(n)}(t)+\varepsilon^2\mathbb{E}_{\ell_1}^{(N)}(t)\right\}\nonumber\\
		&&+\frac{\d}{\d t}\left\{\overline{\mathcal{E}}_{N-1,l_1}(t)+\overline{\mathcal{E}}_{N_0-1,l_0}(t)+\mathcal{E}_{N}(t)+\|\Lambda^{-\varrho}[f_\varepsilon,E_\varepsilon,B_\varepsilon]\|^2\right\}+\sum_{n\leq N_0}\mathbb{D}_{\ell_0}^{(n)}(t)\nonumber\\
		&&+\sum_{N_0+1\leq n\leq N-1}\mathbb{D}_{\ell_1}^{(n)}(t)+\varepsilon^2\mathbb{D}_{\ell_1}^{(N)}(t)+\overline{\mathcal{D}}_{N-1,l_1}(t)+\overline{\mathcal{D}}_{N_0-1,l_0}(t)+\mathcal{D}_{N}(t)
		\lesssim0\nonumber
	\end{eqnarray}
	{ holds for all $0\leq t\leq T$.}
\end{proposition}
\begin{proof}
	Recalling the definition of $\mathbb{D}^{(N)}_{\ell}(t)$ and $\overline{\mathcal{D}}_{N-1,l}(t)$, \eqref{prof-1-cut} tells us that
	\begin{eqnarray}\label{end-1-cut}
		&&\frac{\d}{\d t}\mathcal{E}_{N}(t)+\mathcal{D}_{N}(t)\nonumber\\
		&\lesssim &\left\|E_\varepsilon\right\|_{L^\infty_x}^2(1+t)^{\sigma_{N,0}}(1+t)^{-\sigma_{N,0}}\left\|\langle v\rangle^{\frac 32}\nabla^N_xf_{\varepsilon}\right\|^2\nonumber\\
		&&+\left\|\nabla_xB_\varepsilon\right\|_{L^\infty_x}^2(1+t)^{\sigma_{N,1}}(1+t)^{-\sigma_{N,1}}\left\|\langle v\rangle^{\frac 32}\nabla^{N-1}_x\nabla_v\{{\bf I-P}\}f_{\varepsilon}\right\|^2\nonumber\\
		&&+\mathcal{E}_N(t)\sum_{|\alpha'|+|\beta'|\leq N-1}\|\partial^{\alpha'}_{\beta'} \{{\bf I-P}\}f_{\varepsilon}\langle v\rangle^{\frac32}\|^2\nonumber\\
		&\lesssim&\mathcal{M}_1{\varepsilon^2}\mathbb{D}^{(N)}_{\ell_1}(t)+\mathcal{E}_N(t)\overline{\mathcal{D}}_{N-1,l_1}(t).
	\end{eqnarray}

		By multiplying $(1+t)^{-\epsilon_0}$ into \eqref{end-1-cut}, one has
	\begin{eqnarray}\label{end-2-cut}
		&&\frac{\d}{\d t}\left\{(1+t)^{-\epsilon_0}\mathcal{E}_{N}(t)\right\}+\epsilon_0(1+t)^{-1-\epsilon_0}\mathcal{E}_{N}(t)+(1+t)^{-\epsilon_0}\mathcal{D}_{N}(t)\nonumber\\
		&\lesssim&\mathcal{M}_1{\varepsilon^2}\mathbb{D}^{(N)}_{\ell_1}(t)+\mathcal{E}_N(t)\overline{\mathcal{D}}_{N-1,l_1}(t).
	\end{eqnarray}

	Thus \eqref{cut-end} follows from \eqref{end-1-cut},  \eqref{end-2-cut}, \eqref{large-N}, \eqref{large-N-1}, \eqref{small-N_0} and \eqref{prop-neg-cut-1}, which complete the proof of this proposition.
\end{proof}
\section{Limit to two fluid incompressible Navier-Stokes-Fourier-Maxwell equations with Ohm's law}
\label{Sec:Limits}
In this section, we will derive the two fluid incompressible NSFM equations \eqref{INSFM-Ohm} with Ohm's law from the perturbed two-species VMB as $\eps \rightarrow 0$.

As \cite{Jiang-Luo-2022-Ann.PDE}, we first introduce the following fluid variables
\begin{equation}\label{Fluid-Quanities}
	\begin{aligned}
		\rho_\eps = \tfrac{1}{2} \langle {f}_\eps , {q_2} \sqrt{M} \rangle_{L^2_v} \,, \ u_\eps = \tfrac{1}{2} \langle f_\eps , {q_2} v \sqrt{M} \rangle_{L^2_v} \,, \ \theta_\eps = \tfrac{1}{2} \langle f_\eps , {q_2} ( \tfrac{|v|^2}{3} - 1 ) \sqrt{M} \rangle_{L^2_v} \,, \\
		n_\eps = \langle f_\eps , {q_1} \sqrt{M} \rangle_{L^2_v} \,,\ j_\eps = \tfrac{1}{\eps} \langle f_\eps , {q_1} v \sqrt{M} \rangle_{L^2_v} \,,\ w_\eps = \tfrac{1}{\eps} \langle f_\eps , {q_1} ( \tfrac{|v|^2}{3} - 1 )\sqrt{M} \rangle_{L^2_v} \,.
	\end{aligned}
\end{equation}
We use the similar argument as \eqref{Macro-equation1}, we can deduce the following local conservation laws
\begin{equation}\label{Local-Consvtn-Law}
	\left\{
	\begin{array}{l}
		\partial_t \rho_\eps + \tfrac{1}{\eps} \div_x \, u_\eps = 0 \,, \\
		\partial_t u_\eps + \tfrac{1}{\eps} \nabla_x ( \rho_\eps + \theta_\eps ) + \div_x \, \big\langle \widehat{A} (v) \sqrt{M}\cdot {q_2} , \tfrac{1}{\eps} \mathscr{L} ( \tfrac{f_\eps }{2} ) \big\rangle_{L^2_v} = \tfrac{1}{2} ( n_\eps E_\eps + j_\eps \times B_\eps ) \,, \\
		\partial_t \theta_\eps + \tfrac{2}{3} \tfrac{1}{\eps} \div_x \, u_\eps + \tfrac{2}{3} \div_x \, \big\langle \widehat{B} (v) \sqrt{M}\cdot {q_2} , \tfrac{1}{\eps} \mathscr{L} ( \tfrac{f_\eps}{2} ) \big\rangle_{L^2_v} = \tfrac{\eps}{3} j_\eps \cdot E_\eps \,, \\
		\partial_t n_\eps + \div_x \, j_\eps = 0 \,, \\
		\partial_t E_\eps - \nabla_x \times B_\eps = - j_\eps \,, \\
		\partial_t B_\eps + \nabla_x \times E_\eps = 0 \,, \\
		\div_x \, E_\eps = n_\eps \,, \quad \div_x B_\eps = 0 \,.
	\end{array}
	\right.
\end{equation}
where $\mathscr{L}[\widehat{A} (v) \sqrt{M}\cdot {q_2}]=\left(v\otimes v-\frac{|v|^2}{3}I_3\right)\sqrt{M}\cdot {q_2}\in\ker^{\bot}({\mathscr{L}})$ with $\widehat{A} (v) \sqrt{M}\cdot {q_2}\in\ker^{\bot}({\mathscr{L}})\cdot {q_2}$
and $\mathscr{L}[\widehat{B} (v) \sqrt{M}]=\left(v\left(\frac{|v|^2}2-\frac52\right)\right)\sqrt{M}\cdot {q_2}\in\ker^{\bot}({\mathscr{L}})$ with $\widehat{B} (v) \sqrt{M}\cdot {q_2}\in\ker^{\bot}({\mathscr{L}})$.

Based on Theorem \ref{Main-Thm-1} and Theorem \ref{Main-Thm-2}, the Cauchy problem to \eqref{VMB-F-perturbative} admits a global solution
$(f_\epsilon,E_\epsilon,B_\epsilon)$ belonging to $L^\infty(\mathbb{R}_+;H^N_xL^2_v)$, for the noncutoff cases, from \eqref{Main-Thm-1-1}, one has
\begin{eqnarray}\label{limits-2}
	\mathcal{E}_{N,l}(t)+\mathcal{E}_{N}(t)
	\leq\mathcal{E}_{N,l}(0)+\mathcal{E}_{N}(0)\leq C
\end{eqnarray}
and
\begin{eqnarray}\label{limits-3}
	\int^\infty_0\left\{\mathcal{D}_{N}(t)+\mathcal{D}_{N-1,l}(t)\right\}dt\leq C
\end{eqnarray}
where $C$ is independent of $\varepsilon$. In fact, \eqref{limits-3} tells us that
\begin{eqnarray}\label{limits-4}
	\sum_{|\alpha|\leq N}\int^\infty_0\left\|\partial^\alpha\{{\bf I-P}\}f_{\varepsilon}\right\|^2_{D} dt\lesssim C\varepsilon^2.
\end{eqnarray}
which yields that
\begin{equation}
	\{{\bf I-P}\}f_{\varepsilon}\rightarrow 0, \textrm{\ strongly in $L^2(\mathbb{R}_+;H^N_xL^2_D)$ as $\varepsilon\rightarrow0$.}
\end{equation}
As for the cutoff cases, from \eqref{Main-Thm-2-1}, one also deduces that
\begin{equation}
	\{{\bf I-P}\}f_{\varepsilon}\rightarrow 0, \textrm{\ strongly in $L^2(\mathbb{R}_+;H^N_xL^2_\nu)$ as $\varepsilon\rightarrow0$.}
\end{equation}
By standard convergent method, there exist $f,E,B,\rho,u,\theta,n,w,j$ such that
\begin{eqnarray}\label{limit-weak}
	f_\epsilon&\rightarrow& f,\textrm{weakly$-*$ for $t>0$, weakly in $H^N_x L^2_v$},\nonumber\\
	E_\epsilon&\rightarrow& E,\textrm{weakly$-*$ for $t>0$, weakly in $H^N_x$},\nonumber\\
	B_\epsilon&\rightarrow& B,\textrm{weakly$-*$ for $t>0$, weakly in $H^N_x$},\nonumber\\
	\rho_\epsilon&\rightarrow& \rho,\textrm{weakly$-*$ for $t>0$, weakly in $H^N_x$},\nonumber\\
	u_\epsilon&\rightarrow& u,\textrm{weakly$-*$ for $t>0$, weakly in $H^N_x$},\nonumber\\
	\theta_\epsilon&\rightarrow& \theta,\textrm{weakly$-*$ for $t>0$, weakly in $H^N_x$},\nonumber\\
	n_\epsilon&\rightarrow& n,\textrm{weakly$-*$ for $t>0$, weakly in $H^N_x$},\nonumber\\
	w_\epsilon&\rightarrow& w,\textrm{weakly in $L^2(\mathbb{R}^+;H^N_xL^2_v)$},\nonumber\\
	j_\epsilon&\rightarrow& j,\textrm{weakly in $L^2(\mathbb{R}^+;H^N_xL^2_v)$}.
\end{eqnarray}
where
\begin{eqnarray}\label{f-limit}
	f &=&\left(\rho+n/2\right)\frac{q_1+q_2}{2}{M}^{1/2}+\left(\rho-n/2\right)\frac{q_1-q_2}{2}{M}^{1/2}\nonumber\\
	&&+u\cdot v q_2{{M}}^{1/2}+\theta q_2(| v|^2-3){M}^{1/2}
\end{eqnarray}
with
\begin{equation}\label{Fluid-Quanities-non}
	\begin{aligned}
		\rho = \tfrac{1}{2} \langle {f} , {q_2} \sqrt{M} \rangle_{L^2_v} \,, \ u = \tfrac{1}{2} \langle f , {q_2} v \sqrt{M} \rangle_{L^2_v} \,,
		\\ \theta = \tfrac{1}{2} \langle f , {q_2} ( \tfrac{|v|^2}{3} - 1 ) \sqrt{M} \rangle_{L^2_v} \,,
		n = \langle f , {q_1} \sqrt{M} \rangle_{L^2_v} \,.
	\end{aligned}
\end{equation}
In the sense of distributions, utilizing the uniform estimates \eqref{limits-2}, \eqref{limits-3} and \eqref{limits-4}, applying Aubin-Lions-Simon Theorem and the similar argument as \cite{Jiang-Luo-2022-Ann.PDE}, we can deduce that
$$(u, \theta, n, E, B) \in C(\R^+; H^{N-1}_x ) \cap L^\infty (\R^+; H^N_x) $$
satisfy the following two fluid incompressible NSFM equations with Ohm's law
\begin{equation*}
	\left\{
	\begin{array}{l}
		\partial_t u + u \cdot \nabla_x u - \mu \Delta_x u + \nabla_x p = \tfrac{1}{2} ( n E + j \times B ) \,, \qquad \div_x \, u = 0 \,, \\ [2mm]
		\partial_t \theta + u \cdot \nabla_x \theta - \kappa \Delta_x \theta = 0 \,, \qquad\qquad\qquad\qquad\qquad\quad\ \, \rho + \theta = 0 \,, \\ [2mm]
		\partial_t E - \nabla_x \times B = - j \,, \qquad\qquad\qquad\qquad\qquad\qquad\ \ \ \, \div_x \, E = n \,, \\ [2mm]
		\partial_t B + \nabla_x \times E = 0 \,, \qquad\qquad\qquad\qquad\qquad\qquad\qquad \div_x \, B = 0 \,, \\ [2mm]
		\qquad \qquad j - nu = \sigma \big( - \tfrac{1}{2} \nabla_x n + E + u \times B \big) \,, \qquad\quad\,\ w = \tfrac{3}{2} n \theta \,,
	\end{array}
	\right.
\end{equation*}
with initial data
\begin{equation*}
	\begin{aligned}
		u (0,x) = \mathcal{P} u^{in} (x) \,, \ \theta (0,x) = \tfrac{3}{5} \theta^{in} (x) - \tfrac{2}{5} \rho^{in} (x) \,, \ E(0,x) = E^{in} (x) \,, \ B(0,x) = B^{in} (x) \,.
	\end{aligned}
\end{equation*}
We omit its detail proof for brevity.
Moreover, from the uniform bound \eqref{Main-Thm-1-1} in Theorem \ref{Main-Thm-1}, \eqref{Main-Thm-2-1} in Theorem \ref{Main-Thm-2} and the convergence \eqref{limit-weak}, we have
\begin{eqnarray}
	&&\sup_{t \geq 0} \big( \|f \|^2_{H^N_{x}L^2_v} + \| E \|^2_{H^N_x} + \| B \|^2_{H^N_x} \big) (t)\nonumber\\
	& \leq& \sup_{t \geq 0} \big( \| f_\eps \|^2_{H^N_{x}L^2_v} + \| E_\eps \|^2_{H^N_x} + \| B_\eps \|^2_{H^N_x} \big) (t) \nonumber\\
	& \lesssim&Y^2_{f_\varepsilon,E_\varepsilon,B_\varepsilon}(0)\rightarrow Y^2_{f,E,B}(0)
\end{eqnarray}
as $\eps \rightarrow 0$. Hence
\begin{equation*}
	\begin{aligned}
		&&\sup_{t \geq 0} \big( \|f \|^2_{H^N_{x}L^2_v} + \| E \|^2_{H^N_x} + \| B \|^2_{H^N_x} \big) (t) \lesssim Y^2_{f,E,B}(0).
	\end{aligned}
\end{equation*}
Recalling the definition of $f$ in \eqref{f-limit}, there are positive generic constants $C_h$ and $C_l$ such that
\begin{equation*}
	\begin{aligned}
		C_l \big( \| u \|^2_{H^N_x} + \| \theta \|^2_{H^N_x} + \| n \|^2_{H^N_x} \big) \leq \| f \|^2_{H^N_{x}L^2_v} \leq C_h \big( \| u \|^2_{H^N_x} + \| \theta \|^2_{H^N_x} + \| n \|^2_{H^N_x} \big) \,.
	\end{aligned}
\end{equation*}
Consequently, the solution $(u,\theta,n,E,B)$ to the two fluid incompressible NSFM equations \eqref{INSFM-Ohm} with Ohm's law constructed above admits the energy bound
\begin{equation*}
	\begin{aligned}
		\sup_{t \geq 0} \big( & \| u \|^2_{H^N_x} + \| \theta \|^2_{H^N_x} + \| n \|^2_{H^N_x} + \| E \|^2_{H^N_x} + \| B \|^2_{H^N_x} \big) (t) \lesssim Y^2_{f,E,B}(0).
	\end{aligned}
\end{equation*}
Then the proof of Theorem \ref{Main-Thm-3} is complete.

\appendix

\section{Appendix}
\subsection{Properties of the collision operator for non-cutoff cases}
Take \[w_{\ell,\vartheta}=\langle v\rangle^{\ell}e^{\frac{q\langle v\rangle}{(1+t)^{\vartheta}}}.\]
\begin{lemma}\label{L-noncut}
	Let $\ell\in \mathbb{R}$, $\eta>0$, $0<s<1$ and $\max\{-3,-\frac32-2s\}<\gamma<0$.
	\begin{itemize}
		\item [(i).] It holds that
		\begin{equation}\label{L-noncut-1}
			\langle\mathscr{L}g,g\rangle\gtrsim|\{{\bf{I}}-{\bf{P}}\}g|_{L^2_D}^2.
		\end{equation}
		\item [(ii).] It holds that
		\begin{equation}\label{L-noncut-2}
			\left\langle w_{\ell,\vartheta}^{2}\mathscr{L}g, g\right\rangle\gtrsim |w_{\ell,\vartheta} g|_D^2-C|g|^2_{L^2_{B_C}}.
		\end{equation}
		For $|\beta|\geq 1$, one has
		\begin{equation}\label{L-noncut-3}
			\left\langle w_{\ell,\vartheta}^{2}\partial _\beta \mathscr{L}g,\partial_\beta g\right\rangle\gtrsim \left|w_{\ell,\vartheta}\partial_\beta g\right|_{L^2_D}^2-C\sum_{\beta'<\beta}\left|w_{\ell,\vartheta}\partial_{\beta'} g\right|_{L^2_D}^2 -C|g|^2_{L^2_{B_C}}.
		\end{equation}
	\end{itemize}
\end{lemma}
\begin{proof}
	\eqref{L-noncut-1} has been shown in \cite{AMUXY-JFA-2012}.
	The relevant coercive estimate \eqref{L-noncut-2} and \eqref{L-noncut-3} with exponential weights can be found in \cite{DLYZ-KRM2013}.
	
\end{proof}
\begin{lemma}\label{Gamma-noncut}\quad
	For all $0<s<1$, ${q}>0$ and  $\ell\geq 0$,
\begin{itemize}
	\item [i)] It holds that
	\begin{eqnarray}\label{Gamma-noncut-1}
		&&\left|\left\langle \partial_\beta^\alpha\mathscr{T}(f,g), w_{\ell,\vartheta}^2\partial_\beta^\alpha h\right\rangle\right|\nonumber\\
		&\lesssim&\sum\left\{\left|w_{\ell,\vartheta}\partial^{\alpha_1}_{\beta_1}f\right|_{L^2_
			{\frac\gamma2+s}}
		\left|\partial^{\alpha-\alpha_1}_{\beta_2}g\right|_{L^2_D}+\left|\partial^{\alpha-\alpha_1}_{\beta_2}g\right|_{L^2_{\frac\gamma2+s}}
		\left|w_{\ell,\vartheta}\partial^{\alpha_1}_{\beta_1}f\right|_{L^2_D}\right\}
		\left|w_{\ell,\vartheta}\partial^{\alpha}_{\beta}h\right|_{L^2_D}\nonumber\\
		&&+\min\left\{\left|w_{\ell,\vartheta}\partial^{\alpha_1}_{\beta_1}f\right|_{L^2_v}
		\left|\partial^{\alpha-\alpha_1}_{\beta_2}g\right|_{L^2_{\frac\gamma2+s}},\left|\partial^{\alpha-\alpha_1}_{\beta_2}g\right|_{L^2_v}
		\left|w_{\ell,\vartheta}\partial^{\alpha_1}_{\beta_1}f\right|_{L^2_{\frac\gamma2+s}}\right\}
		\left|w_{\ell,\vartheta}\partial^{\alpha}_{\beta}h\right|_{L^2_D}\nonumber\\
		&&+\sum\left|e^{\frac{{q}\langle v\rangle}{(1+t)^{\vartheta}}}\partial^{\alpha_1}_{\beta_2}g\right|_{L^2_v}
		\left|w_{\ell,\vartheta}\partial^{\alpha-\alpha_1}_{\beta_1}f\right|_{L^2_{\frac\gamma2+s}}
		\left|w_{\ell,\vartheta}\partial^{\alpha}_{\beta}h\right|_{L^2_D},
	\end{eqnarray}
	where the summation $\sum$ is taken over $\alpha_1+\alpha_2\leq \alpha$ and $\beta_1+\beta_2\leq\beta$.
	
	Furthermore, we have
	\begin{eqnarray}\label{Gamma-noncut-2}
		\left|\langle\mathscr{T}(f,g), h\rangle\right|&\lesssim& \left\{|f|_{L^2_{\frac\gamma2+s}}|g|_{L^2_D}+|g|_{L^2_{\frac\gamma2+s}}|f|_{L^2_D}\right.\nonumber\\
		&&\left.+\min\left\{|f|_{L^2_v}|g|_{L^2_{\frac\gamma2+s}},|g|_{L^2}|f|_{L^2_{\frac\gamma2+s}}\right\}\right\}|h|_{L^2_D}.
	\end{eqnarray}
\item [ii)]	For $0<s<1,\ m\geq 4,\ \ell>0$, it holds that
\begin{equation}\label{Gamma-noncut-3}
	\left|w_{\ell,\vartheta}\partial^\alpha\mathscr{T}(f,f)\right|_{L^2_v}\lesssim
	\sum_{\alpha_1+\alpha_2\leq\alpha }\left|w_{\ell,\vartheta}\partial^{\alpha_1} f\right|_{H^3_{\frac\gamma2+s}}
	\left|w_{\ell,\vartheta}\partial^{\alpha_2} f\right|_{H^1_{\frac\gamma2+s}},
\end{equation}
or
\begin{equation}\label{Gamma-noncut-4}
	\left|w_{\ell,\vartheta}\partial^\alpha\mathscr{T}(f,f)\right|_{L^2_v}\lesssim
	\sum_{\alpha_1+\alpha_2\leq\alpha }\left|w_{\ell,\vartheta}\partial^{\alpha_1} f\right|_{L^2_{\frac\gamma2+s}}
	\left|w_{\ell,\vartheta}\partial^{\alpha_2} f\right|_{H^4_{\frac\gamma2+s}}.
\end{equation}
\end{itemize}	
\end{lemma}
\begin{proof}
Actually, for the case of weak angular singularity $0<s<\frac12$, \eqref{Gamma-noncut-1} is shown as in \cite{FLLZ-2018}. For the case of strong angular singularity $\frac12\leq s<1 $, we can follow the proof strategy of \cite{FLLZ-2018} and \cite{DLYZ-KRM2013} with slight modifications to obtain \eqref{Gamma-noncut-1}. For the sake of brevity, we omit the detailed proof. \eqref{Gamma-noncut-2} and \eqref{Gamma-noncut-3}-\eqref{Gamma-noncut-4} have been shown in \cite{AMUXY-JFA-2012} and \cite{DLYZ-KRM2013}.
\end{proof}
\subsection{Properties of the collision operator for angular cutoff cases}
Take \[w_{l,\vartheta}=\langle v\rangle^{l}e^{\frac{q\langle v\rangle^2}{(1+t)^{\vartheta}}}.\]
\begin{lemma}\label{L-cut}
	Let $-3<\gamma\leq 1$,
		\begin{itemize}
			\item [i)]One has
			\begin{equation}\label{L-cut-1}
				\langle \mathscr{L}g, g\rangle\geq |{\bf \{I-P\}}g|^2_{L^2_\nu}.
			\end{equation}
		\item [ii)]
		It holds that
		\begin{equation}\label{L-cut-2}
			\left\langle w_{l,\vartheta}^{2}\mathscr{L}g, g\right\rangle \gtrsim |w_{l,\vartheta} g|_{L^2_\nu}^2-C\|g\|^2_{L^2_{B_C}}.
		\end{equation}
		For $|\beta|\geq 1$, one has
		\begin{equation}\label{L-cut-3}
			\left\langle w_{l,\vartheta}^{2}\partial _\beta \mathscr{L}g,\partial_\beta g\right\rangle \gtrsim \left|w_{l,\vartheta}\partial_\beta g\right|_{L^2_\nu}^2-C\sum_{\beta'<\beta}\left|w_{l,\vartheta}\partial_{\beta'} g\right|_{L^2_\nu}^2 -C\|g\|^2_{L^2_{B_C}}.
		\end{equation}
			\end{itemize}

\end{lemma}
\begin{proof}
	The detail proof of this lemma can be found in \cite{Guo-2003-Invent} and \cite{DLYZ-CMP2017}.
\end{proof}

\begin{lemma}\label{Gamma-cut}
	Let $-3<\gamma<0$, $N\geq4$, one has
	\begin{eqnarray}\label{Gamma-cut-0}
		&&\partial^\alpha_\beta\mathscr{T}_\pm(g_1,g_2)\nonumber\\&\equiv&\sum C_\beta^{\beta_0\beta_1\beta_2}C_\alpha^{\alpha_1\alpha_2} \mathscr{T}^0_\pm\left(\partial^{\alpha_1}_{\beta_1}g_1,\partial^{\alpha_2}_{\beta_2}g_2\right)\nonumber\\ \nonumber
		&\equiv& \sum C_\beta^{\beta_0\beta_1\beta_2}C_\alpha^{\alpha_1\alpha_2}\int_{\mathbb{R}^3\times\mathbb{S}^2}|v-u|^\gamma{\bf b}(\cos\theta)\partial_{\beta_0}[\mu(u)^\frac 12]\left\{\partial^{\alpha_1}_{\beta_1}g_{1\pm}(v')\partial^{\alpha_2}_{\beta_2}g_{2\pm}(u')\right.\\
		&&\left.
		+\partial^{\alpha_1}_{\beta_1}g_{1\pm}(v')\partial^{\alpha_2}_{\beta_2}g_{2\mp}(u')
		-\partial^{\alpha_1}_{\beta_1}g_{1\pm}(v)\partial^{\alpha_2}_{\beta_2}g_{2\pm}(u)
		-\partial^{\alpha_1}_{\beta_1}g_{1\pm}(v)\partial^{\alpha_2}_{\beta_2}g_{2\mp}(u)\right\}d\omega du,
	\end{eqnarray}
	where $g_i(t,x,v)=[g_{i+}(t,x,v), g_{i-}(t,x,v)]$ $(i=1,2)$ and the summations are taken for all $\beta_0+\beta_1+\beta_2=\beta, \alpha_1+\alpha_2=\alpha$.
	\begin{itemize}
		\item[(i).] When $|\alpha_1|+|\beta_1|\leq N$, we have
		\begin{equation}\label{Gamma-cut-1}
		\begin{split}
				&\left\langle w_{l,\vartheta}^2\mathscr{T}^0_\pm\left(\partial^{\alpha_1}_{\beta_1}g_1,\partial^{\alpha_2}_{\beta_2}g_2\right), \partial^{\alpha}_{\beta}g_3\right\rangle\\
				\lesssim&\sum\limits_{m\leq2}\left\{
				\left|\nabla^m_{v}\left\{\mu^\delta\partial^{\alpha_1}_{\beta_1}g_1\right\}\right|_{L^2_v}+\left|e^{\frac{q\langle v\rangle^2}{(1+t)^{\vartheta}}}\partial^{\alpha_1}_{\beta_1}g_1\right|_{L^2_v}\right\}
				\left|w_{l,\vartheta}\partial^{\alpha_2}_{\beta_2}g_2\right|_{L^2_{\nu}}
				\left|w_{l,\vartheta}\partial^{\alpha}_{\beta}g_3\right|_{L^2_{\nu}}\\
				&+\sum\limits_{m\leq2}\left\{
				\left|\nabla^m_{v}\left\{\mu^\delta\partial^{\alpha_1}_{\beta_1}g_1\right\}\right|_{L^2_v}+\left|w_{l,\vartheta}\partial^{\alpha_1}_{\beta_1}g_1\right|_{L^2_v}\right\}
				\left|e^{\frac{q\langle v\rangle^2}{(1+t)^{\vartheta}}}\partial^{\alpha_2}_{\beta_2}g_2\right|_{L^2_{\nu}}
				\left|w_{l,\vartheta}\partial^{\alpha}_{\beta}g_3\right|_{L^2_{\nu}},
			\end{split}
		\end{equation}
		or
		\begin{equation}\label{Gamma-cut-2}
	\begin{split}
		&\left\langle w_{l,\vartheta}^2\mathscr{T}^0_\pm\left(\partial^{\alpha_1}_{\beta_1}g_1,\partial^{\alpha_2}_{\beta_2}g_2\right), \partial^{\alpha}_{\beta}g_3\right\rangle\\
		\lesssim&\sum\limits_{m\leq2}\left\{
		\left|\nabla^m_{v}\left\{\mu^\delta\partial^{\alpha_2}_{\beta_2}g_2\right\}\right|_{L^2_v}+\left|e^{\frac{q\langle v\rangle^2}{(1+t)^{\vartheta}}}\partial^{\alpha_2}_{\beta_2}g_2\right|_{L^2_v}\right\}
		\left|w_{l,\vartheta}\partial^{\alpha_1}_{\beta_1}g_1\right|_{L^2_{\nu}}
		\left|w_{l,\vartheta}\partial^{\alpha}_{\beta}g_3\right|_{L^2_{\nu}}\\
		&+\sum\limits_{m\leq2}\left\{
		\left|\nabla^m_{v}\left\{\mu^\delta\partial^{\alpha_2}_{\beta_2}g_2\right\}\right|_{L^2_v}+\left|w_{l,\vartheta}\partial^{\alpha_2}_{\beta_2}g_2\right|_{L^2_v}\right\}
		\left|e^{\frac{q\langle v\rangle^2}{(1+t)^{\vartheta}}}\partial^{\alpha_1}_{\beta_1}g_1\right|_{L^2_{\nu}}
		\left|w_{l,\vartheta}\partial^{\alpha}_{\beta}g_3\right|_{L^2_{\nu}}.
	\end{split}
		\end{equation}
		\item[(ii).]
		Set $\varsigma(v)=\langle v\rangle^{-\gamma}\equiv \nu(v)^{-1},~l\geq0$, it holds that
		\begin{equation}\label{Gamma-cut-3}
			\begin{split}
				\left|\varsigma^{l}\mathscr{T}(g_1,g_2)\right|^2_{L^2_v}
				&\lesssim\displaystyle\sum_{|\beta|\leq2}\left|\varsigma^{l-|\beta|}\partial_{\beta}g_1\right|^2_{L^2_{\nu}}
				\left|\varsigma^{l}g_2\right|^2_{L^2_{\nu}},\\
				\left|\varsigma^{l}\mathscr{T}(g_1,g_2)\right|^2_{L^2_v}
				&\lesssim\sum_{|\beta|\leq2}\left|\varsigma^{l}g_1\right|^2_{L^2_{\nu}}
				\left|\varsigma^{l-|\beta|}\partial_{\beta}g_2\right|^2_{L^2_{\nu}}.
			\end{split}
		\end{equation}
	\end{itemize}
\end{lemma}
\begin{proof}
	The detail proof of this lemma can be found in \cite{DLYZ-CMP2017}.
\end{proof}
\subsection{The proof of Lemma \ref{mac-dissipation}}\label{Macro}
\begin{proof}
Note that the proof here is not only valid for the non-cutoff cases, but also for the grad cutoff cases. For the sake of convenience, we use $\|\cdot\|_{D\vee \nu}$ to denote the dissipative norm $\|\cdot\|_{D}$ under non-cutoff cases or  $\|\cdot\|_{\nu}$ under cutoff cases.

By applying the macro-micro decomposition $(\ref{macro-micro})$ introduced in \cite{Guo-IUMJ-2004} and by defining moment functions $\mathcal{A}_{mj}(f_\varepsilon)$ and $\mathcal{B}_j(f_\varepsilon),~1\leq m,j\leq3,$ by
	\begin{equation*}
		\mathcal{A}_{mj}(f_\varepsilon)=\int_{{\mathbb{R}}^3}\left( v_m v_j-1\right){M}^{1/2}f_\varepsilon d v,\quad \mathcal{B}_j(f_\varepsilon)=\frac{1}{10}\int_{{\mathbb{R}}^3}\left(| v|^2-5\right) v_j{M}^{1/2}f_\varepsilon d v,
	\end{equation*}
 one can then derive from $(\ref{VMB-F-perturbative})$ a fluid-type system of equations
	\begin{equation}\label{Macro-equation}
		\begin{cases}
			\partial_t\rho_\varepsilon^\pm+\frac1\varepsilon\nabla_x\cdot u_\varepsilon+\frac1\varepsilon\nabla_x\cdot
			\left\langle v{M}^{1/2},\{{\bf I_\pm-P_\pm}\}f_\varepsilon\right\rangle=\left\langle {M}^{1/2}, g_\pm\right\rangle,\\[2mm]
			\partial_t\left(u_{\varepsilon,i}+\left\langle v_i{M}^{1/2},\{{\bf I_\pm-P_\pm}\}f_\varepsilon\right\rangle\right)+\frac1\varepsilon\partial_i\left(\rho_\varepsilon^\pm+2\theta_{\varepsilon}\right)\mp \frac1\varepsilon E_{\varepsilon,i}\\[2mm]
			\quad\quad+\frac1\varepsilon\nabla_x\cdot\left\langle vv_i{M}^{1/2}, \{{\bf I_\pm-P_\pm}\}f _{\varepsilon}\right\rangle=\left\langle v_i{M}^{1/2}, g_\pm+\frac1{\varepsilon^2} \mathscr{L}_\pm f_\varepsilon\right\rangle,\\[2mm]
			\partial_t\left(\theta _{\varepsilon}+\frac16\left\langle (|v|^2-3){M}^{1/2},\{{\bf I_\pm-P_\pm}\}f_\varepsilon\right\rangle\right)+\frac1{3\varepsilon}\nabla_x\cdot u_\varepsilon\\[2mm]
			\quad\quad+\frac1{6\varepsilon}\nabla_x\left\langle (|v|^2-3)v{M}^{1/2},\{{\bf I_\pm-P_\pm}\}f_\varepsilon\right\rangle=\left\langle(|v|^2-3){M}^{1/2},g_\pm-\frac1{\varepsilon^2}\mathscr{L}_\pm f_\varepsilon \right\rangle,
		\end{cases}
	\end{equation}
	and
	\begin{equation}\label{Micro-equation}
		\begin{cases}
			\partial_t[\mathcal{A}_{ii}(\{{\bf I_\pm-P_\pm}\}f _{\varepsilon})+2\theta_{\varepsilon}]+2\frac1\varepsilon\partial_iu_{{\varepsilon},i}
			=\mathcal{A}_{ii}(r_\pm+g_\pm),\\[2mm]
			\partial_t\mathcal{A}_{ij}(\{{\bf I_\pm-P_\pm}\}f _{\varepsilon})+\frac1\varepsilon\partial_ju_{{\varepsilon},i}+\frac1\varepsilon\partial_iu_{{\varepsilon},j}
			+\frac1\varepsilon\nabla_x\cdot \langle v{M}^{1/2},\{{\bf I_\pm-P_\pm}\}f_\varepsilon\rangle
			=\mathcal{A}_{ij}(r_\pm+g_\pm),\\[2mm]
			\partial_t \mathcal{B}_{j}(\{{\bf I_\pm-P_\pm}\}f _{\varepsilon})+\frac1\varepsilon\partial_j\theta_{\varepsilon}=\mathcal{B}_j(r_\pm+g_\pm),
		\end{cases}
	\end{equation}
	where
	\begin{eqnarray}\label{r-G}
		r_\pm&=&- \frac1\varepsilon v\cdot\nabla_x\{{\bf I_\pm-P_\pm}\}f_{\varepsilon}-\frac1{\varepsilon^2}{\mathscr{L}_\pm}f_{\varepsilon},\nonumber\\
		\quad g_\pm&=&\frac12 v\cdot E_{\varepsilon} f_{{\varepsilon},\pm}\mp (E_{\varepsilon}+\frac1\varepsilon v\times B_{\varepsilon})\cdot\nabla_{ v}f_{{\varepsilon},\pm}+\frac1\varepsilon\mathscr{T}_\pm(f_{\varepsilon},f_{\varepsilon}).
	\end{eqnarray}
	Setting
	$$
	G\equiv\left\langle v{M}^{1/2},\{{\bf I-P}\}f_{\varepsilon} \cdot q_1 \right\rangle,
	$$ we can get from (\ref{Macro-equation})-(\ref{Micro-equation}) that
	\begin{equation}\label{Macro-equation1}
		\begin{cases}
			\partial_t\left(\frac{\rho_\varepsilon^++\rho_\varepsilon^-}2\right)+\frac1\varepsilon\nabla_x\cdot u_\varepsilon=0\\[2mm]
			\partial_tu_{\varepsilon,i}+\frac1\varepsilon\partial_i\left(\frac{\rho_\varepsilon^++\rho_\varepsilon^-}2+2\theta_{\varepsilon}\right)+\frac1{2\varepsilon}\sum\limits_{j=1}^3\partial_j\mathcal{A}_{ij}(\{{\bf I-P}\}f_\varepsilon\cdot [1,1])=\frac{\rho_\varepsilon^+-\rho_\varepsilon^-}{2}E_i+\frac1\varepsilon[G\times B]_i,\\[2mm]
			\partial_t\theta _{\varepsilon}+\frac1{3\varepsilon}\nabla_x\cdot u_\varepsilon+\frac5{6\varepsilon}\sum\limits_{i=1}^3\partial_i \mathcal{B}_i(\{{\bf I-P}\}f_\varepsilon\cdot [1,1])=\frac16 G\cdot E_\varepsilon,
		\end{cases}
	\end{equation}
	and
	\begin{equation}\label{Micro-equation1}
		\begin{cases}
			\partial_t[\frac12\mathcal{A}_{ij}(\{{\bf I-P}\}f_\varepsilon\cdot [1,1])+2\theta_{\varepsilon}\delta_{ij}]+\frac1\varepsilon\partial_ju_{\varepsilon,i}+\frac1\varepsilon\partial_iu_{\varepsilon,j}
			=\frac12\mathcal{A}_{ij}(r_++r_-+g_++g_-),\\[2mm]
			\frac12\partial_t \mathcal{B}_{j}(\{{\bf I-P}\}f_\varepsilon\cdot [1,1])+\frac1\varepsilon\partial_j\theta_\varepsilon=\frac12\mathcal{B}_{j}(r_++r_-+g_++g_-).
		\end{cases}
	\end{equation}
	Moreover, by using the third equation of (\ref{Macro-equation1}) to replace $\partial_t\theta$ in the first equation of (\ref{Micro-equation1}), one has
	\begin{equation}\label{Micro-equation2}
		\begin{split}
			&\frac12\partial_t\mathcal{A}_{ij}(\{{\bf I-P}\}f_\varepsilon\cdot [1,1])+\frac1\varepsilon\partial_ju_{\varepsilon,i}+\frac1\varepsilon\partial_iu_{\varepsilon,j}-\frac2{3\varepsilon}\delta_{ij}\nabla_x\cdot u_\varepsilon\\[2mm]
			&-\frac5{3 \varepsilon}\delta_{ij}\nabla_x\cdot \mathcal{B}(\{{\bf I-P}\}f_\varepsilon\cdot [1,1])=\frac12\mathcal{A}_{ij}(r_++r_-+g_++g_-)-\frac13\delta_{ij}G\cdot E_\varepsilon.
		\end{split}
	\end{equation}
	In order to further obtain the dissipation rate related to $\rho^\pm_\varepsilon$ from the formula
	\[
	|\rho_\varepsilon^+|^2+|\rho_\varepsilon^-|^2=\frac{|\rho_\varepsilon^++\rho_\varepsilon^-|^2}{2}+\frac{|\rho_\varepsilon^+-\rho_\varepsilon^-|^2}{2},
	\]
	we need to consider the dissipation of $\rho_\varepsilon^+-\rho_\varepsilon^-$. For that purpose, one can get from $(\ref{Macro-equation1})_1$ and $(\ref{Macro-equation1})_2$ that
	\begin{equation}\label{a_+-a_--original}
		\begin{cases}
			\partial_t(\rho_\varepsilon^+-\rho_\varepsilon^-)+\frac1\varepsilon\nabla_x\cdot G=0,\\[2mm]
			\partial_tG+\frac1\varepsilon\nabla_x(\rho_\varepsilon^+-\rho_\varepsilon^-)-\frac2\varepsilon E_\varepsilon+\frac1\varepsilon\nabla_x\cdot \mathcal{A}(\{{\bf I-P}\}f_\varepsilon\cdot q_1)\\[2mm]
			\qquad \qquad=E(\rho_\varepsilon^++\rho_\varepsilon^-)+\frac2\varepsilon u_\varepsilon\times B_\varepsilon+\left\langle [v,-v]{M}^{1/2},\frac1{\varepsilon^2}\mathscr{L}f_\varepsilon+\frac1\varepsilon\mathscr{T}(f_\varepsilon,f_\varepsilon)\right\rangle.
		\end{cases}
	\end{equation}
	
	Applying $\partial^\alpha$ to $(\ref{Micro-equation1})_2  $, and multiplying to the identity with $\varepsilon\partial^\alpha\partial_j \theta $, and integrating the identity result over $\mathbb{R}^3_x$, one has
	\begin{eqnarray}
		&&\left\|\partial^\alpha\nabla_x \theta_\varepsilon\right\|^2=\sum_{j=1}^3\left(\frac1\varepsilon\partial^\alpha\partial_j\theta_\varepsilon, \varepsilon\partial^\alpha\partial_j\theta_\varepsilon\right)\nonumber\\
		&\lesssim&-\sum_{j=1}^3\left(\frac12\partial_t \partial^\alpha \mathcal{B}_{j}(\{{\bf I-P}\}f_\varepsilon\cdot [1,1]),\varepsilon\partial^\alpha\partial_j\theta_\varepsilon\right)+\sum_{j=1}^3\left(\frac12\partial^\alpha \mathcal{B}_{j}(r_++r_-+g_++g_-),\varepsilon\partial^\alpha\partial_j\theta_\varepsilon\right)\nonumber\\
		&=&-\frac{\d}{\d t}\sum_{j=1}^3\left(\frac12 \partial^\alpha \mathcal{B}_{j}(\{{\bf I-P}\}f_\varepsilon\cdot [1,1]),\varepsilon\partial^\alpha\partial_j\theta_\varepsilon\right)+\sum_{j=1}^3\left(\frac12 \partial^\alpha\mathcal{ B}_{j}(\{{\bf I-P}\}f_\varepsilon\cdot [1,1]),\varepsilon\partial^\alpha\partial_j\partial_t\theta_\varepsilon\right)\nonumber\\
		&&+\sum_{j=1}^3\left(\frac12\partial^\alpha \mathcal{B}_{j}(r_++r_-+g_++g_-),\varepsilon\partial^\alpha\partial_j\theta_\varepsilon\right).\label{4.18}
	\end{eqnarray}
	
	For the second term on the right hand of the second equality in \eqref{4.18}, we have from $(\ref{Macro-equation1})_3$ that
	\begin{equation*}
		\begin{aligned}
			&\sum_{j=1}^3\left(\frac12 \partial^\alpha \mathcal{B}_{j}(\{{\bf I-P}\}f_\varepsilon\cdot [1,1]),\varepsilon\partial^\alpha\partial_j\partial_t\theta_\varepsilon\right)\\
			=&\bigg(\frac12 \partial^\alpha \mathcal{B}_{j}(\{{\bf I-P}\}f_\varepsilon\cdot [1,1]),\varepsilon\partial^\alpha\partial_j\bigg\{-\frac1{3\varepsilon}\nabla_x\cdot u_\varepsilon-\frac5{6\varepsilon}\sum\limits_{i=1}^3\partial_i \mathcal{B}_{\varepsilon,i}(\{{\bf I-P}\}f_\varepsilon\cdot [1,1])+\frac16 G\cdot E_\varepsilon\bigg\}\bigg)\\
			\lesssim&\eta\left\|\partial^\alpha\nabla_x u_\varepsilon\right\|^2+\left\|\partial^{\alpha}\{{\bf I-P}\}f_\varepsilon\right\|_{H^1_xL^2_{D\vee\nu}}+\mathcal{E}_{N}(t)\mathcal{D}_{N}(t),
		\end{aligned}
	\end{equation*}
	while for the last term on the right hand of (\ref{4.18}), we can deduce that
	\begin{equation*}
		\begin{aligned}
			&\sum_{j=1}^3\left(\frac12\partial^\alpha \mathcal{B}_{j}(r_++r_-+g_++g_-),\varepsilon\partial^\alpha\partial_j\theta_\varepsilon\right)\\[2mm]
			\lesssim&\eta\left\|\partial^\alpha\nabla_x u_\varepsilon\right\|^2+\left\|\partial^{\alpha}\{{\bf I-P}\}f_\varepsilon\right\|_{H^1_xL^2_{D\vee\nu}}+\mathcal{E}_{N}(t)\mathcal{D}_{N}(t).
		\end{aligned}
	\end{equation*}
	Thus we can get that
	\begin{equation}\label{G-c}
		\begin{aligned}
			\frac{\d}{\d t}G^{f_\varepsilon}_{\theta_\varepsilon}(t)+\left\|\partial^\alpha\nabla_x \theta_\varepsilon\right\|^2\lesssim\eta\left\|\partial^\alpha\nabla_x u_\varepsilon\right\|^2+\left\|\partial^{\alpha}\{{\bf I-P}\}f_\varepsilon\right\|_{H^1_xL^2_{D\vee\nu}}+\mathcal{E}_{N}(t)\mathcal{D}_{N}(t).
		\end{aligned}
	\end{equation}
	Here
	\[
	G^{f_\varepsilon}_{\theta_\varepsilon}(t)\equiv\sum_{j=1}^3\left(\frac12 \partial^\alpha \mathcal{B}_{j}(\{{\bf I-P}\}f_\varepsilon\cdot [1,1]),\varepsilon\partial^\alpha\partial_j\theta_\varepsilon\right).
	\]
	On the other hand, by using \eqref{Micro-equation2}, one has
	\begin{eqnarray}
		&&-\left\{\frac12\partial_i\partial_t\mathcal{A}_{ii}(\{{\bf I-P}\}f_\varepsilon\cdot [1,1])+\sum_j\partial_j\partial_t\mathcal{A}_{ji}(\{{\bf I-P}\}f_\varepsilon\cdot [1,1])\right\}\nonumber\\
		&&-\frac2\varepsilon\Delta u_{\varepsilon,i}-\frac2\varepsilon\partial_i\partial_i u_{\varepsilon,i}+\frac5{ \varepsilon}\partial_i\nabla_x\cdot \mathcal{B}(\{{\bf I-P}\}f_\varepsilon\cdot [1,1])\nonumber\\
		&=&-\frac12\partial_i\mathcal{A}_{ii}(r_++r_-+g_++g_-)-\sum_j\partial_j\mathcal{A}_{ji}(r_++r_-+g_++g_-)+\partial_i\left[G\cdot E_\varepsilon\right].
	\end{eqnarray}
	Applying $\partial^\alpha$ to the above equality, multiplying it with $\epsilon\partial^\alpha u_{\varepsilon,i}$, and integrating the identity result over $\mathbb{R}^3_x$, then one has
	\begin{eqnarray*}
		&&2\|\partial^\alpha \nabla_xu_\varepsilon\|^2+2\sum_i\|\partial^\alpha \partial_iu_{\varepsilon,i}\|^2\nonumber\\
		&=&\sum_i\left(-\frac2\varepsilon\partial^\alpha \Delta u_{\varepsilon,i}, \varepsilon\partial^\alpha u_{\varepsilon,i}\right)+\sum_i
		\left(-\frac2\varepsilon\partial^\alpha\partial_i\partial_i u_{\varepsilon,i}, \varepsilon\partial^\alpha u_{\varepsilon,i}\right)\nonumber\\
		&=&\sum_i\left(-\frac2\varepsilon\partial^\alpha \Delta u_{\varepsilon,i}-\frac2\varepsilon\partial^\alpha\partial_i\partial_i u_{\varepsilon,i}, \varepsilon\partial^\alpha u_{\varepsilon,i}\right)\nonumber\\
		&=&\sum_i\left(\frac12\partial^\alpha\partial_i\partial_t\mathcal{A}_{ii}(\{{\bf I-P}\}f_\varepsilon\cdot [1,1])+\sum_j\partial^\alpha\partial_j\partial_t\mathcal{A}_{ji}(\{{\bf I-P}\}f_\varepsilon\cdot [1,1]), \varepsilon\partial^\alpha u_{\varepsilon,i}\right)\nonumber\\
		&&+\sum_i\left(-\frac12\partial^\alpha\partial_i\mathcal{A}_{ii}(r_++r_-+g_++g_-)
		-\sum_j\partial^\alpha\partial_j\mathcal{A}_{ji}(r_++r_-+g_++g_-), \varepsilon\partial^\alpha u_{\varepsilon,i}\right)\nonumber\\
		&&+\sum_i\left(\partial^\alpha\partial_i\left[G\cdot E_\varepsilon\right]-\frac5{ \varepsilon}\partial^\alpha\partial_i\nabla_x\cdot \mathcal{B}(\{{\bf I-P}\}f_\varepsilon\cdot [1,1]), \varepsilon\partial^\alpha u_{\varepsilon,i}\right)\nonumber\\
		&=&\frac{\d}{\d t}\sum_i\left(\frac12\partial^\alpha\partial_i\mathcal{A}_{ii}(\{{\bf I-P}\}f_\varepsilon\cdot [1,1])+\sum_j\partial^\alpha\partial_j\mathcal{A}_{ji}(\{{\bf I-P}\}f_\varepsilon\cdot [1,1]), \varepsilon\partial^\alpha u_{\varepsilon,i}\right)\nonumber\\
		&&-\sum_i\left(\frac12\partial^\alpha\partial_i\mathcal{A}_{ii}(\{{\bf I-P}\}f_\varepsilon\cdot [1,1])+\sum_j\partial^\alpha\partial_j\mathcal{A}_{ji}(\{{\bf I-P}\}f_\varepsilon\cdot [1,1]), \varepsilon\partial^\alpha \partial_tu_{\varepsilon,i}\right)\nonumber\\
		&&+\sum_i\left(-\frac12\partial^\alpha\partial_i\mathcal{A}_{ii}(r_++r_-+g_++g_-)
		-\sum_j\partial^\alpha\partial_j\mathcal{A}_{ji}(r_++r_-+g_++g_-), \varepsilon\partial^\alpha u_{\varepsilon,i}\right)\nonumber\\
		&&+\sum_i\left(\partial^\alpha\partial_i\left[G\cdot E_\varepsilon\right]-\frac5{ \varepsilon}\partial^\alpha\partial_i\nabla_x\cdot \mathcal{B}(\{{\bf I-P}\}f_\varepsilon\cdot [1,1]), \varepsilon\partial^\alpha u_{\varepsilon,i}\right)\nonumber\\
		&=&\frac{\d}{\d t}\sum_i\left(\frac12\partial^\alpha\partial_i\mathcal{A}_{ii}(\{{\bf I-P}\}f_\varepsilon\cdot [1,1])+\sum_j\partial^\alpha\partial_j\mathcal{A}_{ji}(\{{\bf I-P}\}f_\varepsilon\cdot [1,1]), \varepsilon\partial^\alpha u_{\varepsilon,i}\right)\nonumber\\
		&&+\sum_i\left(\frac12\partial^\alpha\partial_i\mathcal{A}_{ii}(\{{\bf I-P}\}f_\varepsilon\cdot [1,1])+\sum_j\partial^\alpha\partial_j\mathcal{A}_{ji}(\{{\bf I-P}\}f_\varepsilon\cdot [1,1]), \partial^\alpha \left\{\partial_i\big(\frac{\rho_\varepsilon^++\rho_\varepsilon^-}2+2\theta_\varepsilon\big)\right\}\right)\nonumber\\
		&&-\sum_i\left(\frac12\partial^\alpha\partial_i\mathcal{A}_{ii}(\{{\bf I-P}\}f_\varepsilon\cdot [1,1])+\sum_j\partial^\alpha\partial_j\mathcal{A}_{ji}(\{{\bf I-P}\}f_\varepsilon\cdot [1,1]), \varepsilon\partial^\alpha \left\{\frac{\rho_\varepsilon^+-\rho_\varepsilon^-}{2}E_{\varepsilon,i}+\frac1\varepsilon[G\times B_\varepsilon]_i\right\}\right)\nonumber\\
		&&+\sum_i\left(\frac12\partial^\alpha\partial_i\mathcal{A}_{ii}(\{{\bf I-P}\}f_\varepsilon\cdot [1,1])+\sum_j\partial^\alpha\partial_j\mathcal{A}_{ji}(\{{\bf I-P}\}f_\varepsilon\cdot [1,1]), \frac1{2}\partial^\alpha \partial_i\big(\sum\limits_{j=1}^3\partial_j\mathcal{A}_{ij}(\{{\bf I-P}\}f_\varepsilon\cdot [1,1])\right)\nonumber\\
		&&+\sum_i\left(-\frac12\partial^\alpha\partial_i\mathcal{A}_{ii}(r_++r_-+g_++g_-)
		-\sum_j\partial^\alpha\partial_j\mathcal{A}_{ji}(r_++r_-+g_++g_-), \varepsilon\partial^\alpha u_{\varepsilon,i}\right)\nonumber\\
		&&+\sum_i\left(\partial^\alpha\partial_i\left[G\cdot E_\varepsilon\right]-\frac5{ \varepsilon}\partial^\alpha\partial_i\nabla_x\cdot \mathcal{B}(\{{\bf I-P}\}f_\varepsilon\cdot [1,1]), \varepsilon\partial^\alpha u_{\varepsilon,i}\right)\nonumber
	\end{eqnarray*}
	The second term on the right-hand side of the above equality can be bounded by
	\begin{eqnarray}
		&&\sum_i\left(\frac12\partial^\alpha\partial_i\mathcal{A}_{ii}(\{{\bf I-P}\}f_\varepsilon\cdot [1,1])+\sum_j\partial^\alpha\partial_j\mathcal{A}_{ji}(\{{\bf I-P}\}f_\varepsilon\cdot [1,1]), \partial^\alpha \left\{\partial_i\big(\frac{\rho_\varepsilon^++\rho_\varepsilon^-}2+2\theta_{\varepsilon}\big)\right\}\right)\nonumber\\
		&\lesssim&\eta\|\partial^{\alpha}\nabla_x\{\rho_\varepsilon^++\rho_\varepsilon^-\}\|^2+\eta\|\partial^{\alpha}\nabla_x\theta_{\varepsilon}\|^2
		+\left\|\partial^{\alpha}\nabla_x\{{\bf I-P}\}f_\varepsilon\right\|_{D\vee\nu}^2.
	\end{eqnarray}
	The other terms can be bounded by
	\begin{eqnarray}
		&\lesssim&\eta\|\partial^\alpha\nabla_x u_\varepsilon\|^2+\left\|\partial^{\alpha}\nabla_x\{{\bf I-P}\}f_\varepsilon\right\|_{D\vee\nu}^2+\mathcal{E}_N(t)\mathcal{D}_N(t).
	\end{eqnarray}
	Consequently, one has
	\begin{eqnarray}\label{G-b}
		&&\frac{\d}{\d t}G^{f_\varepsilon}_{u_\varepsilon}(t)+\|\partial^\alpha \nabla_xu_\varepsilon\|^2+\sum_i\|\partial^\alpha \partial_iu_{\varepsilon,i}\|^2\nonumber\\
		&\lesssim&\eta\|\partial^{\alpha}\nabla_x\{\rho_\varepsilon^++\rho_\varepsilon^-\}\|^2+\eta\|\partial^{\alpha}\nabla_x\theta_\varepsilon\|^2+\left\|\partial^{\alpha}\nabla_x\{{\bf I-P}\}f_\varepsilon\right\|_{D\vee\nu}^2+\mathcal{E}_N(t)\mathcal{D}_N(t),
	\end{eqnarray}
	where $G^{f_\varepsilon}_{u_\varepsilon}(t)$ denotes that
	\[G^{f_\varepsilon}_{u_\varepsilon}(t)\equiv\sum_i\left(\frac12\partial^\alpha\partial_i\mathcal{A}_{ii}(\{{\bf I-P}\}f_\varepsilon\cdot [1,1])+\sum_j\partial^\alpha\partial_j\mathcal{A}_{ji}(\{{\bf I-P}\}f_\varepsilon\cdot [1,1]), \varepsilon\partial^\alpha u_{\varepsilon,i}\right).\]
	Next, we estimate $\rho_\varepsilon^++\rho_\varepsilon^-$. To this end, we have from $(\ref{Macro-equation1})_2$ and by employing the same argument to deduce $(\ref{G-c})$ that
	\begin{equation}\label{G-a}
		\begin{split}
&	\frac{\d}{\d t}	G^{f_\varepsilon}_{\rho_\varepsilon}(t)+\|\partial^\alpha\nabla_x(\rho_\varepsilon^++\rho_\varepsilon^-)\|^2\\
\lesssim&\|\partial^\alpha\nabla_xu_\varepsilon\|^2
+\|\partial^\alpha\nabla_x\theta_\varepsilon\|^2
+\left\|\partial^{\alpha}\nabla_x\{{\bf I-P}\}f_\varepsilon\right\|_{D\vee\nu}^2+\mathcal{E}_N(t)\mathcal{D}_N(t)
		\end{split}	
	\end{equation}
	with
	$$
	G^{f_\varepsilon}_{\rho_\varepsilon}(t)\equiv\sum_{i}\left(\partial^\alpha u_{\varepsilon, i},\partial^\alpha\partial_i(\rho_\varepsilon^++\rho_\varepsilon^-)\right).
	$$
	
	Set
	\begin{equation*}
		 G^{f_\varepsilon}_{\rho_\varepsilon^++\rho_\varepsilon^-,u_\varepsilon,\theta_\varepsilon}(t)=G^{f_\varepsilon}_{u_\varepsilon}(t)+\kappa_1G^{f_\varepsilon}_{\theta_\varepsilon}(t)+\kappa_2G^{f_\varepsilon}_{\rho_\varepsilon}(t), , 0<\kappa_2\ll\kappa_1\ll1.
	\end{equation*}
	we can deduce from (\ref{G-c}), (\ref{G-b}), and (\ref{G-a}) that
	\begin{equation}\label{g-f-}
		\begin{aligned}
			\frac{\d}{\d t}G^f_{\rho_\varepsilon^++\rho_\varepsilon^-,u_\varepsilon,\theta_\varepsilon}(t)+\left\|\partial^{\alpha}\nabla_x[\rho_\varepsilon^++\rho_\varepsilon^-,u_\varepsilon,\theta_\varepsilon]\right\|^2
			\lesssim\left\|\{{\bf I-P}\}f_\varepsilon\right\|^2_{H^N_xL^2_{{D\vee\nu}}}+\mathcal{E}_{N}(t)\mathcal{D}_{N}(t).
		\end{aligned}
	\end{equation}
	Finally, for the corresponding estimate on $\rho_\varepsilon^+-\rho_\varepsilon^-$, we have from $\eqref{a_+-a_--original}_2$ that
	\begin{eqnarray}
		&&\left\|\partial^{\alpha}\nabla_x(\rho_\varepsilon^+-\rho_\varepsilon^-)\right\|^2+2\left\|\partial^{\alpha}(\rho_\varepsilon^+-\rho_\varepsilon^-)\right\|^2
		=\left(\partial^{\alpha}\nabla_x(\rho_\varepsilon^+-\rho_\varepsilon^-)-2\partial^\alpha E_\varepsilon,\partial^{\alpha}\nabla_x(\rho_\varepsilon^+-\rho_\varepsilon^-)\right)\nonumber\\
		&=&\left(\partial^{\alpha}\left\{\frac1\epsilon\nabla_x(\rho_\varepsilon^+-\rho_\varepsilon^-)-\frac2\epsilon E_\varepsilon\right\},\varepsilon\partial^{\alpha}\nabla_x(\rho_\varepsilon^+-\rho_\varepsilon^-)\right)\nonumber\\
		&=&\left(\partial^{\alpha}\left\{-\partial_tG-\frac1\varepsilon\nabla_x\cdot A(\{{\bf I-P}\}f_\varepsilon\cdot q_1)\right\},\varepsilon\partial^{\alpha}\nabla_x(\rho_\varepsilon^+-\rho_\varepsilon^-)\right)\nonumber\\
		&&+\left(\partial^{\alpha}\left\{E_\varepsilon(\rho_\varepsilon^++\rho_\varepsilon^-)+\frac2\varepsilon u_\varepsilon\times B_\varepsilon+\left\langle [v,-v]{M}^{1/2},\frac1{\varepsilon^2}\mathscr{L}f_\varepsilon+\frac1\varepsilon\mathscr{T}(f_\varepsilon,f_\varepsilon)\right\rangle\right\},\varepsilon\partial^{\alpha}\nabla_x(\rho_\varepsilon^+-\rho_\varepsilon^-)\right)\nonumber\\
		&=&-\frac{\d}{\d t}\left(\partial^{\alpha}G,\varepsilon\partial^{\alpha}\nabla_x(\rho_\varepsilon^+-\rho_\varepsilon^-)\right)
		+\left(\partial^{\alpha}G,\varepsilon\partial^{\alpha}\nabla_x\partial_t(\rho_\varepsilon^+-\rho_\varepsilon^-)\right)\nonumber\\
		&&+\left(\partial^{\alpha}\left\{-\frac1\varepsilon\nabla_x\cdot \mathcal{A}(\{{\bf I-P}\}f_\varepsilon\cdot q_1)\right\},\varepsilon\partial^{\alpha}\nabla_x(\rho_\varepsilon^+-\rho_\varepsilon^-)\right)\nonumber\\
		&&+\left(\partial^{\alpha}\left\{E_\varepsilon(\rho_\varepsilon^++\rho_\varepsilon^-)+\frac2\varepsilon u_\varepsilon\times B_\varepsilon+\left\langle [v,-v]{M}^{1/2},\frac1{\varepsilon^2}\mathscr{L}f_\varepsilon+\frac1\varepsilon\mathscr{T}(f_\varepsilon,f_\varepsilon)\right\rangle\right\},\varepsilon\partial^{\alpha}\nabla_x(\rho_\varepsilon^+-\rho_\varepsilon^-)\right)\nonumber\\
		&=&-\frac{\d}{\d t}\left(\partial^{\alpha}G,\varepsilon\partial^{\alpha}\nabla_x(\rho_\varepsilon^+-\rho_\varepsilon^-)\right)
		-\left(\partial^{\alpha}G,\varepsilon\partial^{\alpha}\nabla_x
		\frac1\varepsilon\nabla_x\cdot G\right)\nonumber\\
		&&+\left(\partial^{\alpha}\left\{-\frac1\varepsilon\nabla_x\cdot \mathcal{A}(\{{\bf I-P}\}f_\varepsilon\cdot q_1)\right\},\varepsilon\partial^{\alpha}\nabla_x(\rho_\varepsilon^+-\rho_\varepsilon^-)\right)\nonumber\\
		&&+\left(\partial^{\alpha}\left\{E_\varepsilon(\rho_\varepsilon^++\rho_\varepsilon^-)+\frac2\varepsilon u_\varepsilon\times B_\varepsilon+\left\langle [v,-v]{M}^{1/2},\frac1{\varepsilon^2}\mathscr{L}f_\varepsilon+\frac1\varepsilon\mathscr{T}(f_\varepsilon,f_\varepsilon)\right\rangle\right\},\varepsilon\partial^{\alpha}\nabla_x(\rho_\varepsilon^+-\rho_\varepsilon^-)\right)\nonumber\\
		&=&-\frac{\d}{\d t}\left(\partial^{\alpha}G,\varepsilon\partial^{\alpha}\nabla_x(\rho_\varepsilon^+-\rho_\varepsilon^-)\right)+
		\left(\partial^{\alpha}\nabla_x\cdot G,\partial^{\alpha}
		\nabla_x\cdot G\right)\nonumber\\
		&&+\left(\partial^{\alpha}\left\{-\frac1\varepsilon\nabla_x\cdot \mathcal{A}(\{{\bf I-P}\}f_\varepsilon\cdot q_1)\right\},\varepsilon\partial^{\alpha}\nabla_x(\rho_\varepsilon^+-\rho_\varepsilon^-)\right)\nonumber\\
		&&+\left(\partial^{\alpha}\left\{E_\varepsilon(\rho_\varepsilon^++\rho_\varepsilon^-)+\frac2\varepsilon u _\varepsilon\times B _\varepsilon+\left\langle [v,-v]{M}^{1/2},\frac1{\varepsilon^2}\mathscr{L}f_\varepsilon+\frac1\varepsilon\mathscr{T}(f_\varepsilon,f_\varepsilon)\right\rangle\right\},\varepsilon\partial^{\alpha}\nabla_x(\rho_\varepsilon^+-\rho_\varepsilon^-)\right)\nonumber\\
	\end{eqnarray}
	Consequently
	\begin{eqnarray}\label{g-a}
		&&\frac{\d}{\d t}\left(\partial^{\alpha}G,\varepsilon\partial^{\alpha}\nabla_x(\rho_\varepsilon^+-\rho_\varepsilon^-)\right)
		+\left\|\partial^{\alpha}\nabla_x(\rho_\varepsilon^+-\rho_\varepsilon^-)\right\|^2+2\left\|\partial^{\alpha}(\rho_\varepsilon^+-\rho_\varepsilon^-)\right\|^2\nonumber\\
		&\lesssim&\left\|\partial^{\alpha}\nabla_x\{{\bf I-P}\}f_\varepsilon\right\|_{{D\vee\nu}}^2+\mathcal{E}_N(t)\mathcal{D}_N(t).
	\end{eqnarray}

	A suitable linear combination of (\ref{g-f-}) and (\ref{g-a}) gives
	\begin{equation}\label{mac-alpha}
		\begin{aligned}
			&\frac{\d}{\d t}G^{f}_{\rho_\varepsilon^+\pm\rho_\varepsilon^-,u_\varepsilon,\theta_\varepsilon}(t)+
			2\left\|\partial^{\alpha}(\rho_\varepsilon^+-\rho_\varepsilon^-)\right\|^2+\left\|\partial^{\alpha}\nabla_x[\rho_\varepsilon^+\pm\rho_\varepsilon^-,u_\varepsilon,\theta_\varepsilon]\right\|^2\\
			\lesssim&\left\|\{{\bf I-P}\}f_\varepsilon\right\|^2_{H^N_xL^2_{D\vee\nu}}+\mathcal{E}_{N}(t)\mathcal{D}_{N}(t)
		\end{aligned}
	\end{equation}
	where
	\[
	G^{f}_{\rho_\varepsilon^+\pm\rho_\varepsilon^-,u_\varepsilon,\theta_\varepsilon}(t)\equiv G^f_{\rho_\varepsilon^++\rho_\varepsilon^-,u_\varepsilon,\theta_\varepsilon}(t)+\left(\partial^{\alpha}G,\varepsilon\partial^{\alpha}\nabla_x(\rho_\varepsilon^+-\rho_\varepsilon^-)\right),
	\]
	where $|\alpha|\leq N-1$.
	
	Since
	\begin{eqnarray}
		\frac2\varepsilon E_\varepsilon&=&\partial_tG+\frac1\varepsilon\nabla_x(\rho_\varepsilon^+-\rho_\varepsilon^-)+\frac1\varepsilon\nabla_x\cdot \mathcal{A}(\{{\bf I-P}\}f_\varepsilon\cdot q_1)-E(\rho_\varepsilon^++\rho_\varepsilon^-)\nonumber\\
		&&-\frac2\varepsilon u _\varepsilon\times B _\varepsilon-\left\langle [v,-v]{M}^{1/2},\frac1{\varepsilon^2}\mathscr{L}f-\frac1\varepsilon\mathscr{T}(f,f)\right\rangle.
	\end{eqnarray}
	Applying $\partial^\alpha$ to the above equality and multiplying it with $\varepsilon \partial^\alpha E$, one has
	\begin{eqnarray}
		&& 2\|\partial^\alpha E_\varepsilon\|^2=\left(\partial^\alpha\left\{ \frac2\varepsilon E_\varepsilon\right\},\varepsilon\partial^\alpha E_\varepsilon\right) \nonumber\\
		&=&\left(\partial^\alpha\left\{ \partial_tG+\frac1\varepsilon\nabla_x(\rho_\varepsilon^+-\rho_\varepsilon^-)+\frac1\varepsilon\nabla_x\cdot \mathcal{A}(\{{\bf I-P}\}f_\varepsilon\cdot q_1)-E_\varepsilon(\rho_\varepsilon^++\rho_\varepsilon^-)\right\},\varepsilon\partial^\alpha E_\varepsilon\right) \nonumber\\
		&&+\left(\partial^\alpha\left\{ -\frac2\varepsilon u _\varepsilon\times B _\varepsilon-\left\langle [v,-v]{M}^{1/2},\frac1{\varepsilon^2}\mathscr{L}f_\varepsilon-\frac1\varepsilon\mathscr{T}(f_\varepsilon,f_\varepsilon)\right\rangle\right\},\varepsilon\partial^\alpha E_\varepsilon\right) \nonumber\\
		&=&\frac{\d}{\d t}\left(\partial^\alpha G,\varepsilon\partial^\alpha E_\varepsilon\right) -\left(\partial^\alpha G,\varepsilon\partial^\alpha \left\{\nabla_x \times B_\varepsilon  - \tfrac{1}{\eps}G\right\} \right)\nonumber\\
		&&+\left(\partial^\alpha\left\{\frac1\varepsilon\nabla_x(\rho_\varepsilon^+-\rho_\varepsilon^-)+\frac1\varepsilon\nabla_x\cdot \mathcal{A}(\{{\bf I-P}\}f_\varepsilon\cdot q_1)-E_\varepsilon(\rho_\varepsilon^++\rho_\varepsilon^-)\right\},\varepsilon\partial^\alpha E_\varepsilon\right) \nonumber\\
		&&+\left(\partial^\alpha\left\{ -\frac2\varepsilon u _\varepsilon\times B _\varepsilon-\left\langle [v,-v]{M}^{1/2},\frac1{\varepsilon^2}\mathscr{L}f_\varepsilon-\frac1\varepsilon\mathscr{T}(f_\varepsilon,f_\varepsilon)\right\rangle\right\},\varepsilon\partial^\alpha E_\varepsilon\right) \nonumber\\
		&=&\frac{\d}{\d t}\left(\partial^\alpha G,\varepsilon\partial^\alpha E_\varepsilon\right) +\|\partial^\alpha G\|^2-\left(\partial^\alpha G,\varepsilon\partial^\alpha \left\{\nabla_x \times B _\varepsilon\right\} \right)\nonumber\\
		&&+\left(\partial^\alpha\left\{\frac1\varepsilon\nabla_x(\rho_\varepsilon^+-\rho_\varepsilon^-)+\frac1\varepsilon\nabla_x\cdot\mathcal{ A}(\{{\bf I-P}\}f_\varepsilon\cdot q_1)-E_\varepsilon(\rho_\varepsilon^++\rho_\varepsilon^-)\right\},\varepsilon\partial^\alpha E_\varepsilon\right) \nonumber\\
		&&+\left(\partial^\alpha\left\{ -\frac2\varepsilon u _\varepsilon\times B _\varepsilon-\left\langle [v,-v]{M}^{1/2},\frac1{\varepsilon^2}\mathscr{L}f_\varepsilon-\frac1\varepsilon\mathscr{T}(f_\varepsilon,f_\varepsilon)\right\rangle\right\},\varepsilon\partial^\alpha E_\varepsilon\right) \nonumber\\
		&\lesssim&\frac{\d}{\d t}\left(\partial^\alpha G,\varepsilon\partial^\alpha E_\varepsilon\right) -\|\partial^\alpha(\rho_\varepsilon^+-\rho_\varepsilon^-)\|^2+\|\nabla_xG\|_{H^{N}_x}^2+\eta\|\nabla_x B_\varepsilon\|^2_{H^{N-2}_x}+\eta\|\partial^\alpha E_\varepsilon\|^2\nonumber\\
		&&+\frac1{\varepsilon^2}\|\partial^\alpha\{{\bf I-P}\}f_\varepsilon\|^2_{D\vee\nu}+\|\partial^\alpha\nabla_x\{{\bf I-P}\}f_\varepsilon\|^2_{D\vee\nu}+\mathcal{E}_N(t)\mathcal{D}_N(t)\nonumber\\
	\end{eqnarray}
	where $|\alpha|\leq N-1$.
	Consequently
	\begin{eqnarray}\label{E-1-s}
		&&-\frac{\d}{\d t}\left(\partial^\alpha G,\varepsilon\partial^\alpha E_\varepsilon\right) +2\|\partial^\alpha E_\varepsilon\|^2+\|\partial^\alpha(\rho_\varepsilon^+-\rho_\varepsilon^-)\|^2\nonumber\\
		&\lesssim&\|\nabla_xG\|_{H^{N}_x}^2+\eta\|\nabla_x B_\varepsilon\|^2_{H^{N-2}_x}\nonumber\\
		&&+\frac1{\varepsilon^2}\|\partial^\alpha\{{\bf I-P}\}f_\varepsilon\|^2_{D\vee\nu}+\|\partial^\alpha\nabla_x\{{\bf I-P}\}f_\varepsilon\|^2_{D\vee\nu}+\mathcal{E}_N(t)\mathcal{D}_N(t).
	\end{eqnarray}
	For $B_\varepsilon$, it follows that for $|\alpha|\leq N-2$
	\begin{eqnarray}
		\left\|\partial^{\alpha}\nabla_xB_\varepsilon\right\|^2&=&\left(\partial^{\alpha}\nabla_xB_\varepsilon,\partial^{\alpha}\nabla_xB_\varepsilon\right)\nonumber\\
		&=&\left(\partial^{\alpha}\nabla_x\times B_\varepsilon,\partial^{\alpha}\nabla_x\times B_\varepsilon\right)\nonumber\\
		&=&\left(\partial^{\alpha}\left\{\partial_tE_\varepsilon+\frac1\varepsilon G\right\},\partial^{\alpha}\nabla_x\times B_\varepsilon\right)\nonumber\\
		&=&\frac{\d}{\d t}\left(\partial^{\alpha}E_\varepsilon,\partial^{\alpha}\nabla_x\times B_\varepsilon\right)-\left(\partial^{\alpha}E_\varepsilon,\partial^{\alpha}\nabla_x\times \partial_tB_\varepsilon\right)+\left(\partial^{\alpha}\left\{\frac1\varepsilon G\right\},\partial^{\alpha}\nabla_x\times B_\varepsilon\right)\nonumber\\
		&=&\frac{\d}{\d t}\left(\partial^{\alpha}E_\varepsilon,\partial^{\alpha}\nabla_x\times B_\varepsilon\right)-\left(\partial^{\alpha}E_\varepsilon,\partial^{\alpha}\nabla_x\times \left\{-\nabla_x\times E_\varepsilon\right\}\right)+\left(\partial^{\alpha}\left\{\frac1\varepsilon G\right\},\partial^{\alpha}\nabla_x\times B_\varepsilon\right)\nonumber\\
		&=&\frac{\d}{\d t}\left(\partial^{\alpha}E_\varepsilon,\partial^{\alpha}\nabla_x\times B_\varepsilon\right)+\|\partial^{\alpha}\nabla_x\times E_\varepsilon\|^2+\left(\partial^{\alpha}\left\{\frac1\varepsilon G\right\},\partial^{\alpha}\nabla_x\times B_\varepsilon\right)\nonumber\\
		&\lesssim&\frac{\d}{\d t}\left(\partial^{\alpha}E_\varepsilon,\partial^{\alpha}\nabla_x\times B_\varepsilon\right)+\|\partial^{\alpha}\nabla_x\times E_\varepsilon\|^2+\eta\|\partial^{\alpha}\nabla_x\times B_\varepsilon\|^2+\frac1{\varepsilon^2}\|\partial^{\alpha}G\|^2\nonumber
	\end{eqnarray}
	That is
	\begin{eqnarray}\label{B-s}
		&&-\frac{\d}{\d t}\left(\partial^{\alpha}E_\varepsilon,\partial^{\alpha}\nabla_x\times B_\varepsilon\right)+\left\|\partial^{\alpha}\nabla_xB_\varepsilon\right\|^2\nonumber\\
		&\lesssim&\|\partial^{\alpha}\nabla_x\times E_\varepsilon\|^2+\eta\|\partial^{\alpha}\nabla_x\times B_\varepsilon\|^2+\frac1{\varepsilon^2}\|\partial^{\alpha}G\|^2
	\end{eqnarray}
	for $|\alpha|\leq N-2$.
	
	For sufficiently small $\kappa>0$, (\ref{E-1-s})$+\kappa$(\ref{B-s}) gives
	\begin{equation}\label{EB-s-2}
		\begin{aligned}
			&\frac{\d}{\d t}G_{E_\varepsilon,B_\varepsilon}(t)+\left\|E_\varepsilon\right\|_{H^{N-1}_x}^2+\left\|\nabla_xB_\varepsilon\right\|_{H^{N-2}_x}^2+\left\|\rho_\varepsilon^+-\rho_\varepsilon^-\right\|^2_{H^{N-1}_x}\\[2mm]
			\lesssim&\frac1{\varepsilon^2}\sum_{|\alpha|\leq N}\|\partial^\alpha\{{\bf I-P}\}f_\varepsilon\|^2_{D\vee\nu}+\mathcal{E}_N(t)\mathcal{D}_N(t).
		\end{aligned}
	\end{equation}
	Here we set
	\[
	G_{E_\varepsilon,B_\varepsilon}(t)=-\left\{\sum_{|\alpha|\leq N-1}\left(\partial^\alpha G,\varepsilon\partial^\alpha E_\varepsilon\right) +\kappa\sum_{|\alpha|\leq N-2}\left(\partial^{\alpha}E_\varepsilon,\partial^{\alpha}\nabla_x\times B_\varepsilon\right)\right\}.
	\]
	A proper combination of \eqref{mac-alpha} and \eqref{EB-s-2} gives \eqref{mac-dissipation-1}.
	This completes the proof of Lemma \ref{mac-dissipation}.
\end{proof}

\section*{Acknowledgment}

The research of Ning Jiang was supported by grants from the National Natural Science Foundation of China under contract No.11471181 and No.11731008. This work is also supported by the Strategic Priority
Research Program of Chinese Academy of Sciences, Grant No.XDA25010404. The research of Yuanjie Lei was supported by the National Natural Science Foundation of China under contract No.11971187 and No.12171176.
\bigskip

\bibliography{reference}

\begin{thebibliography}{99}\small







































\bibitem{AMUXY-JFA-2012} R Alexandre, Y. Morimoto, S.  Ukai, C. J. Xu and T. Yang, The Boltzmann equation without angular cutoff in the whole space: I, Global existence for soft potential. {J Funct Anal}, 2012, 263: 915--1010

\bibitem{AV-CPAM2002} R. Alexandre and C. Villani,  On the Boltzmann equation for
long-range interaction.  {\em Commun. Pure and Appl. Math.} {\bf 55} (2002),  30-70.

\bibitem{Arsenio} D. Ars\'enio,
 From Boltzmann's Equation to the incompressible Navier-Stokes-Fourier system with long-range interactions. {\em Arch. Ration. Mech. Anal.} {\bf 206} (2012), no. 3, 367-488

\bibitem{AIM-ARMA-2015} D. Ars\'enio, S. Ibrahim and N. Masmoudi,
A derivation of the magnetohydrodynamic system from Navier-Stokes-Maxwell systems. {\em Arch. Ration. Mech. Anal.} {\bf 216} (2015), no. 3, 767-812.

\bibitem{Arsenio-Masmoudi} D. Ars\'enio and N. Masmoudi,
Regularity of renormalized solutions in the Boltzmann equation with long-range interactions. {\em Comm. Pure Appl. Math.} {\bf 65} (2012), no. 4, 508-548.


\bibitem{Arsenio-SRM-2016} D. Ars\'enio and L. Saint-Raymond,
{\em From the Vlasov-Maxwell-Boltzmann system to incompressible viscous electro-magneto-hydrodynamics. Vol. 1}. EMS Monographs in Mathematics. European Mathematical Society (EMS), Zurich, 2019.

\bibitem{BGL-1991-JSP} C. Bardos, F. Golse and C. D. Levermore,
Fluid dynamic limits of kinetic equations I: formal derivation. {\em J. Stat. Phys.}, {\bf 63} (1991), 323-344.

\bibitem {BGL-CPAM1993} C. Bardos, F. Golse, and C. D. Levermore,
Fluid Dynamic Limits of Kinetic Equations II: Convergence Proof
for the Boltzmann Equation, {\em Commun. Pure and Appl. Math.} {\bf 46}
(1993), 667-753

\bibitem {BGL3} C. Bardos, F. Golse, and C. D. Levermore,
The acoustic limit for the Boltzmann equation. {\em Arch. Ration.
Mech. Anal.} {\bf 153} (2000), no. 3, 177-204

\bibitem{b-u}
C. Bardos and S. Ukai, The classical incompressible Navier-Stokes limit of the Boltzmann equation.  {\em Math. Models Methods Appl. Sci.} {\bf 1} (1991), no.2, 235-257.

\bibitem{Biskamp} D. Biskamp,
{\em Nonlinear magnetohydrodynamics.} Cambridge Monographs on Plasma Physics, {\bf 1}. Cambridge University Press, Cambridge, 1993.

\bibitem{Boyer-Fabrie-2013BOOK} F. Boyer and P. Fabrie,
{\em Mathematical tools for the study of the incompressible Navier-Stokes equations and related models.} Applied Mathematical Sciences, {\bf 183}, Springer, New York, 2013.

\bibitem{Briant-JDE-2015} M. Briant,
From the Boltzmann equation to the incompressible Navier-Stokes equations on the torus: a quantitative error estimate. {\em J. Differential Equations}, {\bf 259} (2015), no. 11, 6072-6141.

\bibitem{Briant2017} M. Briant,
Perturbative theory for the Boltzmann equation in bounded domains with different boundary conditions. {\em Kinet. Relat. Models} {\bf 10} (2017), no. 2, 329-371.

\bibitem{Briant-Guo} M. Briant and Y. Guo,
Asymptotic stability of the Boltzmann equation with Maxwell boundary conditions. {\em J. Differential Equations} {\bf 261} (2016), no. 12, 7000-7079.

\bibitem{BMM-arXiv-2014} M. Briant, S. Merino-Aceituno and C. Mouhot,
From Boltzmann to incompressible Navier-Stokes in Sobolev spaces with polynomial weight. {\em Anal. Appl. (Singap.)} {\bf 17} (2019), no. 1, 85-116.

\bibitem{Caflisch}R. Caflisch,
The fluid dynamic limit of the nonlinear Boltzmann equation. {\em Comm. Pure Appl. Math.}, {\bf 33} (1980), no. 5, 651-666.

\bibitem{CIP} C. Cercignani, R. Illner and M. Pulvirenti. {\em The mathematical theory of dilute
gases.} Springer, New York, 1994.

\bibitem{Davidson} P. A. Davidson,
{\em An introduction to magnetohydrodynamics. }
Cambridge Texts in Applied Mathematics. Cambridge University Press, Cambridge, 2001.


\bibitem{DEL-89}
A. De Masi, R. Esposito, and J. L. Lebowitz,
{\it Incompressible Navier-Stokes and Euler limits of the Boltzmann equation.}
 Comm. Pure Appl. Math. {\bf 42} (1989), no. 8, 1189-1214.

\bibitem{DL-CPAM1989} R. DiPerna and P.-L. Lions,
Global weak solutions of Vlasov-Maxwell systems. {\em Comm. Pure Appl. Math.} {\bf 42} (1989), no. 6, 729-757.

\bibitem{DL-Annals1989} R. DiPerna and P.-L. Lions,
On the Cauchy problem for Boltzmann equations: global existence and weak stability. {\em Ann. of Math.} (2) {\bf 130} (1989), no. 2, 321-366.

\bibitem{DLYZ-KRM2013} R. J. Duan, S. Q. Liu, T. Yang and H. J. Zhao,
Stability of the nonrelativistic Vlasov-Maxwell-Boltzmann system for angular non-cutoff potentials. {\em Kinet. Relat. Models} {\bf 6} (2013), no. 1, 159-204.

\bibitem{DLYZ-CMP2017} R. J. Duan, Y. J. Lei, T. Yang and H.J. Zhao,
The Vlasov-Maxwell-Boltzmann system near Maxwellians in the whole space with very soft potentials. {\em Comm. Math. Phys.} {\bf 351} (2017), no. 1, 95-153.

\bibitem{Duan-Yang-Yu-arXiv-VMB} R. J. Duan, D. C. Yang and H. J. Yu, Compressible Euler-Maxwell limit for global smooth solutions to the Vlasov-Maxwell-Boltzmann system. arXiv:2207.01620 [math.AP]

\bibitem{DS-CPAM-11} R. J. Duan and R. M. Strain, Optimal large-time behavior of the Vlasov-Maxwell-Boltzmann system in the whole space. \textit{Comm. Pure. Appl. Math.} \textbf{24} (2011), no. 11, 1497--1546.

\bibitem{Duan-Yang-Zhao-MMMAS-2013}R. J. Duan, T. Yang and H.J. Zhao, The Vlasov-Poisson-Boltzmann system for soft potentials. {\em Math. Models Methods Appl. Sci.} {\bf 23} (2013), no. 6, 979--1028.

\bibitem{FLLZ-2018} Y. Z. Fan,  Y. J. Lei, S.Q. Liu and H.J. Zhao,
The non-cutoff Vlasov-Maxwell-Boltzmann system with weak angular singularity. {\em Sci. China Math.} {\bf 61} (2018), no. 1, 111-136.

\bibitem{GIM2014} P. Germain, S. Ibrahim, and N. Masmoudi,
Well-posedness of the Navier-Stokes-Maxwell equations. {\em Proc. Roy. Soc. Edinburgh Sect. A} {\bf 144} (2014), no. 1, 71-86.

\bibitem{Glassey-1996} R. T. Glassey,
{\em The Cauchy problem in kinetic theory.} Society for Industrial and Applied Mathematics (SIAM), Philadelphia, 1996.

\bibitem{GL} F. Golse and C. D. Levermore,
The Stokes-Fourier and acoustic limits for the Boltzmann
equation. {\em Comm. on Pure and Appl. Math.} {\bf 55} (2002), 336-393.

\bibitem{G-SRM-Invent2004} F. Golse and L. Saint-Raymond,
The Navier-Stokes limit of the Boltzmann equation for bounded collision kernels.  {\em Invent. Math.} {\bf 155} (2004), no. 1, 81--161.

\bibitem{G-SRM2009} F. Golse and L. Saint-Raymond,
The Incompressible Navier-Stokes Limit of the Boltzmann
Equation for Hard Cutoff Potentials.  {\em J. Math. Pures Appl.} (9) {\bf
91} (2009), no. 5, 508--552.

\bibitem{Grad-1958} H. Grad, Principles of the kinetic theory of gases. {\em Handbuch der Physik (herausgegeben von S. Fl$\ddot{u}$gge), Bd. 12, Thermodynamik der Gase}. Springer-Verlag, Berlin, 1958, pp,205-294.

\bibitem{Grad-1963} Grad H., Asymptotic theory of the Boltzmann equation II. Rarefied Gas Dynamics (Laurmann, J.A. Ed.) Vol. 1, Academic Press, New York, pp. 26--59, 1963.

\bibitem{Gressman_Strain-JAMS-2011} P. T. Gressman  and R. M. Strain, Global classical solutions of the Boltzmann equation without angular cut-off. {\it J. Amer. Math. Soc.} {\bf24} (2011), no. 3, 771--847.

\bibitem{GMM} M. P. Gualdani, S. Mischler and C. Mouhot,
Factorization for non-symmetric operators and exponential H-theorem. {\em M\'em. Soc. Math. Fr. (N.S.)} No. {\bf 153} (2017).

\bibitem{GJL-2019} M. Guo, N. Jiang and Y-L. Luo,
From Vlasov-Poisson-Boltzmann system to incompressible Navier-Stokes-Fourier-Poisson system: convergence for classical solutions. arXiv:2006.16514v2 [math.AP], submitted to {\em Asymptot. Anal.}

\bibitem{Guo-2003-Invent} Y. Guo,
The Vlasov-Maxwell-Boltzmann system near Maxwellians. {\em Invent. Math.} {\bf 153}(2003), no. 3, 593-630.

\bibitem{Guo-IUMJ-2004} Y. Guo, The Boltzmann equation in the whole space. {Indiana Univ Math J}, 2004, 53: 1081--1094

\bibitem{Guo-2006-CPAM} Y. Guo,
Boltzmann diffusive limit beyond the Navier-Stokes approximation. {\em Comm. Pure Appl. Math.} {\bf 59}(2006), no. 5, 626-687.

\bibitem{GJJ-KRM2009} Y. Guo, J. Jang and N. Jiang,
Local Hilbert expansion for the Boltzmann equation. {\em Kinet. Relat. Models} {\bf 2} (2009), no. 1, 205-214.

\bibitem{GJJ-CPAM2010} Y. Guo, J. Jang and N. Jiang,
Acoustic limit for the Boltzmann equation in optimal scaling. {\em Comm. Pure Appl. Math.} {\bf 63} (2010), no. 3, 337-361.
\bibitem{Guo-2012-JAMS}Y. Guo, The Vlasov-Poisson-Landau system in a periodic box. {\em  J. Amer. Math. Soc.} {\bf 25} (2012), no. 3, 759--812.

\bibitem{Guo-Wang-CPDE-2012} Y. Guo and Y.-J Wang, Decay of dissipative equation and negative sobolev spaces. \textit{Comm.Partial Differential Equations} \textbf{37} (2012), 2165--2208.

\bibitem{Hosono-Kawashima-2006} T. Hosono and S. Kawashima, Decay property of regularity-loss type and application to
some nonlinear hyperbolic-elliptic system. {\it Math. Models Methods Appl. Sci.} {\bf 16} (2006), no.
11, 1839--1859.

\bibitem{Ibrahim-Keraani-2011-SIMA} S. Ibrahim  and S. Keraani,
Global small solutions for the Navier-Stokes-Maxwell system. {\em SIAM J. Math. Anal.} {\bf 43} (2011), no. 5, 2275-2295.

\bibitem{Ibrahim-Yoneda-2012-JMAA}S. Ibrahim and T. Yoneda,
Local solvability and loss of smoothness of the Navier-Stokes-Maxwell equations with large initial data. {\em J. Math. Anal. Appl.} {\bf 396} (2012), no. 2, 555-561.

\bibitem{Jang-ARMA2008} J. Jang,
Vlasov-Maxwell-Boltzmann diffusive limit. {\em Arch. Ration. Mech. Anal.} {\bf 194} (2009), no. 2, 531-584.

\bibitem{JJ-DCDS2009} J. Jang and N. Jiang,
Acoustic limit of the Boltzmann equation: classical solutions. {\em Discrete Contin. Dyn. Syst.} {\bf 25} (2009), no. 3, 869-882.

\bibitem{Jang-Masmoudi-SIMA2012} J. Jang and N. Masmoudi,
Derivation of Ohm's law from the kinetic equations.  {\em SIAM J. Math. Anal.} {\bf 44} (2012), no. 5, 3649-3669.

\bibitem{JLM-CPDE2010} N. Jiang, C. D. Levermore and N. Masmoudi,
Remarks on the acoustic limit for the Boltzmann equation. {\em Comm. Partial Differential Equations} {\bf 35} (2010), no. 9, 1590-1609.
\bibitem{Jiang-Luo-2022-Ann.PDE} N. Jiang and Y-L. Luo, From Vlasov-Maxwell-Boltzmann system to two-fluid incompressible Navier-Stokes-Fourier-Maxwell system with Ohm's law: convergence for classical solutions. Ann. PDE 8 (2022), no. 1, Paper No. 4, 126 pp.
\bibitem{JL-CMS-2018} N. Jiang and Y-L. Luo,
Global classical solutions to the two-fluid incompressible Navier-Stokes-Maxwell system with Ohm's law. {\em  Commun. Math. Sci. }, {\bf 16} (2018), no. 2, 561-578.

\bibitem{JLZ-VMB-2020} N. Jiang, Y-L. Luo and T-F. Zhang,
Incompressible Navier-Stokes-Fourier-Maxwell system with Ohm's law limit from Vlasov-Maxwell-Boltzmann system: Hilbert expansion approach. arXiv:2007.02286 [math.AP]

\bibitem{JM-CPAM2017} N. Jiang and N. Masmoudi,
Boundary layers and incompressible Navier-Stokes-Fourier limit of the Boltzmann equation in bounded domain I. {\em Comm. Pure Appl. Math.} {\bf 70} (2017), no. 1, 90-171.

\bibitem{JX-SIMA2015}N. Jiang and L. J. Xiong,
Diffusive limit of the Boltzmann equation with fluid initial layer in the periodic domain. {\em SIAM J. Math. Anal.} {\bf 47} (2015), no. 3, 1747-1777.

\bibitem{JXZ-Indiana2018} N. Jiang, C-J, Xu and H. J. Zhao,
Incompressible Navier-Stokes-Fourier limit from the Boltzmann equation: classical solutions.   {\em Indiana University Mathematical Journal.} {\bf 67} (2018), no. 5, 1817-1855

\bibitem{JZ-2020} N. Jiang and X. Zhang,
Sensitivity analysis and incompressible Navier-Stokes-Poisson limit of Vlasov-Poisson-Boltzmann equations with uncertainty. arXiv:2007.00879 [math.AP]

\bibitem{KMN} S. Kawashima, A. Matsumura and T. Nishida,
On the fluid dynamical approximation to the Boltzmann equation at the level of the Navier-Stokes equation. {\em Comm. Math. Phys.}, {\bf 70} (1979), 97-124.
\bibitem{Lei-Zhao-JFA-2014}Y. J. Lei and H. J. Zhao, Negative Sobolev spaces and the two-species Vlasov-Maxwell-Landau system in the whole space. {\it J. Funct. Anal.} {\bf 267} (2014), no. 10, 3710--3757.
\bibitem{Liu-Yang-Yu-PhysicaD} T-P. Liu, T. Yang and S-H. Yu,
Energy method for Boltzmann equation. {\em Phys. D} {\bf 188} (2004), no. 3-4, 178-192.
\bibitem{LM} C. D. Levermore and N. Masmoudi,
From the Boltzmann equation to an incompressible Navier-Stokes-Fourier system.
{\em Arch. Ration. Mech. Anal.}, {\bf 196} (2010), no. 3, 753-809.

\bibitem{Levermore-Sun-2010-KRM} C. D. Levermore and W. Sun,
Compactness of the gain parts of the linearized Boltzmann operator with weakly cutoff kernels. {\em Kinet. Relat. Models}, {\bf 3} (2010), no. 2, 335-351


\bibitem{Lions-Kyoto1994} P.-L. Lions,
Compactness in Boltzmann's equation via Fourier integral operators and applications. I, II. {\em J. Math. Kyoto Univ.} {\bf 34} (1994), no. 2, 391-427, 429-461.

\bibitem{Lions-Kyoto1994-2} P.-L. Lions,
Compactness in Boltzmann's equation via Fourier integral operators and applications. III. {\em J. Math. Kyoto Univ.} {\bf 34} (1994), no. 3, 539-584.

\bibitem{Lions-1996} P.-L. Lions,
{\em Mathematical topics in fluid mechanics. Vol. 1. Incompressible models.} Oxford Lecture Series in Mathematics and its Applications, 3. Oxford Science Publications. The Clarendon Press, Oxford University Press, New York, 1996.

\bibitem{LM3} P. L. Lions and N. Masmoudi,
From Boltzmann equation to Navier-Stokes and Euler equations
I. {\em Arch. Ration. Mech. Anal.}, {\bf 158} (2001), 173-193.

\bibitem{LM4} P. L. Lions and N. Masmoudi,
From Boltzmann equation to Navier-Stokes and Euler equations
II. {\em Arch. Ration. Mech. Anal.}, {\bf 158} (2001), 195-211.

\bibitem{Masmoudi-JMPA2010} N. Masmoudi,
Global well posedness for the Maxwell-Navier-Stokes system in 2D. {\em J. Math. Pures Appl. (9)}, {\bf 93} (2010), no. 6, 559-571.

\bibitem{Masmoudi-SRM-CPAM2003} N. Masmoudi and L. Saint-Raymond,
From the Boltzmann equation to the Stokes-Fourier system in a bounded domain.
{\em Comm. Pure Appl. Math.} {\bf 56} (2003), no. 9, 1263-1293.


\bibitem{mischler2010asens}	S.~Mischler,
\newblock Kinetic equations with {M}axwell boundary conditions.
\newblock \emph{Ann. Sci. \'Ec. Norm. Sup\'er. (4)} \textbf{43} (2010), 719-760.

\bibitem{Mischler-Mouhot} S. Mischler and C. Mouhot,
Kac's program in kinetic theory.  {\em Invent. Math.} {\bf 193} (2013), no. 1, 1-147.

\bibitem{Nishida} T. Nishida,
Fluid dynamical limit of the nonlinear Boltzmann equation to the level of the compressible Euler equation. {\em Comm. Math. Phys.}, {\bf 61} (1978), 119-148.

\bibitem{SRM2010} L. Saint-Raymond,
Some recent results about the sixth problem of Hilbert: hydrodynamic limits of the Boltzmann equation. {\em  European Congress of Mathematics}, 419-439, Eur. Math. Soc., Zurich, 2010.

\bibitem{Simon-1987-AMPA} J. Simon,
Compact sets in the space $L^p (0,T;B)$. {\em Ann. Mat. Pura Appl.}, {\bf 146} (1987), no. 4, 65-96.
\bibitem{Strain-CMP-2006} R. M. Strain, The Vlasov-Maxwell-Boltzmann system in the whole space. {\it Comm. Math. Phys.} {\bf 268} (2006), no. 2, 543--567.
\end{thebibliography}

\end{document}